\documentclass[11pt]{amsart}
\usepackage{amsmath,amssymb}
\usepackage{graphicx}
\usepackage{amsfonts,amscd,latexsym,epsfig, psfrag}

\usepackage{hyperref}
\usepackage[all]{xy}

\newcommand{\er}{{\Diamond}}

\newcommand{\ui}{{\underline i}}
\newcommand{\uj}{{\underline j}}
\newcommand{\ux}{{\underline x}}
\newcommand{\uy}{{\underline y}}
\newcommand{\ug}{{\underline g}}
\newcommand{\uh}{{\underline h}}

\newcommand{\uk}{{\underline k}}
\newcommand{\uv}{{\underline v}}

\newcommand{\ue}{{\underline e}}

\newcommand{\BVG}{{\bB_V^{{\raise.6ex\hbox{\tiny{$\less G$}}}}}}
\newcommand{\GVG}{{\bG_V^{{\raise.6ex\hbox{\tiny{$\less G$}}}}}}

\newcommand{\Br}{{\rm Br}}
\newcommand{\cl}{{\rm cl}}
\newcommand{\Fr}{{\rm Fr}}

\newcommand{\Zz}{{\mathcal Z}}

\newcommand{\Mor}{{\rm Mor}}
\newcommand{\Obj}{{\rm Obj}}

\newcommand{\TQ}{{\Tilde Q}}

\renewcommand{\Hat}{\widehat}

\newcommand{\imm}{{\rm im\,}}

\newcommand{\PSL}{{\rm PSL}}

\newcommand{\Cc}{{\mathcal C}}
\newcommand{\Kk}{{\mathcal K}}

\newcommand{\Jj}{{\mathcal J}}
\newcommand{\Rr}{{\mathcal R}}

\newcommand{\im}{{\rm im\,}}
\newcommand{\less}{{\smallsetminus}}

\newcommand{\supp}{{\rm supp\,}}

\newcommand{\p}{{\partial}}
\newcommand{\al}{{\alpha}}

\newcommand{\be}{{\beta}}

\newcommand{\eps}{{\varepsilon}}
\newcommand{\de}{{\delta}}

\newcommand{\ga}{{\gamma}}
\newcommand{\Ga}{{\Gamma}}
\newcommand{\io}{{\iota}}
\newcommand{\ka}{{\kappa}}
\newcommand{\la}{{\lambda}}
\newcommand{\La}{{\Lambda}}
\newcommand{\si}{{\sigma}}
\newcommand{\Si}{{\Sigma}}

\newcommand{\bY}{{\bf Y}}

\newcommand{\bW}{{\bf W}}

\newcommand{\Uu}{{\mathcal U}}
\newcommand{\Bb}{{\mathcal B}}
\newcommand{\Ww}{{\mathcal W}}
\newcommand{\Oo}{{\mathcal O}}
\newcommand{\Mm}{{\mathcal M}}
\newcommand{\Ss}{{\mathcal S}}
\newcommand{\oMm}{{\overline {\Mm}}}
\newcommand{\ov}{\overline}
\newcommand{\wh}{\widehat}

\newcommand{\id}{{\rm id}}
\newcommand{\rd}{{\rm d}}
\newcommand{\rD}{{\rm D}}
\newcommand{\rT}{{\rm T}}

\renewcommand{\Tilde}{\widetilde}

\newcommand{\TV}{{\Tilde V}}

\newcommand{\coker}{{\rm coker\,}}

\newcommand{\Ee}{{\mathcal E}}
\newcommand{\Ii}{{\mathcal I}}

\newcommand{\Qq}{{{\mathcal Q}}}
\newcommand{\HatQq}{{\wh{\Qq}}}

\newcommand{\Yy}{{\mathcal Y}}
\newcommand{\Vv}{{\mathcal V}}
\newcommand{\Xx}{{\mathcal X}}
\newcommand{\N}{{\mathbb N}}

\newcommand{\Q}{{\mathbb Q}}
\newcommand{\R}{{\mathbb R}}
\newcommand{\C}{{\mathbb C}}
\newcommand{\E}{{\mathbb E}}
\newcommand{\F}{{\mathbb F}}

\newcommand{\Z}{{\mathbb Z}}

\newcommand{\Hom}{{\rm Hom}}

\newcommand{\Pp}{{\mathcal P}}

\newcommand{\uo}{{\underline{o}}}

\newcommand{\uw}{{\underline{w}}}
\newcommand{\bB}{{\bf B}}
\newcommand{\bG}{{\bf G}}

\newcommand{\bE}{{\bf E}}

\newcommand{\bX}{{\bf X}}
\newcommand{\bQ}{{\bf Q}}

\newcommand{\Aut}{{\rm Aut}}

\newcommand{\s}{{\mathfrak s}}

\newcommand{\fo}{{\mathfrak o}}

\newtheorem{theorem}{Theorem}[subsection]
\newtheorem{thm}[theorem]{Theorem}

\newtheorem{cor}[theorem]{Corollary}
\newtheorem{lemma}[theorem]{Lemma}

\newtheorem{prop}[theorem]{Proposition}

\newtheorem{definition}[theorem]{Definition}
\newtheorem{defn}[theorem]{Definition}
\newtheorem{example}[theorem]{Example}

\newtheorem{remark}[theorem]{Remark}
\newtheorem{rmk}[theorem]{Remark}

\numberwithin{figure}{subsection}
\numberwithin{equation}{subsection}

\newcommand{\MS}{{\medskip}}

\newcommand{\NI}{{\noindent}}

\newcommand{\Ti}{\widetilde}

\newcommand{\pr}{{\rm pr}}

\newcommand{\lm}{\Lambda^{\rm max}\,}

\newcounter{qcounter}
\newenvironment{enumlist}
  { \begin{list} {\rm(\alph{qcounter})}{\usecounter{qcounter} 
         \setlength{\topsep}{.5ex}  \setlength{\itemsep}{.5ex} \setlength{\leftmargin}{4ex} } }
   { \end{list} }
   
\newcounter{qicounter}
\newenvironment{enumilist}
   { \begin{list} {\rm (\roman{qicounter})}{\usecounter{qicounter}
     \setlength{\itemsep}{.5ex} \setlength{\leftmargin}{4ex} } }
   { \end{list} }
   
\newcounter{nqicounter}
\newenvironment{nenumilist}
   { \begin{list} {\rm (\roman{nqicounter})}{\usecounter{nqicounter}
     \setlength{\itemsep}{.5ex} \setlength{\leftmargin}{0.5ex} } }
   { \end{list} }

\newenvironment{itemlist}
   { \begin{list} {$\bullet$}
         { \setlength{\topsep}{.5ex}  \setlength{\itemsep}{.5ex} \setlength{\leftmargin}{2.5ex} } }
   { \end{list} }

\newcommand{\leftsub}[2]{{\vphantom{#2}}_{#1}{#2}}

\newcommand\quotient[2]{
        \mathchoice
            {
                \text{\raise1ex\hbox{$#1$}\Big/\lower1ex\hbox{$#2$}}
            }
            {
                #1\,/\,#2
            }
            {
                #1\,/\,#2
            }
            {
                #1\,/\,#2
            }
    }

\newcommand\quot[2]{
                \text{\raise1ex\hbox{$#1$}/\lower1ex\hbox{$\scriptstyle#2$}}
  }

\newcommand\qu[2]{
                \text{\raise.8ex\hbox{$\scriptstyle#1\!$}/\lower.8ex\hbox{$\!\scriptstyle#2$}}
  }

\newcommand\qq[2]{
                \text{\raise.8ex\hbox{$#1\!$}/\lower.8ex\hbox{$#2$}}
}

 \title[Polyfold fundamental classes and globally structured perturbations]{Polyfold fundamental classes and globally structured multivalued perturbations}
  \author{Dusa McDuff}
 \address{Department of Mathematics,
 Barnard College, Columbia University}
\email{dusa@math.columbia.edu}
 \author{Katrin Wehrheim}
  \address{Department of Mathematics,
University of California, Berkeley}
\email{katrin@math.berkeley.edu}
\thanks{This material is based upon work supported by the National Science Foundation under Grants No.DMS-1708916, and DMS-1440140 while the first author was in residence at the Mathematical Sciences Research Institute in Berkeley, California, during the Spring 2018 semester.}
\subjclass{46,53,58}
\date{June 21, 2024}

\begin{document}

\maketitle
\tableofcontents

\begin{abstract}  Work of Hofer--Wysocki--Zehnder  has shown that many spaces of pseudoholomorphic curves that arise when studying  symplectic manifolds may be described as the zero set of a polyfold Fredholm section.  This framework has many analytic advantages. However the methods they develop to extract useful topological information from it are rather cumbersome. 
This paper develops a general construction of a finite dimensional space of multivalued perturbations of a polyfold Fredholm section 
such that almost all elements are regularizing. These perturbation are globally structured and explicitly described, and, in cases where the moduli space has no formal boundary,  permit a transparent definition of its (rational \v{C}ech) fundamental class.
\end{abstract}

\section{Introduction}

Polyfold theory, initiated by Hofer-Wysocki-Zehnder 
in the early 2000s, generalizes the notion of a smooth section of an orbibundle to obtain a nonlinear Fredholm theory for functors between groupoids that promises to make the use of pseudoholomorphic curves in symplectic geometry more widely accessible by combining the following features:

\begin{enumlist}
\item The theory applies to compactified moduli spaces $\oMm$ of pseudoholomorphic curves in the sense that $\oMm \simeq \sigma^{-1}(0)$ is homeomorphic to the solution set of 
a so-called scale-smooth (or sc-) Fredholm section $\sigma:\Bb\to \Ee$ of a strong polyfold bundle $\Ee\to \Bb$.
\footnote{While expected for all known types of compact moduli spaces in symplectic geometry, only Gromov-Witten moduli spaces currently have a published description \cite{HWZgw}
of this form.
There is work in progress \cite{FHsft,jiayong} on the moduli spaces of SFT, Floer theories, and Fukaya categories. 
}

\item The theory provides a transparent approach to regularizing such moduli spaces by (multivalued) perturbations of the section $\sigma$ that achieve transversality while preserving compactness 
and most geometric characteristics 
of the perturbed solution set.\footnote{This theory is fully developed and documented in \cite{TheBook} and earlier publications cited therein.}
\end{enumlist}

Two key points about this theory contrast sharply with the usual treatments of Fredholm operators on Banach spaces.\footnote{For an introduction to scale calculus -- the notion of differentiability on Banach spaces used by \cite{TheBook} -- see e.g.\ \cite{usersguide}.} 

$\bullet$ It is based on a new notion of sc-smoothness, which is designed firstly so that there is an appropriate analog of Fredholm theory, and secondly so that reparametrization groups (such as $\PSL(2,\C)$ act sc-smoothly on spaces $E$ of sc-smooth maps $u:\C P^1\to M$.  This is in sharp contrast to the action of such groups on Banach spaces.  For example the rotation group $S^1$ does not act even continuously on the space of $L^2$-maps $S^1\to \R$.

$\bullet$  The fact that $M$-polyfolds are modelled on the image $r(U)$ of an open subset $U\subset E$ under a scale smooth retraction (rather than simply on the open set $U$ itself) allows one to describe a neighbourhood of a nodal curve that includes stable maps with topologically different domains  as an open subset of a single $M$-polyfold.

Thus  polyfold theory provides good local models for the moduli space (i.e. the solution set  of the given operator), and the challenge is to understand the global structure.  There are other theories with analogous features for different notions of ``Fredholm section'', in particular various versions of Kuranishi theory \cite{fooo,MWiso,pardon} which combines local finite dimensional models into a global section functor $\sigma$ between categories (that are generally neither \'etale nor groupoids).  
Although the current ideas may be used to construct a  (tame) smooth Kuranishi atlas that models the zero set of $\si$,
in this paper, we work only in the infinite dimensional setting, constructing a global stabilization of the operator $\si$, namely a finite dimensional family of multi-sections  such that a generic element is a transverse perturbation.
\MS

\MS

Beyond Gromov-Witten theory, most moduli spaces in symplectic geometry  have formal boundary, and in fact come in `coherent families' whose boundary strata are Cartesian or fiber products of other moduli spaces in the family. Extracting algebraic invariants -- such as Floer homology, or a well defined SFT algebra -- from such families of moduli spaces requires `coherent perturbations' that preserve the identifications of boundary strata. The corresponding challenge of extending multisections prescribed on boundaries is present in both polyfold and Kuranishi theory, and is resolved in \cite{TheBook} by a class of so-called `structurable multisections'. 
However, this involves rather elaborate definitions and constructions, and at present does not yield a transparent general existence result for coherent perturbations. 

This problem of constructing coherent multivalued perturbations is simplified by our construction of a global Fredholm stabilization in Theorem~\ref{thm:globstab}.
This provides a finite dimensional family of perturbations that is sufficiently large to achieve transversality, yet is multivalued in the simplest possible way: each perturbation is the orbit of a single-valued section under a global finite group action. 
This type of multisection has natural extensions from boundaries and corners, so that a general existence proof of coherent perturbations reduces to the project  of extending the global stabilization construction to  coherent setups. 

\subsection{Overview of constructions}\label{ss:overview}

 As defined in~\cite[\S16]{TheBook}, a polyfold $\Bb$ is a topological space that is homeomorphic to the realization of a  topological groupoid $\Xx$ that is equipped with a special kind of smooth structure called a sc-smooth structure.
 A local model (without formal boundary) for a global Fredholm section $\sigma:\Bb\to\Ee$ with zero set $\sigma^{-1}(0)=\oMm$ is given by the quotient by a finite group $G$ of a $G$-equivariant scale-smooth section 
$$
U \to \Rr=\bigcup_{x\in U} \{x\}\times \Pi_x(\F), x \mapsto (x, f(x)).
$$ 
Its base $U=r(\Ti U)$ is the image of a scale-smooth retraction $r:\Ti U\to\Ti U$, $r\circ r = r$, on an open subset $\Ti U\subset\E$ of a scale-Banach space. The total space $\Rr$ 
is the image of a scale-smooth retraction $\Ti U\times\F \to \Ti U\times\F, (x,y) \mapsto (r(x), \Pi_x(y))$ given by linear projections $\Pi_x:\F\to\F$ on another scale-Banach space. The polyfold-Fredholm property of this section requires the existence of a scale-Fredholm\footnote{The Fredholm property in scale calculus implies classical Fredholmness of the linearizations, but due to the lack of classical differentiability requires a `contraction germ normal form'. In practice \cite{HWZgw,w:fred} it follows from `classical differentiability in all but finitely many directions'.
} 
extension $\Ti f: \Ti U\to \F$ with $$
\Ti f^{-1}(0)=f^{-1}(0),\quad \ker\rd_x \Ti f = \ker\rd_x f, \quad \qu{\Pi_x(\F)}{\im\rd_x f} \simeq \qu{\F}{\im\rd_x\Ti f}\,,
$$
 and the resulting local chart for the moduli space is an embedding $f^{-1}(0)/G \subset \sigma^{-1}(0)=\oMm$. 

The crucial construction that yields both transverse perturbations of the polyfold-Fredholm section $\sigma$ and its finite dimensional reductions (Kuranishi charts) is that of a {\bf local Fredholm stabilization} near $x\in\sigma^{-1}(0)$.  Here we construct a finite dimensional vector space $E$ as the product of  
a representative of the cokernel $\qu{\Pi_x(\F)}{\im\rd_x f}$ with the group $G$ that acts on $U$ as above, so that $G$ acts diagonally on 
$U \times E$.  We then define a $G$-equivariant bundle map 
$\tau : U \times E \to \Rr$ covering the identity,  so that $\sigma-\tau$ is transverse to the zero section of $\Rr$ near $(x,0)\in U\times E$.  (For details see Lemma~\ref{lem:localtau}.)
Each $e\in E$ then yields a {\bf local multivalued perturbation} $U \ni x \mapsto \bigl( \tau(x, g e ) \bigr)_{g\in G}\subset \Rr$, which yields $G$-equivariant transversality of $\sigma|_U - \bigl(\tau(\cdot,ge)\bigr)_{g\in G}$ for regular values of the map $(\sigma|_U-\tau)^{-1}(0)\to E$. Alternatively, the same construction can be viewed as giving a {\bf local finite dimensional reduction} 
\begin{align*}
f_\tau \,: &\; U_\tau:= (\sigma|_U-\tau)^{-1}(0)  \to E,\quad  (e,x) \mapsto e ,\\
& \quad \text{with}\quad f_\tau^{-1}(0) / G  \simeq (\sigma^{-1}(0) \cap U)/ G \subset \oMm. 
\end{align*}
\NI
Once this local structure is understood, the main difficulty with both the Kuranishi atlas and polyfold approaches  is to patch such local perturbations into a globally transverse perturbation, while preserving compactness of the perturbed solution space.

The key idea of the current paper  is to model the base polyfold $\Bb$  by a groupoid $\Xx_\Vv$ whose structure is so transparent  that the pullback of the  Fredholm section $\sigma:\Bb\to\Ee$  has a natural family of globally transverse perturbations. Here are the main constructions: 

\begin{enumilist}
\item
Given finitely many local polyfold charts $U_i/G_i\hookrightarrow\Bb$ covering a compact subset $\oMm\subset \bigcup_i U_i/G_i$, Theorem~\ref{thm:reduce} constructs an \'etale groupoid $\Xx_\Vv$ 
with explicitly described object and morphism spaces and
 whose realization $|\Xx_\Vv|$ is equivalent as polyfold to an open neighbourhood $|V|\subset\Bb$ of $\oMm$. 

Moreover, Theorem~\ref{thm:reduceFred} constructs a Fredholm section functor $f_\Vv:\Xx_\Vv\to\Ww_\Vv$ such that the perturbation theory for $|f_\Vv^{-1}(0)|\simeq\oMm$ is equivalent to that for $\sigma$ in the sense that appropriate perturbations for each section have naturally identified solution sets. 

\item
Given data as in (i) and local Fredholm stabilizations $\tau_i : U_i \times E_i \to \Ee|_{U_i}$, Theorem~\ref{thm:globstab} constructs a global Fredholm stabilization as follows: 
There is an \'etale (but generally not proper) subgroupoid $\Xx_\Vv ^{\less G}$ of $\Xx_\Vv$ obtained by removing all morphisms that give rise to isotropy. 
It carries an action of $G:=\prod G_i$ such that $|\Xx_\Vv ^{\less G}|\,/G \simeq |\Xx_\Vv|$. 
We then construct a finite dimensional vector space $E:=\prod_i E_i$ and a $G$-equivariant functor
$\tau: \Xx_\Vv^{\less G}\times E \to \Ww_\Vv$  so that $f_\Vv-\tau$ is transverse
to the zero section.

\item 
Given data as in (ii), Corollary~\ref{cor:multisection} shows that $\tau$ applied to generic, sufficiently small  $e\in E$ induces canonically structured sc$^+$-multisection functors $\La_{\Vv,e}: \Ww_\Vv\to \Q^{\geq 0}$ and $\La_{\ue} : \Ww|_{V}\to \Q^{\ge 0}$ that satisfy the requirements of the polyfold perturbation scheme and have 
easily described solution sets. 

\item
When $\oMm$ has no formal boundary, Theorem~\ref{thm:polyVFC} defines a polyfold fundamental class in the rational \c{C}ech homology of $\oMm$ as the inverse limit of the fundamental classes of the perturbed solution sets. 
In this setting, (iii) identifies the polyfold fundamental classes $[\oMm]_{f_\Vv}=[\oMm]_{\sigma}$, where the latter depends on the initial  polyfold Fredholm section $\si:\Bb\to \Ee$.\footnote
{
The book \cite{TheBook} does not attempt to define such a fundamental class, instead showing in this situation that there is a well defined notion of integration of sc-smooth forms over $\oMm$.}
\end{enumilist}

As outlined in \cite[Rmk.1.3.8]{Mfund},
the current results  easily adapt to construct a Kuranishi atlas
\footnote{Another abstract construction of this type, without identification of perturbations, can be found in \cite{Yang}.} 
$\Kk_\tau$ for $\oMm$ whose perturbation theory is equivalent to the polyfold perturbation theory
for $f_\Vv$ and thus for $\sigma$.
\MS

\NI{\bf Note to the reader:} This paper contains many technical details that  readers should not attempt to master until the outlines of the construction are clear.  The main theorems are clearly stated in  \S\ref{ss:main}. The language here is very formal; readers can consult 
\S\ref{ss:scale} and \S \ref{ss:basic} for background definitions.  We recommend that 
the reader goes next to \S\ref{ss:mainres}, where the main constructions are defined in an abstract setting.  This section uses the language developed in  \S \ref{ss:basic}--\ref{ss:groupact} concerning
 \'etale categories and groupoid completions.   In \S\ref{sec:strict}, we first show in Theorem~\ref{thm:polyVFC}  that the perturbed zero sets of polyfold Fredholm sections carry a fundamental class in rational \v{C}ech homology, and then complete the proofs of the main theorems by applying the abstract constructions developed in \S\ref{sec:genconstr} to the polyfold setting.
\MS

\NI {\bf Acknowledgement:} This paper is the joint work of Dusa McDuff and Katrin Wehrheim. However this final version was prepared by the first named author, who takes responsibility for all inaccuracies and omissions. She warmly thanks Zhengyi Zhou for his many careful comments on an earlier version.

\section{Terminology and main results} \label{ss:terminology}

\subsection{Statement of main theorems}\label{ss:main}  
A sc-Fredholm section $\sigma:\Bb\to\Ee$ of a strong polyfold bundle 
 is a
continuous map between topological spaces together with an equivalence class of functors $f:\Xx\to \Ww$ between topological groupoids, whose realization $|f|:|\Xx|\to|\Ww|$ together with homeomorphisms $|\Xx|\simeq\Bb$, $|\Ww|\simeq\Ee$ induces $\sigma$. 
Throughout this paper, we will work with a fixed representative $f:\Xx\to \Ww$ that is a sc-Fredholm section functor of a strong bundle $\Pp:\Ww\to \Xx$ over an ep-groupoid $\Xx$ modeled on M-polyfolds.\footnote{Questions of equivalence between such representatives are addressed in \cite[\S10-11]{TheBook}.} 
To summarize these notions we use conventions of \cite{TheBook} in denoting object and morphism spaces as $\Obj_\Xx=X$ and $\Mor_\Xx=\bX$; but we will differentiate between object space $X$ and groupoid $\Xx=(X,\bX)$.   In particular, the realization $|\Xx| = X/\!\!\sim\,$ 
is the quotient of $X$ by the equivalence relation generated by the morphisms $\bX$, equipped with the quotient topology.
The following generalizes the notions of orbifolds, orbibundles, and smooth sections to infinite dimensions and scale calculus. The underlying analytic notions of M-polyfold, sc-Fredholm, etc.\ are reviewed in \S\ref{ss:scale} along with the finite dimensional cases.

\begin{definition} \label{def:poly}
\begin{nenumilist}
\item
An {\bf ep-groupoid} as in \cite[Def.7.1.3]{TheBook} is a groupoid $\Xx=(X,\bX)$ equipped with M-polyfold structures\footnote{
As in other regularization constructions of polyfold theory, we assume that $X$
admits sc-smooth partitions of unity. This can be guaranteed by working with M-polyfold charts in Hilbert spaces; see \cite[\S5.5]{TheBook}. In classical PDE applications that means working with Sobolev spaces of class $L^p$ with $p=2$.
Moreover, the boundaries of $X$ and $\bX$ are assumed to be tame in the sense of \cite[\S2.5]{TheBook}.
In the known applications, this is guaranteed by retracts of splicing form; see \cite[Lemma~2.5.5]{TheBook}.} on the object and morphism sets so that 
\begin{itemize}
\item 
source and target maps $s,t:\bX\to X$ are local sc-diffeomorphisms and the unit $X\to \bX, q\mapsto\id_q$, inverse $\bX \to\bX, m\mapsto m^{-1}$, and composition $\bX \leftsub{t}{\times}_s \bX \to \bX, (m,m')\mapsto m\circ m'$ are sc-smooth maps (i.e.\ the groupoid is {\rm \'etale}); 
\item 
every $x\in X$ has a neighbourhood $N(x)$ so that $t:s^{-1}\bigl( {\rm cl}_X(N(x))\bigr) \to X$ is {\rm proper}.
\end{itemize}
In addition, we will throughout assume that the realization $|\Xx|$ is paracompact so that $|\Xx|$ is metrizable by 
 \cite[Thm.7.3.1]{TheBook}.
We call an ep-groupoid $\Xx=(X,\bX)$ {\bf nonsingular} if the {\bf isotropy groups} $G_x:=s^{-1}(x)\cap t^{-1}(x)\subset \bX$ are trivial for all $x\in X$.
Finally, we will say that $\Xx$ has no formal boundary if the boundary $\partial X$ of its object set as defined in \cite[\S2.4]{TheBook} is empty. 

\item
A {\bf strong bundle} as in 
\cite[Def.8.3.1]{TheBook}\footnote{This summarizes the discussion following \cite[Def.8.3.1]{TheBook}; for more details see Definition~\ref{def:bundle}.} 
over the ep-groupoid $\Xx$ is a pair $(P,\mu)$ of a strong bundle $P:W\to X$ and a strong bundle map $\mu:\bX\leftsub{s}{\times}_P W\to W$ 
so that 
$P$ lifts to a functor $\Pp:\Ww\to\Xx$ given by $\Pp(w)=P(w)$, $\Pp(m,w)=m$.
Here $\Ww$ is the ep-groupoid given by 
$\Obj_\Ww=W$, $\Mor_\Ww= \bX \leftsub{s}{\times}_P W$, 
$s(m,w)=w$, $t(m,w)=\mu(m,w)$, $\id_{w}= (\id_{P(w)},w)$, and 
$(m,w)\circ (m',w')= \bigl(m\circ m' , w\bigr)$ 
if $w' = \mu(m,w)$.

We will thus often denote strong bundles by $(\Pp:\Ww\to \Xx, \mu)$ or simply $\Pp:\Ww\to\Xx$.  

\item
A  {\bf sc-Fredholm section functor} of the strong bundle $\Pp:\Ww\to \Xx$ 
as in \cite[Def.8.3.4]{TheBook} is a functor $f:\Xx\to \Ww$ that is sc-smooth on object and morphism spaces, satisfies $\Pp\circ f = \id_\Xx$, and is such that $f:X\to W$ is sc-Fredholm on the M-polyfold $X$.

\item
An {\bf orientation} of a sc-Fredholm section functor $f$ as in \cite[Def.12.5.4]{TheBook} is a continuous section $\fo: X_\infty\to\Oo_f$ of the orientation bundle over the dense subset of smooth objects $X_\infty\subset X$, which is compatible with morphisms as specified in \S\ref{ss:orient}.
\end{nenumilist}
\end{definition}

Given a sc-Fredholm section $f: \Xx\to \Ww$ that describes a compact moduli space $\oMm\simeq|f^{-1}(0)|$, Lemma~\ref{lem:localtau} constructs local Fredholm stabilizations near any $x\in f^{-1}(0)$, as explained in \S\ref{ss:overview}.
Now a finite covering of $|f^{-1}(0)|$ by such local stabilizations can be used in two ways: 
On the one hand, the various fattened solution sets resulting from each local stabilization can be related by appropriate transition data to form a Kuranishi atlas for $\oMm$.
On the other hand, after an appropriate refinement of the polyfold model, we can combine a covering of $\oMm\simeq|f^{-1}(0)|$ by local Fredholm stabilizations into a global finite dimensional stabilization. 
This is our first main result (proven in \S\ref{ss:Vdata}--\ref{ss:stabilize}):

\begin{thm}[{\bf Global stabilization}]\label{thm:globstab}  
Let $f:\Xx\to \Ww$ be a sc-Fredholm section functor 
with compact solution set,
Then choices of local uniformizers and Fredholm stabilizations 
$(U_i,G_i,\tau_i:X\times E_i\to \Ww)_{i\in \{1,\ldots,N\}}$ covering $|f^{-1}(0)|$, a partition of unity,  compatible cover reductions as specified in Theorem~\ref{thm:stabilize}, and suitable 
compactness controlling data $(N,\Uu)$ as in Definition~\ref{def:control-compact}
 induce 
 \begin{itemize}\item [-] an open subset $V\subset X$  such that 
$|V| \subset |\Xx|$ is an open neighbourhood of $|f^{-1}(0)|$ that contains $|\cl_X(\Uu)|$,
\item[-] a finite group $G:=\prod_{i=1}^N G_i$ and 
a finite dimensional vector space $E:=\prod_{i=1}^N E_i$ with $G$-action, and
and 
\item[-] a commutative diagram of \'etale groupoids and functors 
\begin{align}\label{diag:11}
\xymatrix
{
\Xx_\Vv^{\less G}\times E  \ar@{->}[d]^{\pr}\ar@{->}[r]^{ \tau\;\;\;} & \Ww_\Vv: =  \psi^*\Ww \ar@{->}[d]_{\Pp_\Vv}\ar@{->}[r]^{\;\;\;\Psi}& \Ww|_{V} \; \subset\; \Ww\ar@{->}[d]_{\Pp}\\
\Xx_\Vv^{\less G} 
 \ar@{->}[r]^{\io} &   \Xx_\Vv  \ar@/_1pc/[u]_{f_\Vv}\ar@{->}[r]^{\psi} & \Xx|_{V} \;\subset\;  \Xx\ar@/_1pc/[u]_{f}
}
\end{align}
with the following properties:
\end{itemize}
\begin{nenumilist}
\item
$\Xx|_{V}$ and $\Ww|_{V}$ are the full subcategories of $\Xx$ resp.\  $\Ww$ with objects $V$ resp.\ $W|_V = P^{-1}(V)$ as in Lemma~\ref{lem:bundle}. 
In particular their realizations are naturally identified with $|\Xx|_{V}|\simeq |V|\subset |\Xx|$ and $|\Ww|_{V}|\simeq |\Pp|^{-1}(|V|)\subset |\Ww|$. 

\item  
$\Xx_\Vv$ is an ep-groupoid constructed in Theorem~\ref{thm:reduce}, covered by finitely many local uniformizers $(V_I,G_I)_{I\subset\{1,\ldots,N\}}$ so that 
$\Obj_{\Xx_\Vv} = \bigsqcup_I V_I$ and $\Mor_{\Xx_\Vv}(V_I,V_I)=G_I\times V_I$ with $G_I=\prod_{i\in I} G_i$. 
It is equipped with an inner action of $G$ in the sense of Definition~\ref{def:inngpact},
in particular lifting the trivial action on $|\Xx_\Vv|$. 
Further, the action on each $V_I\subset\Obj_{\Xx_\Vv}$ is given by the natural projection from $G\times V_I$
 to the local uniformizer morphisms
 $G_I\times V_I=\Mor_{\Xx_\Vv}(V_I,V_I)$ induced by the projection $G\to G_I$. 
 
The functor $\psi: \Xx_\Vv \to \Xx|_{V}$, given by natural maps $V_{I=\{i_1<\cdot\cdot\}} \to  U_{i_1}$, is finite-to-one, surjective on objects, and
an equivalence of ep-groupoids in the sense of \cite[Def.10.1.1]{TheBook}, in particular inducing a homeomorphism
${|\psi|: |\Xx_\Vv| \stackrel{\simeq}\to |\Xx|_{V}| =  |V|}$. 

\item
The bundle $\Pp_\Vv: \Ww_\Vv: =\psi^*\Ww \to \Xx_\Vv$ and its section $f_\Vv=\psi^* f$ are the pullbacks in the sense of Lemma~\ref{lem:bundle} of $\Pp$ and $f$, as spelled out in Theorem~\ref{thm:reduceFred}.
The bundle $\Ww_\Vv$ is equipped with the natural lift of the inner $G$ action on $\Xx_\Vv$ from Lemma~\ref{lem:actW}, with respect to which both $\Pp_\Vv$ and $f_\Vv$ are equivariant. 
The natural lift of $\psi$ to the pullback bundle functor $\Psi: \Ww_\Vv \to \Ww|_{V}$ of Lemma~\ref{lem:bundle}~(iii) is a strong bundle equivalence 
in the sense of \cite[Def.10.4.1]{TheBook}, and $|\psi|$ restricts to a homeomorphism $|f_\Vv^{-1}(0)|  \stackrel{\simeq}\to |f^{-1}(0)|$. 
The section $f_\Vv$ is sc-Fredholm, and the strong bundle equivalence $\Psi$ pulls back any choice of orientation of $f$ to an orientation of $f_\Vv$. 

\item  
The \'etale groupoid $\Xx_\Vv ^{\less G}$ constructed in Proposition~\ref{prop:Hcomplet} 
is a (nonproper) nonsingular  subgroupoid of $\Xx_\Vv$ with the same objects $\Obj_{\Xx_\Vv^{\less G}}=\bigsqcup_I V_I$ but where enough morphisms have been removed so that $\Xx_\Vv ^{\less G}$ has at most one morphism between any two objects.\footnote
{
Although we do not attempt to formulate a sense in which $\Xx_\Vv ^{\less G}$ is unique, it has a very simple structure:   It is the groupoid completion of an \'etale poset 
${\Qq}$, that is, the object space 
$\bigsqcup_I V_I$ has a natural partial order $\preccurlyeq$ such that  $\Mor_{{\Qq}}(x,y)$ is nonempty exactly if $x\preccurlyeq y$; see Lemma~\ref{lem:complet}.}

The removed morphisms are parts of the group action.
Indeed, $\Xx_\Vv ^{\less G}$
is equipped with an action of $G$ so that the inclusion functor $\io:  \Xx_\Vv ^{\less G}\to  \Xx_\Vv $ is $G$-equivariant in the sense of Definition~\ref{def:gpact} and induces a homeomorphism $|\io|: |\Xx_\Vv ^{\less G}|\,/G 
\stackrel{\simeq}\to 
|\Xx_\Vv|$.

\item 
The category $\Xx_\Vv^{\less G}\times E$ is the product of $\Xx_\Vv^{\less G}$ with the category with object space $E$ and only identity morphisms. 
Then
the functor $\tau: \Xx_\Vv^{\less G}\times E \to \Ww_\Vv$ constructed in Theorem~\ref{thm:stabilize}
is $G$-equivariant, a strong bundle map in the sense of \cite[Def.2.6.1]{TheBook}, and a compactness-controlled Fredholm stabilization in the following sense: 

The data $\bigl(N_\Vv:=N\circ\Psi, \Uu_\Vv:=\psi^{-1}(\Uu)\bigr)$  controls compactness of $f_\Vv$ and there exists a neighbourhood $O_\tau\subset E$ of $0$ such that the family of maps $\bigl(\tau_\ue: \Obj_{\Xx_\Vv}  \to \Obj_{\Ww_\Vv},  \uo\mapsto \tau(\uo,\ue) \bigr)_{\ue\in E}$ -- while generally not functorial -- is $(N_\Vv,\Uu_\Vv)$-regular as follows: 
\begin{itemize}
\item[(1)]
$N_\Vv(\tau_\ue(\uo))< 1$ for all $\ue\in O_\tau$, $\uo\in  \Obj_{\Xx_\Vv} = \Obj_{\Xx_\Vv^{\less G}}$;
\item[(2)]
$\supp\tau_\ue = {\rm cl}\{ \uo\in  \Obj_{\Xx_\Vv} \,|\, \tau_\ue(\uo)\ne 0_\uo \} \subset \Uu_\Vv$ for all $\ue\in O_\tau$; 
\item[(3)]
the map $\Obj_{\Xx_\Vv}\times O_\tau \to \Obj_{\Ww_\Vv},  (\uo,\ue) \mapsto f_\Vv(\uo) - \tau(\uo,\ue)$ is transverse to the zero section of $\Ww_\Vv$ and in general position over the boundary. 
\end{itemize}
\end{nenumilist}
\end{thm}

Full details and proofs are given in Theorems~\ref{thm:reduce}, \ref{thm:reduceFred}, and \ref{thm:stabilize}. 
The construction of such a global finite dimensional stabilization is the common technical work behind two major points of view on how to construct the ``virtual'' fundamental class of a compact moduli space $\oMm\simeq|f^{-1}(0)|$ without formal boundary: 
Moreover, as sketched in \cite[Rmk.1.3.8]{Mfund},
the global stabilization leads directly to the construction of a Kuranishi atlas as well as other finite dimensional models.
%
On the other hand the following Corollary shows how the stabilization also
yields a collection of multivalued perturbations whose solution sets are compact weighted branched orbifolds and can be used for an inverse limit construction of $[\oMm]$ in Theorem~\ref{thm:polyVFC}.

This construction of transverse perturbations from a global stabilization is the same line of argument as in both general polyfold theory and more classical regularization approaches such as \cite{MS}.\footnote{
For example, the stabilization in the trivial isotropy case is the map $\R^m\times X\to W$ in the proof of \cite[Thm.5.3.5]{TheBook}, and is the section $\Bb\times\Jj\to\Ee$ in the proof of  \cite[Prop.3.2.1]{MS}.}
In polyfold theory (and any other abstract construction of perturbations) the main technical complications arise from nontrivial isotropy. This is resolved locally by working with multivalued perturbations, but their patching and extension from the boundary is surprisingly nontrivial; see  \cite[\S14]{TheBook}. 
We avoid the technicalities of patching or extending multivalued perturbations by following the strategy of \cite[\S3.2]{MWiso}: Our technical work is in the construction of an appropriate atlas, given by the \'etale groupoid $\Xx_\Vv ^{\less G}$ that is both reduced (to organize overlaps of charts) and pruned (to remove isotropy), and the global action of a group $G$ that captures all isotropy. 
Then the global stabilization $\tau: \Xx_\Vv^{\less G}\times E \to \Ww_\Vv$  gives rise to a family of single-valued sections $$
f_\Vv-\tau(\cdot,\ue): X_\Vv:=\Obj_{\Xx_\Vv}\to \Obj_{\Ww_\Vv}
$$
 of the topological bundle formed by the object spaces of the bundle $\Ww_\Vv\to \Xx_\Vv$, which fail functoriality only in their lack of $G$-equivariance.\footnote{
Note here that, while we construct the stabilization functor $\tau$ to be $G$-equivariant, 
the single-valued perturbation $f_\Vv-\tau(\cdot,\ue)$ is typically not itself $G$-equivariant because $G$ acts nontrivially on $E$.
} 
Taking the $G$-orbit of such a section yields a transverse multivalued perturbation of $f_\Vv$ which is more globally symmetric and structured than the complicated hierarchy of \cite[\S13]{TheBook}. 
Our notion of a globally structured multisection is given in Definition~\ref{def:multisdef0}.
This notion should allow considerable simplification of the intricate extension constructions in \cite[\S14]{TheBook}. 

The following is a detailed description of these perturbations and perturbed zero sets. 
Here we denote by  $\Q^{\ge 0}$ the category with objects $\Q\cap[0,\infty)$ and only identity morphisms.\footnote{
This category is denoted $\Q^+$ in \cite{TheBook}.}

\begin{cor} \label{cor:multisection}
Let $\tau: \Xx_\Vv^{\less G}\times E \to \Ww_{\Vv}$ be a stabilization functor as in Theorem~\ref{thm:globstab}.  

\begin{nenumilist}
\item 
Each $\ue\in E$ induces sc$^+$-multisection functors $\La_{\Vv,\ue} : \Ww_\Vv\to \Q^{\ge 0}$ given by $\La_{\Vv,\ue}(w):= \tfrac{1}{\# G} \ \# \{ \uh\in G \,|\, \tau(P_\Vv(w) ,\uh*\ue)= w  \}$
and $\La_{\ue}:= \Psi_*\La_{\Vv,\ue}: \Ww|_{V}\to \Q^{\ge 0}$ given by pushforward \cite[Lemma~11.5.2]{TheBook}.  
Their support (i.e.\ ``graph") is
\begin{align*}
\supp \La_{\Vv,\ue}& = G * \tau(X_\Vv \times \{\ue\} ) = {\textstyle \bigcup_{\ug\in G}} \; \tau(X_\Vv \times \{\ug*\ue\} ),\\
 \supp \La_\ue & = \Psi(\supp \La_{\Vv,\ue}) .
\end{align*}
These multisections are structurable in the sense of \cite[\S13.3]{TheBook}, and $\La_{\Vv,\ue}$ 
is globally structured in the sense of Definition~\ref{def:multisdef0}.  
%

\item 
There is a comeagre (and thus dense) subset $O^\pitchfork_\tau\subset O_\tau$ such that 
for all $\ue\in O^\pitchfork_\tau$ the functors 
$\Theta_{\Vv,\ue}: =\La_{\Vv,\ue}\circ f_\Vv : \Xx_\Vv\to\Q^{\ge 0}$ and $\Theta_\ue: =\La_{\ue}\circ f : \Xx|_V\to\Q^{\ge 0}$ are transversal and in general position over the boundary in the sense of \cite[\S15.2]{TheBook}. 
This gives both the structure of compact, tame branched ep$^+$-subgroupoids 
in the sense of \cite[\S9.1]{TheBook}.

\item 
If $f$ is oriented, then for every $\ue\in O^\pitchfork_\tau$ both $\Theta_{\Vv,\ue}$
and $\Theta_\ue$ inherit natural orientations.


\item 
Suppose, in addition to (iii), that $f$ has constant Fredholm index $d\in\Z$, 
and $\psi_f : |f^{-1}(0)|\to \oMm$ is a homeomorphism to a compact 
Hausdorff space $\oMm$. 
Then the perturbed solution sets 
\begin{align*}
& |\supp\Theta_{\Vv,\ue}| = |\{\ux\in X_\Vv \,|\, \Theta_{\Vv,\ue}(\ux)>0\}|\subset |\Xx_\Vv|,\\
& |\supp\Theta_\ue| = |\{x\in X|_V \,|\, \Theta_\ue(x)>0\}|\subset |V|
\end{align*}
carry natural fundamental classes denoted $[\Theta_{\Vv,\ue}], [\Theta_\ue]$ that belong to
$\check{H}_d(\oMm;\Q)$ if $\oMm$ has no formal boundary, and to the relative group 
$\check{H}^\infty_d(\oMm\less\p\oMm;\Q)$ if $\oMm$ has formal boundary $\p\oMm$.\footnote
{See \eqref{eq:Afundc} for notation.}
Further, if there is no formal boundary $\psi_{f_\Vv}:= \psi_f \circ |\psi|$ is a homeomorphism  $|f_\Vv^{-1}(0)|\to \oMm$, and 
the perturbations in (ii) induce the same fundamental classes as defined in Theorem~\ref{thm:polyVFC},
$$
[\oMm]_f
\;=\; 
(\psi_f)_* \;\underset{\leftarrow}\lim \, 
(\psi_k)_* \bigl[ \Theta_{\ue_k} \bigr] 
\;=\; 
(\psi_{f_\Vv})_*  \; \underset{\leftarrow }\lim \, (\psi_{\Vv,k})_*
\bigl[ \Theta_{\Vv,\ue_k} \bigr] \;=\; [\oMm]_{f_\Vv}
 \;\in\; \check H_d(\oMm;\Q) . 
$$

\item

In particular, in the situation of (iv) there is a well defined notion of 
integration over the perturbed solution sets, which satisfies Stokes' Theorem in the sense of \cite[Thm.15.4.3]{TheBook}, and the results can be identified via the homeomorphism $|\psi| \,|_{|\supp\Theta_{\Vv,\ue}|}: |\supp\Theta_{\Vv,\ue}| \to |\supp\Theta_\ue|$. 
%
\end{nenumilist}
\end{cor}

In this Corollary, the identity for the support of the multisection in (i) follows from $G$-equivariance of the functor $\tau$;  the rest of the proof is given at the end of \S\ref{ss:stabilize}.

\subsection{Scale Calculus} \label{ss:scale}

The purpose of this subsection is to give a very brief summary of the key notions of scale calculus (M-polyfold, strong bundle, and sc-Fredholm) that were developed in \cite{TheBook} to replace the notions of Banach manifolds, Banach bundles, and Fredholm sections -- which turned out not to be good fits for global descriptions of moduli spaces of pseudoholomorphic curves. For a more detailed exposition see \cite{usersguide}; for an application-centered introduction see \cite{WISCON}.

An {\bf M-polyfold} without boundary, as defined in \cite[Def.2.3.4]{TheBook}, is analogous to the notion of a Banach manifold: While the latter are locally homeomorphic to open subsets of a Banach space, an M-polyfold is locally homeomorphic to the image of a retract $r:U\to U$ of an open subset $U\subset E$ of a Banach space $E$.\footnote{To ensure sc-smooth partitions of unity, we assume $E$ to be a Hilbert space as in \cite[Cor.5.5.17]{TheBook}. } While $r$ is generally not classically differentiable, it is required to be scale-smooth (or \lq sc-smooth') with respect to a scale structure (or \lq sc-structure') on $E$ -- see \cite[\S 2.2]{usersguide} for an introduction. 

Unless otherwise specified, we will allow M-polyfolds to have ``tame" boundary, arising from local models in $[0,\infty)^k\times E$ satisfying the tameness condition \cite[Def.2.5.2]{TheBook}. Each such M-polyfold\footnote{Any (M-)polyfold in this paper is a ``tame (M-)polyfold" in the language of \cite{TheBook}. See \cite[\S 5.3]{usersguide} for an explanation of the additional condition -- previously called ``neatness''.}
carries a well defined degeneracy index $d_X:X\to\N_0$ given in local models by counting the number of coordinates in $[0,\infty)^k$ that equal to zero. With that, the {\bf boundary} of an M-polyfold is defined as $\partial X:=\{x\in X \,|\, d_X(x)\geq 1\}$, and each connected component $Y\subset \{x\in X \,|\, d_X(x)= 1\}$ inherits the structure of an M-polyfold with degeneracy $d_Y=d_X-1$ by \cite[Thm.2.5.10]{TheBook}. We will call such $Y\subset\partial X$ a {\bf boundary stratum}, while its closure is called a face in \cite{TheBook}. For $l\geq 2$, the connected components of $\{x\in X \,|\, d_X(x)= l\}$  -- called {\bf corner strata} -- then inherit sc-smooth structures by local descriptions as intersections of local faces as in \cite[Prop.2.5.16]{TheBook}. 

As with boundary and corners, most other notions of scale calculus arise by generalizing classical differential geometry with the two notions of sc-calculus and retracts. For example, a sc-smooth map $f:X\to Y$ between M-polyfolds $X,Y$ is called {\bf local sc-diffeomorphism} if for every $x\in X$ and sufficiently small open neighbourhood $O\subset X$ of $x$ (we can assume a retract chart $O\simeq r(U)$)  the image $f(O)\subset Y$ is open and there is a sc-smooth map $f_O^{-1}:f(O)\to O$ that is inverse to $f$, i.e.\ $f_O^{-1}\circ f = \id_{O}$ and $f\circ f_O^{-1}=\id_{f(O)}$.

Given the notion of ep-groupoid in Definition~\ref{def:poly}, we can also view M-polyfolds as the realizations $|\Xx|$ of  nonsingular ep-groupoids $\Xx$. 
On the other hand, the following remark specializes Definition~\ref{def:poly} to scale calculus on finite dimensional Banach spaces.

\begin{remark} \rm
The classical notion of manifolds (with boundary and corners)
 is the class of M-polyfolds whose local models use finite dimensional Banach spaces $E$. In that case scale calculus coincides with classical calculus and retracts are necessarily trivial by \cite[Prop.2.1.2]{TheBook}.
 
Thus the definition of an {\bf orbifold} as the realization $|\Xx|$ of an \'etale proper groupoid $\Xx$ is the special case of Definition~\ref{def:poly}~(i) in which the object space $X$ and morphism space $\bX$ are finite dimensional manifolds.
Indeed, sc-smoothness on finite dimensional manifolds is classical smoothness by \cite[Prop.1.2.1]{TheBook}), 
the \'etale property implies that unit, inverse, and composition are sc-diffeomorphisms by Lemma~\ref{lem:etale}, 
and the properness condition on locally compact object spaces is equivalent to the classical condition of $s\times t: \bX \to X \times X$ being proper by Lemma~\ref{lem:prop1}.

Similarly, the definition of an orbibundle as the realization $|\Pp|:|\Ww|\to|\Xx|$ of a smooth functor between \'etale proper groupoids is the special case of Definition~\ref{def:poly}~(ii) where $W\to X$ is a finite rank vector bundle over a manifold. Here finite rank bundles are the special case of the strong bundle notion below for which the fiber ambient space $F$ is finite dimensional, and hence a smooth family of projections on $F$ has images of constant dimension.

In this special case of an orbibundle $|\Pp|:|\Ww|\to |\Xx|$, the sc-Fredholm sections in Definition~\ref{def:poly}~(iii), as specified below, are precisely the classically smooth sections.
\hfill$\er$
\end{remark}

Back to trivial isotropy but general M-polyfolds, a {\bf strong bundle}, as defined in \cite[Def.2.6.5]{TheBook}, is a sc-smooth surjection $P:W\to X$ between M-polyfolds with linear structures on each fiber $W_x=P^{-1}(x)$ for $x\in X$, and an equivalence class of compatible strong bundle charts. 
Here the notion of `strong' encodes a subbundle $W\supset W[1]\to X$ equipped with another topology and sc-smooth structure such that the fiber inclusions $W[1]_x\hookrightarrow W_x$ are dense and compact in the following sense: 
Local trivializations of $W$ are of the form $\bigcup_{u\in O} \im\Gamma(u) \to O$, where $O=r(U)\subset E$ is a retract as above and $\Gamma(u):F\to F$ is a family of projections on another Banach space $F$.  
Scale-smoothness of $(u,h)\mapsto \Gamma(u)h$ is required with respect to the product sc-structure on $E\times F$, but also with respect to a shifted sc-structure on $E\times F_1$, where $F_1\subset F$ is a compact dense embedding of another Banach space (part of the sc-structure on $F$). 
In this local trivialization, $W[1]\subset W$ is the subset $\bigcup_{u\in O} \Gamma(u)F_1$, equipped with the topology induced by its inclusion in $E\times F_1$. 

Local trivializations are compatible if their transition maps are strong bundle maps. Here a sc-smooth map $\Phi:W\to W'$ between two strong bundles (for example local models as above) is called a {\bf strong bundle map}\footnote{This notion generalizes \cite[Def.2.6.1]{TheBook} as described on \cite[p.67]{TheBook}.} if it covers a sc-smooth map on the base, is linear on fibers, and restricts to a sc-smooth map $W[1]\to W'[1]$.

Finally, a {\bf sc-Fredholm section} $f:X\to W$ of a strong bundle over an M-polyfold as in \cite[Def.3.1.16]{TheBook} is a sc-smooth map satisfying $P\circ f=\id_X$ and sc-Fredholm properties which encode elliptic regularity and a nonlinear contraction property \cite[Def.3.1.6, 3.1.7, 3.1.11]{TheBook}. The latter is a stronger condition than the classical notion of the linearized operator being Fredholm,
and is crucial to ensure an implicit function theorem as shown in \cite{counterex} -- since the differentials of a sc-smooth section with Fredholm linearizations may not vary continuously with the basepoint.

Now an {\bf M-polyfold model for a compact moduli space} $\oMm$ (which we generalize to allow for nontrivial isotropy in Definition~\ref{def:PolyfoldModel}) would consist of a sc-Fredholm section $f: X\to W$ and a homeomorphism $\psi: f^{-1}(0) \to \oMm$. The regularization of such (relatively simple) moduli spaces is achieved in \cite[\S5.3]{TheBook} by showing that different transverse perturbations of $f$ have cobordant solution spaces $(f-s)^{-1}(0)\sim (f-s')^{-1}(0)$ when perturbations are restricted to compactness-controlled (see \S\ref{ss:polybundle}) sc$^+$-sections $s,s': X\to W$. 
This serves as introduction to the last crucial scale calculus notion: An {\bf sc$^+$-section} of a strong bundle $W\to X$ is a section that is sc-smooth $X\to W[1]$ relative to the shifted sc-structure on the dense subbundle $W[1]\subset W$. Since the fibers of $W[1]$ embed compactly into $W$, these sections can be viewed as compact perturbations of the sc-Fredholm section $f$. In particular, adding them does not affect the Fredholm index by \cite[Thm.3.1.18]{TheBook}.

\subsection{Orientation Conventions} \label{ss:orient}

The goal of this subsection (which can be omitted on first reading) is to spell out the notion of an orientation in \cite[Def.12.5.4]{TheBook} for a sc-Fredholm section functor $f:\Xx\to\Ww$ over an ep-groupoid in sufficient detail to make sense of the constructions in this paper. 
This will mainly be a matter of constructing the orientation bundle over the M-polyfold of objects $X$, which will only be possible over a dense subset of so-called ``smooth points":  

M-polyfolds inherit from their sc-smooth structure a filtration $X=X_0\supset X_1 \supset \ldots \supset X_\infty=\bigcap_{i\in\N_0} X_i$ of subsets with their own (metrizable, paracompact) topology \cite[Thm.2.3.10]{TheBook}. Viewed as subsets, each of them are dense in $X$ \cite[p.289]{TheBook}.\footnote{This density holds by definition \cite[Def.1.1.1]{TheBook} for $X=\E=(E_m)_{m\in\N_0}$ a sc-Banach space. In a local chart $X\supset O\simeq r(U)$ it holds since every $x\in O$ is an image $x= r(y)$ of a map $r:U\to U$ on an open subset $U\subset E_0$ of a sc-Banach space, where $r$ is sc-continuous in the sense that $r:U\cap E_m\to U \cap E_m$ is continuous for all $m\in\N_0$. 
Now an approximating sequence $E_m\ni y_i \to y$ lies in $U$ for sufficiently large $i$ and yields an approximation $E_m\cap r(U)\ni r(y_i) \to r(y)=x$.}
The same is true for any polyfold -- the realization $|\Xx|$ of an ep-groupoid -- since the equivalence relation by sc-diffeomorphisms preserves the filtration on the object set, resulting in a filtration $|\Xx|=|X_0|\supset |X_1| \supset \ldots \supset |X_\infty|=\bigcap_{i\in\N_0} |X_i|$ of dense subsets.

Now the {\bf orientation bundle} of a sc-Fredholm section $f:X\to W$ on object level -- where $W$ is a strong bundle over an M-polyfold $X$ -- is constructed in \cite[\S6, \S12.5]{TheBook} as a local homeomorphism $\Oo_f\to X_\infty$ between Hausdorff spaces, where the fiber of $\Oo_f$ over $x\in X_\infty$ consists of the two possible orientations of the line bundle ${\rm DET}(f,x)\to {\rm LIN}(f,x)$ over a convex set of admissible linearizations of $f$ at $x$ \cite[Def.6.2.1]{TheBook}. 

Here the fiber of ${\rm DET}(f,x)$ over each admissible linearization $S:\rT_xX\to W_x$ is the classical determinant line $\det(S) =\lm \ker S \otimes \bigl( \lm \bigl( \qu{W_x}{\im S} \bigr) \bigr)^*$. 
When $f(x)=0$ then the canonical linearization $S={\rm D} f (x):\rT_xX\to W_x$ is admissible, and all other admissible linearizations differ by a compact operator (see the discussion of sc$^+$-operators after \cite[Def.1.1.5]{TheBook}). 
When $f(x)\ne 0$ then all admissible linearizations can be expressed as canonical linearizations $S={\rm D} (f-s) (x):\rT_xX\to W_x$ for some nonlinear sc$^+$-section $s:X\to W$ with $s(x)=f(x)$ (see \cite[Def.6.1.8]{usersguide} for the notion of sc$^+$-section). 

Next, each ${\rm DET}(f,x)$ is given the  structure of a line bundle by viewing it as the restriction of the classical determinant line bundle to a subset of the Fredholm operators $\rT_xX\to W_x$. From the abundance of more or less natural conventions for topologizing this determinant bundle (as discussed in \cite{zinger}) we use the conventions in  \cite[Prop.6.5.8]{TheBook}: 
Given a Fredholm operator $T:E\to F$ between Banach spaces $E,F$, we can topologize the determinant line bundle $\bigcup_{\|S-T\| <\eps} \det(S)$ in a neighbourhood of $T$ as follows: 
The choice of any good left-projection $P:F\to F$ -- satisfying in particular $\im PT = \im P$, see \cite[Def.6.4.3]{TheBook} -- induces an exact sequence in \cite[(6.9)]{TheBook},  
$$
E_{(T,P)} \; : \quad 0 \to \ker(T) \to \ker(PT) \to \qu{F}{\im PT} \to \qu{F}{\im T} \to 0.
$$ 
Then the isomorphism $\Phi_E$ of \cite[Lemma 6.3.3]{TheBook} applied to $E=E_{(T,P)}$ yields an isomorphism $\ga_T^P: \det(T)\to \det(PT)$ as in \cite[(6.12)]{TheBook}. 
This extends to isomorphisms $\ga_S^P: \det(S)\to \det(PS)$ for $\|S-T\|\leq \eps$ guaranteed by \cite[Prop.6.5.1~(1)]{TheBook}, and these are used to pull back the local determinant bundle structure of $\bigcup_{\|S-T\| <\eps} \{S\}\times \det(PS)$ constructed in \cite[Prop.6.5.4]{TheBook}. Independence of choices and continuity of transition maps is proven in \cite[\S6.4-5]{TheBook}.\footnote
{
Here we must take into account the fact that the sign in \cite[Prop.6.4.15]{TheBook} needs to be corrected to $\det( c_1,...,c_m, - r_1, ..., - r_l )$.
 This correction makes this proposition coherent with the rest of \cite[\S6.4]{TheBook} as well as \cite[(8.1.9)]{MW2}.}

This defines the orientation bundle $\Oo_f$ as a set with a 2-to-1 map to $X_\infty$. Topologizing it requires a notion of continuous extension of orientations when the basepoint $x\in X_\infty$ varies, which is given in \cite[Thm.6.6.11]{TheBook}. 

Now an {\bf orientation} of a sc-Fredholm section functor $f:\Xx\to\Ww$ over an ep-groupoid $\Xx$ is a continuous section $\fo:X_\infty\to\Oo_f, x\mapsto \fo_x$ of the orientation bundle over the smooth objects $X_\infty$, that is compatible with the morphisms of $\Xx$ in the sense that $\phi_*\fo_{s(\phi)}=\fo_{t(\phi)}$ for all $\phi\in\bX$. 
Here $\phi_*$ is the second of the following conventions worth spelling out:

\medskip
\noindent 
{\bf General isomorphism convention:}
Any commutative diagram 
\begin{align*}
\xymatrix
{
E  \ar@{->}[d]_{I}\ar@{->}[r]^{S} & F \ar@{->}[d]^{J}   \\
E'  \ar@{->}[r]^{S'}  &   F' 
}
\end{align*}
of Fredholm operators $S:E\to F,S':E'\to F'$ and isomorphisms $I: E \to E'$, $J: F \to F'$ 
induces a commutative diagram for each choice of good left-projection $P:F\to F$ for $S$ 
\begin{align*}
\xymatrix
{
0 \ar@{->}[r] & 
\ker(S)  \ar@{->}[d]_{I}\ar@{->}[r] &  
\ker(PS)  \ar@{->}[d]_{I}\ar@{->}[r]^{S} &
\qu{F}{\im PS}  \ar@{->}[d]^{J} \ar@{->}[r]^{{\rm Id} - P} & 
\qu{F}{\im S} \ar@{->}[d]^{J}  \ar@{->}[r]  & 0 
\\
0 \ar@{->}[r] & 
\ker(S') \ar@{->}[r] &  
\ker(P'S')  \ar@{->}[r]^{S'} &
\qu{F'}{\im P'S'} \ar@{->}[r]^{{\rm Id} - P'} & 
\qu{F'}{\im S'} \ar@{->}[r]  & 0 
}
\end{align*}
Here $P':=JPJ^{-1}$ is the corresponding good left-projection for $S'$, which combines the exact sequences $E_{(S,P)}$ and $E_{(S',P')}$ of type \cite[(6.9)]{TheBook} into a diagram of type \cite[(6.5)]{TheBook}. 
Now \cite[Lemma~6.3.4]{TheBook} induces a natural commutative diagram intertwining the isomorphisms $\gamma^P_S: \det(S)\to \det(PS)$ and $\gamma^{P'}_{S'}: \det(S')\to \det(P'S')$ used in the construction of the determinant bundles near $S$ and $S'$. 
We denote the natural isomorphism $\det(S) \to \det(S')$ in this diagram by 
\begin{align*}
(I,J)_* : \;  \det(S) &\to \det(S'), \;\; \\
(e_1\wedge\ldots e_k)\otimes (f_1\wedge\ldots f_c)^* & \mapsto (I(e_1)\wedge\ldots I(e_k))\otimes (J(f_1)\wedge\ldots J(f_c))^* . 
\end{align*}
Note that this map is well defined since we assumed $J S = S' I$, and hence $I$ resp.\ $J$ restrict to isomorphisms between kernels resp.\ cokernels of the two Fredholm operators.

\medskip
\noindent 
{\bf Morphism convention:}
Since source and target maps $\bX\to X$ are local sc-diffeomorphisms, every morphism $\phi\in\bX$ extends to a local sc-diffeomorphism $\phi:U_x\to U_y$ between neighbourhoods $U_x,U_y\subset X$ of $x=s(\phi)$ and $y=t(\phi)$ as in \cite[\S7.1]{TheBook}. Its differential $\rd\phi(x):\rT_xX\to\rT_yX$ is a sc-isomorphism for smooth points $x\in X_\infty$. 
At smooth points, $\phi$ also lifts to a strong bundle isomorphism $\mu(\phi,\cdot): W|_{U_x} \to W|_{U_y}$, and functoriality of $f$ implies $f\circ\phi = \mu(\phi,f)$ on $U_x$.
Now we can combine these maps to obtain a bijection between admissible linearizations $$
\phi_*: {\rm LIN}(f,x)\to {\rm LIN}(f,y), S \mapsto \mu(\phi,\cdot) \circ S\circ (\rd\phi(x))^{-1}.
$$  
Since these linearizations satisfy the commuting diagram relation $\mu(\phi,\cdot) \circ S = \phi_* S \circ \rd\phi(x)$ with isomorphisms $I=\rd\phi(x)$ between their domains and $J=\mu(\phi,\cdot)$ between their target spaces, we can use the general isomorphism convention to lift $\phi_*$ to a bundle isomorphism $\phi_*: {\rm DET}(f,x)\to {\rm DET}(f,y)$ given by $(I,J)_* : \det(S) \to \det(\phi_* S)$ on each fiber. 
This in turn induces a bijection of orientations $\phi_*: (\Oo_f)_x \to (\Oo_f)_y$ given by restricting the isomorphisms $I=\rd\phi(x)$ and $J:=\mu(\phi,\cdot)$ to the kernels and cokernels of the linearizations. 

\medskip
\noindent 
{\bf Equivalence convention:}
If $\psi:\Yy\to\Xx|_V$ is an equivalence of ep-groupoids \cite[Def.10.1.1]{TheBook} and $f:\Xx\to\Ww$ is sc-Fredholm, then 
any orientation of $f$ has a natural pullback to an orientation of the pullback section $\Ti f:=\psi^*f:\Yy\to\psi^*\Ww$ as follows: 

The section $\Ti f$ satisfies $\Psi\circ \Ti f = f \circ \psi$ by Lemma~\ref{lem:bundle} and is sc-Fredholm by Lemma~\ref{lem:pullback}. Here the pullback bundle functor $\Psi:  \psi^*\Ww \to \Ww$ is a strong bundle equivalence, in particular restricts to isomorphisms between the fibers $\Psi_y: (\psi^*\Ww)_y \to \Ww_{\psi(y)}$, whereas $\psi$ is a local sc-diffeomorphism, hence induces isomorphisms $\rd\psi(y):\rT_y Y \to \rT_{\psi(y)}X$ between the domains of the linearizations of $\Ti f$ and $f$. 
For each $y\in Y$ these induce a bijection between the sets of admissible linearizations $$
\psi^*: {\rm LIN}(f,\psi(y))\to {\rm LIN}(\Ti f,y), S \mapsto \Psi_y^{-1}\circ S \circ \rd\psi(y).
$$ 
Indeed, given a sc$^+$-section $s$ with $s(\psi(y))=f(\psi(y))$, its pullback $\Ti s:= \psi^*s$ is a sc$^+$-section with $\Ti s(y)=\Ti f(y)$, and the corresponding linearizations satisfy
$$\Psi_y\circ \rD_y(\Ti f - \Ti s) = \rD_{\psi(y)} (f-s) \circ \rd\psi(y).
$$

Since this pullback of linearizations satisfies the commuting diagram relation $\Psi_y\circ \psi^*S = S \circ \rd\psi(y)$
with isomorphisms $I=\rd\psi(y)$ between their domains and $J=\Psi_y$ between their target spaces, the general isomorphism convention lifts $\psi^*$ to a bundle isomorphism $\psi^*: {\rm DET}(f,\psi(y))\to {\rm DET}(\Ti f,y)$ given by ${(I,J)_*}^{-1} : \det(S) \to \det(\psi^* S)$ on each fiber. 
This then induces a bijection of orientations $\psi^*: (\Oo_f)_{\psi(y)} \to (\Oo_{\Ti f})_y$ given by restricting the isomorphisms $I=\rd\psi(y)$ and $J=\Psi_y$ to the kernels and cokernels of the linearizations. 

Now any orientation $\fo:X_\infty\to\Oo_f$ of $f$ induces a section $$
\psi^*\fo:= \psi^* \circ \fo\circ \psi : Y_\infty \to \Oo_{\Ti f},$$
 which is an orientation of $\Ti f$ by \cite[Thm.12.5.7~(2)]{TheBook} -- called the pullback orientation.

\medskip
\noindent 
{\bf Solution set convention:}
Given an oriented sc-Fredholm section functor $f:\Xx\to\Ww$ and a suitably transverse -- possibly multi-valued -- perturbation of $f$, the perturbed solution set inherits a natural orientation. 
In the ep-groupoid context this is formulated precisely and proven in \cite[\S15.4]{TheBook}, with local orientations constructed from the following single-valued, trivial isotropy case in \cite[Thm.6.6.18]{TheBook}\footnote{Here we generalize the conventions of \cite[\S6.6]{TheBook} as indicated in the beginning of \cite[\S15.4]{TheBook}.}: 

Suppose $f: X\to W$ is an oriented sc-Fredholm section over an M-polyfold, $s: X\to W$ is a sc$^+$-section, and $x\in (f-s)^{-1}(0)\subset X_\infty$ is a perturbed solution.\footnote{The inclusion of perturbed solution sets in $X_\infty$ is guaranteed by the regularizing property of sc-Fredholm sections, which is preserved by sc$^+$-perturbations.} If $f-s$ is transverse at $x$, that is if the linearization $\rD_x(f-s): \rT_x X \to W_x$ is surjective, then $\ker\rD_x(f-s)$ inherits a natural orientation since $S:=\rD_x(f-s) \in {\rm LIN}(f,x)$ is an admissible linearization and so the given orientation of $$
\det(S)=\lm \ker S \otimes (\lm \qu{W_x}{\im S})^* \cong \lm \ker\rD_x(f-s)
$$
 transfers to the kernel by canonical identification. 
If $(f-s)^{-1}(0)$ also inherits a manifold structure (which requires transversality at all perturbed solutions and good or general position for solutions at the boundary), then the induced pointwise orientations of its tangent spaces $\rT_x (f-s)^{-1}(0)=\ker\rD_x(f-s)$ form a continuous orientation of the tangent bundle since the orientation propagation that topologizes $\Oo_f\to X_\infty$ in \cite[Thm.6.6.11]{TheBook} agrees with classical parallel transport on the tangent bundle over $(f-s)^{-1}(0)$. 

More precisely, let $\phi:[0,1]\to (f-s)^{-1}(0)$ be a smooth path along which $f-s$ is transverse (for boundary considerations see \cite{TheBook}). Then the polyfold construction of orientation propagation $\phi_*:(\Oo_f)_{\phi(0)}\to (\Oo_f)_{\phi(1)}$ in \cite[Lemma~6.6.4-6]{TheBook} simplifies: 
\begin{itemlist}
\item[-]
We can choose the sc$^+$-section $[0,1]\times X \to W, (t,x)\mapsto s(x)$ since $s(\phi(t))=f(\phi(t))$. 
\item[-]
The section $F:[0,1]\times X\to W, (t,x)\mapsto f(x) - s(x)$ is transverse along the graph of $\phi$. 
\item[-]
On the line bundle $\lm L = \bigcup_{t\in[0,1]} \{t\} \times \lm L_t$ with $L_t= \ker\rD_{\phi(t)} (f-s)$ the parallel transport is the classical orientation propagation along the path $\phi:[0,1]\to (f-s)^{-1}(0)$. 
\item[-]
The identification of $\det(\ker\rD_{\phi(t)} (f-s)) \to \lm L_t \otimes (\lm \R^k)^*= \lm L_t$ in \cite[(6.28)]{TheBook} is the identity map since $k=0$. 
\end{itemlist}
So the polyfold orientation propagation $\phi_*$ in this setting agrees with classical orientation propagation on $(f-s)^{-1}(0)$, as claimed. Thus an orientation of $f$ induces an orientation of any appropriately transverse perturbed solution set $(f-s)^{-1}(0)$.

\begin{remark}\label{rmk:equiv-orient} \rm 
The orientation conventions for equivalences and solution sets interact as follows: 
Consider a sc-Fredholm section functor $f:\Xx\to\Ww$ with orientation $\fo:X_\infty\to\Oo_f$. 
Then any local perturbed solution set  $(f-s)^{-1}(0)\cap U$ inherits the structure of an oriented manifold when  $s: X\to W$ is a (not necessarily functorial) sc$^+$-section such that $(f-s)|_U$ is transverse over an open subset $U\subset X$.  

Now consider an equivalence of ep-groupoids $\psi:\Yy\to\Xx|_V$, and the pullback section $\Ti f:=\psi^*f:\Yy\to\psi^*\Ww$ with pullback orientation $\psi^*\fo= \psi^* \circ \fo\circ \psi : Y_\infty \to \Oo_{\Ti f}$.  Suppose that the above perturbation is given on a set $U=\psi(\Ti U) \subset V$ that is the image of a sufficiently small open subset $\Ti U\subset Y$ on which $\psi$ restricts to a sc-diffeomorphism. Then $\Ti s := \psi^*s: \Yy\to \psi^*\Ww$ is a sc$^+$-section with $(\Ti f - \Ti s)|_{\Ti U}$ transverse, and $\psi$ restricts to an diffeomorphism between the manifolds $(\Ti f-\Ti s)^{-1}(0)\cap \Ti U \to (f-s)^{-1}(0)\cap U$ that preserves orientations. 

Indeed, the pullback construction guarantees that $\psi$ maps solutions $$
y\in (\Ti f-\Ti s)^{-1}(0)\cap \Ti U\subset Y_\infty
$$
 to solutions $$
 x:=\psi(y)\in (f-s)^{-1}(0)\cap U\subset X_\infty,
 $$
  sc-diffeomorphisms restrict to diffeomorphisms in finite dimensions, and its linearization $\rd\psi(y): \rT_y (\Ti f-\Ti s)^{-1}(0) \to \rT_x (f-s)^{-1}(0)$ is compatible with the orientations induced by $\psi^*\fo$ and $\fo$ as follows: 

In the absence of cokernels, $\rT_x (f-s)^{-1}(0) = \ker S$ is oriented by the natural identification $\lm \ker S=\det(S)$ with the determinant of the linearization $S:=\rD_x(f-s) \in {\rm LIN}(f,{\psi(y)})$, which is oriented by $\fo(x)$. 
Similarly, $\rT_y (\Ti f-\Ti s)^{-1}(0) = \ker \psi^*S$ is oriented by the natural identification $\lm \ker \psi^* S =\det(\psi^* S)$ with the determinant of the linearization $\psi^* S= \Psi_y^{-1}\circ S \circ \rd\psi(y) \in {\rm LIN}(\Ti f,y)$, which is oriented by $(\psi^*\fo)(y)= \psi^*(\fo(\psi(y)))$.
Here, due to the trivial cokernel, $\psi^*: (\Oo_f)_{\psi(y)} \to (\Oo_{\Ti f})_y$ is given by restricting the isomorphism $\rd\psi(y)$ to the kernels. So the polyfold-theoretic identification $\psi^*$ between the orientation $\fo(x)$ of $\rT_x (f-s)^{-1}(0)$ and the orientation $(\psi^*\fo)(y)= \psi^*(\fo(x))$ of $\rT_y (\Ti f-\Ti s)^{-1}(0)$ coincides with the differential-geometric identification induced by $\rd\psi(y)$. 
\hfill$\er$
\end{remark}

\medskip
\noindent 
{\bf Product convention:} Given $f:\Xx\to\Ww$ sc-Fredholm with orientation $\fo:X_\infty\to\Oo_f$, its constant extension $\Hat f: [0,1]\times\Xx\to [0,1]\times\Ww, (t,x)\mapsto (t,f(x))$ inherits a natural orientation $\Hat\fo:[0,1]\times X_\infty\to\Oo_{\Hat f}$ specified in \cite[Rmk.6.6.19]{TheBook} as follows: 

The map on linearizations ${\rm LIN}(f,x)\to {\rm LIN}(\Hat f, (t,x)), S \mapsto S \Pi$ given by composition with the projection $\Pi:\rT_{(t,x)}[0,1]\times X = \R\times\rT_x X \to \rT_x X$ lifts to a bundle map $
{\rm DET}(f,x)\to {\rm DET}(\Hat f, (t,x))$ given on each fiber by $$
\det(S) \to \det(S \Pi), (a_1\wedge\ldots \wedge a_k) \otimes (b_1\wedge \ldots \wedge b_c)^* \mapsto (1\wedge a_1\wedge\ldots \wedge a_k) \otimes (b_1\wedge \ldots \wedge b_c)^*,
$$
 where $1\in \R=\rT_t[0,1]$ represents the standard orientation of the interval $[0,1]$.

\medskip
\noindent 
{\bf Boundary convention:} Given an oriented sc-Fredholm section functor $f:\Xx\to\Ww$ and a boundary stratum $Y\subset \partial X$, the restriction $f|_Y$ inherits a natural orientation as follows:  

Each $y\in Y_\infty$, considered as boundary point of $X$ has a local model $(0,0)\in r(U)$, where $r : U\to U$  is a sc-smooth retraction on an open subset $U\subset [0,\infty)\times E$. 
Here tameness guarantees $U_\partial:= U \cap \{0\}\times E = \{u\in U \,|\, d_U(u)=1\}$ to be $r$-invariant, inducing a local model $(0,0)\in r(U_\partial)$ for $y\in Y$ with $\rT_y Y \simeq \rD_{(0,0)} r (\{0\}\times E) \subset \im \rD_{(0,0)} r \simeq \rT_y X$. 
To find an outward vector $\nu\in \rT_y X$ that splits $\rT_y X=\R\nu \oplus \rT_y Y$ we need the second tameness condition: This guarantees a subspace $A\subset\{0\}\times E$ such that $\R\times E = \im\rD_{(0,0)} r \oplus A$, and hence we can define $0\neq \nu\in \rT_y X \simeq \im\rD_{(0,0)} r$ to be the unique vector in the decomposition $(-1,0)= \nu + a$ with $a\in A$. Here a different choice of complement $A'\neq A$ induces another vector $\nu'\in \im\rD_{(0,0)} r$ with $\nu'-\nu \in \im\rD_{(0,0)} r \cap \{0\}\times E \simeq \rT_y Y$.

To restrict the orientation of $f$ to an orientation of $f|_Y$ at the point $y\in Y$, choose a linearization $S=\rD_y(f-s)\in {\rm LIN}(f,y)$ with $\rD_y(f-s)(-\nu)=0$ using Lemma~\ref{lem:sc-ext} below and the fact that the direct sum $\R\times E = \im\rD_{(0,0)} r \oplus A$ is in fact a sc-splitting, so that $(-1,0)\in\R\times E_\infty$ yields $-\nu=(1,a) \in \R\times E_\infty$. 
This linearization induces a restricted linearization $S_\partial:= \rD_y(f-s)|_{\rT_y Y}\in {\rm LIN}(f|_Y,y)$ with $\R\nu \oplus \ker S_\partial = \ker S$ and $\im S_\partial = \im S$, so that we can define an isomorphism $\det S_\partial \to \det S$ by mapping $(a_1\wedge\ldots \wedge a_k) \otimes (b_1\wedge\ldots\wedge b_c)^*$ to  $(\nu\wedge a_1\wedge\ldots \wedge a_k) \otimes (b_1\wedge\ldots\wedge b_c)^*$. 

A different choice of outward vector $\nu'\in \nu+\rT_y Y$ gives rise to a continuous family $\nu_\tau := \nu + \tau(\nu'-\nu)\in(\rT_y X)_\infty$ and the construction of Lemma~\ref{lem:sc-ext} yields a continuous family of sc$^+$-sections $s_\tau$ with $s_\tau|_Y$ independent of $\tau\in[0,1]$ and $\rD_y(f-s_\tau)(\nu_\tau)=0$. 
So we obtain a continuous family of isomorphism $\det S_\partial \to \det S_\tau$ mapping 
$$(a_1\wedge\ldots \wedge a_k) \otimes (b_1\wedge\ldots\wedge b_c)^*$$ to  
$$(\nu_\tau\wedge a_1\wedge\ldots \wedge a_k) \otimes (b_1\wedge\ldots\wedge b_c)^*,
$$ 
and can conclude that the induced (inverse) map from orientations of $f$ to orientations of $f|_Y$ at $y$ is independent of choices -- except for the crucial ordering and outward convention that we fixed here. 
This convention is moreover compatible with the boundary orientation convention in \cite[Def.9.3.8]{TheBook} for solution sets (resulting from multi-valued perturbations). 

An analogous extension construction for the stabilizations in \cite[Lemma~6.6.4-6]{TheBook} shows that this pointwise convention is compatible with orientation propagation, thus induces a well defined orientation of $f|_Y$.

\begin{remark}\label{rmk:interval-orient} \rm 
The orientation conventions for products and boundary restrictions interact as follows: 
Given a sc-Fredholm section functor $f:\Xx\to\Ww$ with orientation $\fo:X_\infty\to\Oo_f$, the restrictions of its constant extension $\Hat f: [0,1]\times\Xx\to\Ww$ to the boundary strata $\Xx\to\{i\}\cong\Xx$ for $i=0,1$ are
$\Hat f |_{\{0\}\times\Xx} \cong f$ with orientation $-\fo$ and $\Hat f |_{\{1\}\times\Xx} \cong f$ with orientation $\fo$.
This identification results from the isomorphisms
$\det(S) \to \det(S \Pi =\Hat S)$ and $\det(\Hat S) \to \det(\Hat S_\partial = S)$ 
mapping
$(a_1\wedge\ldots \wedge a_k) \otimes (b_1\wedge \ldots \wedge b_c)^* \mapsto (1 \wedge a_1\wedge\ldots \wedge a_k) \otimes (b_1\wedge \ldots \wedge b_c)^*$ 
and 
$(\nu \wedge a_1\wedge\ldots \wedge a_k) \otimes (b_1\wedge \ldots \wedge b_c)^* \mapsto (a_1\wedge\ldots \wedge a_k) \otimes (b_1\wedge \ldots \wedge b_c)^*$ 
with outward vector $\nu = -1$ for $i=0\in \partial [0,1]$ and outward vector $\nu = 1$ for $i=1\in \partial [0,1]$.   
 \hfill$\er$
\end{remark}

To state and prove existence of the sc$^+$-sections used in the boundary convention we assume familiarity with \cite[ch.1-2]{TheBook}.

\begin{lemma} \label{lem:sc-ext}
Consider a strong bundle retraction $$
R:U\times F\to U\times F, (u,f)\mapsto (r(u),\Ga(u) f)
$$
 covering a sc-smooth retraction $r:U\to U$ on an open neighbourhood $U\subset[0,1)\times E$ of $(0,0)=r(0,0)$, and a sc-smooth section of this local bundle given by $f: r(U) \to F$ satisfying $f(u)\in\Ga(u)F$. 
Given any $\nu=(1,a)\in \im\rD_{(0,0)} r \cap \R\times E_\infty$ there exists a sc$^+$-section $s$ of this local bundle such that $s(0,0)=f(0,0)$ and $\rD_{(0,0)} (f-s)(\nu) = 0$.

\end{lemma}
\begin{proof}
Note that we have $\nu=\dot\gamma(0)$ for the sc-smooth path $\gamma:[0,\eps)\to  r(U), t\mapsto r( t\nu ) $ as in \cite[Lemma 2.4.9]{TheBook} for some $\eps>0$.  Since $[0,\eps)$ consists of smooth points and sc-continuous maps preserve sc-levels, this path takes values in the smooth points, $\gamma(t)=r(t\nu)\in [0,1)\times E_\infty$. 
Moreover, sc-smoothness with this finite dimensional domain is equivalent to smoothness as map $[0,\eps)\to[0,1)\times E_\ell$ for any $\ell\in\N_0$. 
So composing this path with the sc-smooth map $f$ yields a path $[0,\eps)\to F_\infty, t\mapsto f(\gamma(t))$ that is sc-smooth w.r.t.\ any shift of the sc-structure on $F$. 

We now construct a map $s:r(U)\to F$ by $s(t,e):= \Ga( t,e ) f(\gamma(t))$ for $(t,e) \in r(U)$. 
This is a section since $s(u)\in\Ga(u)F$. It is sc-smooth and sc$^+$ since its extension $\Ti s: U \to F, u \mapsto s(r(u))$ can be viewed for each $\ell\in\N_0$ as the composition of sc-smooth maps $r:U\to U$, the linear map $U\to U\times\R, (t,e)\mapsto ((t,e),t)$, 
the map $U\times\R \to U\times F_l, (u,t)\mapsto (u, f(\gamma(t)))$, and $U\times F_l \to F_l, (u,f)\mapsto \Gamma(u)f$.
It satisfies the base point condition $$
s(0,0)= \Ga(0,0) f(\gamma(0))=\Ga(0,0) f(0,0)=f(0,0),
$$
 and we can compute with $g(t):={\pr}_\R(\gamma(t))$ 
\begin{align*}
\rD_{(0,0)}(f-s)(\nu)
&=\tfrac{\rd}{\rd t}\big|_{t=0} (f-s)(\gamma(t)) \\
&= \lim_{t\to 0} \tfrac{ f(\gamma(t)) - s(\gamma(t))}{t}\\
&= \lim_{t\to 0}   \Ga (\gamma(t))  \tfrac{ f(\gamma(t)) - f(\gamma(g(t))  ) }{t} = 0 . 
\end{align*}
To prove the last step, we first claim that $\gamma(t)=r(t\nu)=t\nu + O(t^2) \in \R\times E_0$ and hence $g(t)=t + O(t^2)$. 
To check this note that $r:U\cap \R\times E_1 \to \R\times E_0$ is continuously differentiable so that $r(t\nu)=r(0,0)+t \rD_{(0,0)} r (\nu)+ O(t^2)$. Here $r(0,0)=(0,0)$ and $\rD_{(0,0)} r (\nu) = \nu$ follow from $r\in r(U)$ resp.\ $\nu\in\im\rD_{(0,0)}r$ and the retraction property $r\circ r=r$. 
Next, sc-smoothness of $f\circ\ga$ includes continuous differentiability as map $[0,\frac 12 \eps]\to F_0$ with uniformly bounded derivative. Thus we can estimate for $t\to 0$
$$
t^{-1} \| f(\gamma(t)) - f(\gamma(g(t))  )\|_{F_0} \;\leq\; t^{-1} C | t - g(t) | \;=\; C t^{-1} O(t^2) = O(t)  \;\to\; 0 .
$$ 
Finally, continuity of $\Gamma: U\times F_0 \to F_0$ confirms the last step in $\rD_{(0,0)}(f-s)(\nu)=0$. 
\end{proof}

%
%
%

\section{\'Etale categories and groupoid completions}\label{sec:basic}

This section develops the general theory of \'etale categories in order to have a suitable language in which to discuss the structure and properties of the groupoids $\Xx_\Vv$ and $\Xx_\Vv^{\less G}$ that  appear in Theorem~\ref{thm:globstab}. 
After a general introduction to \'etale categories in \S\ref{ss:basic},
 \S\ref{ss:preunif} defines the notion of preuniformizer, generalizing the local uniformizers considered in \cite{TheBook}.  Lemma~\ref{lem:prop1} gives conditions that imply the properness of a groupoid $\Xx$ and hence, by Corollary~\ref{cor:prop1}, the Hausdorffness of $|\Xx|$.  
This subsection also develops the notion of the groupoid completion of an \'etale category, which is used to construct $\Xx_\Vv^{\less G}$. Then  \S\ref{ss:groupact}  discusses the properties
 of group actions on \'etale categories such as the action of $G$ on $\Xx_\Vv^{\less G}$ in Theorem~\ref{thm:globstab}. 
 The next subsection explains the basic properties of bundles and their sections in more detail, with special mention of the properties of strong bundles in the polyfold setting, that is in \'etale class (d). Globally structured multisections are defined in Definition~\ref{def:multisdef0}, and Proposition~\ref{prop:structmulti} shows that such sections are structurable in the sense of \cite[Def.13.3.6]{TheBook}.
Finally, 
\S\ref{ss:polybundle} concentrates on the polyfold case, summarizing the notions required for ensuring the compactness of perturbed zero sets.

\subsection{\'Etale categories}\label{ss:basic}

Throughout this paper we will work with topological categories. Thus we consider small categories $\Bb$ whose sets 
$B:= \Obj_\Bb$ and $\bB:= \Mor_{\Bb}$ of objects and morphisms   are equipped with topologies such that the source and target maps $s,t: \bB\to B$ as well as unit $u:B\to\bB$ and composition $$
c:\{(m,m')\in\bB\times\bB\,|\, t(m)=s(m')\} \to \bB,\quad (m,m')\mapsto m\circ m',
$$
 are continuous.  Further, any functor between topological categories is assumed continuous on object and morphism spaces.
When working with a groupoid $\Xx = (X,\bX)$ we also assume that the inverse operation $i: \bX\to \bX$ is continuous, i.e.\ $\Xx$ is a topological groupoid. 
For notational clarity we reserve the letter $\Xx$ to denote \'etale groupoids throughout this paper.  
In that case the relation 
$$
x_1\sim_\Xx x_2  \quad :\Longleftrightarrow\quad \Mor_\Xx(x_1,x_2)\neq\emptyset
$$ 
is an equivalence relation, and this defines the {\bf realization} 
$|\Xx|:=\qu{X}{\sim_\Xx}$ as topological space, equipped with the quotient topology. 
Here we denote the set of morphisms between subsets $U_1,U_2\subset X$, for example $U_i=\{x_i\}$, by 
$$
\Mor_\Xx(U_1,U_2) \,:=\; s^{-1}(U_1)\cap t^{-1}(U_2) \;\subset\; \bX .
$$
For more general categories $\Bb$ which may not have inverses,
the realization $|\Bb|:=\qu{B}{\sim_\Bb}$  is defined by the equivalence relation $\sim_\Bb$ on $B$ that is generated by the morphisms.
Again, we equip $|\Bb|$ with the quotient topology such that the quotient map is continuous, 
\begin{align}\label{eq:piB}
\pi_\Bb \, : \; B \; \to\; |\Bb|=\qu{B}{\sim_\Bb} , \qquad x \;\mapsto\; |x| := [x]  .
\end{align}
For any subset $V\subset B$ we denote its image in the realization by $|V|:=\pi_\Bb(V)\subset |\Bb|$, and note that it carries two natural topologies which may not coincide\footnote{
For example, consider $\Bb$ with objects $B = [0,1]\times \{\pm1\}$ and non-identity morphisms
$(x,-1)\mapsto (x,+1)$ and $V = ([0,1/2]\times \{-1\}) \cup ((1/2,1]\times \{1\})\subset B$. 
Then $\pi_\Bb(V)=|\Bb|$ is homeomorphic to the interval $[0,1]$, while $\qu{V}{\sim_\Bb}$ is disconnected in the quotient topology.}: the relative topology from the inclusion into $|\Bb|$ and the quotient topology from the identification $|V| \simeq \qu{V}{\sim_\Bb}$.
However, when $V\subset B$ is open or saturated (that is $\pi_\Bb^{-1}(\pi_\Bb(V))=V$) then the topologies agree by Lemma~\ref{lem:topologies}. 

In the following, we will always work with categories that are \'etale in the sense of the following Definition~\ref{def:etale}. For that purpose we fix a class of metrizable topological spaces and a corresponding class of local homeomorphisms that is closed under compositions and taking local inverses. In particular, all our results will hold for any of the following classes of \'etale spaces and \'etale maps:
\begin{itemize}
\item[(a)] metrizable topological spaces and local homeomorphisms;
\item[(b)]  topological manifolds (assumed to be finite dimensional and second countable -- possibly with specified boundary and corner structure, see below) and local homeomorphisms;
\item[(c)]  smooth manifolds (assumed to be finite dimensional and second countable  -- possibly with boundary and corners, see below) and local diffeomorphisms;
\item[(d)]  M-polyfolds and local sc-diffeomorphisms (see \S\ref{ss:scale}).
\end{itemize}

Unless otherwise specified,  we allow manifolds and M-polyfolds to have boundary and corners, arising from local models in $[0,\infty)^k\times E$ as in e.g.\ \cite[\S 5.3]{usersguide}, giving rise to a degeneracy index $d_X:X\to\N_0$.  
The corresponding \'etale maps are required to restrict to local homeomorphisms resp.\ diffeomorphisms on each {\bf boundary and corner stratum} -- that is connected components of $\{x\in X \,|\, d_X(x)=\ell\}$ for fixed $\ell\geq 1$.

\begin{remark} \label{rmk:open} \rm
\begin{nenumilist}
\item
 A map $\phi:X\to Y$ is called a {\bf local homeomorphism (or diffeomorphism)} if for each $x\in X$ there exist open neighbourhoods\footnote
 {
 We always assume that neighbourhoods are open sets; sometimes, for clarity, we mention this explicitly.}
  $N_x\subset X$ of $x$ and $N_y\subset Y$ of $y:=\phi(x)$ so that the restriction $\phi|_{N_x}:N_x\to N_y$ is a homeomorphism (or diffeomorphism) with respect to the induced topology (or smooth structure) on the open subsets $N_x, N_y$. 

\item
This convention implies that every \'etale map is open. 
Indeed, given an open subset $U\subset X$, we check that its image $\phi(U)\subset Y$ is open by finding for each image $y=\phi(x)$ of some $x\in U$ an open neighbourhood $N'_y\subset Y$ contained in $\phi(U)$. 
We obtain $N'_y:=\phi(N_x\cap U)$ from a choice of open neighbourhoods $N_x\subset X$ of $x$ and $N_y\subset Y$ of $y$ so that $\phi|_{N_x}:N_x\to N_y$ is a homeomorphism. 
Here $N'_y=\phi(N_x\cap U)$ is open in $N_y$ because it is the image under a homeomorphism of the open subset $N_x\cap U \subset N_x$. Finally, openness of $N'_y\subset N_y$ implies openness of $N'_y\subset Y$ since $N_y$ is open. 

\item
The topological categories associated to Kuranishi atlases or structures are almost never \'etale, as the local homeomorphism property of source and target maps rules out overlaps between Kuranishi charts whose domains have different dimension.  \hfill$\er$
\end{nenumilist}
\end{remark}

Once such a class is chosen, we will refer to {\bf \'etale spaces} and {\bf \'etale maps} as the spaces and maps in this class, and define the notion of \'etale category for any chosen class as follows.

\begin{defn}  \label{def:etale}
A topological category $\Bb$ is said to be {\bf \'etale}  
w.r.t.\ a chosen class as above, if its object and morphism spaces $B,\bB$ are \'etale spaces and the structure maps $s,t,c,u$ are \'etale maps.  
If $\Bb$ is a groupoid, then we require in addition that the inverse $i$ is \'etale.  
\end{defn}

\begin{remark}\label{rmk:etale}\rm  
The requirement that composition be \'etale makes sense since each \'etale class is closed under fiber products in the following sense:
Three \'etale spaces $U,V, W$ and two \'etale maps $t:U\to W, s: V\to W$ yield a fiber product 
$U\leftsub{t}{\times}_s V = \big\{(x,y)\in U\times V\ \big| \ t(x) = s(y)\big\}$
that also is an \'etale space in this class. Further, the projections $U\leftsub{t}{\times}_s V \to U$ and $U\leftsub{t}{\times}_s V \to V$ are \'etale maps in the same class. 

Indeed, a neighbourhood of $(x_0,y_0)\in U\leftsub{t}{\times}_s V$ 
consists of points $(x,y)$ where $x\approx x_0$ and $y = s\circ t^{-1}(x)$ for a suitable local   inverse $t^{-1}$ to $t$.
A more detailed statement and proof including boundary and corner structures is \cite[Prop.5.2.20]{TheBook} for the polyfold case. 
 \hfill$\er$
\end{remark}

We will moreover use the following standard notions. 

\begin{defn}\label{def:etale2}   
\begin{nenumilist}
\item
 A functor $f: \Bb\to \Bb'$ between  \'etale categories  is called {\bf \'etale} if the induced maps on object and morphism spaces are  \'etale.  
 \item
 A category $\Bb$ is said to be {\bf nonsingular} if for any two objects $x,y\in B$ there is at most one morphism in $\Mor_\Bb(x,y)$. 
 \item  A nonsingular category $\Bb$ is called a {\bf poset} if the object space $B$ is equipped with a partial order $\preccurlyeq$ such that $x\preccurlyeq y$ iff $\Mor_{\Bb}(x,y)\ne \emptyset$.
\item If for some $x\in B$ the semigroup\footnote
{
A semigroup is a set $S$ with an associative operation $S\times S\to S$ that has an identity element but is not required to have inverses.}
 $\Mor_\Bb(x,x)$ is a group, then we call $G_x:=\Mor_\Bb(x,x)$  the {\bf isotropy group} at $x$. 
 \item
 A topological category $\Bb$ is said to be {\bf locally injective} if $\pi_\Bb: B \to |\Bb|$ is locally injective, i.e.\ every $x\in B$ has a neighbourhood $N_x$ such that $\pi_\Bb: N_x\to |\Bb|$ is injective.
\item 
A groupoid
$\Xx$ is called  {\bf proper} if it satisfies the properness condition of  Definition~\ref{def:poly}(i), i.e.\ every $x\in X$ has a neighbourhood $N(x)$ such that the restricted target map $$
t:s^{-1}\bigl( {\rm cl}_X(N(x))\bigr) \to X
$$
 is proper.\footnote{
A map $f:X\to Y$ is said to be {\bf proper} if $f^{-1}(K)$ is compact for every compact subset $K\subset Y$.}
\item 
An \'etale functor  $f:\Yy\to \Xx$ between groupoids is said to be an {\bf equivalence} if it induces isomorphisms $\Mor_{\Yy} (y,y') \simeq \Mor_{\Xx} (f(y),f(y'))$ between morphism sets and a homeomorphism $|f|:|\Yy|\to |\Xx|$ of the realizations.
 \end{nenumilist}
 \end{defn} 

The various properness notions in the literature are compared in Lemma~\ref{lem:prop1}, and Corollary~\ref{cor:prop1} shows that the realization of  a proper \'etale groupoid $\Xx$ is Hausdorff.

If we consider an \'etale proper groupoid using the class of M-polyfolds and local sc-diffeomorphisms, then the above definition is not quite the same as the notion of ep-groupoid in Definition~\ref{def:poly}, which requires only sc-smoothness of the maps $c, u, i$, rather than local sc-diffeomorphisms. 
However, the next lemma shows that \'etale proper groupoids in this class are the same as ep-groupoids in the sense of polyfold theory \cite[Ch.7]{TheBook}.\footnote
{
For clarity, we will use the word ep-groupoid only if we are working in the 
class of M-polyfolds.}

\begin{lemma}\label{lem:etale}   
Let $\Bb$ be an \'etale category modelled on M-polyfolds in the sense of Definition~\ref{def:poly}.  Then the unit  and composition maps are local sc-diffeomorphisms.  
If in addition $\Bb$ is a groupoid then the inverse map is also a sc-diffeomorphism. 
Hence $\Bb$ is \'etale in the sense of Definition~\ref{def:etale}.  
\end{lemma}
\begin{proof}
It suffices to exhibit local inverses of the  structure maps of $\Bb$.  For the unit $u:B \to \bB, x \mapsto \id_x$ we know that $s\circ u =\id_B$, where $s:\bB \to B$, is a local sc-diffeomorphism; hence $s$  provides the local inverse for $u$. 
For the composition $c: \bB \leftsub{t}{\times}_s \bB \to \bB, (m,m')\mapsto m\circ m'$ we can use the fact that $s\circ c=s \circ \pr_1$. Here $\pr_1: \bB \leftsub{t}{\times}_s \bB \to \bB$ is projection to the first factor, which is a local sc-diffeomorphism by construction of the fiber product $ \bB \leftsub{t}{\times}_s \bB$. Since the source map $s$ also is a local sc-diffeomorphism, this shows that the composition $c$ is a local sc-diffeomorphism. 
Finally, in the case of a  groupoid the inverse $i: \bX \to \bX, m\mapsto m^{-1}$ is its own inverse map since inverses are unique in any groupoid.  \end{proof}

\begin{remark}\label{rmk:restrict}\rm   Let $\Bb= (B, \bB)$ be a category.
For any subset $V\subset B$ we obtain a {\bf restricted category} $\Bb|_V$ by taking the full subcategory of $\Bb$ with objects $V$. 
If $\Bb$ is \'etale and $V$ is open, then Lemma~\ref{lem:topologies} below identifies the realization $|\Bb|_V| = \qu{V}{\sim_\Bb}$ (with the quotient topology) with the subset $|V|=\pi_\Bb(V)\subset|\Bb|$ (with the relative topology) as topological spaces. Moreover, the restriction $\Bb|_V$ to an open subset $V$ is also \'etale, since $\bB|_V:= s^{-1}(V)\cap t^{-1}(V)\subset \bB$ is an open subset, hence inherits \'etale structure, and all structure maps in $\Bb|_V$ are restrictions of the \'etale structure maps in $\Bb$.  
Furthermore, Lemma~\ref{lem:prop1} considers an \'etale proper groupoid  $\Xx$ and shows that the  restriction $\Xx|_V$ to an open subset $V\subset X$ is also an \'etale proper groupoid.

Although the above results apply when $V\subset B$ is any open set, in  applications  such as 
Theorem~\ref{thm:globstab}, $V$ is saturated, i.e.\ of the form $\pi_\Bb^{-1}(\pi_\Bb(V))$. 
 \hfill$\er$
\end{remark}

We will often consider subsets $Z\subset B$ that are not open but {\bf saturated}. For groupoids as in \cite[p.341]{TheBook},
this is equivalent  to requiring that the set is closed under morphisms. (For example, the zero set $Z=f^{-1}(0)\subset X$ of any sc-Fredholm section $f:\Xx\to \Ww$ is saturated.)
In general categories, we call $Z\subset B$ saturated if $Z=\pi_\Bb^{-1}(\pi_\Bb(Z))$.
In that case, as noted above, $|Z|$ carries both the relative topology given by the embedding $|Z|:=\pi_\Bb(Z)\subset|\Bb|$ and the quotient topology given by the quotient map $Z\to \qu{Z}{\sim_\Bb}=|\Zz|$ to the realization of the full subcategory $\Zz$ of $\Bb$ with objects $Z$. 
Fortunately, the two topologies agree in \'etale settings (unlike the Kuranishi situation in \cite[Def.3.1.14]{MW1}).  

\begin{lemma}\label{lem:topologies}  Suppose $Z\subset B$ is an open or saturated subset of the objects in an  \'etale category  $\Bb=(B,\bB)$.
Then the quotient topology on $|Z|=\qu{Z}{\sim_\Bb}$ agrees with the relative topology on $|Z|\subset|\Bb|$.
\end{lemma}
\begin{proof}  In both cases, we will identify  $\qu{Z}{\sim_\Bb}\cong\pi_\Bb(Z)$ so that the quotient map $\pr_Z: Z\to\qu{Z}{\sim_\Bb}$ coincides with the restriction of $\pi_\Bb$ to $Z$.
  If $Z$ is open, then by
 Lemma~\ref{lem:etale1} below, every open subset of $Z$ projects to an open subset of $|\Bb|$. 
In particular, $\pi_\Bb(Z)\subset |\Bb|$ is open. Thus the 
openness of $O\subset\pi_\Bb(Z)$ in the relative topology is equivalent to the openness of $O\subset|\Bb|$,
and (using the lemma again) to the openness of  $\pi_\Bb^{-1}(O) \cap Z = \pi_\Zz^{-1}(O)$, and hence
to the openness of $O$ in the quotient topology on $\qu{Z}{\sim_\Bb}$.


In the case of a saturated subset $Z\subset B$,
 openness of $O\subset\pi_\Bb(Z)$ in the relative topology means $O=\pi_\Bb(Z)\cap V$ for some open set $V\subset|\Bb|$. By definition of the quotient topology on $|\Bb|$ that means $\pi_\Bb^{-1}(V)\subset B$ is open, and thus $Z\cap \pi_\Bb^{-1}(V)\subset B$ is an open subset in the relative topology of $Z\subset B$. Moreover, we have $\pr_Z^{-1}(O)=\pi_\Bb^{-1}(\pi_\Bb(Z)\cap V )= Z \cap \pi_\Bb^{-1}(V )$ since $Z$ is saturated. 
This shows that $\pr_Z^{-1}(O)\subset Z$ is open, and thus $O\subset \qu{Z}{\sim_\Bb}$ is open in the quotient topology. 

Conversely, openness of $O\subset \qu{Z}{\sim_\Bb}$ in the quotient topology means $\pr_Z^{-1}(O)\subset Z$ is open. Here $Z\subset B$ is equipped with the relative topology, so $\pr_Z^{-1}(O)= Z\cap U$ for some open subset $U\subset B$. Now $\pi_\Bb(U)\subset|\Bb|$ is open by Lemma~\ref{lem:etale1} below.
Moreover, we have $\pi_\Bb(Z\cap U)=\pi_\Bb(Z)\cap\pi_\Bb(U)$ since $Z$ is saturated.
This identifies $O=\pi_\Bb(\pr_Z^{-1}(O))\subset\pi_\Bb(Z)$ as the intersection $O=\pi_\Bb(Z)\cap\pi_\Bb(U)$ with an open set, thus as open set in the relative topology of $\pi_\Bb(Z)\subset|\Bb|$. 
\end{proof}

Here and throughout, topological considerations in \'etale categories and  \'etale proper groupoids are simplified by the following.

\begin{lemma} \label{lem:etale1} 
 \begin{nenumilist}
\item 
The projection $\pi_\Bb:B\to |\Bb|$ is an open map in any \'etale category.

\item
For any \'etale proper groupoid $\Xx$ we have ${\rm cl}_{|\Xx|}(|U|)=|{\rm cl}_X(U)|$ when $U\subset X$ is saturated or $U\subset X$ is precompact.
\end{nenumilist}  
\end{lemma}
\begin{proof}  
(i) is a generalization of \cite[Prop.7.1.8]{TheBook}. 
To prove it, let $U$ be an open subset of $B$.  Then $\pi_\Bb(U)$ is open in $|\Bb|$ iff $\pi_\Bb^{-1}(\pi_\Bb(U))$ is open in $B$, i.e.\ iff every $y\in \pi_\Bb^{-1}(\pi_\Bb(U)$ has an open neighbourhood $U_y\subset \pi_\Bb^{-1}(\pi_\Bb(U))$.  For all such $y$ there is a finite chain of morphisms
 $m_1,\cdots, m_k$   in $\Bb$ with $s(m_1) = y,\ t(m_k) = y_k\in U$  as follows
  $$
y=:y_0\; \stackrel{m_1}\rightarrow\; y_1\;\stackrel{m_2} \leftarrow \; y_3\; \cdots \stackrel{m_k}\rightarrow \; y_k \in U.
 $$ 
(To achieve such a zigzag chain from an arbitrary chain,  one can  insert appropriate identity morphisms, and/or compose pairs of morphisms that go in the same direction.)
  Since each $m_i$ extends to an injective \'etale map from a neighbourhood of its source onto a
 neighbourhood of its target,
the end point $y_k=y$ has an open neighbourhood $N_y$ in $B$ consisting of points that are equivalent to some point in $x\in U$. 
Hence $\pi_\Bb(U)$ is open, because every point $y\in \pi_\Bb^{-1}(\pi_\Bb(U))$ 
has a neighbourhood $N_y\subset \pi_\Bb^{-1}(\pi_\Bb(U))$.

(ii) is proven by the same arguments as in \cite[Lemma~12.4.4]{TheBook}. 
\end{proof}

We end this section by explaining the notation used in \cite{TheBook} for bundles.
 This is sufficient  preparation for the discussion of group actions on bundles in \S\ref{ss:groupact}.
A much fuller treatment of the basic definitions is given in \S\ref{ss:bundle}.

\begin{definition}\label{def:bundle0} A  {\bf  bundle} $\Pp:\Ww\to \Xx$ over an \'etale groupoid $\Xx$
is given by an \'etale space $W$ with continuous  surjections
$P: W\to X$ and $\mu:\bX \leftsub{s}{\times}_P W\to W$ such that
\begin{nenumilist} \item
$P$ is a locally trivial vector bundle in the appropriate category.
 \item $\mu$ is  \'etale and is a lift of the target map $t:\bX\to X$ in the sense that $P\circ\mu(\cdot,w)=t$ for all $w\in W$. Further, it is compatible with the relevant linear and groupoidal structures.
\end{nenumilist} 
As in \cite[\S8.3]{TheBook}, these conditions ensure that 
$P$ lifts to a functor $\Pp:\Ww\to \Xx$ given by  
$\Pp(w) = P(w),\; \Pp(m,w) =m$, where $\Ww = (W, \bW)$ is the \'etale groupoid  with
objects $W$ and
\begin{align}\label{eq:Wbundle0}
& \bW= \bX \leftsub{s}{\times}_P W, \quad 
(s\times t)(m,w)=(w, \mu(m,w)), \\ \notag &\id_{w}= (\id_{P(w)},w),\quad 
 (m,w)\circ (m',w')= 
 \bigl(m'\circ m , w\bigr) 
\quad \mbox{ if } w' = \mu(m,w).
\end{align}

A {\bf section} of $(P,\mu)$ is an \'etale map $f:X\to W$ that satisfies $P\circ f=\id_X$, and is compatible with morphisms in the sense that  $f(t(m))=\mu(m,f(s(m)))$ for all $m\in\bX$.\end{definition}

One can check as in Lemma~\ref{lem:bundlemaps} that each section $f:X\to W$ extends to a unique \'etale functor $f:\Xx\to \Ww$.

 \subsection{Preuniformizers, properness, and groupoid completions}\label{ss:preunif}

 An important property of ep-groupoids is the existence of local uniformizers \cite[Def.7.1.20]{TheBook} centered at any $x\in X$ -- roughly speaking, a neighbourhood $\simeq \qu{U}{G_x}$ of $|x|\in|\Xx|$ given as the quotient of a natural action of the isotropy group $G_x: = \Mor_\Xx(x,x)$ on a neighbourhood $U\subset X$ of $x$. 
We will generalize this notion to \'etale categories -- i.e.\ allowing any \'etale class from \S\ref{ss:basic}, and not requiring groupoid structure. 
Throughout, we consider left actions of a group $G$ on a space $U$, that is
$(hg)*x= h*(g*x)$ for $h,g\in G, x\in U$. 
 
\begin{defn}\label{def:preunif}  
A {\bf  preuniformizer} for an \'etale category $\Bb=(B,\bB)$ consists of
\begin{nenumilist}
\item an open  subset $U\subset B$;
\item an \'etale action $G \times U \to U, (g,y)\mapsto g*y$ of a finite  group $G$;
\item an \'etale injection $\Ga: G \times U \to \Mor_\Bb(U,U)\subset \bB$ satisfying the following:  
\begin{enumlist}
 \item 
 $\Ga$ identifies the morphisms in its image with the group action in the sense that\footnote{The inversion of order in the last identity compared with \cite{TheBook} results from our use of categorical notation for the composition $\circ$ of morphisms, while group actions are left actions both here and in \cite{TheBook}.} 
\begin{align}\label{eq:groupact}
 s\bigl(\Ga(g,y)\bigr)=y,\qquad t\bigl(\Ga(g,y)\bigr)=g*y,  \qquad  \Ga(hg,y)= \Ga(g,y)\circ \Ga(h, g*y); 
\end{align}
\item $\Ga(G\times U)$ contains all connected components of $\Mor_\Bb(U,U)$ that have a morphism $m$ with $s(m) = t(m)$;
\item  the induced map $|\Ga|:\qu{U}{G}\to |\Bb|$ is locally injective and has bounded fibers in the sense that 
$\sup_{p\in |U|} \#\{[x] \in \qu{U}{G} \,|\, |x|=p  \} <\infty$. 
\end{enumlist}
\end{nenumilist}
We will  abbreviate this data by $(U,G, \Ga)$ or even $(U,G)$, and also identify it with the induced  faithful functor of the {\bf translation groupoid}\footnote{
A category of the form $(U, G\times U)$ is called a translation groupoid in \cite[p.294 in \S7.1]{TheBook}. }
 $(U, G\times U)$ into the subcategory $(U, \Mor_\Bb(U,U))$ of $\Bb$.   The induced map $|\Ga|: \qu{U}{G}\to |\Bb|$ is called the {\bf footprint map}. 

If $\Xx$ is an \'etale groupoid and $\Ga: G \times U \to \Mor_\Xx(U,U)$ is surjective, then, as in \cite[Def.7.1.20]{TheBook}
$(U,G, \Ga)$ is called a {\bf local uniformizer}.   
A  {\bf local uniformizer around} $x\in X$ is a local uniformizer $(U,G, \Ga)$ such that $x\in U$ and we have $t\circ\Ga(\cdot,x)\equiv x$. In particular, $\Ga(\cdot,x):G\to G_x=\Mor_\Xx(x,x)$ identifies $G$ with the isotropy group of $x$. 
\end{defn}

\begin{rmk}\label{rmk:locunif}\rm  
\begin{nenumilist}
\item
If $\Xx$ is an ep-groupoid, then local uniformizers for $\Xx$  exist near any object $x\in X$ by \cite[Prop.7.1.19]{TheBook}, where the isotropy group $G\simeq \Mor_\Xx(x,x)$ is finite by \cite[Prop.7.1.12]{TheBook}. Moreover, by lifting a neighbourhood basis of $|x| \in |\Xx|$, 
we may choose the domain $U$  to be arbitrarily small.
This proof, although given in the polyfold case, also applies to the \'etale proper groupoids in the other \'etale categories considered here.

\item
Our notion of preuniformizer is significantly more general than that of a local uniformizer; in particular the footprint map $| \Ga|:  \qu{U}{G}\to |\Bb|$ is not in general injective -- though it is assumed locally injective, and is locally surjective by Lemma~\ref{lem:unif}~(ii). 

However, our notion of local uniformizer is the same as in \cite[Def.7.1.20]{TheBook}; in particular, in this case $|\Ga|$ is a homeomorphism to its image by Lemma~\ref{lem:unif}~(iii).
\hfill$\er$
\end{nenumilist}
\end{rmk}

Here are some general properties of preuniformizers and thus local uniformizers.

\begin{lemma} \label{lem:unif}   
The following holds for any preuniformizer $(U,G, \Ga)$ in an \'etale category~$\Bb$.

\begin{nenumilist}
\item
For all $x\in U$, the semigroup\footnote{See the footnote in Definition~\ref{def:etale2}~(iv).} 
$\Mor_\Bb(x,x)$ is isomorphic via $\Ga(\cdot,x)$ 
to the subgroup  $G_x = \{g\in G \, \big|\, t(\Ga(g,x))=x\}\subset G$. 
In particular, each $\Mor_\Bb(x,x)$ is a group, and these groups are isomorphic as $x$ varies in an orbit of $G$.
\item
The subset $|U|\subset |\Bb|$ is open and the induced map $|\Ga|: \qu{U}{G} \to |U|$ is a local homeomorphism
with respect to the quotient topology on $\qu{U}{G}$ and the subspace topology on $|U|\subset |\Bb|$.
 \item   If  $\Bb$ is a groupoid, and $(U,G, \Ga)$ is  a local uniformizer, then
  $|\Ga|: \qu{U}{G} \to |U|$ is injective and hence a homeomorphism.
 \end{nenumilist}
 \end{lemma}
 
 \begin{proof}   
 Since $\Ga$ intertwines products in $G$ with composition of morphisms, it induces an injective homomorphism $\Ga_x: G_x \to  \Mor_\Bb(x,x) \subset G_x$.
Condition (iii)(b) in Definition~\ref{def:preunif} implies that $\Ga_x: G_x \to  \Mor_\Bb(x,x)$ is surjective and hence an isomorphism, which implies that
$G_x\simeq \Mor_\Bb(x,x)$ for all $x\in U$. 
The second claim in (i) then holds because it holds for the translation groupoid $(U, G\times U)$.   
\MS
 
 To prove (ii), notice first  that $|U|$ is open in $|\Bb|$ because  the subset $U\subset B$ is open by condition (i) in Definition~\ref{def:preunif},
and  the \'etale property of $\Bb$ implies that  $\pi_\Bb$ is an open map 
 by Lemma~\ref{lem:etale1}.   
Further, the quotient map $\si:U\to \qu{U}{G} $ is open.
Since the morphisms in the translation groupoid $(U, G\times U)$  inject into those of $\Bb$, each $G$-orbit maps to a single equivalence class of $\sim_\Bb$.  Hence there is an induced surjective map $|\Ga|:\qu{U}{G} \to |U|$.   To see that $|\Ga|$ is continuous, 
consider $O': = |\Ga|^{-1}(O)$ where $O\subset |U|$ is open.  Then $\si^{-1}(O') = \pi_\Bb^{-1}(O)\cap U$ is an open set, since $\si: U\to\qu{U}{G} $ is an open map, so that  $O' = \si(\si^{-1}(O'))$ is also an open set.  

It remains to show that each $y\in U$ has an open neighbourhood $N$  such that $|\Ga||_{\si(N)}$ is a homeomorphism to its image.  
To prove this, we adapt the arguments in  \cite[Prop.7.1.15,7.1.19]{TheBook}.   
Given $y\in U$ 
choose an open neighbourhood $N\subset U$ such that $|\Ga|: \si(N)\to |U|$ is injective.  By shrinking $N$ we may suppose that it is $G_y$-invariant and disjoint from its images under  all $g\in G\less G_y$.  Then $\si:N\to \si(N)\subset \qu{U}{G}$ factors through a homeomorphism $\qu{N}{G_y} \stackrel{\simeq}\to \si(N)$.
We claim that the continuous map $|\Ga||_{\si(N)}:\si(N)\to |N|$ is  a homeomorphism onto $|N|\subset |U|$ (which is an open set since $\pi_\Bb$ is an open map).  Since $|\Ga||_{\si(N)}$  is injective by construction we need only check that its inverse is continuous.  Thus for each open set $O\subset \si(N)\subset \qu{U}{G}$ we must check that $|\Ga|(O)$ is open.
But the open set $\si^{-1}(O)$ maps onto $O$, so that $|\Ga|(O) = \pi_\Bb (\si^{-1}(O))$ and hence is open because $\pi_\Bb$ is an open map. This completes the proof of (ii).\MS

The hypothesis that $\Bb$ is a groupoid implies that two points $x,y$ in $U$ are equivalent under $\sim_\Bb$ if and only if there is a morphism $m\in \bB$ with $s(m) = x, t(m)=y$.  Since this morphism lies in the full subcategory $\Bb|_U$ which is isomorphic to the translation groupoid $(U,G\times U)$, the map $\qu{U}{G} \to |U|$ is injective.  Hence (iii) follows from (ii).
 \end{proof}
  
We now give an example of an \'etale category $\Bb$ that is the union of two preuniformizers, both with trivial groups $G$.  We show in (ii) that although $|\Bb|$ is not usually Hausdorff, the category $\Bb$ is proper in the sense of Definition~\ref{def:etale2}, which illustrates the fact that properness is not a useful notion  in the context of \'etale categories that are not groupoids.
The setup is slightly more elaborate than is needed here to make our point, but it is a special case  of the constructions discussed in \S\ref{ss:Qq}.

\begin{example}\label{ex:branch}\rm
\begin{nenumilist}
\item
Let $\TV\subset U\subset \R^k$ be nested open subsets,
and suppose that $G$ is  a finite group that acts freely on $\TV$ in such a way that the action extends freely to an open  neighbourhood $U'\subset U$ 
of the closure $\cl( \TV)$ of $\TV$ in $U$.
Let  $\rho:\TV\to V: = \qu{\TV}{G}$ be the quotient map, and 
consider  the \'etale category $\Bb$ with objects $V \sqcup U$ and morphisms consisting of the identity morphisms  together with 
  a component\footnote
  {
  To be consistent with the notation in \cite{MW1,MW2,MWiso} we will interpret these morphisms as going from $V$ to $U$, though in some ways it might be more natural to think of them as going in the other direction.}
   $\Mor_\Bb(V,U)\simeq \TV$ of morphisms from $V$ to $U$ that  \lq\lq quotients out by the action of $G$ on $\TV$."  More precisely, if we denote the elements of $\TV$ by $y$ and suppress the inclusion map $\TV\to U$ 
 we have
  $$
\Mor_\Bb(V,U) \simeq   \TV,\quad   s\times t\,: \TV\to V\times U, \;\;  y\mapsto  \bigl(\rho(y), y \bigr).
  $$
While the projection map $\pi_\Bb|_V$ is injective (where $\pi_\Bb$ is as in \eqref{eq:piB}),
 the map  $\pi_\Bb|_U$   is not, because the points in $|\TV|=|V|\subset |\Bb|$ 
 have $\# G$ preimages in $U$.  
Nevertheless,  the map $\pi_\Bb|_U$ is locally injective because $G$ acts freely on $U'$.  Thus $\Bb$ is covered by two preuniformizers with domains $V, U$ respectively and trivial groups $G = \{\id\}$.

\item
We claim  that if $\Bb$ is as in (i) its realization $|\Bb|$ is not Hausdorff unless $G=\{\id\}$ or the frontier $\Fr_{U}(\TV): = \cl (\TV)\less \TV$ is empty.  To see this, suppose given $x\in \Fr_{U}(\TV)\subset U$ and $\id\ne g\in G$. Then $x, g(x)$ are two distinct points in
$U\less V$, and have distinct images $|x|, |g(x)|$ in $|\Bb|$. Now choose any neighbourhoods $N_{|x|}, N_{|g(x)|}$ of $|x|, |g(x)|$ in $|\Bb|$ and consider their inverse images $N_x$, $N_{g(x)}$ in $U$. Since these sets are open and contain $x, g(x)$ respectively, both $g(N_x)$ and $N_{g(x)}$ are neighbourhoods of $g(x)$ in $U$.  Therefore  $g(N_x)\cap N_{g(x)}$ is a neighbourhood of $g(x)$, and hence must have nonempty intersection with $\TV$. But if $y\in g(N_x)\cap N_{g(x)}\cap \TV$, then
$|y|\in N_{|x|} \cap  N_{|g(x)|}$
because $|N_x| = N_{|x|}$ and  $|N_{g(x}| = N_{|g(x)|}$ by construction.     Therefore 
 $N_{|x|}, N_{|g(x)|}$ are never disjoint.
  
On the other hand, the map $ s\times t\,: \TV\to V\times U$ is proper because the map $\rho: \TV\to V$ is surjective and proper.  Further,  every point $x\in B = V\sqcup U$ does have a neighbourhood $N(x)$ such that the map $t: s^{-1}(\cl_B N(x))\to B$ is   proper. 
 Indeed, if $x\in U$  and we choose $N(x)\subset U$, the only morphisms in $s^{-1}(\cl_{B} N(x))$ are identity maps, and the existence of $N(x)$  is clear.  Further,  since
  $\rho(\TV) = V$ then for any $x\in V$ we may take $N(x)\subset V$ to be compact, so that  $s^{-1}(\cl_{B} N(x)) = (\cl_{V} N(x))\sqcup \rho^{-1}(\cl_{V} N(x))$;  thus   $t: s^{-1}(\cl_{B} N(x)) \to V\sqcup U$ is proper.
  \hfill$\er$
\end{nenumilist}
\end{example}

We next discuss various conditions on a topological groupoid $\Xx: =(X,\Xx)$
 that guarantee that its realization is Hausdorff.  
We will assume that the object space $X$ is first countable, but we do not assume that $\Xx$ is \'etale.

\begin{lemma}\label{lem:prop1} 
Let $\Xx$ be a topological groupoid such that $X$ is first countable.

\begin{nenumilist}
\item 
If the space $X$  is locally compact and 
$s\times t: \bX\to X \times X$ is proper, then $\Xx$ is proper 
in the sense of Definition~\ref{def:etale2}.

\item  
If  $\Xx$ is proper in the sense of Definition~\ref{def:etale2}, then 
the map $s\times t: \bX\to X \times X$ is proper and has closed graph. 
%

\item 
If $\Xx$  is \'etale and has local  uniformizer  $(U,G)$, then for every closed subset $C$ of $ U$ 
such that $|C|$ is closed in $|\Xx|$ the map $t: s^{-1}(C)\to X$ is proper.

\item  
If the map $\pi: X\to |\Xx|$ is open\footnote{For \'etale groupoids this is guaranteed by Lemma~\ref{lem:etale1}~(i).} and $s\times t: \bX\to X\times X$ has closed graph, then $|\Xx|$ is Hausdorff.
\end{nenumilist}
\end{lemma}
\begin{proof}   
To prove (i) suppose that $s\times t: \bX \to X \times X$ is proper. Then, the local compactness assumption on $X$ provides for any $x\in X$ a neighbourhood $N(x)\subset X$ whose closure ${\rm cl}_X(N(x))$ is compact. Now, given a compact subset $K\subset X$, its 
preimage under $t:s^{-1}\bigl( {\rm cl}_X(N(x))\bigr) \to X$ is 
$$
s^{-1}\bigl( {\rm cl}_X(N(x))\bigr) \cap t^{-1}(K) =  (s\times t)^{-1} \bigl( {\rm cl}_X(N(x)) \times K \bigr)
 \subset \bX,
 $$
which is also compact by properness of $s\times t$ and compactness of the Cartesian product ${\rm cl}_X(N(x)) \times K\subset X\times X$. 
\MS

To prove the first claim in (ii), suppose that every $x\in X$ has a neighbourhood $N(x)\subset X$ such that $t:s^{-1}\bigl( {\rm cl}_X(N(x))\bigr) \to X$ is proper, and denote by $(\pr_i)_{i=1,2}$ the projection of $X\times X$ onto its $i$th factor. Then the preimage in $\bX$ of a compact subset $K\subset X\times X$ under $s\times t: \bX \to X \times X$ is compact because 
$s^{-1}\bigl( {\rm cl}_X(N(x))\bigr)\cap t^{-1} \bigl(\pr_2(K)\bigr)$ is compact for every  $x$ by choice of $N(x)$, and the compactness of $\pr_1(K)$ implies that it  has a finite covering by sets of the form $N(x)$.

To prove the second claim in (ii),
 it suffices to observe that if $s\times t$ does not have closed graph, there is a convergent sequence $(x_k,y_k) \to (x_\infty, y_\infty)$ in $X\times X$, with $(x_k,y_k)\in \imm(s\times t)$, but 
$(x_\infty, y_\infty) \notin  \imm(s\times t)$.  Then the compact set $K =\{(x_k,y_k) | k\in\N\} \cup \{ (x_\infty, y_\infty)\}$ does not have compact preimage. 

%
\MS

The statement in (iii) generalizes  \cite[Prop.7.5.10]{TheBook}, which characterizes the  
``proper'' subsets of an ep-groupoid $\Xx$  (subsets $C$ for which the map $t: s^{-1}(C)\to X$ is proper). 
Here we do not assume $\Xx$ to be proper, and so need to give a direct proof.
Let $K\subset X$ be compact.  Since  $U\cap \pi^{-1}(|C|)= \bigcup_{g\in G} gC$, 
it suffices to prove the claim when $C = U\cap \pi^{-1}(|C|)$.    Further, because $X$ is  first countable, it suffices to prove that every sequence $\phi_k\in 
s^{-1}(C)\cap t^{-1}(K) $  has a convergent subsequence.  Since $K$ is compact, we can assume that $y_k: = t(\phi_k)\to y_\infty\in K$.
Consider the sequence $x_k: = s(\phi_k)\in C\cap \pi^{-1}(\pi(K))$.  Since
 $|C|\cap |K|$ is the intersection of a compact set with a closed set and hence compact, and since $\pi:U\to |U|$ quotients out by the action of the finite group $G$, the set $C\cap \pi^{-1}(\pi(K)) = 
C\cap  \pi^{-1}(|C|\cap |K|) $ is also compact.  Hence, without loss of generality we may assume that  $x_k$ also converges, with limit $x_\infty$. 
Since $\pi: C\to |C|$ is continuous and $|x_k| = |y_k|$ by construction, we have $|x_\infty| = \lim |x_k| =  \lim |y_k| = |y_\infty|$, where here we make essential use of the fact that $|U|$ is Hausdorff, which holds because the map $\psi: \qu{U}{G} \to |U|$ is a homeomorphism by Lemma~\ref{lem:unif}~(iii).  Therefore there is an element $m_\infty\in \Mor_\Xx(x_\infty, y_\infty)$, and
 it is now easy to find a convergent subsequence of $(\phi_k)$.  Indeed,
 because  the maps $s,t$ are \'etale, there is a  convergent sequence $\phi_k'\in \bX$ with limit $m_\infty$ and with $t(\phi'_k) = y_k$ and $s(\phi_k')\in U$ for all $k$.  Since $|s(\phi_k')| = |y_k| = |x_k|$ for all $k$, there are elements $g_k\in G$ such that  
$$
x_k = g_k*s(\phi_k') = s\bigl(\Ga(g_k,x)\circ \phi_k'\bigr),
$$
 which implies by the groupoid property of $\Xx$ that there are elements $h_k\in G_{x_k}$ such that
$\phi_k = \Ga(h_k,x)\circ \Ga(g_k,x)\circ \phi'_k$  for all $k$.  Since there are only finitely many possibilities for $h_k,g_k$, the given sequence $\phi_k$ must have a convergent subsequence.  
\MS

To prove (iv)
suppose that $x,y$ are two points in $X$ 
such that every neighbourhood  of $|x|$ intersects every neighbourhood of $|y|$. 
Since the map $\pi:X\to |\Xx|$ is  open 
we may take neighbourhoods of the form $\pi(N_x), \pi(N_y)$ where $N_x,N_y$ are open neighbourhoods of $x,y$ respectively, and hence find  sequences 
 $x_k\to x$ and $y_k\to y$ such that $x_k\sim y_k$. Because $\Xx$ is a groupoid, $x_k\sim y_k$ exactly if there is a morphism $x_k\to y_k$,  i.e.\  exactly if  $(x_k, y_k)\in s\times t (\bX) \subset X\times X$.  Since  $(x_k,y_k)\to (x,y)$, the closed graph condition  implies that  $(x,y)\in 
s\times t\,(\bX)$.  Hence $|x|=|y|$, i.e.\  $|\Xx|$ is Hausdorff. A similar argument applies if $\pi$ is closed, since every point of $X$ has a countable basis of closed neighbourhoods.
\end{proof}

\begin{cor} \label{cor:prop1} 
\begin{nenumilist}
\item 
The realization $|\Xx|$ of any \'etale proper groupoid $\Xx$ is Hausdorff.
\item 
Suppose the \'etale  groupoid $\Xx$ has Hausdorff realization $|\Xx|$, and is covered by a finite number of local uniformizers $(U_I, G_I)_{I\in \Ii}$ in the sense that $|\Xx|=\bigcup_{I\in\Ii} |U_I|$.  
Let $U : = \bigcup_{I\in\Ii} U_I \subset X$, and define $\Xx|_U$ (as in Remark~\ref{rmk:restrict}) to be the full subcategory of $\Xx$ with objects $U$.
Then the categories $\Xx$ and  $\Xx|_U$ are both proper.
 \end{nenumilist}
\end{cor}

\begin{proof}  If $\Xx$ is \'etale, the projection $\pi:X\to |\Xx|$ is open by Lemma~\ref{lem:etale1}, and the map $s\times t: \bX\to X\times X$ has closed graph by Lemma~\ref{lem:prop1}~(ii).  Therefore (i) follows from Lemma~\ref{lem:prop1}~(iv).

Towards (ii), we first show that the realization $|\Xx|$ is metrizable when we assume it to be Hausdorff and covered $|\Xx|=\bigcup_{I\in\Ii} |U_I|$ by a finite number of local uniformizers $(U_I, G_I)_{I\in \Ii}$. 
Note first that each $U_I$ is metrizable since it is a subset of $X$, which is metrizable by 
definition of the \'etale spaces (a)-(d).   By the Nagata--Smirnov metrization theorem (see \cite[Thm.40.3]{Mu}, this is equivalent to saying that $U_I$ is regular and has a  basis for its topology that is a countable union of locally finite collections  of open sets. Both these properties are inherited by the quotient $|U_I| = U_I/G_I$; regularity of the quotient follows from Exercise 7(b) on p 199 of \cite{Mu}, while the second condition holds because the projection $U_I\to |U_I|$ is open and globally finite-to-one so that the image in $|U_I|$ of a  locally finite collection  of open sets in $U_I$  has the same properties. Hence the quotient $|U_I|$ is metrizable.
Thus, because $|U_I|$ is open in $|\Xx|$
by Lemma~\ref{lem:unif}~(ii), the space $|\Xx|$ is locally metrizable. 
Further, because metrizability is equivalent to paracompactness, Hausdorffness 
and local metrizability by Smirnov's metrizability theorem (see~\cite[Thm.42.1]{Mu}),
 each  $|U_I|$ is paracompact.  It follows that  $|\Xx|$, which is the finite union of the open sets $|U_I|$, is also paracompact.  Since $|\Xx|$ is Hausdorff by hypothesis,
  Smirnov's metrizability theorem implies that $|\Xx|$ is metrizable.

To prove properness of $\Xx|_U$ in (ii), we must check  that every $x\in U$ has a neighbourhood $N(x)$ so that $t: s^{-1}\bigl(\cl_X(N(x))\bigr)\to X$ is proper. 
By assumption, we have $x\in U_I$ for some local uniformizer $(U_I,G_I)$, and applying 
 Lemma~\ref{lem:prop1}~(iii) to this uniformizer guarantees properness of $t: s^{-1}(C)\to X$ for any closed subset $C$ subset  of $U_I$ such that $|C|\subset |\Xx|$ is closed.   Thus it suffices to find a closed neighbourhood 
 $C\subset U_I$ of $x$ such that $|C|$ is closed in $|\Xx|$. 
 To this end, consider the closed set $|\Xx|\less |U_I|$.  Since $|x|\notin |\Xx|\less |U_I|$ and $|\Xx|$ is metrizable and hence normal (by \cite[Thm.32.3]{Mu}), the point  $|x|$ (which is a closed set since $|\Xx|$ is Hausdorff) has an open neighbourhood $W\subset|\Xx|$ whose closure $\cl_{|\Xx|}(W)$ is disjoint from $|\Xx|\less |U_I|$ and hence is contained in $|U_I|$.  
 Therefore $C := \pi^{-1}\bigl(\cl_{|\Xx|}(W)\bigr)$ is contained in $U_I$.  Moreover, because
 $\cl_{|\Xx|}(W)$ is 
 a closed neighbourhood of $|x|$ in the open set $|\Xx|$ its pullback $C$ to $U_I$
 is a closed neighbourhood of $x$ in $U_I$.
 
It is now straightforward to check that the category $\Xx$ itself is proper. Indeed, by assumption, for every $y\in X$  there is some  $x\in U_I$  such that $|x| = |y|$.  Hence, as in the proof of Lemma~\ref{lem:etale1}(i) there is a local \'etale bijection $\phi:N'(y)\stackrel{\simeq}\to N'(x)$ from a neighbourhood $N'(y)$  of $y$ onto a
neighbourhood $N'(x)\subset U$ of $x$.  Now choose a neighbourhood $N(x)$ of $X$ such that
$\cl_X(N(x))\subset N'(x)$ and $t: s^{-1}\bigl(\cl_X(N(x))\bigr)\to X$ is proper.  Then 
the \'etale map $$
s^{-1}\bigl(\cl_X(N(x))\bigr) \to s^{-1}\bigl(\cl_X(N(y))\bigr),\quad
m\mapsto \phi\circ m
$$
has an \'etale inverse. It follows that $t: s^{-1}\bigl(\cl_X(N(y))\bigr)\to X$ is proper.

 This completes the proof of the lemma.
\end{proof}

Finally, we discuss groupoid completions of \'etale categories $\Bb$. 
For that purpose recall from Definition~\ref{def:etale2} that the sets $B_x: = \Mor_\Bb(x,x)$ are generally semigroups as they may lack inverses.

\begin{defn}\label{def:grcompl}  
Let $\Bb=(B, \bB)$ be a category. Then a pair $(\Hat\Bb, \io)$ consisting of a groupoid $\Hat\Bb$ together with an injective functor $\io:\Bb\to \Hat{\Bb}$ is called a {\bf groupoid completion} of the category $\Bb$ if  the following holds:
\begin{itemize}
\item[{\rm (a)}]  $\io:B\to \Hat B$ is bijective and induces a bijection $|\Bb|\to |\Hat{\Bb}|$,
\item[{\rm (b)}] for all $x\in B$, the functor $\io$ induces an isomorphism $B_x\to \Hat B_x: = \Mor_{\Hat\Bb}(x,x)$.
 \end{itemize}
 If $\Bb$ is \'etale, then a pair $(\Hat\Bb, \io)$ as above is an {\bf \'etale groupoid completion} if in addition
 \begin{itemize}
 \item[{\rm (c)}]  the category $\Hat\Bb$ has an \'etale structure so that $\iota$ is \'etale.
\end{itemize}
\end{defn}

\begin{rmk}\rm\label{rmk:grcompl0} If $\Bb$ is \'etale and has \'etale groupoid completion $\Hat\Bb$, then by conditions (a) and (c) we can identify the object spaces of both categories.  Then by (a) and the fact that the realization of a category has the quotient topology, the two realization maps $\pi_\Bb:B\to |\Bb|$ and 
$\pi_{\Hat\Bb}:\Hat B\to |\Hat\Bb|$ can also be identified, and have the same topological properties. For example, if one is locally injective, so is the other.
\hfill$\er$
\end{rmk}

The following lemma establishes some useful properties of groupoid completions.

\begin{lemma}\label{lem:gpcomplet} 
\begin{nenumilist}
\item 
If $\Bb$ has a groupoid completion, then each $B_x$ is a group, and these groups are isomorphic as $x$ ranges over an equivalence class with respect to $\sim_\Bb$. 

\item If $\Bb$ is \'etale and there is an injective \'etale functor $f:\Bb\to \Xx$ to an  \'etale groupoid $\Xx$  that induces an isomorphism $f: B_x\stackrel{\simeq}\to G_{f(x)}: = \Mor_\Xx(x,x)$ for all $x\in B$ and an injection $|f|: |\Bb|\to |\Xx|$ that is a homeomorphism to its image, then $\Bb$ has an \'etale  groupoid completion, namely $(\Xx|_{f(B)}, f)$, where $\Xx|_{f(B)}$ is the full subcategory of $\Xx$ with objects $f(B)$.

\item  A  nonsingular  category $\Bb$  has a unique   groupoid completion $\Hat\Bb$ up to functorial isomorphism.  It is necessarily nonsingular, and is \'etale and/or  locally injective  if $\Bb$ is.

\item If the nonsingular category $\Bb$ is a poset such that each equivalence class $[x]: = \{y\in B\ | \ |x|=|y|\}$ in $B$ has a unique minimum with morphisms to all other elements.
Then every functor $f:\Bb\to \Xx$ whose target is a groupoid extends to a functor
 $\Hat f: \Hat\Bb\to \Xx$.  Further   if $\Bb, \Xx$ and $f:\Bb\to \Xx$ are \'etale, then so is the functor $\Hat f$. 
\end{nenumilist}
 \end{lemma}
 \begin{proof}   
 To prove (i), note first that the isomorphisms $B_x\to \Hat B_x$ in (b) provide inverses for $B_x$.  
Next, in any groupoid $\Xx$ the morphism $m\in \Mor_\Xx(x,y)$ induces an isomorphism 
$$
\phi_m: G_x\to G_y,\quad \al\mapsto \be: = m^{-1}\circ \al\circ m: y\to x\to x\to y.
$$
Since  the above identity may be written $m\circ \be = \al\circ m$, it follows that if $\Bb$ has a  groupoid completion 
the equation $m\circ \be = \al\circ m$ in $\bB$ must have a unique solution $\be\in B_y$ for each $\al\in B_x, m\in \Mor_\Bb(x,y)$,
and moreover that the assignment 
\begin{align}\label{eq:alm}
\phi_m:\quad \al\mapsto \be=:\phi_m(\al),\;\;\mbox{ where } m\circ \be = \al\circ m
\end{align}
 is a group  isomorphism $B_x\to B_y$.   This proves (i).
 \MS
 
 To prove (ii), notice that because $f$ is \'etale, $f(B)$ is an open subset of $X$. Further, 
if $\Xx$ is an \'etale groupoid and $U\subset X$ is open then $\Xx|_U$ is also an \'etale groupoid.  Hence $f: \Bb\to \Xx|_{f(B)}$ satisfies the conditions in Definition~\ref{def:grcompl}.
 \MS
 
 To prove (iii), notice that  by the definition of nonsingularity (see Definition~\ref{def:etale})  we have  $|\Mor_\Bb(x,y)| \le 1$ for all $x,y\in B$.  Therefore  condition  (b) implies that any groupoid completion $\Hat\Bb$ is also nonsingular.  By 
 condition  (a)  we must therefore have $$
 \Mor_{\Hat\Bb}: = \bigl\{(x,y)\in B\times B \ \big| \ x\sim_\Bb y\bigr\},
 $$
  with structure maps $s,t$ given by projection onto the first and second factors respectively.  Then $\Hat\Bb$ with this set of morphisms is a groupoid.  Next note that $ \Mor_{\Hat\Bb}$  can be given the structure of an \'etale space in such a way that all structure maps are \'etale.  But this follows because, as in the proof of Lemma~\ref{lem:etale1}, each added morphism $y_0\to y_k$ in $\Hat\Bb$ is formed from a zigzag chain 
  $y_0\; \stackrel{m_1}\rightarrow\; y_1\;\stackrel{m_2} \leftarrow \; y_3\; \cdots \stackrel{m_k}\rightarrow \; y_k $ of morphisms in $\Bb$.  Finally, the claim about local injectivity holds by Remark~\ref{rmk:grcompl0}.
  \MS
  
In the situation of (iv), denote by $\Hat\Bb$ the unique groupoid completion of $\Bb$ provided by (iii).  
For any pair $x,y\in B$ with $x\sim_\Bb y$,  denote by $x_{\min}\in B$  the root of the equivalence class $[x]$ so that there are unique morphisms $m_x: x_{\min}\to x,   \ m_y: x_{\min}\to y $ in $\Bb$.  If  $m_{xy} \in \Hat \Bb$ 
is the unique morphism in $\Mor_{\Hat\Bb}(x,y)$,  define $\Hat f(m_{xy}) = f (m_x)^{-1}\circ f(m_y)$.  Then $\Hat f$ takes identity maps to identity maps and is compatible with inverses.  To see it is compatible with composition, notice that if $x\sim_\Bb y \sim_\Bb z$, then 
$$
m_{xz} = m_{xy}\circ m_{yz} = \bigl((m_x)^{-1}\circ m_y\bigr)\circ \bigl((m_y)^{-1}\circ m_z\bigr) = 
(m_x)^{-1}\circ m_z.
$$
It follows that 
$$
\Hat f (m_{xz}) = f(m_x)^{-1}\circ f(m_z) = f(m_x)^{-1}\circ f(m_y)\circ f(m_y)^{-1}\circ f(m_z)
=  \Hat f(m_{xy})\circ \Hat f(m_{yz}).
$$
Finally, if $\Bb, \Xx, f$ are \'etale, then the induced map $\Hat f: \bB\to \Hat\bB$ is a composite of \'etale maps and hence itself \'etale.    
\end{proof}

\begin{rmk}\label{rmk:gpcompl}\rm
This notion of groupoid completion is rather weak.  For example, if $\Bb$ is a nonsingular category  and  $\phi:\Bb\to \Xx$ is any functor whose target $\Xx$ is a groupoid, it is not necessarily true that $\phi$ extends to the groupoid completion $\Hat\Bb$.  Suppose for example that $B = \{x_1,x_2,x_3,x_4\}$ and that
there are morphisms  $m_{ij}\in \Mor_\Bb(x_i,x_j)$  for 
$(i,j)=(1,3), (1,4), (2,3), (2,4)$.
Then $\Hat\Bb$ has a unique morphism between any pair of objects.  However, in order for $\phi$ to extend to $\Hat\Bb$ we would need $\phi(m_{13})\circ \phi(m_{23})^{-1} = \phi(m_{14})\circ \phi(m_{24})^{-1} $, a relation that may or may not hold. 
\hfill$\er$
\end{rmk}

The following example shows that groupoid completion generally does not preserve properness, and thus realizations of \'etale categories are often not Hausdorff.

 \begin{example}\label{ex:branch1}\rm  
 Consider the \'etale category $\Bb$ constructed in Example~\ref{ex:branch} with objects $B = V\sqcup U$ and $\Mor_{\Bb}(V,U) \simeq \TV\subset U$.  
This category is nonsingular and $\pi_\Bb: B\to |\Bb| $ is locally injective, hence $\Bb$ has a groupoid completion $\Hat\Bb$ by Lemma~\ref{lem:gpcomplet}~(iii).  
Besides the morphisms in $\Bb$, $\Hat\Bb$ has morphisms from $\TV\subset U$ to $V$ given by inverting the elements of $\Mor_{\Bb}(V,U)$
and also morphisms with both source and target in $U$ given by the action of $G$ on $\TV\subset U$:
\begin{align*}
\Mor_{\Hat\Bb}(U,U)& =  U \sqcup \bigl(\TV\times (G\less  \{\id\})\bigr), \\
s\times t(x,\ga) & = (\ga^{-1}*x, x),\quad x\in \TV, \ga\in G\less \{\id\}.
\end{align*}
Thus $\Hat\Bb$ is \'etale.
However, it follows from Example~\ref{ex:branch}~(ii) that $|\Hat\Bb| = |\Bb| $ is not Hausdorff if $\Fr_U(\TV) \ne \emptyset$ and $G\ne {\id}$. In this case, $\Hat\Bb$ is not  proper, since  properness implies Hausdorffness by Lemma~\ref{lem:prop1}~(ii) and (iv) and Lemma~\ref{lem:etale1}.
\hfill$\er$
 \end{example}

\subsection{Group actions on \'etale categories}\label{ss:groupact}

We begin with a natural notion of the action of a finite group on a category.
Here we denote by $\Aut(\Bb,\Bb)$ the group of invertible functors $\Bb\to \Bb$ on a given category $\Bb$.
If $\Bb$ is topological, then we require functors in $\Aut(\Bb,\Bb)$ to act continuously on object and morphism spaces.
If $\Bb$ is \'etale, then we require functors in $\Aut(\Bb,\Bb)$ to be \'etale.

\begin{defn}\label{def:gpact}
An {\bf action of a finite group $G$ on a (possibly topological or \'etale) category $\Bb$} is a homomorphism $F:G\to \Aut(\Bb,\Bb), g \mapsto F_g$, that is satisfying $F_{hg}= F_g \circ F_h$.

If $G$ acts on two (possibly topological or \'etale) categories $\Bb,\Bb'$ via the homomorphisms $F, F'$, then a (continuous or \'etale) functor $\psi:\Bb\to \Bb'$ is said to be
{\bf $G$-equivariant} if  $F_g\circ \psi = \psi\circ F'_g$ for all $g\in G$.  
\end{defn}

\begin{rmk}\rm Above we write functor composition in categorical notation so that, for example, $(F_g\circ F_h) (x)=F_h(F_g(x))$ for $x\in B$. On the other hand, the action $x\mapsto g*x$  of a group $G$ on a space $X$  satisfies the identity $(hg)*x = h*(g*x)$.
\end{rmk}

\begin{lemma} \label{lem:act} Let $G$ be a finite group.
 \begin{nenumilist}
 \item
 Any action of  $G$ on a topological  category $\Bb$ induces a continuous action of $G$  on the realization $|\Bb|$ of $\Bb$.  

\item
Let $\Bb = (B,\bB)$ be a (possibly topological resp.\ \'etale) category.  Then
$G$ acts on $\Bb$ if and only if $G$ acts
 on the spaces $B, \bB$ by (continuous resp.\ \'etale) maps 
 $F_g: x\mapsto g*x, F_g: m\mapsto g * m$, in such a way that
  for all $x\in B,\; g,h\in G,\; m, m'\in \bB$ we have  
\begin{align}\label{eq:gpident}  g* \id_x = \id_{g*x}, \quad & g*(m\circ m') = (g*m)\circ (g*m'), \\ \notag
& (s\times t)(g * m) = \bigl(g* s(m),g* t(m)\bigr).
\end{align}
\item  
Suppose $\Xx = (X,\bX)$ is a groupoid and $\al:G\times X\to \bX$ is a map with
\begin{align}\label{eq:gpident2}
t(\al(g,x)) = x ,\quad 
\al(\id,x) = \id_x , \quad 
\al(hg,x) = \al(g,s(\al(h,x)))\circ \al(h,x)
\end{align}
for all $g,h\in G, \; x\in X$.
Then this induces an action of $G$ on $\Xx$ as in (ii) with 
\begin{equation}\label{eq:gpident3}
g*x := s(\al(g^{-1},x)), \qquad
g*m  := \al(g^{-1},s(m)) \, \circ\, m\,\circ\, \al(g,g*t(m))
\end{equation}
for $x\in X$ and $m\in \bX$. 
If $\Xx$ is a topological resp.\ \'etale groupoid and $\al$ is continuous resp.\ \'etale, then this construction yields an action of $G$ on the topological resp.\ \'etale category $\Xx$. Further, the induced action on $|\Xx|$ is trivial.

\item[\rm (iii$'$)]
Suppose $\Xx = (X,\bX)$ is a groupoid and $\ov\al:G\times X\to \bX$ is a map with
\begin{align}\label{eq:gpident2'}
s(\ov\al(g,x)) = x ,\quad 
\ov\al(\id,x) = \id_x , \quad 
\ov\al(hg,x) = \ov\al(g,x)\circ \ov\al(h, t(\ov\al(g,x)) )
\end{align}
for all $g,h\in G, \; x\in X$.
Then this induces an action of $G$ on $\Xx$ as in (ii) with 
\begin{equation}\label{eq:gpident3'}
g*x := t(\ov\al(g,x)), \qquad
g*m  := \ov\al(g^{-1},g*s(m)) \, \circ\, m\,\circ\, \ov\al(g,t(m))
\end{equation}
for $x\in X$ and $m\in \bX$. 
If $\Xx$ is a topological resp.\ \'etale groupoid and $\al$ is continuous resp.\ \'etale, then this construction yields an action of $G$ on the topological resp.\ \'etale category $\Xx$.
Further, the induced action on $|\Xx|$ is trivial.

\item 
Let $\Bb=(B,\bB)$ be a category that is nonsingular and locally injective. 
Suppose that $G$ is a finite group that acts on the object set $B$ and preserves the equivalence relation $x\sim_\Bb y \Longrightarrow g*x\sim_\Bb g*y$ for all $x,y,g$.   
Then this action extends uniquely to $G$ actions on the category $\Bb$ and its groupoid completion $\Hat\Bb$. 

Moreover, if $\Bb$ is a topological resp.\ \'etale category and the action of $G$ on $B$ is contiunuous resp.\ \'etale, then this extension is a $G$-action on the  topological resp.\ \'etale category and its  topological resp.\ \'etale groupoid completion.  
\end{nenumilist}
\end{lemma}

\begin{proof}  In (i), by hypothesis, for each $g\in G$ there is a functor $F(g):\Bb\to \Bb$, and we denote 
the induced maps on objects and morphisms by $F_g (x) = g*x, F_g(m) = g*m$.  Because $F(g)$ is a functor, each $m\in \Mor_\Bb(x,y)$ yields $g*m \in \Mor_\Bb(F_g(x),F_g(y))=  \Mor_\Bb(g*x,g*y)$.
To see that the action $(g,x)\mapsto g*x$ of $G$ on $B$ preserves the equivalence relation on $B$,  
 note that $x\sim_\Bb y$ iff there is a chain of morphisms from $x$ to $y$ as follows
$$
x: = x_0 \stackrel{m_1}\rightarrow  x_1 \stackrel{m_2}\leftarrow  x_2  \cdots x_{k-1} \stackrel{m_k}\rightarrow  x_k : = y. 
$$
Because 
$m\in \Mor(x,y) \Longrightarrow g*m\in \Mor(g*x,g*y)$, applying $g\, *$ to this chain gives a similar chain of 
morphisms from $g*x$ to $g*y$.    Therefore $x\sim_\Bb y$ iff $g*x\sim_\Bb g*y$, so that
 the action of $G$ preserves the equivalence relation on $B$ and hence descends to an action $|x|\mapsto g*|x|$ on $|\Bb|$.   This action commutes with the projection $\pi_\Bb: B\to |\Bb|$, that is 
 $\pi_\Bb(g*x) = g*(\pi_\Bb(x))$.  It follows that the action of $G$ on $|\Bb|$ is continuous; indeed for every open subset $V\subset |\Bb|$ and every $g\in G$, 
we have $\pi_\Bb^{-1}(g^{-1}*V) =  g^{-1}*(\pi_\Bb^{-1}(V))$.
This proves (i). \MS

To prove (ii), first note that each functor $F(g)\in\Aut(\Bb,\Bb)$ consists of two maps $F_g: B\to B$, $F_g:\bB\to\bB$ whose continuity or \'etaleness and three algebraic properties in \eqref{eq:gpident} follow from $F(g):\Bb\to\Bb$ being a (continuous resp.\ \'etale) functor. The fact that both are $G$-actions follows from $F$ being a homomorphism: 
$$
(hg)*x = F_{hg}(x) = (F_g \circ F_h) (x) = F_h(F_g(x)) = h*(g*x)$$
 and analogously on morphisms.

Conversely, assume the actions on object and morphism spaces satisfy \eqref{eq:gpident} (and possibly are continuous resp.\ \'etale).  Then we must check that each $g\in G$ induces an invertible functor $F(g)=(F_g,F_g)$ on $\Bb=(B,\bB)$ and that $F(hg) = F(g)\circ F(h)$ holds for all $g,h\in G$. 
Then \eqref{eq:gpident} ensures that each $F(g)$ is a functor. It inherits continuity resp.\ \'etaleness directly from the two actions. 
The homomorphism property is equivalent to the maps on objects and morphisms both satsifying $F_{hg} = F_g\circ F_h$ for all $g,h\in G$, which holds since $G\times X \to X, (g,x)\mapsto F_g(x)$ and $G\times \bX \to \bX, (g,m)\mapsto F_g(m)$ are $G$-actions by hypothesis.
%
This hypothesis also guarantees that $F_{\id}:X\to X$ and $F_{\id}:\bX\to\bX$ are the identity maps and hence $F(\id)$ is the identity functor $\id_\Bb$. Now the homomorphism property guarantees that each $F(g)$ is invertible since $F(g)\circ F(g^{-1})= F(g^{-1} g) =F(\id)=\id_\Bb$ and $F(g^{-1})\circ F(g)= F(g g^{-1}) =F(\id)=\id_\Bb$, which completes the proof of (ii). \MS

To prove (iii) we must check that \eqref{eq:gpident3} defines $G$-actions which satisfy \eqref{eq:gpident}. 
The action property on objects follows from the source identity included in $\al(g^{-1}h^{-1},x) = \al(h^{-1},s(\al(g^{-1},x)))\circ \al(g^{-1},x)$ to obtain
$$
(hg)*x = s \bigl( \alpha ( (hg)^{-1} , x ) \bigr) = s \bigl( \al(h^{-1},s(\al(g^{-1},x))) \bigr) 
= h * (g * x) . 
$$
This guarantees an action $*:G\times X\to X$, and then by definition $\al(g,x)\in \Mor(g^{-1}*x,x)$.
With that we can rewrite the above identity as
$\al(g^{-1}h^{-1},x) = \al(h^{-1},g*x)\circ \al(g^{-1},x)$ 
and use this to check the action property on morphisms: Denoting $x := s(m)$ and $y:=t(m)$ 
\begin{align*}
(hg)*m &= \al((hg)^{-1},x) \, \circ\, m\,\circ\, \al(hg, (hg)* y)\\
& = \al(h^{-1},g*x)\circ \al(g^{-1},x)         \circ\, m\,\circ\ \al(g, h^{-1}*( (hg)*y) ) \circ \al(h, (hg)*y)\\
& = \al(h^{-1},g*x)\circ ( g * m )  \circ \al(h, h*(g*y))\  =  h * ( g*m ).
\end{align*}
To check that these actions satisfy \eqref{eq:gpident}, first use the composition property in the form
$\id_y = \al(g g^{-1},y) = \al(g^{-1},g^{-1}*y)\circ \al(g,y)$ applied to $y=g*x$ to obtain
$$
g* \id_x = \al(g^{-1},x) \circ \id_x \circ \al(g,g*x) =  \al(g^{-1},x) \circ \al(g,g*x) = \id_{g*x}. 
$$ 
Next, the conditions on $(s\times t)(g*m)$ may be checked by inspecting \eqref{eq:gpident3}: 
$$
s(g*m) = s( \al(g^{-1},s(m))  ) = g* m, 
\qquad
t(g*m) = t( \al(g,g*t(m)) ) = g*t(m)) . 
$$
Third, compatibility with composition of morphisms $m,m'$ with $t(m) = s(m')=x$ holds by
\begin{align*}
 g * (m\circ m') & = \al(g^{-1},s(m))\, \circ\, m\,\circ m'\, \,\circ\, \al(g,g*t(m')) \\
&   = 
\Bigl(\al(g^{-1}, s(m))\, \circ\, m\,\circ  \al(g,g*t(m))\Bigr) \circ\Bigl(\al(g^{-1}, s(m'))\, \circ\, m'\, \,\circ\, \al(g,g*t(m'))\Bigr)\\
& =  ( g*m)\circ (g*m'),
\end{align*}
where the second equality holds because $\id_x = \al(g^{-1}g,x) = \al(g,g*x)\circ \al(g^{-1},x)$ by \eqref{eq:gpident2}.

If $\Xx$ is a topological resp.\ \'etale groupoid and $\al$ is continuous resp.\ \'etale, then the maps in \eqref{eq:gpident3} are continuous resp.\ \'etale since they arise by compisition with the continuous resp.\ \'etale structure maps of $\Xx$. 
Hence this yields an action of $G$ on the topological resp.\ \'etale category $\Xx$ in the sense of Definition~\ref{def:gpact}. 
Finally, the induced action on $|\Xx|$ is trivial because $x\sim g*x$ via the morphism $\al(g^{-1},g*x)$. This proves (iii).\MS

We prove 
 (iii$'$) by noting that it is equivalent to (iii). Indeed, an $\al$-action on $\Xx$ in the sense of (iii) is equivalent to an $\ov\al$-action on $\Xx^{op}$ in the sense of  (iii$'$), where $\ov\al(g,x) = \al(g^{-1},x)^{op}$ and vice versa. Now, $\Xx$ and $\Xx^{op}$ are isomorphic by the 
 functor that is the identity on objects and $m\mapsto (m^{-1})^{op}$ on morphisms. Using this isomorphism, an $\al$-action on $\Xx$ induces an $\ov\al$-action on $\Xx^{op}$ by $\ov\al(g,x) = \al(g,x)^{-1}$.
\MS

To prove (iv), 
note first that because the morphisms in $\Bb$ are uniquely  determined by their source and target, the action of $G$ on $B$ lifts to a unique action on the morphism space $\bB$ that satisfies the conditions in (ii).  Moreover, as in the proof of in Lemma~\ref{lem:gpcomplet}~(iii),
 this action  is
  \'etale because locally the map $m\mapsto g*m$ may be written as the composite $m\mapsto t^{-1}\circ \bigl(g* (s(m))\bigr)$, where $t^{-1}$ is a suitable local inverse to $t$.
  Therefore
  the homomorphism $F:G\to \Aut(\Bb,\Bb)$ exists by (i).  Moreover, as in (i), because the $G$-action preserves identity maps it necessarily commutes with  the inverse map.  Hence $F$ extends to a homomorphism $F:G\to \Aut(\Hat\Bb, \Hat\Bb)$.
\end{proof} 

The description of a $G$-action on $\Xx$ by a function $\al: G\times X\to \bX$  as in \eqref{eq:gpident2}, or equivalently by $\ov\al: G\times X\to \bX$ as in \eqref{eq:gpident2'}, means that the $G$-action is induced by appropriate morphisms in $\Xx$.  
Furthermore, we will see below that such actions naturally lift to bundles. 
For the latter purpose, it is much more natural to use a description with $\ov\al$ satisfying $s(\ov\al(g,x)) = x$ because, as in Definition~\ref{def:bundle0},  the structure map  $\mu:\bX\leftsub{s}{\times}_P W\to W$ of a bundle $\Ww\to\Xx$ uses the source of a morphism rather than its target. This motivates the next definitions. 

\begin{defn}\label{def:inngpact}
 \begin{nenumilist}
 \item 
 An {\bf inner action} of $G$ on a (possibly topological or \'etale) groupoid $\Xx$  is 
an action induced by a subset of the morphisms of $\Xx$ as follows:  
There is a (continuous resp.\ \'etale) function $\al:G\times X\to \bX$ that satisfies \eqref{eq:gpident2'}, 
so that the actions on $X, \bX$ are given by \eqref{eq:gpident3'}. 
In this case we also say that $G$ {\bf acts on $\Xx$ by inner automorphisms}.

\item  
If $\Pp:\Ww\to \Xx$ is a bundle over a (possibly topological or \'etale) groupoid, 
then we say that an action of $G$ on $\Xx$ {\bf lifts to} $\Ww$ if $G$ acts on $\Ww$ in such a way that $\Pp$ is $G$-equivariant.
\end{nenumilist}
\end{defn}

\begin{rmk}\label{rmk:inngpact}\rm   
\begin{nenumilist}
\item
If $G$ acts on a topological or \'etale groupoid $\Xx$ by inner automorphisms, then the induced action of $G$ on $|\Xx|$ is trivial by Lemma~\ref{lem:act}~(iii).
\item 
Given a local uniformizer $(U,G,\Ga)$ of an \'etale groupoid $\Xx$ as in Definition~\ref{def:preunif}, the \'etale injection $\Ga: G \times U \to \Mor_\Xx(U,U)\subset \Xx$ induces a local action of $G$ by inner automorphisms on the restricted groupoid $\Xx|_U=(U,\Mor_\Xx(U,U))$. 
Indeed, $\ov \al:=\Ga : G\times U \to \Mor_\Xx(U,U)$ satsfies \eqref{eq:gpident2'} by the conditions on $\Ga$ in \eqref{eq:groupact}:   
$s(\ov\al(g,x)) = s(\Ga(g,x)) = x$ follows directly. 
Compatibility with composition follows from the composition and target property of $\Ga$
\begin{align*}
\ov\al(hg,x)& \,=\, \Ga(hg,x) \,=\, \Ga(g,y) \circ \Ga(h, g*x)\\
& \,=\, \Ga(g,x) \circ \Ga(h, t(\Ga(g,x) ) \,=\, \ov\al(g,x)\circ \ov\al(h, t(\ov\al(g,x)) ). 
\end{align*}
Moreover, this implies $$
\ov\al(\id,x)\circ \ov\al(\id, t(\ov\al(\id,x)) ) = \ov\al(\id\, \id,x) = \ov\al(\id,x),
$$
 where we have $$
 t(\ov\al(\id,x))=t(\Ga(\id,x))=\id * x = x.
 $$ So this implies $\ov\al(\id,x)\circ \ov\al(\id, x ) = \ov\al(\id,x)$, and then composition with $\ov\al(\id, x )^{-1}$ implies $\ov\al(\id,x) = \id_x$. 

So when discussing $G$-actions in local uniformizers, we may write $\Ga$ instead of $\ov\al$. In \cite[\S7.1]{TheBook} this is called the natural representation of $G=G_x$. 
\hfill$\er$
\end{nenumilist}
\end{rmk}

\begin{lemma} \label{lem:actW}  
Let $\Pp:\Ww\to \Xx$ be a bundle with structure map $\mu$ as in Definition~\ref{def:bundle0} over an \'etale groupoid.
Then every inner action of a finite group $G$  on $\Xx$ has a natural lift to an inner action on $\Ww$. 
In fact, let $\ov\al_X:G\times X\to \bX$ be the \'etale map determining the $G$-action on $\Xx$ as in Lemma~\ref{lem:act}~(iii$'$).  Then a natural $G$-action on $\Ww$ is given as in Definition~\ref{def:inngpact}~(i) by 
$$
\ov\al_W: G\times W \to \bW=\bX \leftsub{s}{\times}_P W,
\qquad \ov\al_W(g,w)  := \bigl(\ov\al_X(g,P(w)), w\bigr) .
$$
The induced actions on objects and morphisms as in Lemma~\ref{lem:act}~(ii) are 
$$
g*w =  \mu\bigl( \ov\al_X(g,P(w)) , w \bigr) , \qquad
g*(m,w)  = ( g*m , g*w ) .
$$
Further, the functor $\Pp:\Ww\to\Xx$ and all sections $f:\Xx\to \Ww$ are $G$-equivariant.
\end{lemma}

\begin{proof}
To establish the inner action on $\Ww$ via Lemma~\ref{lem:act}~(iii$'$), note that $\ov\al_W$ is a composition of \'etale maps, hence \'etale. So it remains to check the three properties in \eqref{eq:gpident2'}.  
The source identity follows from the definition of $\ov\al_W$ and the bundle construction \eqref{eq:Wbundle0} 
$$
s(\ov\al_W(g,w)) = s \bigl(\ov\al_X(g,P(w)), w\bigr) = w . 
$$
The identity property requires in addition the identity property of $\ov\al_X$
$$
\ov\al_W(\id,w) = \bigl(\ov\al_X(\id,P(w)), w\bigr) = \bigl( \id_{P(w)} , w\bigr) = \id_w . 
$$
The composition property follows from the composition property of $\ov\al_X$, the composition in \eqref{eq:Wbundle}, and the bundle property $P \mu( m , w ) = t(m)$ applied to $m= \ov\al_X(g,P(w))$ 
\begin{align*}
\ov\al_W(hg,w) &= \bigl(\ov\al_X(hg,P(w)) \,,\, w\bigr)  \\
&= \bigl(  \ov\al_X(g,P(w))\circ \ov\al_X(h, t(\ov\al_X(g,P(w))) )  \,,\, w\bigr)  \\
&=\bigl(\ov\al_X(g,P(w)) \,,\, w\bigr)  \circ \bigl(\ov\al_X(h, P \mu(\ov\al_X(g,P(w)), w ) ) \,,\, t(\ov\al_W(g,w)) \bigr)  \\
&= \ov\al_W(g,w)\circ \ov\al_W(h,  \mu(\ov\al_X(g,P(w)), w ) ) \\
&= \ov\al_W(g,w)\circ \ov\al_W(h, t(\ov\al_W(g,w)) ).
\end{align*}
Now \eqref{eq:gpident3'} and \eqref{eq:Wbundle} provide the induced actions on obects and morphisms: 
\begin{align*}
g*w &= t\bigl( \ov\al_W(g,w) \bigr) =  \mu\bigl( \ov\al_X(g,P(w)) , w \bigr) , \\
g*(m,w) &= \ov\al_W(g^{-1},g*s(m,w)) \, \circ\, (m,w) \,\circ\, \ov\al_W(g,t(m,w)) \\
&= \bigl( \ov\al_X(g^{-1}, P(g*w) ) , g*w \bigr) \, \circ\, (m,w) \,\circ\, \bigl( \ov\al_X(g, t(m)) ,  \mu(m,w) \bigr) \\
&= \bigl( \ov\al_X(g^{-1},  g*s(m) ) \circ m \circ \ov\al_X(g, t(m)) , g*w \bigr) \\
&= ( g*m , g*w )
\end{align*}
using the fact that 
$P(g*w)= P \mu\bigl( \ov\al_X(g,P(w)) , w \bigr) =  t( \ov\al_X(g,P(w)) ) = g*P(w) = g*s(m)$. 

Next, $G$-equivariance of $\Pp:\Ww\to\Xx$ in the sense of Definition~\ref{def:gpact} follows similarly by checking it on objects and morphisms:
\begin{align*}
P(g*w)  
&\,=\; P  \mu\bigl( \ov\al_X(g,P(w)) , w \bigr) \;=\; t \bigl( \ov\al_X(g,P(w)) \bigr) \;=\; g* P(w)
, \\
P(g*(m,w)) 
&\,=\; P( (g*m,g*w) ) \;=\; g * m \;=\; g * P((m,w)) . 
\end{align*}
Finally, any section $f:\Xx\to\Ww$ in the sense of Definition~\ref{def:bundlemaps} 
is given by a map $f:X\to W$ satisfying $P\circ f = \id$ and $f(t(m))=\mu(m,f(s(m)))$. 
It lifts to a functor by Lemma~\ref{lem:bundlemaps}, given on morphisms by $f(m)= \bigl(m,f(s(m))\bigr)$.
These properties are used to check $G$-equivariance on objects and morphisms: 
\begin{align*}
g* f(x) 
&\,=\;  \mu\bigl( \ov\al_X(g, P(f(x))) , f(x) \bigr)  
\;=\;  \mu\bigl( \ov\al_X(g, x ) , f(x) \bigr)  \\
&\,=\;  \mu\bigl(\ov\al_X(g,x) , f(s(\ov\al_X(g,x)))\bigr)
\;=\;  f( t (\ov\al_X(g,x) ))
\;=\; f(g*x) 
, \\
g * f(m)
&\,=\;  g * \bigl(m,f(s(m))\bigr)
\;=\;  \bigl( g *m,  g * f(s(m)) \bigr) \\
&\,=\;  \bigl( g *m,  f(g * s(m)) \bigr)
\;=\;  \bigl( g *m,  f(s(g *m)) \bigr)
\;=\; f(g*m)  , 
\end{align*}
where the last steps used equivariance of $f$ on objects and $s(g *m) = s (\ov\al(g^{-1},g*s(m)) =  g * s(m)$ from \eqref{eq:gpident3'} and \eqref{eq:gpident2'}. 
\end{proof}

\begin{rmk}\label{rmk:actW} \rm 
\begin{nenumilist}
\item
The above constructions are functorial as follows.  Let $ \Psi: \Ww\to \Ww'$ be a bundle map over
an \'etale functor $\psi:\Xx\to \Xx'$ between \'etale groupoids as in Definition~\ref{def:bundlemaps} and Lemma~\ref{lem:bundlemaps}, and suppose given inner actions $\ov \al_X, \ov \al_{X'}$ of $G$ on $\Xx$ and $\Xx'$. 
We say that the functor $\psi:\Xx\to \Xx'$  is  {\bf equivariant with respect to these inner actions} 
(or $(G, \al,\al')$-equivariant) if
$$
\psi({\ov\al}_X(g,x)) = \ov\al_{X'}(g,\psi(x)):\ G\times X\to \bX'.
$$   
It is straightforward to check that in this case  $\Psi$ is $G$-equivariant with respect to the lifted inner actions on the bundles; in other words 
 $\Psi\circ \ov\al_W= \ov\al_{W'}\circ (\id_G\times \Psi) $.
 
\item
Note that the condition in (i) is stronger than the equivariance condition for an arbitrary action since we now assume that $\psi$ takes the morphism ${\ov\al}_X(g,x)$ that implements the action $x\mapsto g*x$ of $G$ on $X$ to that which implements the action $\psi(x)\mapsto g*\psi(x)$
of $G$ on $X'$.  This condition does imply that $\psi$ is $G$-equivariant 
in the sense of Definition~\ref{def:gpact}, but, unless the categories $\Xx, \Xx'$ are effective, the two conditions may not be the same.  

To avoid confusion, unless specific mention is made to the contrary, whenever a functor is said to be $G$-equivariant this should be understood in the weaker sense of Definition~\ref{def:gpact}.
\hfill$\er$
\end{nenumilist}
\end{rmk}

For the final construction of this section we return to general (possibly topological or \'etale) categories.

\begin{lemma} \label{lem:BbGa}
Suppose that the finite group $G$ acts on the category $\Bb$.  

\begin{nenumilist}
\item   
This induces a category $\Bb^{\times G}$ with object space $B$, morphism space $\bB\times G$ and structure maps given for all $x\in B, m\in \bB, g\in G$ by
\begin{align*}
&(s\times t)\, (m,g) = \bigl(g^{-1}*s(m),\, t(m)\bigr), 
\quad 
\id^{\Bb^{\times G}}_x = (\id^\Bb_x,\id),\\
&\qquad\qquad (m,g)\circ (m',h) = \bigl((h*m) \circ m',\, hg\bigr).
\end{align*}

\item   
If $\Bb$ is \'etale and $G$ acts by \'etale maps, then the category $\Bb^{\times G}$ is \'etale, 
 and the identification $\Obj_\Bb=B= \Obj_{\Bb^{\times G}}$ induces  a homeomorphism $|\Bb|/G  \stackrel{\simeq}\to |\Bb^{\times G}|$, where $G$ acts on $|\Bb|$ as in Lemma~\ref{lem:act}~(i).

Moreover, if $\Bb$ is a (possibly \'etale) groupoid, then $\Bb^{\times G}$ is an (\'etale) groupoid with inverse given by
$$
(m,g)^{-1} = \bigl(g^{-1}*m^{-1}, g^{-1}\bigr).
$$

\item  
If $G$ also acts on another category $\Cc$, then any $G$-equivariant functor $\phi: \Bb\to \Cc$ naturally extends to a functor $\phi^{\times G}: \Bb^{\times G}\to \Cc^{\times G}$, given by $\phi^{\times G}(m,g)=(\phi(m),g)$ on morphisms.

\item 
Let $\Bb$ be a nonsingular, locally injective, \'etale category, and consider the
associated  $G$ action on 
its groupoid completion $\Hat\Bb$ constructed in Lemma~\ref{lem:act}~(iv).
Then
 the category $\Hat\Bb^{\times G}$ is an \'etale groupoid completion of $\Bb^{\times G}$.

\item  
If $G$ acts on a groupoid $\Xx$ by inner automorphisms, then the identity functor $\io:\Xx\to \Xx$ extends to a  functor $\io^{\times G}: \Xx^{\times G}\to \Xx$ given on morphisms by
\begin{align}\label{eq:gpident4}
(m,g)\mapsto \al(g,s(m))\circ m \in \Mor_\Xx(g^{-1}*s(m), t(m)),
\end{align}
where $\al$ is as in Lemma~\ref{lem:act}~(iii).
\end{nenumilist}
%
%
%
%
%

\end{lemma}

\begin{proof}
We first note that the formulas in (i) restrict to the formulas for $\Bb$ when $g=\id$.  
We next assume that the source and target map are as given, and then check that the composition formula is  consistent with the source and target maps, and that it is associative.
To this end, note that $(m,g)\circ (m',h)$ is defined only if $t(m) = t(m,g) = s(m',h) = h^{-1}* s(m')$, which implies
$h* t(m) = s(m')$.  Since $h* t(m) = t(h*m)$ by assumption (i), the composite $(h*m) \circ m'$ is defined.
Further
\begin{align*}
&s\bigl((m,g)\circ (m',h)\bigr) = s(m,g) = g^{-1}* s(m) = (hg)^{-1} h * s(m) = s\bigl((h*m) \circ m'\bigr), \\
&  t\bigl((m,g)\circ (m',h)\bigr) = t(m',h) = t(m') = t\bigl((h*m) \circ m'\bigr).
\end{align*}
Therefore, the composition formula makes sense.  To prove associativity 
one uses the identity $g*(m\circ m') = (g*m)\circ (g*m')$
to check that
\begin{align*}
\bigl( (m,g)\circ (m',h)\bigr)\circ (m'', k) & =  \bigl((h*m) \circ m',\, hg\bigr) \circ (m'', k) \\
& =  \bigl(k*((h*m) \circ m')\circ m'',\, khg\bigr)\\
& =  \bigl(k*h*m \circ (k*m')\circ m'',\, khg\bigr)\\
&  = (m,g)\circ  \bigl((m',h)\circ (m'', k)\bigr).
\end{align*}
This proves that $\Bb^{\times G}$ is a well defined category.  

To check (ii), notice that because
$G$ is a finite group that acts by \'etale maps by assumption, the structure maps in $\Bb^{\times G}$  are \'etale, 
which means that $\Bb^{\times G}$ is \'etale, as claimed.  Moreover, if $\Bb$ is a groupoid, 
one checks the formula for the inverse by noting
 that $$
 (s\times t) \bigl(g^{-1}*m^{-1}, g^{-1}\bigr) = (t\times s)(m,g),
 $$
  and then checking that
the composition formula gives 
$$
(g^{-1}*m^{-1}, g^{-1}\bigr)\circ  (m,g) = \bigl((g*(g^{-1}*m^{-1}))\circ m, \id\bigr) = \bigl(m^{-1}\circ m, \id\bigr) = (\id_{t(m)}, \id).
$$
Next, we prove the claim about the realizations of these categories. 
Note that if $x\sim_\Bb y$ then $x\sim_{\Bb^{\times G}} g*y$ since we can simply adjoin to the above chain of morphisms  in $\Bb$ an extra morphism $(\id_y,g): y\to g*y$.  
Therefore there is a well defined map $|\Bb|/G\to |\Bb^{\times G}|$, which is necessarily surjective.  To see that it is injective,
notice that by using the identities $
(m,g) =  (m,\id)\circ (\id_{t(m)}, g) = (\id_{s(m)}, g) \circ (g*m,\id)$, one can convert any chain of morphisms
$$
x: = x_0 \stackrel{(m_1,g_1)}\rightarrow  x_1 \stackrel{(m_2, g_2)}\leftarrow  x_2  \cdots x_{k-1} \stackrel{(m_k, g_k)}\rightarrow  x_k : = y 
$$
in $\Bb^{\times G}$ to a chain of morphisms in $\Bb$ followed by one morphism of the form $(\id_z,h)$. It follows that
$x\sim_{\Bb^{\times G}} y$ if and only if there is $h\in G$ such that $x\sim_{\Bb} h^{-1}*y$.    Thus $h*|x| = |y|$, which proves injectivity.
\MS

To prove (iii), define $\phi^{\times G}: \Mor_{\Bb\times G}\to \Mor_{\Cc\times G}:\ (m,g)\mapsto (\phi(m),g).$  
This is compatible with source and target maps because
\begin{align*}
(s\times t)  (\phi(m),g) & = \bigl(g^{-1}*s(\phi(m)), t(\phi(m)\bigr) =\bigl(s(g^{-1}*\phi(m)), t(\phi(m))\bigr)\\
& = \bigl(s(\phi(g^{-1}* m)), t(\phi(m))\bigr) =  \phi\bigl(s(g^{-1}*m), t(m)\bigr),
\end{align*}
where in the second equality we use the fact that the source map in $\Cc$ is $G$-equivariant (see \eqref{eq:gpident}), and in the third used the fact that the functor $\phi$ is $G$-equivariant. 
To see that this is compatible with composition, we note:
\begin{align*}
\phi^{\times G}\bigl((m,g)\circ (m',h)\bigr) & = \phi^{\times G} \bigl((h*m) \circ m',\, hg\bigr)\\
& =\bigl( \phi((h*m) \circ m'),\, hg\bigr)  =\bigl( \phi(h*m) \circ \phi(m'),\, hg\bigr) \\ 
& = \bigl((h* \phi(m)) \circ \phi(m'),\, hg\bigr) = (\phi(m), g)\circ (\phi(m'),h)\\
& = \phi^{\times G}(m,g)\circ \phi^{\times G}(m',h),
\end{align*}
where the fourth equality uses the $G$ equivariance of $\phi$.
One can check similarly that if $\Bb, \Cc$ are groupoids,  $\phi^{\times G}$ is compatible with inverses.  This completes the proof of (iii).

\medskip

Now consider (iv).
To check that $\Hat\Bb^{\times G}$ is an \'etale groupoid completion of $\Bb^{\times G}$ we need to see that the conditions in Definition~\ref{def:grcompl} hold.   
Now (a) and (c) hold because the inclusion $\io: \Bb^{\times G}\to \Hat \Bb^{\times G}$ is  an \'etale functor by part (iii) above that induces an homeomorphism on the realization by part (ii) above.  Further, (b) holds because $\Bb$ and $\Hat\Bb$ are nonsingular so that the 
isotropy group 
of $ \Bb^{\times G}$ and 
$\Hat \Bb^{\times G}$ at $x\in B$ both equal $\{g\in G \ | \ g*x = x\}$.   This proves (iv).

\MS

Towards (v), note  first that because $(s\times t)(m,g) = (g^{-1}*s(m), t(m))$ by the definition in (i)
and $(s\times t)(\al(g,x)) = (g^{-1}*x,x)$
 by \eqref{eq:gpident2},
 the morphisms
\begin{align}\label{eq:defG}
(m,g)\quad \mbox{ and } \;\;\;  \io^{\times G}(m,g)= \al(g,s(m)) \circ m
\end{align}
do have the same  the source and target.  
To check compatibility of $\io^{\times G}$ with composition, 
note that
$(m,g)\circ (m',h)$
 is defined when $t(m,g) = t(m) = s(m',h) = h^{-1}*s(m')$.
If we  write $$
x: = s(m), \quad h*y: = h*t(m)= s(m'), 
$$
 then  because
$\al(hg, h* x) = \al(g,x)\circ \al(h,h*x)$ by \eqref{eq:gpident2} we have
\begin{align*}
 \io^{\times G}\bigl( (m,g)\circ (m',h)\bigr) &=
 \io^{\times G}\bigl((h*m)\circ m', hg\bigr)\\
& = \al(hg, s(h* m))\circ  (h*m)\circ m' \;\; \mbox{ by }\eqref{eq:gpident4}\\
& = \al(hg, h*x)\circ  (h*m)\circ m'\\
& = \bigl( \al(g,x)\circ \al(h,h*x)\bigr)  \circ \bigl( \al(h^{-1},x) \circ m \circ \al(h, h*y)\bigr) \circ m',  \;\;\mbox{ by }\eqref{eq:gpident3} \\
& =  \al(g,x)  \circ m \circ  \al(h, h*y) \circ m'\;\;\mbox{ by }\eqref{eq:gpident2}\\
& =  \io^{\times G}(m,g)\circ  \io^{\times G}(m',h) \;\;\mbox{ by }\eqref{eq:defG}.
\end{align*}
Since $ \io^{\times G}$ preserves identity maps, it is a functor as claimed.
\end{proof}

\begin{rmk}\label{rmk:timesG}\rm To make sense of the rather complicated formulas in (v) above, notice that the category $\Xx^{\times G}$  is formed by adding morphisms to $\Xx$ that act essentially by left multiplication by $g\in G$.  More precisely, the morphism $\io^{\times G}(m,g)$ of $\Xx$ first multiplies by $g$ and then acts by $m$.  
Because by assumption $\Xx$ supports an inner action of $G$, there is an action $m\mapsto g*m$ of $G$ on $\Mor_\Xx$ as in \eqref{eq:gpident3}. But this should be thought of as a conjugation action since it acts on both source and target by multiplication by $g$. In fact, it follows readily from \eqref{eq:gpident3} that
$$
g*m = \io^{\times G}(m,g^{-1})\circ \io^{\times G} (\id_{g*t(m)}, g).
$$
In particular, we do not equip the category $\Xx^{\times G}$ with a $G$ action and do not claim that the functor $\io^{\times G}$ is in any sense $G$-equivariant.
 \hfill$\er$
\end{rmk}

\subsection{Bundles, sections and structured multisections}\label{ss:bundle}

We now give a much more thorough treatment of the notion of bundle, as preparation for the discussion of compactness controlling data for polyfold bundles in \S\ref{ss:polybundle}.
We will generalize \cite[Def.8.3.1]{TheBook} to a notion of bundle over any \'etale groupoid.  
For that purpose we need to add notions of admissible fibers, linear structures, local trivializations, etc to the \'etale classes (a)-(d) at the beginning of \S\ref{ss:basic}. However, the only instance used in this paper is strong bundles over \'etale groupoids in the class (d) of M-polyfolds, whose properties were summarized in Definition~\ref{def:poly} and are reviewed in more detail here.
Both this section and the next one can be omitted on first reading.

\begin{defn}\label{def:bundle}
A  {\bf  bundle} $\Pp:\Ww\to \Xx$ over an \'etale groupoid $\Xx$
is given by an \'etale space $W$ with continuous maps\footnote{Both maps are surjections due to their algebraic properties below: $P(0_x)=x$ and $\mu(1_{P(w)},w)=w$.} 
$P: W\to X$ and $\mu:\bX \leftsub{s}{\times}_P W\to W$ and the following additional properties and structures: 
\begin{nenumilist}
\item Each fiber $P^{-1}(x)$ for $x\in X$ carries the structure of a vector space, in particular contains a zero vector $0_x\in W$ with $P(0_x)=x$. Moreover, $P$ is equipped with an equivalence class of compatible local trivializations: 

For each $x\in X$ there is an injection $\Phi:P^{-1}(V)\hookrightarrow O\times F$ covering a chart $\phi: V\overset{\simeq}{\to} O$ of an open neighbourhood $V\subset X$ of $x$, which preserves fibers $\phi\circ P|_{P^{-1}(V)}=\pr_O\circ\Phi$ and restricts to bounded linear isomorphisms on fibers, where we make the following requirements in the different \'etale classes:

\begin{itemize}
\item[{\rm (a)}] $\phi$ is the identity map on a metric space $V=O$, $F$ is a finite rank vector space, $\Phi$ is a homeomorphism to $O\times F$;

\item[{\rm (b,c)}] $\phi$ is a chart of the (topological/smooth) manifold $X$, $F$ is a finite rank vector space, $\Phi$ is a homeomorphism to $O\times F$;

\item[{\rm (d)}] $\phi$ is an M-polyfold chart as in \S\ref{ss:scale}, $F$ is a sc-Banach space, $\Phi$ is a homeomorphism to a strong bundle retract $\bigcup_{u\in O}\im\Ga(u)\subset O\times F$ as in \S\ref{ss:scale} and \cite[Def.2.6.2]{TheBook}.
\end{itemize}
Two such local trivializations $\Phi,\Phi'$ are compatible if $$
\Phi'\circ\Phi^{-1}:\Phi(P^{-1}(V'\cap V))\to \Phi'(P^{-1}(V'\cap V))
$$
 is a (a),(b) homeomorphism, (c) diffeomorphism, (d) sc-diffeomorphism with respect to each of the two sc-structures  $O\times F \to O'\times F'$ and $O\times F_1 \to O'\times F'_1$. 
 
Two coverings of $W$ by atlases of compatible local trivializations are equivalent if all local trivializations in their union are compatible.

\item $\mu$ is a lift of the target map $t:\bX\to X$ in the sense that $P\circ\mu(\cdot,w)=t$ for all $w\in W$, it is \'etale and, in class (d), sc-smooth for each of the two sc-structures on $O\times F$.

\item $\mu$ is compatible with the linear and groupoidal structures in the sense that

\begin{itemize}
\item[$-$]
 $\mu(m,\cdot): P^{-1}(s(m))\to P^{-1}(t(m))$ is linear for each morphism $m\in\bX$;

\item[$-$]
 $\mu(1_x,w)=w$ for all $x\in X$ and $w\in P^{-1}(x)$;

\item[$-$]
$\mu(m\circ m',w)=\mu(m',\mu(m,w))$ for  all $m,m'\in\bX$ with $t(m)=s(m')$, $w\in P^{-1}(s(m))$.\footnote{Comparing with \cite[Def.8.3.1]{TheBook} note that we write composition of morphisms in categorical notation.}  
\end{itemize}
\end{nenumilist}

\MS\NI
As in \cite[\S8.3]{TheBook}, these conditions ensure that 
$P$ lifts to a functor $\Pp:\Ww\to \Xx$ given by  
$\Pp(w) = P(w),\; \Pp(m,w) =m$, where $\Ww = (W, \bW)$ is the \'etale groupoid  with
objects $W$ and
\begin{align}\label{eq:Wbundle}
& \bW= \bX \leftsub{s}{\times}_P W, \quad 
(s\times t)(m,w)=(w, \mu(m,w)), \\ \notag &\id_{w}= (\id_{P(w)},w),\quad 
 (m,w)\circ (m',w')= 
 \bigl(m'\circ m , w\bigr) 
\quad \mbox{ if } w' = \mu(m,w).
\end{align}
\end{defn}

The following definition generalizes the notions of sections and bundle maps from the polyfold context \cite[Def.8.3.2-4]{TheBook} to general \'etale bundles. 

\begin{definition}\label{def:bundlemaps}
Let $(P,\mu)$ be a bundle over an \'etale groupoid $\Xx$ as in Definition~\ref{def:bundle}. 
\begin{nenumilist}
\item 
A {\bf section} of $(P,\mu)$ is a map $f:X\to W$ that satisfies $P\circ f=\id_X$, is compatible with morphisms in the sense that  $f(t(m))=\mu(m,f(s(m)))$ for all $m\in\bX$, and is \'etale in the sense that $f$ is (a,b) continuous, (c) smooth, (d) sc-smooth. 
\item
An {\bf \'etale bundle map} from $(P,\mu)$ to another bundle $(P',\mu')$ covering an \'etale functor $\psi: \Xx\to \Xx'$ to another \'etale groupoid $\Xx'=(X',\bX')$ is a map ${\Psi:W\to W'}$ that restricts to linear maps between fibers $\Psi|_{P^{-1}(x)}: P^{-1}(x) \to {P'}^{-1}(\psi(x))$ and is compatible with morphisms in the sense that $\Psi(\mu(m,w))=\mu'(\psi(m),\Psi(w))$ for all $(m,w)\in\bX\leftsub{s}{\times}_P W$. 
Depending on the \'etale class, we moreover require that $\Psi$ is (a,b) continuous, (c) smooth, (d) sc-smooth with respect to both sc-structures. 
\end{nenumilist}
\end{definition}

\begin{lemma} \label{lem:bundlemaps}
The conditions in Definition~\ref{def:bundlemaps} guarantee extensions to functors: 
\begin{nenumilist}
\item 
Any section $f:X\to W$ lifts to an \'etale functor $f:\Xx\to\Ww$ in the same \'etale class as $f$. The lifted map on morphisms $\bX\to \bX\leftsub{s}{\times}_P W$ is given by $m\mapsto f(m): = \bigl(m,f(s(m))\bigr)$.  
\item
Any \'etale bundle map $\Psi:W\to W'$ lifts to a functor $\Psi:\Ww\to\Ww'$ that covers $\psi$ in the sense that $\Pp'\circ\Psi=\psi\circ\Pp$. The lifted map on morphisms $\Psi:\bX\leftsub{s}{\times}_P W\to\bX'\leftsub{s}{\times}_{P'} W'$ is given by $(m,w)\mapsto(\psi(m),\Psi(w))$. It has the same \'etaleness properties as $\Psi$;  in particular in \'etale class (d) $\Psi$ is a strong bundle map in the sense of  \S\ref{ss:scale}. 
\end{nenumilist}
\end{lemma}

\begin{proof}  We prove (i) first note that $f$ takes identity morphisms to identity morphisms since $f(\id_x) = (\id_x, f(x))$ by construction and $(\id_x, f(x)) = \id_{f(x)}$ by \eqref{eq:Wbundle}.  
Further, it follows from the definition of the category $\Ww$  in \eqref{eq:Wbundle} that 
$$
s\times t\bigl(f(m)\bigr) = s\times t \bigl(m,f(s(m))\bigr)  = \bigl(f(s(m)),\ \mu(s(m),f(s(m))\bigr),
$$
Therefore, if $t(m) = s(m')$ we have $$
f(s(m')) = f(t(m)) = \mu\bigl(s(m),f(s(m))\bigr) = t(f(s(m)))
$$
where the second equality holds by the compatibility condition.  Therefore 
the composite $f(m)\circ f(m')$ is defined and by \eqref{eq:Wbundle} we have
\begin{align*}
f(m)\circ f(m') &= \bigl(m,f(s(m))\bigr) \circ \bigl(m',f(s(m'))\bigr) = \bigl(m\circ m', f(s(m))\bigr) \\ 
&= \bigl(m\circ m', f(s(m\circ m'))\bigr)= f(m\circ m').
\end{align*}
This proves (i).
The proof of (ii) is outlined after \cite[Def.8.3.2]{TheBook}.  
We check compatitibility of $\Psi: \bX\leftsub{s}{\times}_P W\to\bX'\leftsub{s}{\times}_{P'} W'$ with source and target maps:
\begin{align*}
&s(\Psi (m,w)) = s (\psi(m), \Psi(w)) = \Psi(w) = \Psi(s(m,w)),\\
 & t (\Psi (m,w)) = t (\psi(m), \Psi(w))  = \mu'(\psi(m),\Psi(w)) =  \Psi(\mu(m,w)) = \Psi(t(m,w)).
\end{align*}
Therefore, the composite $\Psi(m,w)\circ \Psi(m',w')$ is defined whenever $(m,w)\circ (m',w')$ is defined.  
Moreover,
 \begin{align*}
\Psi(m,w)\circ \Psi(m',w') & = (\psi(m), \Psi(w))\circ (\psi(m'), \Psi(w'))
= \bigl(\psi(m)\circ \psi(m'), \Psi(w) \bigr) \\
&  =  \bigl(\psi(m\circ m'), \Psi(w) \bigr) = \Psi(m\circ m', w) = \Psi\bigl((m,w)\circ (m',w')\bigr).
\end{align*}
Finall, $\Psi$ preserves identity maps because
$$
\Psi (\id_w) = \Psi\bigl((\id_{P(w)},w)\bigr) = (\id_{P'(\psi(w))}, \Psi(w)) = \id_{\Psi(w)}.
$$
Hence $\Psi:\Ww\to \Ww'$ is a functor, as claimed. It satisfies $\Pp'\circ\Psi=\psi\circ\Pp$ on objects by the given properties of $\Psi:W\to W'$ and on morphisms by construction.
Moreover, the functor $\Psi$ inherits its \'etaleness properties directly from the \'etale bundle map $\Psi$ and the underlying \'etale functor $\psi$. 
In particular, in the polyfold case (d), the notion of strong bundle map is satisfied since $\Psi$ covers the sc-smooth map $\psi$ on the base, is linear on fibers and restricts to a sc-smooth map $W[1]\to W'[1]$ by Definition~\ref{def:bundlemaps}~(ii).
\end{proof}

\begin{lemma} \label{lem:bundle}
Let $\Pp:\Ww\to \Xx$ be a bundle as in Definition~\ref{def:bundle} given by a pair $(P,\mu)$, and let $f:\Xx\to\Ww$ be a section as in Definition~\ref{def:bundlemaps}. 
\begin{nenumilist}
\item
Every open subset $U\subset X$ induces a {\bf restricted bundle} $\Pp|_U:\Ww|_U\to \Xx|_U$ given by restricting $P$ to $W|_U:=P^{-1}(U)$ and $\mu$ to $\bX|_U\,\leftsub{s}{\times}_P W|_U$, where $\bX|_U:=s^{-1}(U)\cap t^{-1}(U)$.
The resulting bundle category $\Ww|_U$ is the full subcategory of $\Ww$ with objects $W|_U$, and realization $|\Ww|_U| = |P|^{-1}(|U|)$. 

The section $f$ then induces a {\bf restricted section} $f|_U:\Xx|_U \to \Ww|_U$ given by restricting $f|_U:U\to W|_U$. 

\item
Every \'etale functor $\psi: \Yy=(Y,\bY) \to \Xx$ induces a {\bf pullback bundle} $\psi^*\Ww\to \Yy$ given by \begin{align}\label{eq:bundlemu}
P_\Yy :&\; \psi^*W := Y \leftsub{\psi}{\times}_{P} W  \to Y,  \quad (y,w)\mapsto y , \\ \notag
\mu_\Yy : & \;  \bY \leftsub{s}{\times}_{P_\Yy} \psi^*W  \to \psi^*W, \quad
(m, (y,w) ) \mapsto \bigl(t(m), \mu(\psi(m), w )  \bigr) .
\end{align}
In particular, if $\Pp:\Ww\to \Xx$  is a strong polyfold bundle (i.e.\ a bundle in the \'etale category (d), so is the pullback $\psi^*\Ww\to \Yy$.
\item
The pullback bundle in (ii) has a {\bf pullback bundle functor} $\Psi:  \psi^*\Ww \to \Ww$ that is the functor induced by an \'etale bundle map, given on objects and morphisms by 
\begin{align}\label{eq:bundlemap}
 & \psi^*W = \{ (y,w) \,|\, y\in Y , w\in W_{\psi(y)} \} \to W , \quad (y,w)\mapsto w , \\ \notag
& \bY\leftsub{s}{\times}_{P_\Yy} \psi^* W\to \bX\leftsub{s}{\times}_{P}W, \quad (m, (y,w))\mapsto
(\psi(m),w ).
\end{align}
It is a {\bf local bundle isomorphism} in the sense that for open subsets $U\subset Y$ on which $\psi|_U$ has an \'etale inverse, there is an \'etale bundle map $\Psi|_U^{-1}: \Ww|_{\psi(U)} \to \psi^*\Ww|_U$ inverse to $\Psi$.

\item
Given a pullback bundle as in (ii), the section $f: X\to W$ induces a {\bf pullback section} $\psi^*f:\ Y \to \psi^*W$ given  by $\psi^*f(y)=\bigl(y, f(\psi(y)) \bigr)$.   Correspondingly, there is a pullback section functor
$\psi^*f:\Yy \to \psi^*\Ww$ satisfying $\Psi\circ\psi^*f = f\circ\psi$.  

\item
Assume in addition that $\psi$ is an equivalence in the sense of Definition~\ref{def:etale2}~(vii). Then $\Psi$ in (iii) is an \'etale bundle equivalence in the sense of \cite[Def.10.4.1]{TheBook}, that is an \'etale bundle map and local bundle isomorphism (as above) such that $\Psi:  \psi^*\Ww \to \Ww$ is an equivalence of \'etale groupoids in the sense of Definition~\ref{def:etale2}~(vii). 
In the \'etale category (d), it is a strong bundle equivalence in the sense that it also restricts to an equivalence
$\Psi:  \psi^*\Ww[1] \to \Ww[1]$.
Moreover, in (iv) an equivalence $\psi$ induces a homeomorphism $|\psi|: |(\psi^*f)^{-1}(0)| \stackrel{\simeq}\to |f^{-1}(0)|$.
\end{nenumilist}
\end{lemma}

\begin{proof}
We saw in Remark~\ref{rmk:restrict} that when $U$ is open
the restricted category $\Xx|_U$ is also \'etale.    Similarly, $\Ww|_U$ is \'etale, and has realization  
$|\Ww|_U| = |P|^{-1}(|U|)$.   All the other properties 
listed in Definition~\ref{def:bundle}~(i) are local in the base space $X$ in the sense that they involve properties of the restriction of $\Pp$ to a neighbourhood of a point $x\in X$, and hence they automatically hold for the map $\Ww|_U\to \Xx|_U$.  Similarly, because the morphisms in $\Xx|_U$ are defined
to consist of those morphisms in $\Xx$ with both source and target in $U$, 
the properties in Definition~\ref{def:bundle}~(ii), (iii) automatically hold for the restriction.  This proves (i). 
\MS

To construct the pullback bundle in (ii), $\psi^*W = Y \leftsub{\psi}{\times}_{P} W$ is topologized as a subset of $Y\times W$. Then continuity of $P_\Yy$ and $\mu_\Yy$ follows from the corresponding properties of $P,\mu$.  Similarly, the algebraic properties of $\mu_\Yy$ in Definition~\ref{def:bundle}~(ii),(iii) follow from the properties of $\mu$ and functoriality of $\psi$.

The fibers $(\psi^*W )_{y} = \{ (y,w) \,|\,  \psi(y)=P(w) \}=\{y\}\times W_{\psi(y)}$ are given vector space structures by identification with the fibers of $W$. These identifications fit together into the globally defined and continuous map $\Psi$ in \eqref{eq:bundlemap}. So $\Psi$ is a bounded isomorphism on fibers by construction. We will use this map to specify the \'etale structure of the fiber products by finding compatible trivializations as in Definition~\ref{def:bundle}~(i) -- before we can establish the \'etale property of $\mu_\Yy$ in (ii). 

Given any $y\in Y$, we can find an open neighbourhood $V_\Yy\subset Y$ so that the restriction $\psi|_{V_\Yy}:V_\Yy\overset{\simeq}{\to} V\subset X$ is a sc-diffeomorphism to its image. We can moreover choose $V_\Yy$ sufficiently small for $V$ to support a local trivialization $\Phi:P^{-1}(V)\hookrightarrow O\times F$ of $W$ that covers an \'etale chart $\phi: V\overset{\simeq}{\to} O$ around $\psi(y)\in X$.
In this neighbourhood, $\Psi|_{\psi^*W|_{V_\Yy}}$ is a homeomorphism since its inverse $w\mapsto (\psi^{-1}(P(w)),w)$ is continuous. (This also provides the local inverse in (iv).)
The composition $\Phi\circ \Psi : \psi^*W|_{V_\Yy} \to O\times F$ then provides a local trivialization of $\psi^*W$, and any two such local trivializations are compatible since $\Psi$ cancels out so that the transition maps are the same as those for $W$. 

To check the \'etale property of $\mu_\Yy$ in Definition~\ref{def:bundle}~(ii) we moreover need to lift the local bundle chart $\Phi\circ \Psi : \psi^*W|_{V_\Yy} \to O\times F$ to a chart for the morphism bundle 
$$
\bY \leftsub{s}{\times}_{P_\Yy} \psi^*W|_{V_\Yy}
=\bigl\{  (m, (y,w) ) \,|\, m\in s^{-1}(y), y\in V_\Yy, w\in W_{\psi(y)} \bigr\}  . 
$$
Near a fixed morphism $(m_0,(y_0,w_0))$, this is done by shrinking $V_\Yy$ further and fixing an inverse of the source map $s_0^{-1}:V_\Yy\to\bY$ centered at $s_0^{-1}(y_0)=m_0$. Then a chart for the pullback morphism bundle is obtained by precomposing the local trivialization of $W|_V$ with the continuous map $\Pi: \bY \leftsub{s}{\times}_{P_\Yy} \psi^*W, (m,(y,w))\mapsto w$, whose local inverse is $w\mapsto(s_0^{-1}(y_w), y_w,w)$ with $y_w:=\psi^{-1}(P(w))$. 
In these local charts, $\mu_\Yy$ takes the form
\begin{align*}
(\Phi\circ \Psi) \circ \mu_\Yy  \circ (\Phi\circ \Pi)^{-1} :  \;  
\Phi(W_V) &\to \bY \leftsub{s}{\times}_{P_\Yy} \psi^*W  \to \psi^*W \to O\times F , \\
\Phi(w) \mapsto \bigl(s_0^{-1}(y_w),y_w, w \bigr) 
&\mapsto \bigl(t(s_0^{-1}(y_w)), \mu(\psi(s_0^{-1}(y_w)), w )  \bigr) \mapsto \Phi\bigl(\mu( s_{\Xx,0}^{-1}(P(w)), w ) \bigr) 
\end{align*}
with $s_{\Xx,0}^{-1}= \psi \circ s_0^{-1}\circ \psi^{-1}: V \to s_\Xx^{-1}(V)$ the local inverse to the source map of $\Xx$ obtained by pullback of $s_0^{-1}$ with $\psi$. 
This coincides with $\mu$ in the local trivialization $\Phi$ (and its lift to morphisms via $s_{\Xx,0}^{-1}$), hence $\mu_\Yy$ inherits all required \'etaleness properties from the assumed \'etaleness properties of $\mu$. 
\MS

To prove (iii), note that $\Psi$ on objects is the bundle map \eqref{eq:bundlemap} which describes the identification of fibers in the construction of the pullback bundle and is used in (ii) to transfer local bundle charts from $W$ to $\psi^*W$. 
So the map $\Psi:\psi^*W\to W$ in the corresponding local charts of $\psi^*W$ and $W$ is simply represented by the identity map -- thus $\Psi$ is linear, a bounded isomorphism on fibers, and \'etale, in particular in class (d) it is sc-smooth in both sc-structures. The same holds for its locally defined inverse -- noting that the restriction  $\Psi: (\psi^*W)|_U \to W|_{\psi(U)}$ has an \'etale inverse if and only if the map $\psi|_U: U\to \psi(U)$  does.
In summary, $\Psi:\psi^*W\to W$ is a local (strong in class (d) bundle isomorphism. 

The same holds on morphisms since -- as with all bundle constructions -- the map on morphisms is obtained from the map on objects by pullback with the source map. This transfers the \'etaleness properties to morphism level. 
In detail, we apply Lemma~\ref{lem:bundlemaps}  after checking compatibility with morphisms in the sense of Definition~\ref{def:bundlemaps}~(ii): We have $\Psi(\mu_\Yy(m,(y,w)))=\mu(y,w))= \mu(\psi(m),\Psi(y,w))$ for all $(m,(y,w))\in\bY\leftsub{s}{\times}_{P_\Yy} \psi^* W$.
Then the lifted map on morphisms is 
$(m,(y,w))\mapsto (\psi(m),\Psi(y,w))=(\psi(m),w)$ as in \eqref{eq:bundlemap}. 

This proves (iii) and begins the proof of the first part of (v), where we assume that $\psi$ is an equivalence in the sense of Definition~\ref{def:etale2}~(vii), i.e.\ it identifies morphisms $\Mor_{\Yy} (y,y') \simeq \Mor_{\Xx} (\psi(y),\psi(y'))$ and induces a homeomorphism $|\psi|:|\Yy|\to |\Xx|$. 

To check that $\Psi: \psi^*\Ww \to \Ww$ (and in the polyfold \'etale class (d), also its restriction $\Psi: \psi^*\Ww[1] \to \Ww[1]$) is an equivalence of \'etale groupoids we first need to verify that for any given $(y,w),(y',w')\in\psi^*W$ (so in class (d), in particular for pairs 
$
(y,w),(y',w')\in\psi^*W[1]$) the map $(m, (y,w))\mapsto (\psi(m),w )$ restricts to a bijection from
\begin{align*}
&\Mor_{\psi^*\Ww}((y,w),(y',w')) =\\
&\qquad   \bigl\{  ( m , (y,w) ) \,\big|\,  m\in\Mor_\Yy, s(m)=y, t(m)=y', \mu(\psi(m),w)=w' \bigr\}
\end{align*}
to $\Mor_{\Ww}(\Psi(y,w),\Psi(y',w')) = \Mor_{\Ww}(w,w')$, where
$$
\Mor_{\Ww}(w,w') =  \bigl\{  ( n , w ) \,\big|\,  n\in \Mor_\Xx, s(n)=\psi(y), t(n)=\psi(y'), \mu(n,w)=w' \bigr\} .
$$
Since $y,y'$ and $w,w'$ are fixed, this is equivalent to $\psi$ mapping 
$$\bigl\{ m \in\Mor_\Yy(y,y') \,| \, \mu(\psi(m),w)=w'  \bigr\}$$
bijectively to $$
\bigl\{  n \in \Mor_\Xx(\psi(y), \psi(y') ) \,|\, \mu(n,w)=w' \bigr\}. $$ 
Here we know that $\psi : \Mor_\Yy(y,y') \to \Mor_\Xx(\psi(y), \psi(y') )$ is a bijection, hence its restriction
$$\bigl\{ m \ \,| \, \mu(\psi(m),w)=w'  \bigr\} \to \bigl\{  n  \,|\, \mu(n,w)=w' \bigr\}$$ is injective. To check that it is surjective, consider $n\in \Mor_\Xx(\psi(y), \psi(y') )$ with $\mu(n,w)=w'$. From the given bijectivity we know $n=\psi(m)$ for some $m\in \Mor_\Yy(y,y')$, and can deduce $\mu(\psi(m),w)=\mu(n,w)=w'$, which confirms surjectivity of the restricted map. 

The second requirement for $\Psi: \psi^*\Ww \to \Ww$ to be an equivalence of \'etale groupoids is that the induced map $|\Psi|: |\psi^*\Ww| \to |\Ww|$ is a homeomorphism. 
This map is continuous maps since $\Psi$ is continuous and a functor. To check that $|\Psi|$ is bijective, note that the fibers of $|\psi^*\Ww|$ resp.\ $|\Ww|$ over points represented by $y\in Y$ and $\psi(y)\in X$ are the quotients of $\{y\}\times W_x$ resp.\ $W_x$ by the morphisms for fixed $y=y'$ that we identified above. So the fact that $\Psi$ restricts to isomorphisms of fibers $\{y\}\times W_x \to W_x, (y,w)\mapsto w$ implies that $|\Psi|$ is a bijection on fibers. It moreover covers the homeomorphism $|\psi|:|\Yy|\to |\Xx|$ and hence is globally bijective.  Finally, $|\Psi|$ is an open map and hence a homeomorphism because the local inverses of $\Psi$ in (iii) descend to local inverses of $|\Psi|$ -- using again the fact that we identified the morphisms above. 

In the polyfold class (d), we additionally need $|\Psi| \bigl( |\psi^*\Ww[1]| \bigr) = |\Ww[1]|$ to ensure that the restriction of $|\Psi|$ induces a homeomorphism $|\psi^*\Ww[1]| \to |\Ww[1]|$.
This is guaranteed by the fact that $\Psi$ is a strong bundle map, i.e.\ preserves the sc-filtration that defines the shift ``$[1]$'' and has local inverses that are strong bundle maps as well. This proves the first assertion in (v). 

Going back to (iv), the map $\psi^*f: Y\to \psi^*W, y\mapsto \bigl(y, f(\psi(y))\bigr)$  satisfies 
$P_\Yy\circ (\psi^*f) = \id_Y$, and is \'etale because the map $y\mapsto f(\psi(y))$ is a composite of \'etale maps and hence \'etale. To complete the proof of (iv) it remains to check that $\psi^*f$ is compatible with morphisms, i.e.\ that for all morphisms $m'$ in $\Yy$ we have 
 $$
(\psi^*f)(t(m')) = \mu_\Yy\bigl(m', (\psi^*f)(s(m'))\bigr) \in \psi^*W.
$$
Since $f$ is compatible with morphisms we know
$f(t(m)) =  \mu_\Xx\bigl(m, f(s(m))\bigr)$ for all morphisms $m$ in $\Xx$, where for clarity we write $\mu_\Xx$ for the operator $\mu$ in the bundle $\Ww\to \Xx$. 
Therefore, using the definition of $\mu_\Yy$ in \eqref{eq:bundlemu}
we have
\begin{align*}
\mu_\Yy\bigl(m', (\psi^*f)(s(m'))\bigr) & = \bigl(t(m'),  \mu_\Xx(\psi(m'),  f(\psi(s(m')))\bigr)\\
& = \bigl(t(m'),  \mu_\Xx(\psi(m'),  f(s(\psi(m')))\bigr)\\
& =  \bigl(t(m'),  \mu_\Xx(m,  f(s(m))\bigr)\quad \mbox{ where } m: = \psi(m')\\
& =  \bigl(t(m'),  f(t(\psi(m')))\bigr) \quad \mbox{ since $f$ is a section}\\
& = (\psi^*f)(t(m')),
\end{align*}
where the last equality uses the fact that $t(\psi(m')) = \psi(t(m'))$ and the definition $\psi^*(f)(y) =
(y, f(\psi(y))$. That proves (iv). 

Finally, we prove the second part of (v). By definition, an equivalence $\psi$ induces a homeomorphism $|\psi|: |\Yy|\to |\Xx|$. So to establish the homeomorphism $|\psi|: |(\psi^*f)^{-1}(0)| \stackrel{\simeq}\to |f^{-1}(0)|$ it remains to show that the image of $|(\psi^*f)^{-1}(0)|$ under $|\psi|$ is $|f^{-1}(0)|$. This is guaranteed by the fact that the pullback construction gives $y\in (\psi^*f)^{-1}(0) \Leftrightarrow \psi(y)\in f^{-1}(0)$.  This completes the proof.
\end{proof}

To end this subsection, we discuss  multisections and their structurability. As we will see, these are generally needed 
to achieve transversely cut out perturbed solution sets
in the presence of isotropy.

\begin{defn} \label{def:multisdef0}
Let $\Pp:\Ww\to \Xx$ be an \'etale bundle.  A {\bf multisection} of  $\Pp$ is a functor 
\begin{align}\label{eq:multisdef0}
\Lambda:\Ww\to\Q^{\geq 0}:=\Q\cap[0,\infty),
\end{align}
that can in local charts $W|_V$ be represented as $$
 \Lambda(w)=\sum_{\{ i \,|\, w=\s_i(P(w))\}} \sigma_i
 $$ 
 in terms of finitely many \'etale sections\footnote
 {
 The $\s_i$ here are simply maps $V\to W$ with no functorial properties, rather than  sections of bundles as in Definition~\ref{def:bundlemaps}.}
  $\s_i:V\to W|_V
 $ and weights $\sigma_i>0$. 

  A {\bf global section structure} for a multisection functor $\La$ of a bundle $\Ww\to\Xx$ is a finite family of sections $(\s_i: X \to W)_{i\in \Ii}$ such that 
\begin{align}\label{eq:global0}
\La(w)=\tfrac 1{|\Ii|} \#\{ i \in \Ii \,|\, \s_i(P(w))=w\}.
\end{align}
A multisection functor $\La$ is called {\bf globally structured} 
if it is equipped with a global section structure $(\s_i)_{i\in \Ii}$ and a global natural correspondence
$$
\ka\, : \; \Mor_{\Xx} \to \{  b: \Ii\to \Ii \; \text{bijection} \}
$$
satisfying
\begin{enumilist}
\item
$\ka$ is  locally constant;
\item for all $m\in\Mor_\Xx$ and $i\in\Ii$.
\begin{align}\label{eq:global1}
\s_{\ka(m)(i)} (t(m)) = \mu(m, \s_i(s(m))) \in W_{t(m)}.
\end{align}
\end{enumilist} 
\end{defn}

\begin{defn}\label{def:struct}\rm 
Let $\Xx$ be an \'etale proper category.   A collection of open neighbourhoods $(U(x))_{x\in X}$ of the points $x\in X$ is called a {\bf good system of neighbourhoods} if
\begin{itemize}\item[-] for every $x\in X$ the target map $t: s^{-1}(\cl_X(U(x))\to X$ is proper;
\item[-] each $U(x)$ is equipped with the natural action of the isotropy group $G_x$ at $x$;
\item[-] for any pair $x,x'\in X$ and each connected component $\Si$ of $\Mor(U(x),U(x'))$  the maps $s:\Si\to U(x)$ and $t:\Si\to U(x')$ are \'etale bijections onto their images.
\end{itemize}
This system is said to be {\bf connected} if each set $U(x)$ is connected.  
\end{defn}

\begin{lemma}\label{lem:struct1} 
  If $(U(x))_{x\in X}$ is a good system of neighbourhoods, and  $U(x')\subset U(x)$ is any 
$G_x$-invariant open subset of $U(x)$, then 
$(U'(x))_{x\in X}$ is a good system of neighbourhoods. In particular,
any \'etale category with a good system of neighbourhoods  also has a good connected  system
 of neighbourhoods.  Further, the intersection $(U(x)\cap U'(x))_{x\in X}$ of any two good systems
 of neighbourhoods is also a good system of neighbourhoods.  \end{lemma}
 
 \begin{proof} The first claim  is immediate.  To prove the second claim, given any good system of neighbourhoods $(U(x))_{x\in X}$, define $U'(x)$ to be the connected component of $U(x)$ containing $x$.  Then $U'(x)$ is $G_x$-invariant, and the other conditions are immediate.
\end{proof}

\begin{lemma}\label{lem:struct2}  Let $\Xx$ be an \'etale proper category with a 
good system of neighbourhoods  and let $\La$ be  a multisection  with a global structure provided by
 the sections $(\s_i)_{i\in \Ii}$ and correspondence $\ka$.
   For $x\in X$, let $\Ii_x$  be the quotient of $\Ii$ by the equivalence relation $i\sim_x j$ if $\s_i(y) = \s_j(y)$ for all $y$ near $x$, and denote the quotient map by $\pr_x: \Ii\to \Ii_x$. 
  Then we may arrange that the following conditions hold.
   \begin{itemize}\item[\rm(a)]  Each $U(x)$ is connected and $\Ii_x = \Ii_y$ for all $y\in U(x)$.
   Moreover $i\sim_x j $ exactly if there is $z\in U(x)$ such that $\s_i(z) = \s_j(z)$.
     \item[\rm(b)] 
   For all $x\in X$ and all composable pairs of morphisms $m, m'$ with $s(m)\in U(x)$, we have %
\begin{align}\label{eq:simXI}
    \ka(m\circ m')(i) \sim_{x}  \ka(m')\circ \ka(m)(i), \quad \forall i\in \Ii.
\end{align}
   \item[\rm(c)] 
       The action of the isotropy  group $G_x$ on $\Ii$ given by $\ka$ descends to an action on $\Ii_x$ via the identity
    $$
   g* \ui = \pr_x\bigl(\ka(m_{g})(r_x(i))\bigr)\in \Ii_x,\quad 
   $$
   where $m_{g} \in\Mor(x,x)=G_x$ is the morphism $g\in G_x$, and $r_x:\Ii_x\to \Ii$ is any right inverse to $\pr_x$. In particular, $\pr_x$ is $G_x$-equivariant, 
\end{itemize}
\end{lemma}
\begin{proof} First notice that
if $(U(x))_{x\in X}$ is a good system of neighbourhoods, and  $U(x')\subset U(x)$ is any 
$G_x$-invariant open subset of $U(x)$, then 
$(U'(x))_{x\in X}$ is a good system of neighbourhoods. Therefore we may assume that each $U(x)$ is connected by replacing it by its connected component containing $x$.  Further, it
 follows immediately from the definition that if $i\sim_x j$ then $i\sim_y j$ for all $y$ in some open neighbourhood of $x$.
Therefore, because $\Ii$ is a finite set,  the two equivalence relations $\sim_x, \sim_y$  must agree on some $G_x$-invariant neighbourhood of $x$.  Therefore we can shrink the $U(x)$ so that the first claim in (i) holds. The second claim in (i) holds because $\Ii$ is finite, and $i\not\sim_x j$ if and only if there is $z$ near $x$ such that $\s_i(z)\ne \s_j(z)$.

Towards (b), notice first that the equivalence in \eqref{eq:simXI} holds when $s(m) = x$ because
\begin{align*}
\s_{ \ka(m\circ m')(i)}(t(m\circ m'))& =  \mu\bigr(m\circ m', \s_i(s(m\circ m'))\bigl)\\
& = \mu\bigr(m',\mu(m, \s_i(s(m))\bigl)\\
& =  \mu\bigr(m', \s_{\ka(m)i} (t(m))\bigl)\quad \mbox{by }\  \eqref{eq:global1}\\
& = \s_{\ka(m')\circ \ka(m)(i)}(t(m')).
\end{align*}
Since $\ka$ is locally constant by definition and the relation $\Ii_x$ is locally constant by (a), 
this identity continues to  hold when we vary the composable pair $m,m'$ so that $s(m)$ is close to $x$.
This proves (b).

Note finally that (c) holds because \eqref{eq:global1} and part (a) above together imply that
 the $G_x$ action on $\Ii$ preserves the  relation $\equiv_x$.
\end{proof}

\begin{rmk}\rm  (i) If we assume given  $(\s_i)_{i\in \Ii}$ and $\ka$  satisfying the conclusions of 
Lemma~\ref{lem:struct2}, and then
define $\La$ by \eqref{eq:global0},  the fact that $\La$  is a functor follows from 
\eqref{eq:global1}.
\MS

\NI (ii)  In general it is not true that the correspondence $\ka$ in Definition~\ref{def:multisdef0} 
is compatible with composition, that is  there could be a composable pair
  $m,m'$ of morphisms in $\Xx$ such that the two bijections
$\ka(m\circ m'), \ka(m)\circ \ka(m'): \Ii\to \Ii$ are different.  However by
Lemma~\ref{lem:struct2}~(b)  this does hold locally modulo the equivalence relation $\sim_x$.
\hfill$\er$
\end{rmk}

We now rephrase Lemma~\ref{lem:struct2} in the language used in \cite{TheBook}. As above, we denote by $r_x:\Ii_x\to \Ii$ any map satisfying $\pi_x(r_x(\ui)) = \ui, \ \forall\  \ui\in \Ii_x$.

 \begin{lemma}\label{lem:struct3}   In the situation of Lemma~\ref{lem:struct2}, the following holds.
  \begin{itemize}   \item[\rm(a)] 
 $\bigl(U(x), (\s_{r_x(\ui)})_{\ui\in \Ii_x}\bigr)$ is a symmetric section structure for $\La$ in the sense of \cite[Def.13.2.5]{TheBook}.
   \item[\rm(b)]  For each pair $x,x'\in X$ the map
$$
\tau_{x,x'}: 
   \Mor(U(x), U(x')) \to {\rm Biject}(\Ii_x, \Ii_{x'}), \quad m\mapsto \pr_{x'}\circ \ka(m)\circ r_x
   $$
is a correspondence in the sense of \cite[Def.13.3.4]{TheBook} between    the local structures
$\bigl(U(x), (\s_{r_x(\ui)})_{\ui\in \Ii_x}\bigr)$ and $\bigl(U(x'), (\s_{r_{x'}(\uj)})_{\uj\in \Ii_{x'}}\bigr)$.
 \end{itemize} 
\end{lemma}

\begin{proof}
To prove (a), we need to check the two conditions in \cite[Def.13.2.5]{TheBook}. The second condition follows immediately from \eqref{eq:global0} and the fact that  $i\sim_x j$ implies that  $\s_i(y) = \s_j(y)$ for all $y\in U(x)$ so that $\s_{r_x(\ui)}(y)$ is independent of the choice of lift $r_x:\Ii_x\to \Ii$.  The first condition is written  in our notation as 
$$
\s_{g*\ui}(g*y) = \mu(m_{g,y}, \s_{r_x(i)}(y)),\quad\forall\ y\in U(x),\ \ui\in \Ii_x,\ g\in G_x,
$$
where $m_{g,y}\in \Mor(U(x),U(x))$ is the morphism $y\mapsto g*y$.  This is an immediate consequence of \eqref{eq:global1}.

To prove (b), we must check the three conditions in \cite[Def.13.3.4]{TheBook}.  The first condition, i.e. local constancy of $\tau_{x,x'}$, holds because the map $\ka$ in \eqref{eq:global1} is locally constant.\footnote
{
In view of the notation $\tau_{x,x'}$ used in \cite{TheBook}  for these local correspondences, it would have been more natural to write $\tau$ instead of $ \ka$ for the global correspondence in Definition~\ref{def:multisdef0}. However, this might cause confusion with the use of $\tau$ to denote the functor $\Xx_\Vv^{\less G}\times E \to \Ww_\Vv$ in Theorem~\ref{thm:globstab}.
}
 The third condition holds because it is a restatement of \eqref{eq:global1}. The second condition 
states that  $\tau_{x,x'}$ is compatible with the actions of $G_x, G_{x'}$ on $U(x), U(x')$. 
This can be written as
$$
\tau_{x,x'}(g* m *g') = g'* \tau_{x,x'}(m)* g: \Ii_x\to \Ii_{x'},
$$
 for all $m\in \Mor(U(x),U(x')),\ g\in G_x, \ g'\in G_{x'},$
where on the left hand aside $g* m *g'$ is the morphism that first multiplies by $g$ then does $m$ and then multiplies by $g'$,
%
and on the right denote by $g$ (resp. $g'$) the map  $\Ii_x\to \Ii_x: \ui\mapsto g*\ui$
(resp. $\Ii_{x'}\to \Ii_{x'}: \uj\mapsto g'*\uj$). Again, this identity is a straightforward consequence of the 
compatibility of $\ka$ with composition stated in \eqref{eq:simXI}.
\end{proof}

Here is the main result about multisections.

\begin{prop}\label{prop:structmulti}  If $\Xx$ has a good system of neighbourhoods $(U(x))_{x\in X}$, then
every globally structured multisection $(\La, (\s_i)_{i\in \Ii}, \ka)$ is structurable in the sense of \cite[Def.13.3.6]{TheBook}.   Moreover $\La$ determines a unique structured multisection  $[\La, \Uu, \Ss, \tau]$  in the sense of \cite[Def.13.3.8]{TheBook}.
\end{prop}
\begin{proof} The first claim holds provided that the maps $\tau_{x,x'}$ defined in Lemma~\ref{lem:struct3}(b) above  satisfy  the additional identities
\begin{align*}
{\rm (i)}&\qquad \tau_{x,x}(m_{g,y})(\ui)  = g*\ui,\quad \forall g\in G_x, y\in U(x), \ui\in \Ii_x,\\
{\rm (ii)}&\qquad   \tau_{x',x}(m) =  \bigr(\tau_{x,x'}(m^{-1})\bigl)^{-1}, \quad \forall m\in \Mor(U(x'), U(x)),
\end{align*}
where the morphism $m_{g,y}: y\to g*y$ in (i) implements the action of $g$.\footnote
{These identities are simplified forms of those in \cite{TheBook} because we consider $\ka$ and the associated correspondences $\tau_{x,x'}$  to take values in the set of bijections $\Ii\to \Ii$, rather than in the set of correspondences $\Ii\stackrel{a}\leftarrow \Ii\stackrel{b}\to \Ii$ as
in \cite{TheBook}.}
Both identities follow immediately from our definitions. This proves the first claim.

 By \cite[Def.13.3.6]{TheBook} a {\bf structured multisection} is an equivalence class of tuples $(\La, \Uu, \Ss,\tau)$ where $\La$ is a multisection, $\Uu$ is a good system of neighbourhoods, $\Ss$ is a collection of local sections that determine $\La$ via local versions of \eqref{eq:global0}, and $\tau = (\tau_{x,x'})$ is a family of correspondences as in Lemma~\ref{lem:struct3}.
Since any two good systems of neighbourhoods $\Uu, \Uu'$ are equivalent under restriction,
and  the collection $\Ss$ is uniquely determined up to equivalence
by restricting the fixed sections $(\s_i)_{i\in \Ii}$, 
it follows straightforwardly from the definitions in \cite{TheBook} that the equivalence class 
$[\La, \Uu, \Ss,\tau]$ of this tuple is unique.
\end{proof}

\subsection{Compactness control} \label{ss:polybundle}

We specify to the polyfold \'etale class (d) for this subsection.  
Since polyfolds are typically infinite dimensional, we cannot expect to find precompact neighbourhoods of $|f^{-1}(0)|\subset|\Xx|$. Instead, compactness of the perturbed solution sets will be achieved by the following notions from  \cite[\S12.2, \S15.3]{TheBook}.

\begin{defn}\label{def:control-compact}
Let $f:\Xx\to \Ww$ be a sc-Fredholm section functor of a strong bundle ${\Pp:\Ww\to \Xx}$ as in Definition~\ref{def:poly}.
\begin{nenumilist}
\item
An {\bf auxiliary norm} on $\Ww$ is a continuous map $N : W[1] \to [0,\infty)$ compatible with morphisms (i.e. $N(s(m))=N(t(m))$ for $m\in\Mor_\Ww$), which restricts to a complete norm on each fiber $W[1]\cap P^{-1}(x)$, and such that $N(h_k)\to 0, P(h_k)\to x\in X$ implies convergence $h_k\to 0_x$ in the topology of $W[1]$ specified in  \S\ref{ss:scale}.  
\item
A {\bf saturated open neighbourhood of the solution category} is an open subset $\Uu\subset X$ that contains $f^{-1}(0):=\{x\in X \,|\, f(x)=0_x\}$. 
\item
A pair $(N,\Uu)$ of data as above {\bf controls the compactness of $\mathbf f$} if the closure of 
$\bigl| \{ x\in \Uu \,|\, f(x)\in W[1], N(f(x))\leq 1 \} \bigr| \subset |\Xx|$ is compact. 
\end{nenumilist}
\end{defn}

\begin{rmk}\rm \label{rmk:control-compact}
Compactness of $|f^{-1}(0)|\subset|\Xx|$ implies the existence of compactness controlling data $(N,\Uu)$ by \cite[Thm.12.4.2]{TheBook} (since we assume paracompactness of $|\Xx|$ in Definition~\ref{def:poly}). 
This will allow to preserve compactness when perturbing the solution set.

To compare perturbed solution sets arising from different choices of compactness control data it is useful to note the following: 
Given compactness control data $(N,\Uu)$ for $f$, any pair $(N',\Uu')$ of a smaller saturated open set $f^{-1}(0)\subset \Uu'\subset \Uu$ and larger auxiliary norm $N'\geq N$ also controls compactness of $f$. This is since $\{ x\in \Uu' \,|\, N'(f(x))\leq 1 \}$ is contained in $\{ x\in \Uu \,|\, N(f(x))\leq 1 \}$, and hence the closure of $\bigl| \{ x\in \Uu' \,|\, f(x)\in W[1], N'(f(x))\leq 1 \} \bigr|$ is a closed subset of the compact set ${\rm cl}_{|\Xx|}\bigl( \bigl| \{ x\in \Uu \,|\, f(x)\in W[1], N(f(x))\leq 1 \} \bigr|\bigr)$, and thus compact. 

In particular, if $V\subset X$ is a saturated open subset that contains $f^{-1}(0)$, then we can choose compactness controlling data $(N,\Uu')$ for $f$ with $\Uu'\subset V$ by applying the above observation to  any choice of compactness controlling data $(N,\Uu)$ and $\Uu':=\Uu\cap V$. 
 \hfill$\er$
\end{rmk}

The pullback constructions in Lemma~\ref{lem:bundle} are compatible with the sc-Fredholm notion and compactness control as follows.

\begin{lemma}\label{lem:pullback}   
Suppose that $f:\Xx\to \Ww$ is a sc-Fredholm section functor of a strong bundle ${\Pp:\Ww\to \Xx}$ and that $\psi: \Yy\to \Xx|_V$ is an equivalence of ep-groupoids in the sense of Definition~\ref{def:etale}~(vii), where $\Xx|_V$ denotes the restriction to an open (but not necessarily saturated) 
 subset $V\subset X$ as in Lemma~\ref{lem:bundle}~(i).  Then $\psi^* f: \Yy\to \psi^*\Ww$ is sc-Fredholm. 

If, moreover, the zero set $|f^{-1}(0)|$ is compact and contained in $|V|$, then we can find compactness controlling data $(N,\Uu)$ for $f$ as in Definition~\ref{def:control-compact} with $|{\rm cl}_X(\Uu)|\subset |V|$. 
And given such data $(N,\Uu)$, the pullback data $\bigl(N_\Yy: = N\circ \Psi, \ \Uu_\Yy = \psi^{-1}(\Uu)\bigr)$ controls the compactness of $\psi^*f$, where $\Psi:\psi^*\Ww\to\Ww$ is the pullback bundle functor from Lemma~\ref{lem:bundle}~(iii).

\end{lemma} 
\begin{proof}  
The sc-Fredholm property \cite[Def.12.1.1]{TheBook} of the functor $\psi^* f$ just requires the sc-Fredholm property \cite[Def.3.1.16]{TheBook} for the section of the object bundle $\psi^*f: Y\to \psi^*W, y\mapsto \bigl(y, f(\psi(y))\bigr)$. The three properties encoded in the sc-Fredholmness of $f$ --- (1) sc-smoothness, (2) regularization $f(x)\in W_{m,m+1} \Rightarrow x\in X_{m+1}$, (3) local coordinates for base and bundle in which the section has a contraction form --- are all local properties on the base. So to check that these properties transfer to $\psi^* f$ it suffices to consider domains on which $\psi$ restricts to a sc-diffeomorphism. On such domains they follow -- as in \cite[Thm.10.4.5]{TheBook} -- from (1) the chain rule in scale calculus, (2) the fact that sc-maps preserve the sc-indices, (3) composing local coordinates for the base with $\psi$.

Now assume that $|f^{-1}(0)|\subset |V|$ is compact. 
To find compactness controlling data with $|{\rm cl}_X(\Uu)|\subset |V|$ recall that $|\Xx|$ is metric and $\im|\psi|=|V|\subset|\Xx|$ is an open subset that contains $|f^{-1}(0)|$. So, since $|\Xx|$ is a normal topological space, we can find an open neighbourhood $N_S\subset|\Xx|$ of the closed subset $S=|f^{-1}(0)|$ whose closure does not intersect the closed subset $|\Xx|\less |V|$. Its preimage $V':=\pi_\Xx^{-1}(N_S)\subset\Obj_\Xx$ is a saturated open neighbourhood of $f^{-1}(0)$, so as in Remark~\ref{rmk:control-compact} we can choose compactness controlling data $(N,\Uu)$ for $f$ with $\Uu\subset V'$. This construction gives the desired closure control -- using continuity of $\pi_\Xx$ --
$$
|{\rm cl}_X(\Uu)| \subset |{\rm cl}_X\bigl(V'=\pi_\Xx^{-1}(N_S)\bigr) |
\subset |{\rm cl}_{|\Xx|} (N_S)| \subset |V|  . 
$$
Next, let $N: \Ww[1]\to [0,\infty)$ be an auxiliary norm. Then we claim that $N_\Yy:=N\circ\Psi$ defines an auxiliary norm $\psi^*\Ww[1]\to [0,\infty)$ as in Definition~\ref{def:control-compact}. 
Indeed, it is continuous since $N$ and $\Psi$ are continuous, compatible with morphisms since $\Psi$ is a functor, and restricts to a complete norm on each fiber since pullback bundle functor $\Psi:\psi^*\Ww\to\Ww$ from Lemma~\ref{lem:bundle}~(iii) restricts to a strong sc-isomorphism on each fiber. 
The convergence property above (and \cite[Def.12.2.1~(2)]{TheBook}) is preserved under pullback: 
Given $(y_n,w_n)\in \Obj_{\psi^*\Ww}[1] = Y \leftsub{\psi}{\times}_{P} W[1]$ with $P_\Yy(y_n,w_n)=y_n\to y \in Y$ and
$N_\Yy(y_n,w_n) = N(w_n) \to 0$ we have $w_n\in \Obj_{\Ww}[1]$ with $P(w_n)= \psi(y_n)\to  \psi(y) \in X$ since $\psi$ is continuous. Then convergence of the auxiliary norm $N(w_n) \to 0$ implies $w_n \to 0_{\psi(y)}$ in $\Obj_{\Ww}[1]$, which implies the required convergence $(y_n,w_n) \to (y,0_{\psi(y)})=0_y$ in $\Obj_{\Ww_\Yy}[1]$. 

Finally, to check that $\Uu_\Yy:=\psi^{-1}(\Uu)\subset Y$ together with $N_\Yy$ controls compactness of $\psi^*f$ as in Definition~\ref{def:control-compact}, first note that $\Uu_\Yy$ is open since $\psi$ is continuous, saturated since $\psi$ is a functor and $\Uu$ is saturated in $\Xx$, and a neighbourhood of $(\psi^*f)^{-1}(0)$ since $|\psi|$ identifies the zero sets and $|f^{-1}(0)|\subset |\Uu|$. Now it remains to show that 
\begin{align*} 
 \bigl|   \bigl\{ y \in \Uu_\Yy \subset Y \,\big|\, N_\Yy(\psi^*f(y))\leq 1  \bigr\}     \bigr|   
&= \bigl|   \bigl\{ y \in  Y \,\big|\, \psi(y)\in\Uu, N(f(\psi(y)))\leq 1  \bigr\}     \bigr|   \\
&= \bigl| \psi^{-1}  \bigl\{ x \in \Uu \subset X \,\big|\,  N(f(x))\leq 1  \bigr\}     \bigr|   \\
&= |\psi|^{-1}   \bigl| \bigl\{ x \in \Uu \subset X \,\big|\,  N(f(x))\leq 1  \bigr\}     \bigr|  
\end{align*}
has compact closure in $|\Yy|$. Since $(N,\Uu)$ controls compactness of $f$, we know that $B_f:= \bigl| \bigl\{ x \in \Uu \subset X \,\big|\,  N(f(x))\leq 1  \bigr\}   \bigr|$ has compact closure in $|\Xx|$. Moreover, $|\psi|:|\Yy|\to |V|\subset|\Xx|$ is a homeomorphism by Definition~\ref{def:etale2}~(vii). 
So $|\psi|$ intertwines closures and maps compact sets to compact sets as long as we can guarantee that the closure of $B_f$ is contained in $|V|$. This is why we chose $\Uu\subset X$ with $|{\rm cl}_X(\Uu)|\subset |V|$, as this guarantees
${\rm cl}_{|\Xx|}(B_f)  \subset {\rm cl}_{|\Xx|}(|\Uu|)  = \bigl|{\rm cl}_X(\Uu)\bigr| \subset |V|$, 
where we used Lemma~\ref{lem:etale1}~(ii) to identify closures. 
Thus $(N_\Yy,\Uu_\Yy)$ controls compactness of $\psi^*f$. 
\end{proof}

To activate the compactness control when perturbing a sc-Fredholm section $f: X\to W$ over an M-polyfold $X$, we can now restrict perturbations to sc$^+$-sections $s:X\to W[1]$ that are bounded by $N(s(x))\leq 1$ (or in fact any other constant by \cite[Thm.12.4.3]{TheBook}), so that 
the perturbed zero set 
$$
\bigl| (f-s)^{-1}(0)\bigr| = \bigl| \{ x\in \Uu \,|\, f(x)=s(x) \}\bigr| \subset
\bigl| \{ x\in \Uu \,|\, f(x)\in W[1], N(f(x))\leq 1 \} \bigr|
$$ is automatically compact.

Next, the presence of isotropy in any \'etale proper groupoid $\Xx$ generally requires the use of multivalued perturbations to achieve transversely cut out perturbed solution sets. 
In the case of a sc-Fredholm section $f:\Xx\to\Ww$ over an ep-groupoid $\Xx$, we additionally need sc$^+$-perturbations, leading to the formalism of \cite[\S13.2]{TheBook}: A {\bf sc$^+$-multisection} is a functor 
\begin{align}\label{eq:multisdef}
\Lambda:\Ww\to\Q^{\geq 0}:=\Q\cap[0,\infty),
\end{align}
 which can in local charts $W|_V$ be represented as $$
 \Lambda(w)=\sum_{\{ i \,|\, w=\s_i(P(w))\}} \sigma_i
 $$ in terms of finitely many sc$^+$-sections $\s_i:V\to W[1]|_V
 $ and weights $\sigma_i>0$. 
The composition of such a multisection $\Lambda:\Ww\to\Q^{\geq 0}$ with $f:\Xx\to\Ww$ yields a functor $\Lambda\circ f : \Xx\to\Q^{\geq 0}$, whose support $$
S_{\Lambda\circ f}:=\{ x\in X \,|\, \Lambda(f(x))>0 \}\subset X
$$ represents the perturbed solution set. 
Indeed, we have $V\cap S_{\Lambda\circ f}=\bigcup_i \{ x\in V \,|\, f(x)=\s_i(x) \}$ locally.
Moreover, each ``branch'' $(f-\s_i)^{-1}(0)$ has a natural weight $\sigma_i$, and these add up to a weighting function $\Lambda\circ f|_{S_{\Lambda\circ f}}$. 
Transversality of these multivalued perturbations as in \cite[\S15.2]{TheBook} requires transversality of each local section $f-\s_i$, and requirements of compactness control and transversality are gathered in the following notion of regularity.

\begin{defn}\label{def:NUregular}
Let $f:\Xx\to \Ww$ be a sc-Fredholm section functor of a strong bundle ${\Pp:\Ww\to \Xx}$ as in Definition~\ref{def:poly}, and suppose $(N,\Uu)$ controls the compactness of $f$ as in Definition~\ref{def:control-compact}.
Then a sc$^+$-multisection $\Lambda:\Ww\to\Q^{\ge 0}$ as in \eqref{eq:multisdef}
is called {\bf $\mathbf(N,\Uu)$-regular} (as a perturbation of $f$) if it satisfies\footnote{This is a special case of $(N,\Uu)$-admissibility \cite[Def.15.3.3]{TheBook} used in \cite[Cor.15.3.10]{TheBook}. }: 
\begin{itemlist}
\item[(1)] 
$N(\Lambda)<1$, which is shorthand for: $N(w)<1$ for all $w\in W$ with $\Lambda(w)>0$. 
\item[(2)]
$\text{dom-supp}(\Lambda)\subset \Uu$, which is shorthand for: 

the closure of the set  $\{ x\in X \,|\, \exists w\in P^{-1}(x)\less\{0\} : \Lambda(w)>0 \}$ is contained in $\Uu$. 
\item[(3)]  
$\rT_{(f,\Lambda)}$ is surjective on $\supp(\Lambda\circ f)$, which is shorthand for: 

each linearization $\rD(f-\s_i)(x)$ at a solution $f(x)=\s_i(x)$ of a local section $\s_i$ 

representing $\Lambda$ is surjective; see \cite[Def.15.2.1]{TheBook}.
\end{itemlist}
If $\Xx$ has nonempty boundary, then we require in addition: 
\begin{itemlist} 
\item[(3a)]  nothing; 
\item[(3b)]  the kernel of each linearization at a solution $x\in\partial X$ is in good position \cite[Def.3.1.24]{TheBook} to the partial quadrant in a chart around $x$;
\item[(3c)]  the kernel of each linearization at a solution $x\in\partial X$ is in general position \cite[Def.5.3.9]{TheBook} to the partial quadrant in a chart around $x$.
\end{itemlist}
\end{defn}

The results in this paper hold for all three versions (3a--c) of boundary transversality when e.g.\ regularity is transferred under pullback. 
Perturbation constructions for a single moduli space can usually achieve the strongest condition (3c), and so do ours. The less strict condition (3b) only becomes relevant for the coherent perturbation of a collection of moduli spaces, whose Cartesian products form each others' boundaries, as discussed in \cite{FHsft}.
Either condition (3b) or (3c) induces a tame smooth structure on the perturbed solution set in the sense of \cite[Def.9.3.6]{TheBook}: The composition of a sc-Fredholm section $f:\Xx\to\Ww$ and a regular sc$^+$-multisection $\Lambda:\Ww\to\Q^{\ge 0}$ is a branched ep$^+$-subgroupoid $\Lambda\circ f:\Xx\to\Q^{\ge 0}$.

\section{General construction principles}\label{sec:genconstr}

The purpose of this section is to develop general constructions of \'etale categories, their groupoid completions, and group actions in the context that will be relevant for the proof of Theorem~\ref{thm:globstab} but is valid for any of the \'etale classes discussed in \S\ref{ss:basic}.  We also discuss the structure of multisections.  Section~\ref{ss:mainres} states the main results and explains their relevance to Theorem~\ref{thm:globstab} 
 and Corollary~\ref{cor:multisection}, while subsequent subsections give the proofs.
The key fact that makes our later  constructions possible  is the simple structure of the category $\Xx_\Vv^{\less G}$.  This is the groupoid completion of a poset in which each equivalence class has a unique element with morphisms to all the other equivalent elements; see Proposition~\ref{prop:Hcomplet}~(i).  This permits the construction of the stabilization functor $\tau$ in Proposition~\ref{prop:MMGaE} and Theorem~\ref{thm:stabilize}.

\subsection{The structure of the categories \texorpdfstring{$\Xx_\Vv$}{X V} and \texorpdfstring{$\Xx_\Vv^{\less G}$}{X V less G.}} \label{ss:mainres}

The \'etale proper groupoid $\Xx_{\Vv}$ in
Theorem~\ref{thm:globstab} is generated by \'etale data of the following type,  called for short {\bf  \'etale data of type $V$}.

 \begin{enumlist} 
\item
$G_1, \ldots, G_N$ is a finite collection of finite groups. 
Then for $I\subset \{1,\dots,N\}=:A$
 we denote $G_I: = \prod_{i\in I} G_i$,  for short writing $G: = G_A: = \prod_{i\in A} G_i$.
For $I\subset J$\footnote{
Here, as elsewhere, the notation $I\subset J$ includes the possibility that $I=J$.} we consider $G_I\subset G_J$ as natural subgroup and write $G_J\to G_I, g\mapsto g|_I$ for the natural projection.
\item 
$V_J$ is a (possibly empty)  \'etale space equipped with an \'etale action of $G_J$ for each $J\subset A$.
\item 
$\TV_{IJ}\subset V_J$ for each $I\subset J$ is a $G_J$-invariant open subset 
on which   $G_{J\less I}$ acts freely,
and $\rho_{IJ}: \TV_{IJ}\to V_I$ is a composite $\TV_{IJ}\to \qu{\TV_{IJ}}{G_{J\less I}} \to V_I$ of the quotient map with an injective \'etale map with open image $V_{IJ}:=\rho_{IJ}(\TV_{IJ})\subset V_I$ with the following properties:
 
\begin{itemlist} 
\item[] {\bf (identity)} For all $I\subset A$ we have $\TV_{II} = V_{II} = V_I$ and $\rho_{II} = \id$. 
\item[] {\bf (separation)}
For $H\subset I, J$ we have $\cl(V_{HI})\cap \cl(V_{HJ} ) = \emptyset$ unless $I\subset J$ or $J\subset I$. \\
\phantom{{\bf (separation)}} For $H,I\subset J$ we have $\cl(\TV_{IJ})\cap \cl(\TV_{HJ} ) = \emptyset$ unless $I\subset H$ or $H\subset I$.
\item[] {\bf (composition)}
For $H\subset I\subset J$ we have  $\TV_{HJ}\cap \TV_{IJ} = \rho_{IJ}^{-1}(\TV_{HI}\cap V_{IJ})$ and $$
\rho_{HJ}|_{\TV_{HJ}\cap \TV_{IJ}} = \rho_{HI}\circ \rho_{IJ}.
$$
\item[] {\bf (equivariance)}
$\rho_{IJ}:\TV_{IJ}\to V_I$ is equivariant with respect to the projection $G_J\to G_I$. 
\item[] {\bf (closed graph)} for all $I\subset J$, ${\rm graph}(\rho_{IJ}) \subset V_J\times  V_I$ is closed.
\end{itemlist}
\end{enumlist}

\MS\noindent
Before explaining how to build the \'etale groupoid $\Xx_\Vv$  in Theorem~\ref{thm:globstab} from data of this kind, we introduce the following notation.
%

Here and throughout, 
we often write $\rho_I: = \rho_{I\bullet}: \TV_{I\bullet}\to V_I$, where $\bullet$ stands for any $J\supset I$.  
For any $I\subset J$, we also write  $\rho_J:\TV_{IJ}\to V_J$ for the inclusion map.
We define:
\begin{align}\label{eq:vee}
I\vee J: = 
\begin{cases}
\max\{I,J\}= J\, , & I \subset J; \\
\max\{I,J\}= I\, , & J \subset I; \\
\emptyset \, ,& \text{else};
\end{cases}
 \quad 
I\wedge J: = 
\begin{cases}
\min\{I,J\}= I \, ,& I \subset J; \\
\min\{I,J\}= J\, , & J \subset I; \\
\emptyset\, , & \text{else};
\end{cases} .
\end{align}
Further, we say that the subsets $I_1,\dots, I_k$ are {\bf nested}  if they may be rearranged into an increasing chain
\begin{align}\label{eq:vee1}
I_{\ell_1}\subset I_{\ell_2}\subset \cdots \subset I_{\ell_k}.
\end{align}
In particular, we write $I\sim J$ if $I,J$ are nested. Note  that 
 the sets $I\wedge J$ and $I\vee J$ are empty unless $I\sim J$.
More generally, for nested sets $I_1,\dots,I_k$ we define
\begin{align}\label{eq:vee1a}
I_1\vee \cdots \vee I_k = \max\{I_1,\dots,I_k\}, \quad I_1\wedge \cdots \wedge I_k = \min\{I_1,\dots,I_k\}.
\end{align}
As above, if $I_1,\dots,I_k$ are not nested both $I_1\vee \cdots \vee I_k$ and $I_1\wedge \cdots \wedge I_k$ are defined to be empty.

\begin{prop}\label{prop:MMGa1}   \'Etale data of type $V$
as in (a)--(c) above
  determines an \'etale groupoid $\Xx_\Vv=(X_\Vv, \bX_\Vv)$ with  
\begin{align*}
X_\Vv &:=  \textstyle  \bigsqcup_{ I\subset A}   V_I , &\qquad 
s (I,J,y,g) &:= (I, g^{-1}*\rho_I(y) ),  \\
\bX_\Vv &:=\textstyle \bigsqcup_{I\sim J} \TV_{(I\wedge J)(I\vee J)} \times G_{I\wedge J}  , 
& \qquad   t(I,J,y,g) &:=(J, \rho_J(y) ) ,\\
\id_{(I,x)}&:= (I,I,x,\id), &\qquad (I,J,y, g)^{-1}&:= (J,I, g^{-1} *y, g^{-1}),
\end{align*}
where we use the notation $(I,x) := x\in V_I$ for objects and $ (I,J,y, g):=(y,g)\in\TV_{IJ}\times G_{I\wedge J}$ for morphisms. 
If $t(I,J,x, g)=s(J,K,y,h)$, then $I,J,K$ are necessarily nested, and we define the composition by
 \begin{align} \label{eq:MMGa3} 
&(I,J,x, g) \circ (J,K,y,h) :=
\bigl( I,K,  \rho_{I\vee K}(v) , k(hg)|_{I\wedge K} \bigr) ,
\end{align}
where $(v,k)\in V_{I\vee J\vee K}\times G_{I\wedge K} $  
is 
given as follows\footnote
{
For an explanation of these formulas see Remark~\ref{rmk:rule}.}:
\begin{itemize}\item if  $I\subset K\subset J$ then $(v,k) = (y,\id)$;
\item if $K\subset I\subset J$ then  $(v,k) = \bigl((g|_{I\less K})^{-1}* y,\id\bigr)$;
\item if $I\subset J\subset K$ then  $(v,k) = (y,\id)$;
\item if $J\subset I\subset K$ then $(v,k) = (y,k)$ where $k\in G_{I\less J}$  is uniquely determined by\; $\rho_I(y) = h*k*x$;
\item if  $K\subset J\subset I$ then $(v,k) = \bigl((g|_{J\less K})^{-1}*h*x,\id\bigr)$;
\item  if   $J\subset K\subset I$ then $(v,k) = (k*h*x,k)$ where  $k\in G_{K\less J}$ is  uniquely determined by\; $
y = h*k*\rho_K(x)$ .
\end{itemize}

Further, this construction has the following properties:
\begin{enumilist}
\item 
$\Xx_\Vv$ is an \'etale proper groupoid whose realization  $|\Xx_\Vv|=\bigcup_{J\subset A} |V_J|$ is the union of the realizations of  the finite set of local uniformizers $(V_J, G_J)$.
\item
$\Xx_\Vv$ is equipped with an inner action of $G$ that is determined via Lemma~\ref{lem:act}~(iii)
by the action of $G$ on $X_\Vv= \sqcup_I V_I$ given by the map
\begin{align}\label{eq:XVgpactY}
\al:  G\times X_\Vv\to \bX_\Vv,\quad (g,(I,x))\mapsto (I,I,x,g|_I), 
\end{align}
which in particular induces the action $g*(I,x) = (I, g|_I*x)$ on objects.
\end{enumilist}
\end{prop}

The resulting groupoid $\Xx_\Vv$ will generally have nontrivial isotropy arising from the $G_J$-action on $V_J$. However, we will be able to extract a nonsingular groupoid $\Xx_\Vv^{\less G}$ from this data by 
 restricting consideration to the free part of the actions.  The next result explains how to construct 
 $\Xx_\Vv^{\less G}$ as the groupoid  completion (see Definition~\ref{def:grcompl}) of a nonsingular \'etale category $\Qq$ that has very simple structure.

\begin{prop}\label{prop:Hcomplet}   
\'Etale data of type $V$ defines a nonsingular, locally injective, \'etale category $\Qq = (Q,{{\bQ}})$ with 
\begin{align*}
Q &:= {\textstyle  \bigsqcup_{ I\subset A} }  V_I , &\qquad 
{{\bQ}} &:={\textstyle \bigsqcup_{I\subset J} \TV_{IJ} } , \\
s\times t \,: \; {{\bQ}} &\to Q\times Q, &\quad (I,J,y)&\mapsto \bigl( (I,\rho_{IJ}(y) ), (J,y )\bigr), \\
\id & \,: \; Q \to {{\bQ}} , &\quad (I,x)&\mapsto (I,I,x); \\
c &\,:\;  {{\bQ}}  \leftsub{t}{\times}_s {{\bQ}} \to {{\bQ}}, &\quad \bigl( (H,I,y) , (I,J,z) \bigr) &\mapsto (H,I,y)\circ (I,J,z) := (H,J,z).
\end{align*}
Further ${\Qq}$ has the following properties.

\begin{enumilist}\item  It is a  {\bf poset}; in other words 
its object space can be equipped with a partial order relation $\preccurlyeq$  such that  
\begin{align}\label{eq:pord} 
(I,x)\preccurlyeq(J,y)
\quad \Longleftrightarrow \quad \Mor_\Qq\bigl((I,x),(J,y)\bigr)\ne \emptyset.
\end{align}
Further, each equivalence class in $Q$ has a unique minimum with morphisms to all other elements. In the equivalence class containing $(J,y)$, that minimum is $(H_y,\rho_{H_y J}(y))$, with $H_y: = \bigcap \{H: y\in \TV_{HJ}\}$.
\item
 Let $\Xx_\Vv^{\less G}: = \wh{\Qq}$ be the \'etale  groupoid completion of  $\Qq$ 
given by Lemma~\ref{lem:gpcomplet}~(iii).
Then there 
are morphisms in $\Xx_\Vv^{\less G}$   between $V_I$ and $V_J$ only if $I,J$ are nested, and  for $I\subset J$ the
 morphism spaces are given as follows :
\begin{align}\label{eq:MorHM}
\Mor_{\Xx_\Vv^{\less G}}(V_I,V_J) &:=
 {\textstyle \bigcup_{\emptyset\ne H\subset I} } \; \bigl(\TV_{IJ}\cap \TV_{HJ}\bigr)\times G_{I\less H},\\ \notag
 &= \bigl\{ (y,g )\in V_J\times G\ \big| \ y\in \TV_{IJ}, g \in G_{I\less H_{y}}\bigr\},
\end{align}
where we denote $H_y: = \min \{H \,|\, y\in \TV_{HJ}\}$.  The structure maps\footnote{
Since $\Xx_\Vv^{\less G}$ is nonsingular, the source and target map determine the inverse and composition maps uniquely. An explicit formula for the composition can be found in \eqref{eq:MMGa3}.
} 
of $\Xx_\Vv^{\less G}$ are given by the same formulas as those of $\Xx_\Vv$ in Proposition~\ref{prop:MMGa1}.
\end{enumilist}
 \end{prop}

The proof is given in \S\ref{ss:Qq}.  
Note that $| \Xx_\Vv^{\less G}|$  is not in general Hausdorff.  However the next result shows that the category 
$ \Xx_\Vv^{\less G}$ supports a $G$-action and that the quotient $| \Xx_\Vv^{\less G}|/G \simeq |\Xx_\Vv|$  is Hausdorff.

\begin{prop}\label{prop:MMGa2} 
Let $\Xx_\Vv, \Qq,  \Xx_\Vv^{\less G}(=\wh{\Qq})$ be constructed from \'etale data as in Propositions~\ref{prop:MMGa1} and \ref{prop:Hcomplet}. 
Then the inclusion $\io: \Qq\to \Xx_\Vv$ factors as 
$$
\Qq\;\to\; \Xx_\Vv^{\less G}\; \; \to\; \Xx_\Vv.
$$
Further $G$ acts on $\Xx_\Vv^{\less G}$, and the inclusion $\io$ extends 
to functors $\io^{\times G}: \Qq^{\times G}\to (\Xx_\Vv^{\less G})^{\times G}\to  \Xx_\Vv$ that induce identifications 
of the realizations
$$
|\Qq^{\times G}| \; \simeq \; 
|(\Xx_\Vv^{\less G})^{\times G}| \;\simeq\;|\Xx_\Vv^{\less G}|/G\; \stackrel{\simeq}\to\; |\Xx_\Vv| .
$$
\end{prop}

We next explain the construction of the bundle $\Xx_\Vv^{\less G} \times E$ and the functor $\tau$ in Theorem~\ref{thm:globstab}.
This bundle is simply the product of the category $\Xx_\Vv^{\less G}$ with the category $E$, whose  object space is a finite dimensional vector space 
$E$ on which $G$ acts and whose  morphisms consist only of identity morphisms.    This category  $\Xx_\Vv^{\less G} \times E$  has a $G$ action as follows.

\begin{lemma}\label{lem:MMGaE}  Let $\Xx_\Vv$  and  $\Qq$ be as in Proposition~\ref{prop:MMGa2} (so that $\Qq$ has groupoid completion $\wh{\Qq}:=\Xx_\Vv^{\less G}$),
and let $E$ be a finite dimensional vector space on which $G$ acts.
Consider the following product category $\Qq\times E$:
\begin{align*}
Q\times E = {\textstyle\bigsqcup_{I}} V_I\times E,&\quad 
{{\bQ}}\times \bE = {\textstyle \bigsqcup_{I\subset J}} \TV_{IJ}\times E,   \\
(s\times t)\, (I,J,y,e) &= \bigl( (I,\rho_{IJ}(y), e), (J,y,e)\bigr),
\end{align*}
with groupoid completion $\wh{\Qq}\times E$.  
Then:
\begin{enumilist}
\item  The category $\Qq\times E$ is \'etale, and $G$ acts on both $\Qq\times E$ and its groupoid completion $\HatQq\times E = (\Xx_\Vv^{\less G})\times E$  by the product action on objects and morphisms. 


\item  There is a bundle $\pr: \Xx_\Vv^E\to \Xx_\Vv$ with fiber $E$ that supports a $G$-action and is 
such that the functorial inclusion $\io: \Xx_\Vv^{\less G} \to \Xx_\Vv$ defined in Proposition~\ref{prop:MMGa1}~(ii) lifts to a $G$-equivariant functor
$\io_E: \Xx_\Vv^{\less G}\times E \to \Xx_\Vv^E$ that induces a homeomorphism
$$
|(\Xx_\Vv^{\less G}\times E)^{\times G}| \;\simeq\; |(\Xx_\Vv^{\less G}\times E)|/ G \simeq 
|\Xx_\Vv^E|/G.
$$
\end{enumilist}
\end{lemma}

The proof is given in Lemma~\ref{lem:MMGaE0} below. This lemma also describes precise conditions under which the quotient $|\Xx_\Vv^E|/G$ is not Hausdorff. The point is this:
The bundle $\Xx_\Vv^E$ in (iii) above is constructed from the spaces $V_J\times E$ by adjoining morphisms given by the {\it diagonal} group action of $G$ on $V_J\times E$. On the other hand, 
the partial group actions that appear as  morphisms in $\Xx_\Vv^{\less G}\times E$  are given by $G_{J\less I}$-actions on $\TV_{IJ}\times E$ that are trivial on the $E$ factor.  We will see in  
Lemma~\ref{lem:MMGaE0}~(ii) below that  the space $|\Xx_\Vv^E|/G $ is not  Hausdorff
when these actions are different and  the covering of $|\Qq|$ given by the sets $|V_I|$ is sufficiently complicated..

The  $G$-equivariant functor 
$$
\tau:  \Xx_\Vv^{\less G} \times E \to \Ww_\Vv
$$
 whose existence is claimed in Theorem~\ref{thm:globstab}~(v) will be constructed in \S\ref{ss:stabilize} 
 from a family of local stabilizations of the original Fredholm section $f: \Xx\to \Ww$.
It has the following formal structure.

\begin{prop}\label{prop:MMGaE}  In the situation of Lemma~\ref{lem:MMGaE}, suppose given 
a bundle $\Pp: \Ww_\Vv\to \Xx_\Vv$  over $\Xx_\Vv$
with induced inner $G$-action as defined in Lemma~\ref{lem:actW}.
 Suppose further that there is a functor  ${\tau}: \Qq \times E \to \Ww_\Vv$  such that
 \begin{itemize}\item[-]  the induced map on object spaces is $G$-equivariant
 \item[-]    its restriction to each fiber $(I,y)\times E$ is linear with kernel  that includes the subspace $E_{A\less H_y}$, where $H_y=  \min \{H \,|\, y\in \TV_{HI}\}$ as in \eqref{eq:MorHM}.
 \item[-] 
 ${\tau}\circ \Pp = \io\circ \pr: \Qq\times E\to \Xx_\Vv$, 
 \end{itemize}
where $\io: \Xx_\Vv^{\less G} \to \Xx_\Vv$ is the  functorial inclusion from Proposition~\ref{prop:MMGa2}.
Then:  
\begin{enumilist}\item ${\tau}$ extends uniquely to a functor $\tau:\Xx_\Vv^{\less G} 
\times E \to \Ww_\Vv$.

\item Further $\tau$ factors as the composite  $\tau= \tau_E\circ \io_E$ in the following 
commutative diagram of functors
\begin{align}\label{diag:13}
\xymatrix
{
\Xx_\Vv^{\less G} \times E\ar@{->}[d]^{\pr}\ar@{->}[r]^{\;\quad \io_{E}}  & \Xx_\Vv^E \ar@{->}[d]^{\pr}\ar@{->}[r]^{\;\;\tau_E }& \Ww_\Vv \ar@{->}[d]_{\Pp}\\
\Xx_\Vv^{\less G}  \ar@{->}[r]^{\io} &   \Xx_\Vv \ar@{->}[r]^{\;\; =} & \Xx_\Vv,
}
\end{align}
where $ \Xx_\Vv^E$ and the $G$-equivariant functor $\io_E$ are  as in Lemma~\ref{lem:MMGaE}~(ii).
Further,  $\tau_E$ is $G$-equivariant, in the sense that it takes the $G$-action on $\Xx_\Vv^E$ defined in Lemma~\ref{lem:MMGaE}~(ii) to that on $\Ww_\Vv$.
\end{enumilist}
\end{prop}

For the proof, see Proposition~\ref{prop:MMGaE0}.

Finally we discuss the construction of  multisections of $\Ww_\Vv\to \Xx_\Vv$. For basic definitions, see Definition~\ref{def:multisdef0}.

\begin{prop} \label{prop:globsym}  Let the bundle $\Ww_\Vv\to \Xx_\Vv$, group $G$, and functor $\breve\tau: \Qq\times E \to \Ww_\Vv$ with extension $\tau:\Xx_\Vv^{\less \Ga}\times E\to \Ww_\Vv$ be as
  in Proposition~\ref{prop:MMGaE}.
  \begin{itemlist}
  \item [{\rm (i)}] 
Each $e\in E$ induces a  multisection functor $\La_{\Vv,e} :  \Ww_\Vv\to \Q^{\ge 0}$ 
given by
\begin{align}\label{eq:pushLa300}
 \textstyle
\La_{\Vv,e}(w) &= \tfrac{1}{|G|} \ \# \{ h\in G \,|\, \tau(P_\Vv(w) ,h*e)= w  \},
\end{align}
where $\#$ denotes the number of elements in a finite set. 
\item[{\rm (ii)}]  This multisection $\La$ is globally structured in the sense of Definition~\ref{def:multisdef0} by the sections $(\s_g)_{g\in G}$ and correspondence $\ka: \Mor_{\Xx_\Vv}\to \{b: G\to G, {\rm bijection}\}$ given by
\begin{align}
\s_g(I,x) &= \tau(I,x, g*e),\	\quad 
\ka(I,J,y,g)(h) = g*h.
\end{align}
\item[{\rm (iii)}] 
 The triple $(\La, (\s_i)_{i\in \Ii}, \ka)$ is structurable in the sense of \cite[Def.13.3.6]{TheBook}, and determines a unique structured multisection  $[\La, \Uu, \Ss,\tau]$  in the sense of \cite[Def.13.3.8]{TheBook}, where $(U(x)_{x\in \Xx_\Vv}$ is
the system of good neighbourhoods in  $\Xx_\Vv$ defined in Lemma~\ref{lem:neigh}.
\end{itemlist}
 \end{prop}
 
 The  proof,
 together with a few other results about multisections,  may be found  at the end of 
\S\ref{ss:GaE}.

 \subsection{The category \texorpdfstring{$\Qq$}{Q} and its relation to $\Xx_\Vv$}\label{ss:Qq}

We now construct the categories $\Qq$, $\Xx_\Vv^{\less G}$ and $\Xx_\Vv$, and prove
Propositions~\ref{prop:MMGa1},~\ref{prop:Hcomplet}, and~\ref{prop:MMGa2}.  We begin by
 constructing the categories $\Qq$ and $\Xx_\Vv^{\less G}$, establishing part (i) and most of (ii) of Proposition~\ref{prop:Hcomplet}.
We then construct two \'etale categories that are related to  $\Qq$.
The first new category is $\Qq^{\times G}$ and  is obtained simply by multiplying the morphism space of $\Qq$ by $G$ as in Lemma~\ref{lem:BbGa}, while the second one is the groupoid  $\Xx_\Vv$   
 in Propositions~\ref{prop:MMGa1} and~\ref{prop:MMGa2} that
 incorporates the $G$-action in a more subtle way.  As we explained in \S\ref{ss:mainres} the groupoid $\Xx_\Vv$ 
 has 
the same objects $\bigsqcup_I V_I$ as $\Qq$, but for each $J$ contains the full translation groupoid defined by the action of $G_J$ on $V_J$ rather than the \lq partial' actions of $G_{J\less I}$ on 
$\TV_{IJ}\subset V_J$ that lie in $\Mor_{\Qq}$, and it takes quite a bit of work to establish the  
 composition rule  \eqref{eq:MMGa3}  in $\Xx_\Vv$.  
Finally Lemma~\ref{lem:compos} exhibits $\Xx_\Vv$ as the groupoid completion of a category $\Bb_\Vv$ that only has morphisms from $V_I$ to $ V_J$ when $I\subset J$, and also shows that each functor $\phi:\Bb_\Vv\to \Zz$ (where $\Zz$ is an arbitrary groupoid) has a unique extension to a functor  $\Xx_\Vv\to \Zz$. This result is an essential ingredient in the construction of the functor $\psi: \Xx_\Vv\to \Xx$ in Theorem~\ref{thm:globstab}~(ii).

\begin{lemma}\label{lem:complet}     Given \'etale data of tyye $V$
 let $\Qq = (Q,{{\bQ}})$ be 
the \'etale
category with
\begin{align*}
Q &= {\textstyle  \bigsqcup_{ I} } V_I, \qquad 
{{\bQ}} ={\textstyle \bigsqcup_{I\subset J } \TV_{IJ} } , \\
 \id: &\;\; Q\to {\Qq}, \quad \; (I,y)\mapsto (I,I,y), \\ 
\notag
(s\times t): &\;\;\ {{\bQ}}\to Q\times Q,\quad (I,J,y)\mapsto 
 \bigl( (I,\rho_{IJ}(y) ), (J,y )\bigr), \\
 (I,J,y)&\circ (J,K,z) = (I,K,z),\quad\mbox{ if } I\subset J\subset K,
\end{align*}
Then the following holds.
\begin{enumilist}\item
$\Qq$ is a nonsingular \'etale category such that $\pi_\Qq: Q\to |\Qq|$ is locally injective.  Moreover, $\Qq$ is a poset, i.e.\   the  relation $\preccurlyeq $  on its object space given by  
\begin{align}\label{eq:pord2} 
(I,x)\preccurlyeq(J,y)
 \Longleftrightarrow \Mor_\Qq\bigl((I,x),(J,y)\bigr)\ne \emptyset 
 \end{align}
 is a partial order.  Further, each equivalence class in $Q$ has a unique minimum with morphisms to all other elements. In the equivalence class containing $(J,y)$ that minimum is $(H_y,\rho_{H_y J}(y))$, with $H_y: = \bigcap \{H: y\in \TV_{HJ}\}$.
 
\item 
 Let $\wh{\Qq}$ be the \'etale  groupoid completion of  $\Qq$ (which exists by Lemma~\ref{lem:gpcomplet}~(iii), (iv)).
Then there 
are morphisms in $\wh{\Qq}$   between $V_I$ and $V_J$ only if $I,J$ are nested, and  for $I\subset J$ the
 morphism spaces are given as follows:
\begin{align}\label{eq:MorHM0}
\Mor_{\wh{\Qq}}(V_I,V_J) &  =
 {\textstyle \bigcup_{\emptyset\ne H\subset I} } \; \bigl(\TV_{IJ}\cap \TV_{HJ}\bigr)\times G_{I\less H},\\ \notag
 &= \bigl\{ (I,J, y,g )\in V_J\times G\ \big| \ y\in \TV_{IJ}, g \in G_{I\less H_{y}}\bigr\},
\end{align}
where for  $I\subset J$ we define
\begin{align}\label{eq:MorHM1}
 (s\times t)\,(I,J,y,g _{I\less H_y})& = \bigl(\,(I,\, g _{I\less H_y}^{-1} * \rho_{IJ}(y)), \; (J,y)\bigr). 
 \end{align}
 Further (again for $I\subset J$)
\begin{align}\label{eq:MorHM2}
 \Mor_{\wh{\Qq}}(V_J,V_I) = \bigl\{ (I,J,y,g)^{-1}: = (J,I,g^{-1}*y,g^{-1}) \ | \ (I,J,y,g)\in  \Mor_{\wh{\Qq}}(V_I,V_J)\bigr\}.
 \end{align}
 \end{enumilist}
 \end{lemma}

\begin{proof}  It follows from the identity $\rho_{IK} = \rho_{IJ}\circ \rho_{JK}$ for $I\subset J\subset K$  
that $\Qq$ is a well defined, nonsingular  category.  It is \'etale because for each $I\subset J$ the map $\rho_{IJ}$ quotients out by a free action of $G_{J\less I}$.
Moreover, $\preccurlyeq $ is a partial order; in particular it is antisymmetric because the identity property implies 
that for any pair of distinct elements $(I,x), (J,y)$ at most one of the properties $(I,x) \preccurlyeq (J,y)$ and 
 $(J,y) \preccurlyeq (I,x)$ holds.

Next observe that as in the proof of Lemma~\ref{lem:etale1}, two distinct elements   $(I,x), (J,y)$ lie in the same equivalence class of $\sim_\Qq$  if and only if there is a chain from $(I,x)$ to $(J,y)$ of the following type:
\begin{align}\label{eq:chain}
(I,x) =: (I_0,x_0)\preccurlyeq (I_1,x_1) \succcurlyeq  (I_2,x_2)\preccurlyeq \cdots \succcurlyeq (I_k,x_k) := (J,y).
\end{align}
Now, in a triple of the form
 $(I,x) \preccurlyeq  (K,z)\succcurlyeq (J,y)$ we must have $z\in \TV_{IK}\cap \TV_{JK}$ which implies that $I,J,K$ are nested.
Similarly,  the existence of a triple of the  form $(I,x) \succcurlyeq  (K, z)\preccurlyeq (J,y)$ implies that $z \in V_{KI}\cap V_{KJ}$, so that again $I,J,K$ are nested.  Hence  the sets $(I_i)_{0\le i\le k}$ in any chain such as \eqref{eq:chain} must be nested.  
In particular, every $y\in V_J$ lies in a set of the form $\TV_{H_yJ}$ for a unique  minimal set $H_y$, that is contained in all other sets $I$ such that $(J,y)\sim_\Qq (I,y')$. 
Further, because $H_y$ is minimal,  there is a unique element  $z: = \rho_{H_y J}(y)$ in  $V_{H_y}$  in the equivalence class $[(J,y)]$. Hence the equivalence class of $(J,Y)$ has a unique minimum with morphisms to all other elements. 

To see that  $\pi_\Qq: Q\to |\Qq|$ is locally injective, notice that
because the set $\TV_{H_yJ}$ is open, the second separation property implies that $y$ has a neighbourhood $N_y\subset \TV_{H_yJ}\subset V_J$ that  intersects  $\cl(\TV_{FJ})$  only if $H_y\subset F$.  It follows that $H_y = H_{y'}$ for all $y'\in N_y$. In other words, each equivalence 
class $[(J,y')]$ for $y\in N_y$ contains a unique element of the form $(H_y, z')$, where $z': = \rho_{H_y J}(y')$. 
Since $\rho_{H_y J}: \TV_{H_yJ}\to V_{H_y}$ quotients out by a free action of $G_{J\less H_y}$, we may 
shrink $N_y$ so  that the map $\rho_{H_yJ}|_{N_y}$ is injective.   It follows that $\pi_\Qq: N_y\to |\Qq|$ is injective.  This completes the proof of (i).
\MS

To prove (ii), note first that because the completion $\HatQq$ is nonsingular, each morphism is determined by its source and target.  
In particular,
 the formulas \eqref{eq:MorHM0} and \eqref{eq:MorHM1} do describe morphisms from $V_I$ to $V_J$ in $\HatQq$ in the case $I\subset J$.  When $I\subset J$,  we find using \eqref{eq:MorHM2} that
\begin{align*}
 (s\times t)\bigl((I,J,y,g)^{-1}\bigr) &=  (s\times t)(J,I, g^{-1}*y, g^{-1})\\
 &= \bigl((J,y), (I, \rho_{IJ}(g^{-1}*y)\bigr) =  (t\times s)(I,J,y,g)
\end{align*}
 since $\rho_{IJ}$ commutes with the action of $G_I$.  Therefore the description of  $\Mor_{\HatQq}(V_J,V_I)$ in 
   \eqref{eq:MorHM2} is consistent with the source and target maps, as is the formula for the inverse.

It  remains to show that if $I\subset J$, every element in $\Mor_{\HatQq}(V_I,V_J)$ has the form given in \eqref{eq:MorHM0}.
The remarks above imply that
 if $I\subset J$ and $(I,x)\sim_\Qq(J,y)$,  we can replace the chain \eqref{eq:chain} by one of the form 
\begin{align}\label{eq:chainJ}
(I,x) \succcurlyeq \bigl(H_y, z\bigr) \preccurlyeq (J,y),\quad z: = \rho_{H_y I}(x) = \rho_{H_yJ}(y)=\rho_{H_yI}\bigl(\rho_{IJ}(y)\bigr).
\end{align}
Thus the elements $x$ and $\rho_{IJ}(y)$ of $V_I$  have the same image under $\rho_{H_y I}$, so that they must both lie in the same $G_{I\less H_y}$ orbit.  
Hence there is a unique element $g_{I\less H_y}\in G_{I\less H_y}$ such that 
$x = g_{I\less H_y}^{-1}*\rho_{IJ}(y)$.   Thus this morphism in $\HatQq$ from $(I,x)$ to $(J,y)$ has the same source and target as the tuple 
$$
(I,J, y, g_{I\less H_y}), \quad \mbox{ where }\;\;  x = g_{I\less H_y}^{-1}*\rho_{IJ}(y),
$$
and hence may be identified with it since $\wh{\Qq}$ is nonsingular.  This completes the proof.
\end{proof}

\begin{rmk}\label{rmk:complet}\rm
\begin{nenumilist}
\item   
Both $\Qq$ and $\HatQq$ have realizations that are unions of the realizations of finitely many preuniformizers $\bigl((V_I,\id)\bigr)_{I\in \Ii}$; see Definition~\ref{def:preunif}.   But just as in Example~\ref{ex:branch} these are not uniformizers.
 
\item
Above, given $y\in V_J$,  we defined $H_y: =  \min \{H: y\in \TV_{HJ}\}$.  We remark here that
 the set $H_y$ does not depend on the choice of $J$ but only on the equivalence class $[(J,y)]$ of the element $(J,y)$. Indeed by \eqref{eq:chainJ} we could equally well have defined $H_y$ to be the minimum of the sets 
 $I$ that appear in this equivalence class. 
 
 \item 
 As in Example~\ref{ex:branch} the space $|\Qq| = |\wh{\Qq}| =|\Xx_\Vv^{\less G}|$ is not in general Hausdorff.  
\hfill$\er$
\end{nenumilist}
\end{rmk}

\begin{lemma}\label{lem:MMGa}  
The \'etale categories $\Qq$ and $\wh{\Qq}$ constructed in 
Lemma~\ref{lem:complet} support an action of $G$.
\end{lemma}

\begin{proof}  We define the action of $G$  on the object space $Q: = \bigsqcup_I V_I$ by setting
$$
g* (I,x) = (I, g_I*x), \quad x\in V_I, \ g\in G,
$$
where $g_I $ denotes the projection of $g: = (g_i)_{i\in \{1,\dots,N\}}\in G$ to $g_I: = (g_i)_{i\in I} \in G_I$. This action preserves the equivalence relation $\sim_{\Qq}$ on $Q$ because we have assumed that $\rho_{IJ}$ is equivariant with respect to the projection  $G_J\to G_I$.  
Therefore,
since the nonsingular category $\Qq$ is locally injective,  
the existence of the $G$ action on  $\Qq$ and hence on $\wh{\Qq}$ follows from Lemma~\ref{lem:act}~(iv).
\end{proof}

We now prove Propositions~\ref{prop:MMGa1} and ~\ref{prop:MMGa2} which we restate for the convenience of the reader.

\begin{prop}\label{prop:MMGa10}    Every collection of \'etale data of type $V$ defines 
 an \'etale groupoid $\Xx_\Vv=(X_\Vv, \bX_\Vv)$ with 
\begin{align}\label{eq:MMGa20}
X_\Vv &:=  \textstyle  \bigsqcup_{ I\subset A}   V_I , &\qquad 
s (I,J,y,g) &:= (I, g^{-1}*\rho_I(y) ),  \\ \notag
\bX_\Vv &:=\textstyle \bigsqcup_{I\sim J} \TV_{(I\wedge J)(I\vee J)} \times G_{I\wedge J}  , 
& \qquad   t(I,J,y,g) &:=(J, \rho_J(y) ) ,\\ \notag
\id_{(I,x)}&:= (I,I,x), &\qquad (I,J,y, g)^{-1}&:= (J,I, g^{-1} *y, g^{-1}).
\end{align}
If $t(I,J,x, g)=s(J,K,y,h)$, then $I,J,K$ are necessarily nested, and we define the composition by
 \begin{align} \label{eq:MMGa30} 
&(I,J,x, g) \circ (J,K,y,h) :=
\bigl( I,K,  \rho_{I\vee K}(v) , k(hg)|_{I\wedge K} \bigr) ,
\end{align}
where $(v,k)\in V_{I\vee J\vee K}\times G_{I\wedge K} $  
is 
given as follows:
\begin{itemize}\item if  $I\subset K\subset J$ then $(v,k) = (y,\id)$;
\item if $K\subset I\subset J$ then  $(v,k) = \bigl((g|_{I\less K})^{-1}* y,\id\bigr)$;
\item if $I\subset J\subset K$ then  $(v,k) = (y,\id)$;
\item if $J\subset I\subset K$ then $(v,k) = (y,k)$ where $k\in G_{I\less J}$  is uniquely determined by\; 
\begin{align}\label{eq:JIK}
\rho_I(y) = h*k*x;
\end{align}
\item if  $K\subset J\subset I$ then $(v,k) = \bigl((g|_{J\less K})^{-1}*h*x,\id\bigr)$;
\item  if   $J\subset K\subset I$ then $(v,k) = (k*h*x,k)$ where  $k\in G_{K\less J}$ is  uniquely determined by\; 
\begin{align}\label{eq:JKI}
y = h*k*\rho_K(x) .
\end{align}
\end{itemize}
Further,

\begin{enumilist}
\item 
$\Xx_\Vv$ is an \'etale proper groupoid whose realization  $|\Xx_\Vv|=\bigcup_{J\subset A} |V_J|$ is the union of the realizations of  the finite set of local uniformizers $(V_J, G_J)$.

\item
$\Xx_\Vv$ supports an inner action of $G$ that is determined as in \eqref{eq:gpident3} by the map
\begin{align}\label{eq:XVgpactY0}
\al: G \times X_\Vv\to \bX_\Vv, \quad (g, (I,x))\mapsto (I,I,x,g|_I).
\end{align}

\item
The  inclusion $\io: \Qq\to \Xx_\Vv$ factors as 
$$
\Qq\;\to\; \HatQq = \Xx_\Vv^{\less G}\; \;\to\; \Xx_\Vv,
$$
 and extends to functors $\io^{\times G}: \Qq^{\times G}\to \HatQq^{\times G}\to  \Xx_\Vv$ that  induce the identification 
 $$
|\Qq^{\times G}| \; \simeq \; 
|\HatQq^{\times G}| \;\simeq\;|Q|/G\; \stackrel{\simeq}\to\; |\Xx_\Vv|
$$
 of the realizations.  
 \end{enumilist}
\end{prop}

\begin{rmk}\label{rmk:rule}  \rm The composition rule in \eqref{eq:MMGa30} above is most complicated (i.e.\  $k\ne \id$) when $J$ is minimal among $\{I,J,K\}$, since  that is the case in which the compatibility in $V_J$ required for the composite to be defined gives the least information.
For example, if $J\subset I\subset K$, the composite $(I,J,x,\id)\circ (J,K,y,\id)$ is defined if  
$\rho_J(x) = \rho_J(y)$ (where $x\in \TV_{JI}\subset V_I, y\in \TV_{JK}\subset V_K$).  Therefore by  the composition property in  Lemma~\ref{lem:complet}, the element $\rho_I(y)\in \TV_{JI}$ must be defined and lie in the same $G_{I\less J}$-orbit as $x$, which implies by the freeness of the action of $G_{I\less J}$ on $\TV_{JI}$  that there
is a unique $k\in G_{I\less J}$ such that $\rho_I(y) = k*x$ as in \eqref{eq:JIK}. 
Note also that if $g, h\ne \id$, then $g,h\in G_J$ and so both commute with $k\in G_{I\less J}$.
Further, as $x,y$ vary, $k\in G_{I\less J}$ is a locally constant function, though it need not be constant on the components of the morphism space 
unless the principal $G_{I\less J}$-bundle $\TV_{JI}\to V_{JI}$ is trivial.
\hfill$\er$
\end{rmk}

\begin{proof} [Proof of Proposition~\ref{prop:MMGa10}] \;
 {\bf Step 1:} {\it 
If composition is defined then $I,J,K$ are nested.} 

The composition $(I,J,x, g) \circ (J,K,y,h) $ is defined only if $$
t(I,J,x,g) = \rho_J(x) = s(J,K,y,h) = \rho_J(h^{-1}*y).
$$
Further we know that $I,J$ are nested and $J,K$ are nested.    Hence $I,J,K$ are nested except possibly if  $J$ is either maximal or minimal.
But if $J$ is maximal, then the condition $t(I,J,x,g)= s(J,K,y,h)$ implies that $\TV_{IJ}\cap \TV_{KJ} \ne \emptyset$ so that 
$I,J,K$ are nested by the separation axiom in \S\ref{ss:mainres}~(c).  Similarly, if $J$ is minimal, we have  $V_{JI}\cap V_{JK} \ne \emptyset$
which again implies that $I,J,K$ are nested.
\MS

\NI
 {\bf Step 2:} {\it The
 rule \eqref{eq:MMGa30} is consistent with source and target maps. }
 
We must  check
that
\begin{align*}
& {\rm (A)}:\quad   s(I,J,x,g) = s\bigl( I,K,  \rho_{I\vee K}(v) , k(hg)|_{I\wedge K} \bigr),  \\
& {\rm  (B)}:\;\;\quad 
(K, \rho_K(y)) = t(J,K,y,h) = t\bigl( I,K,  \rho_{I\vee K}(v) , k (hg)|_{I\wedge K} \bigr) = (K, \rho_K(v)),
\end{align*}
when
$
  t(I,J,x,g) = \rho_J(x) = s(J,K,y,h) = \rho_J(h^{-1}*y), 
  $
  i.e.\  when 
  \begin{align}\label{eq:hx=y}&  h*\rho_J(x) = \rho_J(y),\;\;  \mbox{ where } h\in G_{J\wedge K}.
  \end{align}
  We first prove (B).  First notice that
 (B) holds whenever $y = v$, and hence when $K$  is maximal or $I\subset K\subset J$.
 If $K\subset I\subset J$ then $v = (g|_{I\less K})^{-1}*y$ and $h\in G_K$.  Therefore $\rho_K(v) = \rho_K(y)$ since  $\rho_K$ quotients out by the action of $G_{I\less K}$.
  If $I$ is maximal and $K$ is minimal,  then 
$  v = 
 (g|_{J \less K})^{-1} * h* x 
 \in V_I$ and $h\in G_K$.  Therefore $$
 \rho_K(v) = \rho_K(h*x) = h*\rho_K(x) = \rho_K(y)
 $$
   because $\rho_K$ is $G_K$-equivariant  and constant on the orbits of $G_{J\less K}$. 
  Finally if $I$ is maximal and $J$ is minimal, the compatibility condition \eqref{eq:hx=y}  it implies that  $y\in  V_K$ lies in the same $G_{K\less J}$-orbit as $\rho_K(h*x)$.  Hence there is 
  a unique element  $k\in G_{K\less J}$  such that $y = k*h*\rho_K(x)$.
By definition we take  $v = k*h*x$ in this case.  Hence again $\rho_K(v) = y = \rho_K(y)$, as required.  
  Thus (B) holds in all cases.
  \MS
  
Next observe that (A) will hold if
\begin{align}\label{eq:(A)}
g^{-1}*\rho_I(x) =(k(hg)|_{I\wedge K})^{-1}*\rho_I \rho_{I\vee K}(v) = (g|_{I\wedge K})^{-1}* (h|_{I\wedge K})^{-1}* k^{-1}*\rho_I (v).
\end{align}
We first consider the three cases with $v=y$.  
\begin{itemlist}
\item
If $I\subset J\subset K$ then $x\in V_J$, $g\in G_I, h\in G_J$, $k=\id$  and $v = y\in V_K$, so that 
we must check $$g^{-1}*\rho_I(x) = g^{-1}* (h|_I)^{-1}*\rho_I(y).$$
  But because $\rho_I$ is equivariant with respect to the projection $G_J\to G_I$ we have
$(h|_I)^{-1}*\rho_I(y) = \rho_I(h^{-1}*y)$, so that  this follows from the compatibility condition  \eqref{eq:hx=y}.  

\item  If $I\subset K\subset J$, then $v=y \in V_J$, $h\in G_K$ and $k = \id$,
and the same argument holds. 
\item 
  If $J\subset I \subset K$, then $x\in V_I$, $v=y\in V_K$ and by definition of $k$  we have $ \rho_I(y) = h*k*x$ where $h\in G_J, k\in G_{I\less J}$ commute.   Since $h = h|_{I\wedge K}$ and $g= g|_{I\wedge K}$, this equation is equivalent to \eqref{eq:(A)}.

\item If $K\subset I\subset J$ then $h\in G_K$, and $g\in G_I$ decomposes as the product $g = g|_K\, g|_{I\less K}$, where these factors commute.  Further $k = \id$ and  $v=(g_{I\less K})^{-1}*y\in V_J$, while the  compatibility equation \eqref{eq:hx=y} reads $y = h *x$.
  In this case,  \eqref{eq:(A)} can be written as
$$
g^{-1}*\rho_I(x) = (g|_{K})^{-1}* h^{-1}* (g|_{I\less K})^{-1}* \rho_I (y) 
$$
which follows from  $y = h *x$ because  $h\in G_K$ commutes with $g|_{I\less K}$.
\item  If $K\subset J \subset I$, then $v = (g|_{J\less K})^{-1}*h*x \in V_I$ and $k=\id$.   Since $h\in G_K$, $h$ commutes with 
$g|_{J\less K}$.  Hence
$$
  (g|_{K})^{-1}* h^{-1} *v =  (g|_{K})^{-1}* (g|_{J\less K})^{-1}*x =g^{-1}*x
$$
as is required by \eqref{eq:(A)}.
\item Finally, if  $J\subset K\subset I$ then $v = k*h*x\in V_I$ where $k\in G_{K\less J}$ is the unique element such that $y = h*k*\rho_K(x)$. By definition, $g,h\in G_J$ so that $g = g|_{I\wedge K}$ and $h = h|_{I\wedge K}$.  
Therefore \eqref{eq:(A)} is equivalent to  $x = h^{-1}*k^{-1} *v$, which holds by the definition of $v$.
\end{itemlist}
This completes the proof of Step 2.\MS

\NI {\bf Step 3:}  
{\it The formula for the inverse 
is consistent with that for composition.}

To check this, we note that
\begin{align*}
(I,J,x,g)\circ \bigl((I,J,x,g)\bigr)^{-1}&  = (I,J,x,g)\circ(J,I,g^{-1}*x, g^{-1}) \\
& = (I,I, \rho_I(v), k)\;\;\mbox{ with } (v,k) \mbox{ as in } \eqref{eq:MMGa30}\\
&  = (I,I, g^{-1}*\rho_I(x), \id)
\end{align*}
because if $K=I\subset J$ we have $k = \id, \ \rho_I(v) = \rho_I(g^{-1}*x) = g^{-1}*\rho_I(x)$; while if $J\subset I=K$ we have 
$x,y\in V_I $, $v=y$ and $k\in G_{I\less J}$ so that $y = g^{-1}*k*x$, which implies
$k=id$  by \eqref{eq:hx=y}. \MS

\NI {\bf Step 4:}  {\it There is a well defined inclusion $\HatQq \to \Xx_\Vv$; in particular, the structure maps in  $\HatQq$ are given by \eqref{eq:MMGa20}, \eqref{eq:MMGa30}.}

Lemma~\ref{lem:complet} shows that $\Qq$ has a unique nonsingular groupoid completion $\HatQq$ whose structure maps are determined by its source and target maps.   
This lemma also shows that
these maps $s,t$  are given by the formulas in \eqref{eq:MMGa20},  as is the formula for the inverse.
Further,  the composition rule in $\Qq$ is given by the case $I\subset J\subset K$ of \eqref{eq:MMGa30}.  
Hence to check that composition in $\HatQq$ is given by \eqref{eq:MMGa30}, in view of Step 2 it remains to verify that this rule preserves the subset $\Mor_{\HatQq}$ of $\Mor_{\Xx_\Vv}$.  To this end, recall from the
proof of Lemma~\ref{lem:complet} that
 each equivalence class $[(J,y)]$ in $\Obj_{\HatQq}$ contains a unique element of the form $(H_y, z)$ where $H_y$ is the minimal set $I$ such that  $(J,y)\sim (I,x)$ for some $x\in V_I$.  Further,
  the description of $\Mor_{\HatQq}(V_I,V_J)$ in
\eqref{eq:MorHM0} and \eqref{eq:MorHM2} together with  the identity $(I,J,z,\ell)^{-1} = (J,I, \ell^{-1}*z,\ell^{-1})$ implies that
\begin{align}\label{eq:QI}
(I,J,z,\ell)\in 
\Mor_{\HatQq}\;\Longleftrightarrow \; \ell\in G_{I\wedge J\less H_z}.
\end{align} 
Therefore, to see that $\Mor_{\HatQq}$ is closed under the composition rule  \eqref{eq:MMGa30}, it
 remains to check that if $(I,J, x,g), (J,K,y,h)$ both satisfy the condition in \eqref{eq:QI} and 
$t(I,J, x,g) = s(J,K,y,h)$, then the element 
$(I,K,z, \ell)$  given by \eqref{eq:MMGa30}  
also satisfies this condition.  
But in this case, there is equality $H_x = H_y = H_z$ and $H_x\subset I\wedge J\wedge K$.
Further  $\ell = k(hg)|_{I\wedge K}$  where $g\in G_{I\wedge J\less H_z}, h\in G_{J\wedge K\less H_z}$, and $k\in 
 G_{I\wedge K\less H_y}$ by construction.  Therefore $\ell\in  G_{I\wedge K\less H_z}$ as required.
Hence the composition defined in  \eqref{eq:MMGa30} does restrict to $\Mor_{\HatQq}$.  This completes the proof.
\MS

\NI {\bf Step 5:}   {\it The composition rule in  \eqref{eq:MMGa30}  is associative.}

 Although associativity  is proved in  \cite[Prop.2.3]{Morb},\footnote{
 As will become more evident in \S\ref{ss:Vdata}  below, the category $\Xx_\Vv$  is the subcategory of the category $\bG_\Kk$ considered in  \cite[Prop.2.3]{Morb} that corresponds to a reduction $\Vv$ of the covering  given by a strict orbifold atlas.
 }
we give an independent proof here for the sake of completeness.   
We will use the fact that the action of $G_I$ on $V_I$ given as in \eqref{eq:gpident2} by  $$
\al_I(g,x) = (I,I,x,g): g^{-1}*x\mapsto x
$$ 
yields 
actions of
$G_I$ and $G_J$ on the morphism space $\Mor_{\Xx_\Vv}(V_I,V_J)$ as follows: 
\begin{align*}
G_I \times\Mor_{\Xx_\Vv}(V_I,V_J)\; \to \; \Mor_{\Xx_\Vv}(V_I,V_J), \qquad & (g,m) \; \mapsto\;  g\circ m := \al_I(g, s(m) ) \circ m ,
\\ \notag
G_J \times \Mor_{\Xx_\Vv}(V_I,V_J) \; \to\;  \Mor_{\Xx_\Vv}(V_I,V_J), \qquad & (g,m)\;  \mapsto\;  m\circ g := m \circ \al_J(g, g*t(m) ) . 
\end{align*} 
Further,  $G_{I\wedge J}$ acts by conjugation on $\Mor_{\Xx_\Vv}(V_I,V_J)$ via $m\mapsto m^g: = g\circ m \circ g^{-1}$.
Notice that the precise morphism in $\Xx_\Vv$ represented by $g$ in the notation $g\circ m$ (respectively $m\circ g$) depends only on the source $s(m)$ (respectively the target $t(m)$).

Next observe that the  rule \eqref{eq:MMGa30} applied  with $I=J$  implies that every morphism $(I, K, x,g) $ in $\Xx_\Vv$ with $x\in V_{I\vee K}, g\in G_{I\wedge K}$  may be uniquely decomposed as a product
\begin{align}\label{eq:QQ03}
(I,K, x,g) : & = \al_I\bigl(g|_{H_x}, (g|_{I\wedge K\less H_x})^{-1} *\rho_I(x)\bigr)\circ (I,K,x, g|_{I\wedge K\less H_x}),
\\ \notag
& = g|_{H_x}\circ m, \quad m: = (I,K,x, g|_{I\wedge K\less H_x}) \in \Mor_{\HatQq}.
\end{align}
Similarly, taking $J=K$, we have
\begin{align*}
(I,K, x,g) : & = (I,K,(g|_{H_x})^{-1}* x, g|_{I\wedge K\less H_x})\circ \al_K\bigl( g|_{H_x}\rho_K(x)\bigr)\\ \notag
& =  m'\circ g|_{H_x}, \quad m': = (I,K,(g|_{H_x})^{-1}* x, g|_{I\wedge K\less H_x}) \in  \Mor_{\HatQq}.
\end{align*}
Thus, for $m = (I,J,x,g)\in  \Mor_{\HatQq}$ and $\ell \in H_x$, we have
\begin{align}\label{eq:QQ04}
\ell  \circ m = m^{\ell} \circ \ell, \quad \mbox{where }  \quad m^{\ell} = (I,J,\ell^{-1}*x,\ell^{-1}g\ell) \in  \Mor_{\HatQq}.
\end{align}
In particular, this conjugation action takes $\HatQq$ to itself.
Further, 
\begin{align}\label{eq:QM1}
m_1^\ell  \circ  m_2^\ell  = (m_1\circ m_2)^\ell ,\quad \forall\;\; \ell\in  G_{H_x},\;\; x: = s(m_1),
\end{align}
since the morphisms on both sides have the same source and target and
lie in the nonsingular groupoid  $\HatQq$. 
More generally, we claim that if $m \in \Mor_{\HatQq}$ and $\ell_1,\ell_2 \in G_{H_x}$ then
\begin{align}\label{eq:QM2}
(\ell_1 \circ m) \circ \ell_2 = \ell_1 \circ (m_1 \circ \ell_2) = :  \ell_1 \circ m_1 \circ \ell_2 \in \Mor_{\Xx_\Vv}.
\end{align}
This follows easily from the fact  that  a morphism $m = (I,J,x,g)$ in $\Xx_\Vv$ is uniquely determined by its source $s(m) = g^{-1}*\rho_I(x)$,
target $t(m) = \rho_J(x)$, and group element $g\in G_{I\wedge J}$.  (For example, if $J\subset I$ then $x = g*s(m)$.)
A similar argument shows that 
\begin{align}\label{eq:QM3}
(\ell \circ m_1) \circ m_2 = \ell \circ (m_1 \circ m_2), \qquad m_1, m_2 \in \Mor_{\HatQq}, \ell \in G_{H_{s(m_1)}}.
\end{align}
Finally we claim that if $$
m_1 = (I,J,x,g),\quad m_2 = (J,K,y,h)   \in \Mor_{\HatQq},   \qquad \ell\in G_{H_x}
$$
are morphisms in $\wh{\Qq}$ then 
\begin{align}\label{eq:QM4}
(m_1\circ \ell)\circ m_2 = m_1\circ (\ell \circ m_2)= : m_1\circ \ell \circ m_2.
\end{align}
To prove this, notice first that the two sides of \eqref{eq:QM4} clearly have the same source and target.  
 Further the formulas in \eqref{eq:MMGa30} show that if $J$ is not minimal  they have the same group element, namely $(g\ell h)|_{I\wedge K} = (\ell gh)|_{I\wedge K}$,
 since $g\in G_{I\wedge J\less H_x}$ commutes with $\ell$.
  If $J\subset I\subset K$, 
 then the group element 
for $(m_1\circ \ell)\circ m_2$  (resp.\ $m_1\circ (\ell \circ m_2)$) is $k_1(g\ell h)|_J$ (resp.\ $k_2(g\ell h)|_J$, where \eqref{eq:JIK} shows that  $k_1, k_2\in G_{I\less J}$ satisfy 
$$
\rho_I(y) = h*k_1*(\ell*x),\qquad \rho_I(y) = \ell* h*k_2*x.
$$
But these two conditions are the same because $\ell\in G_{H_x}$ commutes with the elements $h,k_i,g \in G_{I\less H_x}$.
A similar argument establishes \eqref{eq:QM4} in the case  $J\subset I\subset K$.

To prove associativity in $\Mor_{\Xx_\Vv}$
\begin{align}\label{eq:QM5}
\Bigl((I,J,x,g)\circ (J,K,y,h)\Bigr)\circ (K,L,z,k) = (I,J,x,g)\circ \Bigl((J,K,y,h)\circ (K,L,z,k)\Bigr),
\end{align}
note that these composites are defined only if the sets $I,J,K,L$ are nested and  $H_x= H_y = H_z \subset I\wedge J\wedge K\wedge L$.  Thus by \eqref{eq:QQ03} it suffices  to prove
$$
\bigl( (\ell_1\circ m_1) \circ (\ell_2\circ m_2)\bigr) \circ (\ell_3\circ m_3) = 
 (\ell_1\circ m_1) \circ \bigl( (\ell_2\circ m_2)\circ (\ell_3\circ m_3)\bigr),
 $$
 where $\ell_i \in G_{H_x}$ and $m_i\in \Mor_{\HatQq}$ for $i=1,2,3$.  But
 \begin{align*} 
\bigl( (\ell_1\circ m_1) \circ (\ell_2\circ m_2)\bigr) \circ (\ell_3\circ m_3) & = 
 \bigl(  (\ell_1\circ  m_1 \circ \ell_2) \circ m_2\bigr) \circ (\ell_3\circ m_3)\;\;\mbox{by } \eqref{eq:QM2}\\ 
& = (\ell_1 \circ \ell_2 \circ m_1^{\ell_2}) \circ m_2\bigr) \circ  (\ell_3\circ m_3)\;\;\mbox{by } \eqref{eq:QQ04}\\ 
& = \bigl((\ell_1 \circ \ell_2) \circ ( m_1^{\ell_2} \circ m_2)\bigr) \circ  (\ell_3\circ m_3)\;\;\mbox{by } \eqref{eq:QM3}\\ 
& =  \bigl((\ell_1 \circ \ell_2) \circ ( m_1^{\ell_2} \circ m_2) \circ  \ell_3\bigr)\circ m_3)\;\;\mbox{by } \eqref{eq:QM2}\\ 
& = (\ell_1\circ \ell_2) \circ \ell_3 \circ (m_1^{\ell_2 \ell_3}\circ m_2^{\ell_3}) \circ m_3 \;\;\mbox{by } \eqref{eq:QM2}.
\end{align*}
A similar  argument shows that 
 $$
(\ell_1\circ m_1) \circ  \bigl( (\ell_2\circ m_2)\circ (\ell_3\circ m_3) \bigr) =
\ell_1\circ (\ell_2 \circ \ell_3) \circ m_1^{\ell_2 \ell_3}\circ (m_2^{\ell_3} \circ m_3).
$$
Since composition in the group $G_{H_x}$ and in the category $\HatQq$  is associative, the
composition operation in $\Xx_\Vv$ is also  associative.  
\MS

\NI {\bf Step 6:}  
{\it $\Xx_\Vv$ is an \'etale, proper groupoid.}

The above steps  show that $\Xx_\Vv$ is a well defined groupoid. Further
the formulas given for composition are \'etale, because the element $k$ belongs to a finite group and so is locally  constant as $x,y$ vary.  (See Remark~\ref{rmk:rule}.)    Therefore the map $\pi:X_\Vv\to |\Xx_\Vv|$ is open by Lemma~\ref{lem:etale1}, and so, by Lemma~\ref{lem:prop1}~(iv) to see that $|\Xx_\Vv|$ is Hausdorff it suffices to show that  the relation $\sim_{\Xx_\Vv}$ on $X_\Vv$ has closed graph in $X_\Vv\times X_\Vv$.   Because this property is preserved by taking   inverses, it suffices to prove this for the components $\Mor_{\Xx_\Vv}(V_I, V_J)$ with $I\subset J$.
But for each $J$ the set $\Mor_{\Xx_\Vv}(V_J, V_J) \simeq V_J\times G_J$ is a finite union of disjoint components indexed by $g\in G_J$,  and for each such component ${\rm graph}(s\times t) = \{(g^{-1}*(y),y) \ | \ y\in V_J\}$, which is closed in $V_J\times V_J$.
Further, if $I\subset J$, every morphism $(I,J, y,g )\in \Mor_{\Xx_\Vv}(V_I,V_J)$  decomposes uniquely as the composite 
$$
(J,J,y,g )\circ (I,J, g ^{-1}*y, \id), \quad  y\in \TV_{IJ},\ g\in G_{J\less I},
$$
 of a morphism $V_J\to V_J$ given by the action of $g\in G_J$ on $V_J$,  followed by a morphism given by the projection $\rho_{IJ}$, where, by hypothesis,  
  ${\rm graph}(\rho_{IJ})\subset V_J\times V_I$ is closed. Hence the graph of 
 $\Mor_{\Xx_\Vv}(V_I,V_J)$ is a finite union (indexed by $G_J$) of closed sets.
Hence  $s\times t:\bX_\Vv\to X\times X$ has closed graph as claimed. 
 \MS

 \NI {\bf Step 7:}  
{\it Completion of the proof.}  

It follows from the explicit description in \eqref{eq:MMGa20} of the morphisms in $\Xx_\Vv$ that each pair $(V_J,G_J)$ is a local uniformizer for $\Xx_\Vv$. Thus $|\Xx_\Vv|$ is the union of the realizations   $|V_J|$  of these local uniformizers. The rest of (i) is proved in Step 6.
Thus it
remains to prove (ii) and (iii).  
\MS

Observe first that $G: = \prod_{i=1}^N G_i$ acts on $X_\Vv = \bigsqcup_I V_I$ via the map 
\begin{align*}
 G \times V_I\to V_I,\quad (g, I,x)\mapsto (I,g|_I*x).
\end{align*}
 Define $\al: G\times X_\Vv \to \bX_\Vv$ by
\begin{align}\label{eq:gpactY}
 \al(g,(I,x)) = (I, I, x, g|_I), \mbox{ so that } (s\times t)(\al(g, (I,x))) = \bigl((I,g|_I^{-1}*x), (I, x)\bigr).
\end{align}
 Then
\begin{align*}
 \al(g,(I, h|_I^{-1}*x)))\circ \al(h,(I, x)) & = (I, I, h|_I^{-1}*x, g|_I)\circ (I,I,  x, h|_I) \\
& = (I,I,  x , h|_Ig|_I) \\
&  = \al(hg, (I,x)) ,
\end{align*}
so that by Lemma~\ref{lem:act}~(iii) there is an associated inner action of $G$ on $\Xx_\Vv$. One can check using \eqref{eq:gpident3} and the formulas in \eqref{eq:MMGa30} that it acts on morphisms by\footnote
{
It is easiest to check this first when $I\subset J$ and then to use the formula for the inverse.}
\begin{align}\label{eq:gpactY1}
g*(I,J,y,h) = \bigl(I,J, g|_{I\vee J}*y, g|_{I\wedge J} h (g|_{I\wedge J})^{-1}\bigr)
\end{align}
  This proves (ii).
\MS

\MS

Finally,  observe that 
the inclusions $\Qq\to \Xx_\Vv$ and $\HatQq\to \Xx_\Vv$,
 are $G$-equivariant in the sense of Definition~\ref{def:gpact}, and hence
induce  functors 
$$
\Qq^{\times G}\to \Xx_\Vv^{\times G},\quad  \HatQq^{\times G}\to \Xx_\Vv^{\times G}
$$
by Lemma~\ref{lem:BbGa}~(iii). 
Composing these with the functor $\Xx_\Vv^{\times G}\to \Xx_\Vv$  of 
Lemma~\ref{lem:BbGa}~(v)  we obtain functors 
 $$
\Qq^{\times G}\to  \HatQq^{\times G}\to \Xx_\Vv.
$$
  Note that in general the functor $ \HatQq^{\times G}\to \Xx_\Vv$  is not   an equivalence of groupoids because the 
isotropy group $\Mor_{\HatQq^{\times G}}((I,x),(I,x))$ of 
the  object $(I,x)\in \HatQq^{\times G}$ contains the whole group $\prod_{j\notin I} G_j$ rather than being a subgroup of $G_I$.    Nevertheless, by Lemma~\ref{lem:BbGa}~(ii,iv),  we know $|\Qq^{\times G}| = |\HatQq^{\times G}| \simeq |\Qq|/G  $ while, because $G$ acts on $\Xx_\Vv$ by inner automorphisms, we have $| \Xx_\Vv^{\times G}|\simeq |\Xx_\Vv|$ by Lemma~\ref{lem:act}~(iii).
Since these categories all have the same object space, $|\io^{\times G}|: |\HatQq^{\times G}|\to |\Xx_\Vv|$ is a homeomorphism.
This proves (iii).  
\end{proof}

\begin{rmk} \label{rmk:QY}\rm 
\begin{nenumilist}
\item   
 The groupoid $\Xx_\Vv$ is {\it not} a groupoid completion of $\io(\Qq^{\times G})$ (cf. Definition~\ref{def:grcompl}) because, as we saw at the very end of the above proof,  the  
isotropy groups
of elements in $\io(\Qq^{\times G})$ are much too large.   \MS

\item
By Step 4 of the above proof,  we may identify $\HatQq$ with the (nonfull!) subcategory $\Xx_\Vv^{\less G}$ of $\Xx_\Vv$ that has the same object space $X_\Vv$ as $\Xx_\Vv$ and that is generated as a groupoid by the projections $\rho_{IJ}: \TV_{IJ}\to V_{IJ}\subset V_I$ that relate the different components of $X_\Vv$.  Thus we have removed the morphisms that depend explicitly on the group actions, except for those needed to form the groupoid completion.
\hfill$\er$
\end{nenumilist}
\end{rmk}

The following lemma shows that $\Xx_\Vv$ is the groupoid completion\footnote
{See Definition~\ref{def:grcompl}} of its subcategory $\Bb_\Vv$ that has the same objects but morphisms between $V_I,V_J$ only for $I\subset J$.  Indeed, it is the groupoid completion in a very strong sense: any functor from $\Bb_\Vv$ to a groupoid $\Zz$ must extend to $\Xx_\Vv$.  (Contrast this with the example in Remark~\ref{rmk:gpcompl}.)  This lemma is an essential ingredient of the proof of  Theorem~\ref{thm:reduce}  below.

\begin{lemma}\label{lem:compos}  Let 
$\Xx_\Vv$ be as in Proposition~\ref{prop:MMGa1}, and denote by
$\Bb_\Vv\subset \Xx_\Vv$ the subcategory with the same object space but morphisms from $V_I$ to $V_J$ only if
$I\subset J$.  (Thus $\Bb_\Vv = \io^{\times G}(\Qq^{\times G})$.)   
Then every functor $\phi: \Bb_\Vv\to \Zz$ from $\Bb_\Vv$ to an
arbitrary groupoid has a unique extension to  a functor $\Xx_\Vv\to \Zz$.
\end{lemma}

\begin{proof}  
 Denote an arbitrary element in  $\Mor_{\Xx_\Vv}(V_I,V_J)$ by $m_{IJ}$ and for $I\supsetneq J$ define
\begin{align}\label{eq:phiIJK0}
    \phi(m_{IJ}) : =     \bigl(\phi(m_{IJ}^{-1})\bigr)^{-1}.
\end{align}
 Since $(s\times t) (m^{-1}) = (t\times s)(m)$, it follows from the assumptions that $\phi$ is compatible with the structure maps $s,t$ of $\Xx_\Vv$ as well as preserving identities and inverses.  Hence it remains
to show that if $m_{IJ}\circ m_{JK}=m_{IK}$ in $\Xx_\Vv$ then
\begin{align}\label{eq:phiIJK}
\phi(m_{IJ})\circ \phi(m_{JK}) = \phi(m_{IK}).
\end{align}
We discuss the various cases in turn.
\begin{itemlist}
\item
If $I\subset J\subset K$ then the identity $m_{IJ}\circ m_{JK}=m_{IK}$ is a relation in $\Bb_\Vv$ so that \eqref{eq:phiIJK} holds because $\phi|_{\Bb_\Vv}$ is a functor.
\item
If $K\subset J \subset I$ then $$
m'_{KI}: = m_{IK}^{-1} = m_{JK}^{-1}\circ m_{IJ}^{-1} =: m_{KJ}'\circ m_{JI}'
$$ 
is a relation in $\Bb_\Vv$  and so is preserved by $\phi$.  Therefore because $\phi(m^{-1}) = \bigl(\phi(m)\bigr)^{-1}$ we have
$$
\bigl(\phi(m_{IK})\bigr)^{-1} = \bigl(\phi(m_{JK})\bigr)^{-1}\circ \bigl(\phi(m_{IJ})\bigr)^{-1}
= \bigl(\phi(m_{IJ})\circ \phi(m_{JK})\bigr)^{-1},
$$
where the second equality is a relation in $\Zz$, 
which implies \eqref{eq:phiIJK} in this case as well.
\item If $I\subset K\subset J$ then we can rewrite the relation $m_{IJ}\circ m_{JK}=m_{IK}$
as $m_{IJ} = m_{IK}\circ m_{KJ}'$, where $m_{KJ}' = m_{JK}^{-1}$.   Then, because the latter is a relation in $\Bb_\Vv$ we have
$$
\phi(m_{IJ}) = \phi(m_{IK})\circ \phi(m_{KJ}') = \phi(m_{IK})\circ \phi(m_{JK}^{-1})= \phi(m_{IK})\circ (\phi(m_{JK}))^{-1},
$$
 which easily implies 
\eqref{eq:phiIJK}.
By inverting this relation as in the second bullet, we find that \eqref{eq:phiIJK} also holds if $J\subset K\subset I$.
\item Finally if $J\subset I\subset K$ we rewrite $m_{IJ}\circ m_{JK}=m_{IK}$
as $m_{JK} = m_{JI}'\circ m_{IK}$ where $m_{JI}' = (m_{IJ})^{-1}$, and then argue as above.
By inverting, we may also obtain \eqref{eq:phiIJK} in the case    $K\subset I\subset J$.
\end{itemlist}
This completes the proof.
 \end{proof}

\begin{rmk}\rm
We saw in Proposition~\ref{prop:MMGa10}(iii) that  the functor $\io^{\times G}: \Qq^{\times G}\to \Xx_\Vv$ induces a homeomorphism $|\HatQq|/G\to |\Xx_\Vv|$.  It follows that each morphism $m\in \Xx_\Vv$ decomposes as a product of a morphism in $\HatQq$ with a morphism given by the action of $G$.  This product can be written in either order since
$$
(I,J, x,g) = (I,I, \rho_I(x),g)\circ (I,J, x,\id)= (I,J, g^{-1}*x,\id)\circ (J,J, \rho_J(x), g), \quad g\in G_{I\wedge J}.
$$
In terms of the map $\ov\al_{X_\Vv}: G\times X_\Vv\to \bX_\Vv, (g, (I,x))\mapsto (I,I,g*x,g)$ in Lemma~\ref{lem:act}~(iii$'$) that defines the group action, the above identity can be written
\begin{align}\label{eq:2sidedg}
 \ov\al(g, g^{-1}*\rho_I(x))\circ (I,J, x,\id) = (I,J, g^{-1}*x,\id)\circ \ov\al(g, g^{-1}*\rho_J(x)).
 \end{align}
\hfill$\er$
\end{rmk}

Finally we show that the category $\Xx_\Vv$ has a natural system of good neighbourhoods.


\begin{lemma}\label{lem:neigh} The category $\Xx_\Vv$ carries a natural  system of good neighbourhoods
as in Definition~\ref{def:struct}.
\end{lemma}
\begin{proof} Recall that    
a collection of open neighbourhoods $(U(x))_{x\in X}$ of the points $x\in X$ 
of an \'etale proper category is called a {\bf good system of neighbourhoods} if
\begin{itemize}\item[-] for every $x\in X$ the target map $t: s^{-1}(\cl_X(U(x))\to X$ is proper;
\item[-] each $U(x)$ is equipped with the natural action of the isotropy group $G_x$ at $x$;
\item[-] for any pair $x,x'\in X$ and each connected component $\Si$ of $\Mor(U(x),U(x'))$  the maps $s:\Si\to U(x)$ and $t:\Si\to U(x')$ are \'etale bijections onto their images.
\end{itemize}
 Given $|y|\in |\Xx_\Vv|$, first note that 
if $U(|y|)$ is any neighbourhood of $|y|\in |\Xx_\Vv|$ such that $\cl(U(|y|)) \subset \bigcap_{|y|\in |V_J|} \, |V_J| $,
then the properness condition  holds by Lemma~\ref{lem:prop1}(iii).  
We also suppose that, with $H_{|y|} = \bigcap_{|y|\in |V_J|} J$  as in Lemma~\ref{lem:complet}~(i),  for one (and hence any) lift $x\in  V_{H_{|y|}}$ of $y$ to $H_{|y|}$, the components of the inverse image $\pi^{-1}(U(|y|)\cap V_{H_{|y|}}$ 
  are freely permuted by the group $G_{H_{|y|}}/G_x$, where $G_x$ is the isotropy group at $x$.  It follows that for any lift $y\in V_J$ of $|y|$  the components of $\pi^{-1}(U(|y|)\cap V_J$ are  freely permuted by the group $G_J/G_x$, and we pick $U(y)$
   to be the component of $\pi^{-1}(U(|y|)\cap V_J$ that contains $y$.  Then $U(y)$ is  invariant under the action of the isotropy group $G_y$.  To check the last condition, 
 suppose first that $U(x)\subset V_I$, 
 $U(x')\subset V_J$ where $I\subset J$,  and consider the component of  $\Mor(U(x),U(x'))$ containing $m=(I,J,y,g)$.  Then $t(m) = y\in U(x')\subset \TV_{IJ}$ and $s(m) = g*\rho_I(y)$, where $\rho_I:\TV_{IJ}\to V_I$  quotients by the free action of $G_{J\less I}$.  But by assumption the images of $U(x')$ by $G_{J\less I}$ are disjoint.  Therefore given $s(m)$ and $g$ (both uniquely determined by $m$), there is at most one possible choice for $y$.   Since the argument when $J\subset I$ is similar, this completes the proof.
 \end{proof}
 

\subsection{Bundles over \texorpdfstring{$\Xx_\Vv^{\less G}$}{X V less G} and \texorpdfstring{$\Xx_\Vv$}{X V}}\label{ss:GaE}

In order to model the category $\Xx_\Vv^{\less G}\times E$  in \eqref{diag:11}.
we will consider the product of the category $\Qq$ as above with the trivial category with objects
$E  $ and only identity morphisms.  Here
we suppose that $E_i$ is a vector space on which $G_i$ acts, so that $G: = \prod_{i\in A} G_i$ acts on $E  : = \prod_{i\in A} E_i$ by the product action (where $A:= \{1,\dots, N\}$).    The statement in (ii) below should be contrasted with the fact that, when the maps $\rho_{IJ}$ have closed graph,  the space $|\HatQq^{\times G}| \simeq |\Xx_\Vv|$ is Hausdorff by Proposition~\ref{prop:MMGa1}~(i).  Again, we restate the relevant results from \S\ref{ss:mainres} for the convenience of the reader.

\begin{lemma}\label{lem:MMGaE0} 
 Let $\Xx_\Vv$ and $\Qq$ be as in Proposition~\ref{prop:MMGa10}, and 
 consider the product category $\Qq\times E =(Q\times E, {{\bQ}}\times \bE)  $ 
as follows:
$$
Q\times E  = {\textstyle \bigsqcup_{I\subset A} V_I\times E  ,\quad 
{{\bQ}}\times \bE  =  \bigsqcup_{I\subset J} \TV_{IJ}\times E  ,   }
$$
with groupoid completion $\HatQq\times E =  (\Xx_\Vv^{\less G})\times E$.  Suppose further  that $G$ acts on $E  $ as above.
Then:
\begin{enumilist}\item
 The category $\HatQq\times E $ is \'etale, and  $G$ acts on both $\Qq\times E  $ and its groupoid completion $\HatQq\times E  $  by the product action on objects and morphisms. 
\item  If the action of $G$ in $E$ is trivial, then
$|(\Qq\times E  )^{\times G}|=|(\HatQq\times E )^{\times G}|$ is Hausdorff. On the other hand,
if  $G_j$ acts nontrivially on $E_j$ for some $j$,
and  for some  $J$ with $j\in J$ and some $H\subset (J\less \{j\} )$ there is a point $x\in \Fr(\TV_{HJ})\subset V_J$ that does not lie in $\TV_{FJ}$ for any $F\subset (J\less \{j\} )$, then
 the realization
$|(\Qq\times E )^{\times G}|= |(\HatQq\times E  )^{\times G}|$  is not Hausdorff.

\item  There is a bundle $\pr: \Xx_\Vv^E\to \Xx_\Vv$ with fiber $E$ that supports a $G$-action (given on objects as in (i))
such that the functorial inclusion $\io: \HatQq=\Xx_\Vv^{\less G} \to \Xx_\Vv$ defined in Proposition~\ref{prop:MMGa10}~(iii) lifts to a $G$-equivariant functor
$\io_E: (\Xx_\Vv^{\less G})\times E \to \Xx_\Vv^E$ that induces a homeomorphism
$$
|(\HatQq\times E)^{\times G}| = |(\Xx_\Vv^{\less G}\times E)^{\times G}| \;=\; |\Xx_\Vv^{\less G}\times E|/ G \simeq 
|\Xx_\Vv^E|/G.
$$
Further, under the conditions detailed in (ii) above, $|\Xx_\Vv^E|/G $ is not  Hausdorff.
\end{enumilist}
\end{lemma}

\begin{proof} Part (i) follows from the fact that 
 the projection maps $\rho_{IJ}$ are equivariant 
with respect to the projection $G_J\to G_I$.  

To prove (ii), first note that if the action of $G$ on $E$ is trivial then
$|(\HatQq\times E  )^{\times G}| = (|\HatQq|/G)\times E = |\Xx_\Vv|\times E$ by   Proposition~\ref{prop:MMGa10}~(iii), and so is Hausdorff.
To prove the second claim in (ii), we show that under the stated conditions there is a sequence $(x_k,e)\in 
Q\times E$ whose image in $ |(\HatQq\times E  )^{\times G}|$ has two different limits. 
Consider a point $x\in  V_J$ that lies in the frontier
$\Fr(\TV_{HJ}): = \cl(\TV_{HJ}) \cap \cl(V_J\less \TV_{HJ})$ of $\TV_{HJ}$ but  does not lie in $\TV_{FJ}$ for any $F\subset (J\less \{j\} )$.\footnote
{
For example, we might take  $H = J\less \{j\}$ and $x\in \Fr(\TV_{HJ}) \less
\bigcup_{F\subset H} \TV_{FJ}$ (assuming this set is nonempty).}
By taking $I=J$ in  \eqref{eq:MorHM0},  we see that for $F\subset J$,  $\HatQq$ only contains the partial actions of $G_{J\less F}$ on $\TV_{FJ}$. 
Therefore if $g\in G_j\less \{\id\}$ there is no morphism in $\HatQq$ with source $(J,x)$ and target $(J,g*x)$.  
Thus $(J,x)\not\sim_{\HatQq} (J, g*x)$.  Now
 choose   a convergent sequence $x_k\to x$, where $x_k\in \TV_{HJ}$.
Then we have $(J,x_k)\sim_{\HatQq} (J,g*x_k)$ since the action of $G_{J\less H}$ on $\TV_{HJ}$ is included in the morphisms of $\HatQq$.  
We then have 
 $$
(J, x_k,e)\sim_\Cc (J,g*x_k, e),\quad (J, x_k,e)\sim_\Cc (J,g*x_k, g*e).
$$
But  if we choose $(g,e)\in G_j\times E_j$ so that $g*e\ne e$ we have 
$$
(J, x,e)\not \sim_\Cc (J, g*x, e).
$$
Thus the sequence $|(x_k,e)|$ has two limits   $|(x,e)|$ and $|(g*x,e)|$ in $|\Cc|$, as claimed.
\MS

To prove (iii), consider the category $\Xx_\Vv^E = (X_\Vv^E, \bX_\Vv^E)$ with
\begin{align}\label{eq:XVE}
X_\Vv^E: = {\textstyle \bigsqcup_{I\in \Ii}\, V_I\times E  ,\quad  \ \bX_\Vv^E: =  \bigsqcup_{I,J\in \Ii, I\sim J} }\, \TV_{(I\wedge J)(I\vee J)}\times G_{I\wedge J}\times  E  , 
\end{align}
where 
 \begin{align}\label{eq:MMGaE2}
(s\times t)\,(I,J,y,g,e) &= \bigl(\,(I,\, g^{-1}* \rho_{I}(y), g^{-1}*e ), \; (J,  \rho_{J}(y), e)\bigr)  \\ \notag
\id_{(I,x,e)} :=  (I,I,x, \id,e) ,& \quad (I,J,x, g,e)^{-1}:= (J,I, g^{-1} *x, g^{-1}, g^{-1} *e),
\end{align}
and (where defined) composition is given by
 \begin{align} \label{eq:MMGaE3}
&(I,J,x, g,e) \circ (J,K,y,h, e') :=
\bigl( I,K,  \rho_{I\vee K}(v) , k(hg)|_{I\wedge K}, e' \bigr) 
\end{align} 
where $v,k$ are defined in \eqref{eq:MMGa3}. Then $\pr: \Xx_\Vv^E\to \Xx_\Vv$ is a bundle with projection functor given by
$$
 (I,x,e)\mapsto (I,x),\quad   (I,J,y,g,e)\mapsto (I,J,y,g),
$$
and structural map 
$$
\mu:\bX \,\leftsub{s}{\times}_{\pr} (X\times E) \to X\times E, \quad 
\bigl( (I,J,y,g), (I,g^{-1}*\rho_{IJ}(y),e)\bigr) \mapsto (J,y,g*e). 
$$
Then $G$ acts on the object space $X_\Vv^E = Q\times E$ by the product action as in (i) above, and acts on the morphism space $\bX_\Vv^E$ by the following lift of the action on $\Xx_\Vv$ given in \eqref{eq:gpactY1}:
\begin{align} \label{eq:MMGaE4}
g*(I,y,e) & = (I, g|_I*y, g*e), \\ \notag
g*\bigl(I,J,y,h,e\bigr)
&  =\bigl(I,J, g|_{I\vee J}*y, g|_{I\wedge J}\,h\,g|_{I\wedge J}^{-1},\, g*e\bigr).
\end{align}
We now define $\io_E: \HatQq\times E  \to \Xx_\Vv^E$ to be the identity map on objects and the inclusion on morphisms.
This is equivariant, as claimed.  Hence the final claim in (iii) follows from (ii) and the fact that
$|(\Qq\times E)^{\times G}| \simeq |\Qq\times E|/ G| $ by Lemma~\ref{lem:BbGa}~(iii). 
\end{proof}

\begin{rmk}\label{rmk:GXVE} \rm  
Notice that the $G$-action on $\Xx_\Vv^E$ is {\it not} an inner action, and hence is not lifted from the inner $G$-action on $\Xx_\Vv$ by the process of Lemma~\ref{lem:actW}.   For example,  because the group $G_{A\less I}$ acts trivially on $V_I\subset X_\Vv$, the lifted inner action is also  trivial  on $V_I\times E_A$.  But the action of $G_{A\less I}$ described in \eqref{eq:MMGaE4} need not be trivial on the $E_A$ factor in $V_I\times E_A$. This is the  reason for the conditions on the kernel of the functor $\tau$  in Proposition~\ref{prop:MMGaE0}.
\hfill$\er$
\end{rmk}

We next explain the formal properties of the  $G$-equivariant functor $\tau:  \Xx_\Vv^{\less G}\times E\to \Ww_\Vv$ whose existence is claimed in Theorem~\ref{thm:globstab}~(v).   Recall from Proposition~\ref{prop:Hcomplet}~(ii) that 
$ \Xx_\Vv^{\less G}= \wh{\Qq}$ is the groupoid completion of the nonsingular \'etale category $\Qq$.
It supports an action of $G$ by Lemma~\ref{lem:MMGa}.

\begin{prop}\label{prop:MMGaE0}  In the situation of Lemma~\ref{lem:MMGaE0}, suppose given 
a bundle $\Pp: \Ww_\Vv\to \Xx_\Vv$  over $\Xx_\Vv$
with induced inner $G$-action as in Definition~\ref{def:inngpact}.
 Suppose further that there is a 
 functor  ${\tau}:  \Qq \times E \to \Ww_\Vv$  such that
 \begin{itemize}\item[-]  the induced map on object spaces is $G$-equivariant
 \item[-]    its restriction to each fiber $(I,y)\times E$ is linear with kernel  that includes the subspace $E_{A\less H_y}$, where $H_y = \min \{H\ \big| \ y\in \TQ_{HI}\}$;  
 \item[-] 
 $\Pp\circ {\tau} = \io\circ \pr: \Qq\times E\to \Xx_\Vv$, 
 \end{itemize}
where $\io: \Qq \to \Xx_\Vv$ is the  functorial inclusion from Proposition~\ref{prop:MMGa2}.

Then:  
\begin{enumilist}\item ${\tau}$ extends uniquely to a functor $\tau:\Xx_\Vv^{\less G} 
\times E \to \Ww_\Vv$.

\item  Moreover
 $\tau$
 factors as the composite  $\tau= \tau_E\circ \io_E$ in the following 
commutative diagram of functors
\begin{align*}
\xymatrix
{
(\Xx_\Vv^{\less G})\times E\ar@{->}[d]^{\pr}\ar@{->}[r]^{\;\quad \io_{E}}  & \Xx_\Vv^E \ar@{->}[d]^{\pr}\ar@{->}[r]^{\;\;\tau_E }& \Ww_\Vv \ar@{->}[d]_{\Pp}\\
\Xx_\Vv^{\less G}  \ar@{->}[r]^{\io} &   \Xx_\Vv \ar@{->}[r]^{\;\; =} & \Xx_\Vv,
}
\end{align*}
where $ \Xx_\Vv^E$ is as in Lemma~\ref{lem:MMGaE0}~(iii).

\item
Both functors $\tau_E$ and $\tau =\tau_E\circ \io_E$ are  $G$-equivariant.  In particular, $\tau(I,x,g*e) = g*\tau(I, g^{-1}*x, e).$
\end{enumilist}
 \end{prop}

\begin{proof}  We saw in Lemma~\ref{lem:complet}~(i) that the category ${\Qq}$ satisfies the conditions in Lemma~\ref{lem:gpcomplet}~(iv); namely it is a poset in which each equivalence class has a unique minimum with morphisms to all other elements. Since the product ${\Qq}\times E$ also has this property, the functor ${\tau}$ extends by Lemma~\ref{lem:gpcomplet}~(iv). This proves (i).

Towards (ii), 
consider the full subcategory  
$(\Xx_\Vv^{\less G}\times E )'$ of $\Xx_\Vv^{\less G}\times E  $
with the following object space:
\begin{align}\label{eq:defYE'}
(X_\Vv\times E)'& \ = {\textstyle \bigsqcup_I } \bigl\{(I,x,e) \in V_I\times E
\ \big| \ e\in E_{H_x} \bigr\} \\ \notag
&\ = {\textstyle \bigsqcup_I \; \bigl( \bigcup_{K\subsetneq H\subset I} } \bigl( \TV_{HI} \less \TV_{KI}\bigr) \times E_H \bigr)
\end{align}
where $H_x$ is as in \eqref{eq:MorHM0}.  Thus 
\begin{align}\label{eq:Hx}
H_x = \min \bigl\{H\ | \ |x| \in |V_H|\subset |\Xx_\Vv^{\less G}|\bigr\} =  \min \bigl\{H\ | \ |x| \in |V_H|/G= |\Xx_\Vv^{\less G}|/G\bigr\} = : H_{|x|}
\end{align}
where the above equality holds because the sets $V_H$ are $G_H$-invariant.
Below 
 we denote by $\pr_E'$  the (non continuous) projection map $X_\Vv\times E \to (X_\Vv\times E)' $ that maps the fiber $(I,x)\times E  $ at $(I,x)$ to 
$(I,x)\times E_{H_x}$ with kernel $E_{A\less H_x}$.  Note that this extends to a $G$-equivariant functor  $\pr_E: \Xx_\Vv^{\less G}\times E  \to  (\Xx_\Vv^{\less G}\times E  )'$. 

Denote by  $ (\Xx_\Vv^E)'$ the corresponding full subcategory of $ \Xx_\Vv^E$ with objects $(X_\Vv\times E)'$.
Since the group $G_{A\less I}$ acts trivially on  $\bigl\{(I,x,e) \ \big| \ e\in E_{H_x} \bigr\} \subset V_I\times E_I$, the action of $G$ on $  (\Xx_\Vv^E)'$ is an inner action given by
lifting the inner action on $ \Xx_\Vv$ as follows:
\begin{align}\label{eq:inner'}
\ov\al_{E'}: (X_\Vv\times E)' \times G \to \Mor_{ (\Xx_\Vv^E)'},&\quad  ((I,x,e),g) \mapsto  (I,I,g|_I*x,g|_I, g|_I*e).
\end{align}
Hence, because the only  morphisms in $ (\Xx_\Vv^E)'$ that are not in $(\Xx_\Vv^{\less G}\times E)'$ come from the $G$-action, the functoriality of the lifting construction (see Remark~\ref{rmk:actW}) implies that
 any functor $(\Xx_\Vv^{\less G}\times E)' \to \Ww$  that 
 \begin{itemize}
 \item[(a)]  covers the inclusion $\Xx_\Vv^{\less G}\to \Xx_\Vv$  on the base, and
 \item [(b)] is given by a $G$-equivariant map $(X_\Vv\times E)' \to W$ on objects that is fiberwise linear,
 \end{itemize}
  extends to a functor $ (\Xx_\Vv^E)'\to \Ww$. 

Now consider the statement in the proposition.  Since the categories $\Xx_\Vv^{\less G}\times E  $ and $ \Xx_\Vv^E$ have the same object space, the only claim that requires proof is that the map  $X_\Vv^E \to W$ on object spaces induced  by $\tau$  does extend to a functor $\tau_E$ between these categories.
In other words we have to show that the fact that $\tau$ is $G$-equivariant on objects implies that it respects the additional morphisms in
$ \Xx_\Vv^E$ corresponding to the group action.
But the hypotheses on the functor $\tau$ imply that the induced map on objects factors through the projection $\pr'_E: X_\Vv\times E  \to  (X_\Vv\times E)'$. 
Thus $\tau$  induces a functor $(\Xx_\Vv^{\less G}\times E  )'\to \Ww$ that satisfies the conditions (a), (b) above and hence extends to a 
$G$-equivariant functor
$\tau'_E: (\Xx_\Vv^E)'\to \Ww$.
We now define $\tau_E : = \tau_E'\circ \pr_{\Vv,E}$, where $ \pr_{\Vv,E}: \Xx_\Vv^E\to  (\Xx_\Vv^E)'$ is the extension of $\pr_E: \Xx_\Vv^{\less G}\times E  \to (\Xx_\Vv^{\less G}\times E  )'$ to $\Xx_\Vv^E$. This satisfies both (ii) and (iii).
\end{proof}

\begin{rmk}\rm\label{rmk:Gtau}
 The $G$-equivariance of $\tau$ implies that it extends to a functor
$$
\tau^G:  (\Xx_\Vv^{\less G}\times E)^{\times G}\to \Ww_\Vv,
$$
 that by Proposition~\ref{prop:MMGaE0}
factors as
\begin{align*}
 (\Xx_\Vv^{\less G}\times E)^{\times G}  \longrightarrow (\Xx_\Vv^E)^{\times G}\;
 \longrightarrow ((\Xx_\Vv^E)')^{\times G}
  \stackrel{\io^{\times G}}\longrightarrow (\Xx_\Vv^E)'\;\stackrel{\tau'_E}\longrightarrow\; \Ww_\Vv,
\end{align*}
where $(\Xx_\Vv^E)'$ has objects as in \eqref{eq:defYE'} and
$\io^{\times G}$ is the functor described in Lemma~\ref{lem:BbGa}~(v), 
which exists because, as noted  above,  the $G$ action on $(\Xx_\Vv^E)'$ is an inner action.
We saw in the preceding proposition that $\tau'_E: (\Xx_\Vv^E)'\to \Ww_\Vv$ covers the identity map on the base $\Xx_\Vv$ and that the $G$ action on both spaces is the lift of the inner action on the base.  Therefore, there are enough morphisms $m$ in the base for the  $G$-equivariance of $\tau_E'$ to be captured in the identity
\begin{align}\label{eq:mutauE}
\mu_W\bigl(m,\tau'_E(w)\bigr) = \tau'_E\bigl(\mu_E(m, w)\bigr), \quad m\in \Mor_{\Xx_\Vv},\; w\in \Obj_{(\Xx_\Vv^E)'}, \;P(w) = s(m).
\end{align}
where $\mu_\bullet$ is the structure function of the relevant bundle that expresses how morphisms of the base lift to morphisms on the total space; see Definition~\ref{def:bundle}.
\hfill$\er$
\end{rmk}

We now use the results of Proposition~\ref{prop:MMGaE0} to construct  multisections of the bundle
$\Ww_\Vv\to \Xx_\Vv$ with properties as described in Proposition~\ref{prop:globsym}.
As explained at the beginning of \S\ref{ss:polyVFC}, a multisection $\La$ of a  bundle $\Pp:\Ww\to \Xx$ 
is a functor $\La:\Ww\to \Q^{\ge 0}$ 
with a {\bf local section structure} near each $x_0\in X$ as follows:  for each $x_0\in X$ there is an open set $U(x_0)\subset X$ 
and finitely many sc$^+$-sections $\s_1,\dots,\s_k: U(x_0)\to W$ with associated positive rational numbers $\si_1,\dots,\si_k$  such that $\sum_i \si_i = 1$ and 
$$
\La(w) = {\textstyle \sum_{\{i\in \{1,\dots k\}\, | \,  \s_i(P(w))=w\}} \si_i.}
$$
The support $\supp(\La)$ is given by
$$
\supp(\La) = \{w\in W \ | \ \La(w)>0\}. 
$$
 If $\La$ is a multisection of $\Ww\to\Xx$ and $\psi:\Xx'\to \Xx$ is any functor, then
the formula $\La': = \La\circ \Psi$ defines a  {\bf pullback  multisection} of the pullback bundle 
$\Psi:\psi^*\Ww\to \Xx'$ defined in Lemma~\ref{lem:bundle}~(ii).
The first part of the next lemma shows that there is a notion of the {\bf pushforward of a multisection} of 
$\psi^*\Ww\to \Xx'$, 
provided that the functor $\psi:\Xx'\to \Xx$ is sufficiently well behaved.
(Compare similar results in \cite[Thm.13.4.1]{TheBook}).  The second part is exactly what is needed 
to show that the functor $\tau$ in Proposition~\ref{prop:MMGaE0} can be used to push forward multisections.
 In the applications in \S\ref{ss:reduce} and \S\ref{ss:stabilize} we will, of course, need to
 take into account  the  extra analytic conditions required to control compactness and smoothness.

\begin{lemma}\label{lem:pushLa}   \begin{enumilist}\item
Let $\psi:\Xx'\to \Xx$ be an equivalence,
let $\Ww\to \Xx$ a bundle with pullback  
$\psi^*\Ww\to \Xx'$, and denote by $\Psi:\psi^*\Ww\to\Ww$
the lift of $\psi:\Xx'\to \Xx$ from Lemma~\ref{lem:bundle}.  
If 
$\La': \psi^*\Ww\to \Q^{\ge0}$ 
is a multisection, then the following defines a multisection  of $\Ww\to \Xx$:
\begin{itemlist}
\item[-]  for each $w\in W$ there is $w'\in W'$ and a morphism in $\Ww$  from $w$ to  $\Psi(w')$, and we define 
\begin{align}\label{eq:pushLa}
\psi_*(\La')(w): = \La'(w').
\end{align}
\end{itemlist}
\item
Suppose that $\Ww'\to \Xx'$ and $\Ww\to \Xx$ are two bundles and that the functor $\Psi: \Ww'\to \Ww$  is the lift of  a functor  $\psi:\Xx'\to \Xx$ that is the identity map on $X' = X$ and is such that $\Psi(\Ww')$ is a full subcategory of $\Ww$  that  is invariant under the morphisms $(m,w)\in \bW$ in $\Ww$ in the sense that any morphism in $\Ww$ with source or target in the image $\Psi(W')$ is in fact contained in the image $\Psi(\bW')$.
  Then, given a multisection  
$\La': \Ww'\to \Q^{\ge0}$,
the following formula defines a multisection  of $\Ww\to \Xx$:
\begin{align}\label{eq:pushLa2}
\Psi_*(\La')(w): & =  \sum_{w' \in \Psi^{-1}(w)} \La'(w'),
\end{align}
where the sum is defined to be $0$ when $\Psi^{-1}(w)=\emptyset$. 
\end{enumilist}
\end{lemma}
\begin{proof}
Consider (i).  To see that  $\psi_*(\La'): W\to \Q^{\ge 0}$ is well defined, we note first that  every $w\in W$ is equivalent w.r.t. $\sim_{\Ww}$ to some element $\Psi(w')$ since
$|\Psi|:|\Ww'|\to |\Ww|$  is surjective, and second that the injectivity of $|\Psi|$ and the fact that $\La'$ is constant on the equivalence classes of $\sim_{\Ww'}$ imply that $\La(w)$ does not depend on the choice of $w'$ such that
$ \Psi(w')\sim_{\Ww} w.$ 
Further,  $\psi_*(\La)$ is a functor because, by construction, it is constant on the equivalence classes of  $\sim_{\Ww}$.
Finally, because $\psi_*(\La)$ is a functor, every morphism $m: x_0\mapsto x_1$ in $\Xx$ extends to a locally invertible \'etale map and has a family of fiberwise lifts 
that take a local section structure for $\La$  at $x_0$ to one at $x_1$.  Hence,
because $|\psi|:|\Xx'|\to |\Xx|$ is surjective, it suffices to check that there is a 
local section structure near every point $x_0 \in \psi(X')$.  To this end, choose $x_0'\in \psi^{-1}(x_0)$ and $U(x_0')$ so small that 
$\psi'|_{U(x_0')}$ is injective. Then, because $\psi'$ is \'etale, $U(x_0) = \psi'(U(x_0'))$ is a neighbourhood of $x_0$ and 
the local section structure $\bigl(\s_i': U(x_0')\to W', \si_i'\bigr)$ at $x_0'$ pushes forward to the local section structure 
$\bigl(\s_i = \Psi\circ \s_i\circ \psi^{-1}: U(x_0)\to W, \si_i'\bigr)$ at $x_0$.
This completes the proof of (i).

In  (ii), the map on the base is the identity, while the map on the fiber need be  neither injective not surjective.
Nevertheless, the given formula for $\Psi_*(\La')$  defines a functor on $\Ww$ because the image $\Psi(\Mor_{\Ww'})$ contains all morphisms in $\Ww$ with source or target in $\Psi(W)$.  Note also that 
the numbers  $\Psi_*(\La')(w)$ sum to $1$ over each fiber of $\Ww\to \Xx$, and that because $\psi$ is the identity on objects, each family $(\s'_i)$ of local sections at $x$ for $\La'$ gives rise to a family $(\s_i: = \Psi\circ \s_i')$ of local sections for $\La$.  
%
\end{proof}

It remains to prove Proposition~\ref{prop:globsym}  that
we now restate for the convenience of the reader.
 
\begin{prop} \label{prop:globsym0}  Let the bundle $\Ww_\Vv\to \Xx_\Vv$, group $G$, and functor $\breve\tau: \Qq\times E \to \Ww_\Vv$ with extension $\tau:\Xx_\Vv^{\less \Ga}\times E\to \Ww_\Vv$ be as
  in Proposition~\ref{prop:MMGaE}.
  \begin{itemlist}
  \item [{\rm (i)}] 
Each $e\in E$ induces a  multisection functor $\La_{\Vv,e} :  \Ww_\Vv\to \Q^{\ge 0}$ 
given by
\begin{align}\label{eq:pushLa30}
 \textstyle
\La_{\Vv,e}(w) &= \tfrac{1}{|G|} \ \# \{ h\in G \,|\, \tau(P_\Vv(w) ,h*e)= w  \},
\end{align}
where $\#$ denotes the number of elements in a finite set. 
\item[{\rm (ii)}]  This multisection $\La_{\Vv,e}$ is globally structured in the sense of Definition~\ref{def:multisdef0} by the sections $(\s_g)_{g\in G}$ and correspondence $\ka: \Mor_{\Xx_\Vv}\to \{b: G\to G, {\rm bijection}\}$ given by
\begin{align}\label{eq:global2}
\s_g(I,x) &= \tau(I,x, g*e),\	\quad 
\ka(I,J,y,g)(h) = g*h.
\end{align}
\item[{\rm (iii)}] 
 The triple $(\La_{\Vv,e}, (\s_i)_{i\in \Ii}, \ka)$ is structurable in the sense of \cite[Def.13.3.6]{TheBook}, and determines a unique structured multisection  $[\La_{\Vv,e}, \Uu, \Ss, \tau]$  in the sense of \cite[Def.13.3.8]{TheBook}, where $(U(x)_{x\in \Xx_\Vv}$ is
the system of good neighbourhoods in  $\Xx_\Vv$ defined in Lemma~\ref{lem:neigh}.
\end{itemlist}
 \end{prop}

\begin{proof} To prove (i)
we must show that if 
 the bundle $\Ww_\Vv\to \Xx_\Vv$, group $G$, and functor $\tau: \Xx_\Vv^{\less G}\times E \to \Ww_\Vv$ 
are as in Proposition~\ref{prop:MMGaE0}, then each
 $e\in E$ determines a  multisection functor $\La_{\Vv,e} :  \Ww_\Vv\to \Q^{\ge 0}$ 
given by
\begin{align}\label{eq:pushLa3}
 \textstyle
\La_{\Vv,e}(w) &= \tfrac{1}{|G|} \# \{ h\in G \,|\, \tau(P_\Vv(w) ,h*e)= w  \},
\end{align}
where $|G|$ denotes the order of the group $G$. 

For fixed $e\in E$, consider the functor $\La^{\less G}_e: \Xx_\Vv^{\less G}\times E \to \Q^{\ge 0}$ defined by setting
\begin{align*}
\La^{\less G}_e(x, e') & = \begin{cases}  1/|G| \;\mbox{ if } e' \in G*e, \\
   0 \;\mbox{ otherwise}.\end{cases}
\end{align*}
This is a $G$-equivariant functor with support $\Xx_\Vv^{\less G}\times \{G*e\}$, and hence clearly is a multisection of the bundle
$ \Xx_\Vv^{\less G}\times E\to  \Xx_\Vv^{\less G}$.  Because $\La_e^{\less G}$  is $G$-invariant, we may also consider it as a multisection of
the bundle $ (\Xx_\Vv^{\less G}\times E)^{\times G}\to  (\Xx_\Vv^{\less G})^{\times G}$ (see Lemma~\ref{lem:BbGa} and
Proposition~\ref{prop:MMGa10}~(iii)).  Next note that because  it is $G$-equivariant,  $\tau$ extends to a functor
$\tau^{G}:  (\Xx_\Vv^{\less G}\times E)^{\times G}\to \Ww_\Vv$, and it is immediate that its image 
is invariant under the morphisms of $\Ww_\Vv$ in the sense defined in Lemma~\ref{lem:pushLa}~(ii).  Hence
there is a well defined pushforward multisection
$(\tau^G)_*(\La^{\less G}_e)$, and it follows  by comparing  \eqref{eq:pushLa2} with \eqref{eq:pushLa3} that
$$
(\tau^G)_*(\La^{\less G}_e) = \La_{\Vv,e}.
$$
This completes the proof of (i).

To prove (ii), note that because $\ka$ is locally constant, it suffices to check that \eqref{eq:global2} holds; namely we require that that for all morphisms $m$ in $\Xx_\Vv$
$$
\s_{\ka(m)(h)}(t(m)) = \mu\bigl(m,\s_h(s(m))\bigr) \in W_{t(m)}.
$$
With $m = (I,J,y,g)$ we therefore require that 
$\s_{gh}(J,\rho_J(y)) = \mu\bigl(m,\s_h(I, g^{-1}*\rho_I(y))$.  Equivalently, since 
$\s_h(I,g^{-1}*\rho_I(y)) = \tau(I,g^{-1}*\rho_I(y), h*e)$, we require
 that
$$
\tau(J,\rho_J(y),gh*e)  = \mu\bigl((I,J,y,g),\tau(I,g^{-1}*\rho_I(y), h*e)\bigr)  
$$
Since every morphism  in $\Xx_\Vv$ is the composite of a morphism $(I,J,y,\id)$ in 
$\Xx_\Vv^{\less G}$ with a morphism $(J,J,y,g)$ given by the $G$ action in $\Xx_\Vv$, it suffices to check  the above identity for these two types of morphisms.  But this identity holds for the morphisms in $\Xx_\Vv^{\less G}$ since $\tau$ is a functor $\Xx_\Vv^{\less G}\times G \to \Ww_\Vv$ by Proposition~\ref{prop:MMGaE0}~(i), and it holds for the morphisms that give the $G$ action since $\tau$ is $G$-equivariant by 
Proposition~\ref{prop:MMGaE0}~(iii).

Finally the claims in (iii) hold by Proposition~\ref{prop:structmulti}.
\end{proof}

\section{Proofs of Main Theorems}\label{sec:strict}

This section proves Theorem~\ref{thm:globstab} and Corollary~\ref{cor:multisection}. Thus we consider a sc-Fredholm section functor  $f:\Xx\to \Ww$ of a strong bundle ${\Pp:\Ww\to \Xx}$ with compact solution set $S:=|f^{-1}(0)|\subset |\Xx|$. 
We begin in \S\ref{ss:Vdata} by constructing suitable \'etale data as described in  \S\ref{ss:mainres} from a family of local uniformizers for the ep-groupoid $\Xx$ that covers a neighbourhood of the compact zero set  of the polyfold-Fredholm section $f: \Xx\to \Ww$.  We then apply the constructions in \S\ref{sec:genconstr} to obtain the bottom row of the diagram \eqref{diag:11}. 
This is the content of Theorem~\ref{thm:reduce}.     
In Theorem~\ref{thm:reduceFred} we construct the pullback bundle $\Ww_\Vv\to \Xx_\Vv$ of the strong bundle $\Ww\to \Xx$, and establish the existence and properties of the right hand square in \eqref{diag:11}.   Finally, under the assumption that the local uniformizers are chosen so that they extend to local stabilizations as in Lemma~\ref{lem:localtau} that are compatible with suitable compactness controlling data,  Theorem~\ref{thm:stabilize} constructs  the left hand square  in \eqref{diag:11}.
The proof of Theorem~\ref{thm:globstab} is given in \S\ref{ss:stabilize}.
Corollary~\ref{cor:multisection} then follows by applying Proposition~\ref{prop:globsym}.
 
 \subsection{Fundamental classes induced by Fredholm sections over polyfolds}
\label{ss:polyVFC}

In this subsection we consider a compact moduli space $\oMm$ without formal boundary and of expected dimension $d$ (given by the Fredholm index of local models).  For example, $\oMm$ might be the Gromov-Witten moduli space $\oMm_{g,k}(M,J,A)$ of $J$-holomorphic curves in homology class $A\in H_*(M)$ in a closed symplectic manifold $M$, of fixed genus $g\geq 0$ and with $k\in\N_0$ marked points, with $d=\dim M + 2c_1(A)+2k +6g -6$. 
In that setting, a Kuranishi atlas $\Kk$ for $\oMm$ consists of 
an interconnected family of local models without boundary, and we showed in \cite{MWiso} how to construct a fundamental class $[\oMm]_\Kk \in \check H_d(\oMm;\Q)$ by compatibly perturbing the local models.
The main purpose of this section is to construct an analogous rational \v{C}ech homology class\footnote
{See the Appendix for a brief discussion of rational \v{C}ech homology.}
 $[\oMm]_f\in \check H_d(\oMm;\Q)$ for polyfold descriptions of $\oMm$ as in Theorem~\ref{thm:polyVFC}. 
We also briefly discuss the situation when $\oMm$ has nonempty formal boundary.

\begin{remark}\label{rmk:cech}\rm   {\bf (\v{C}ech homology and Fundamental classes)} 
We use rational \v{C}ech  homology here due to its identification with singular homology for finite simplicial complexes and its tautness property explained in \cite[Remark 8.2.4]{MW2} and Appendix~\ref{app:Cech}.
This allows for a perturbative construction of $[\oMm]$ as follows: 
Suppose the moduli space is identified $\oMm\simeq S\subset Y$ with a compact subset $S$ of a metric space $Y$. Then we can construct a fundamental class $[S]\in\check{H}_d(S;\Q)$ and push it forward  to $[\oMm]\in\check{H}_d(\oMm;\Q)$ via  $\phi:S\stackrel{\simeq} \to \oMm$.
To construct $[S]$ we choose a nested sequence of open subsets $S \subset B_k\subset B_{k-1} \subset Y$ such that $\bigcap_{k\in\N} B_k=S$.
For example, the balls $B_k:=\{ p\in Y \,|\, \inf_{s\in S} d(p, s) < 1/k \}$ of radius $1/k$ provide such nested neighbourhoods of $S$. 
Then these inclusions induce a system of maps $\check{H}_d(S;\Q)  \to \check{H}_d(B_{k+1};\Q)\to \check{H}_d(B_k;\Q)$, and the meaning of tautness is that their inverse limit equals the \v{C}ech homology
$\check{H}_d(S;\Q) \;\simeq\; \underset{\leftarrow }\lim\, \check{H}_d(B_k;\Q)$.

We can use this limit to define a fundamental class
$$
[S] := \underset{\leftarrow }\lim\, (\phi_k)_*[Z_k] \;\in\;\check{H}_d(S;\Q)
$$ 
by constructing continuous maps $\phi_k: Z_k \to B_k$ from sets $Z_k$ (such as those in Lemma~\ref{lem:epfund}) that carry natural 
fundamental classes $[Z_k]\in \check{H}_d(Z_k;\Q)$. 
Existence of this limit follows if we can ensure compatibility $(\phi_k)_*[Z_k] = (\iota_{k\ell})_*\bigl((\phi_\ell)_*[Z_\ell]\bigr)\in\check{H}_d(B_k;\Q)$ for $k<\ell$, where $\iota_{k\ell}:B_\ell\to B_k$ is the inclusion map.  
This in turn can be achieved by constructing cobordisms  $Z_{k\ell}$ with boundary $\partial Z_{k\ell}\simeq Z^-_k \sqcup Z_\ell$ and continuous maps $\phi_{k\ell}: Z_{k\ell} \to B_k$ that restrict to $\phi_k$ resp.\ $\iota_{k\ell}\circ\phi_\ell$ on the boundary.
 \hfill$\er$
\end{remark}


\begin{defn} \label{def:PolyfoldModel}
A {\bf polyfold model without boundary} $\Ee=\bigl(f:\Xx\to\Ww, \psi \bigr)$ 
of Fredholm index $d\in\Z$ for a compact Hausdorff space $\oMm$ consists of the following data: 

\begin{itemlist}
\item 
${\Pp:\Ww\to \Xx}$ is a strong bundle over an ep-groupoid $\Xx$ with empty boundary $\partial X=\emptyset$. 
\item
$f:\Xx\to \Ww$ is an oriented sc-Fredholm section functor of constant Fredholm index $d\in\Z$.
\item
$\psi: |f^{-1}(0)| \to \oMm$ is a homeomorphism. 
\end{itemlist}
\end{defn}

Here the homeomorphism $|f^{-1}(0)| \to \oMm$ guarantees that the zero set $|f^{-1}(0)|\subset|\Xx|$ is compact. 
On the other hand, the homeomorphism $|f^{-1}(0)| \to \oMm$ then induces a metrizable topology on $\oMm$ since $|f^{-1}(0)|\subset|\Xx|$ is a subset of the metrizable space $|\Xx|$ by \cite[Thm.7.3.1]{TheBook}. 
Thus rational \v{C}ech homology $\check H_d(\oMm;\Q)$ is well defined, and we will apply its tautness properties as  
explained in \cite[Remark 8.2.4]{MW2} and Appendix~\ref{app:Cech} to construct the  fundamental class $[\oMm]\in \check H_d(\oMm;\Q)$.

\begin{lemma}\label{lem:epfund}
Suppose that $\p\Xx = 0$ and $f:\Xx\to \Ww$ is an oriented sc-Fredholm section functor  of constant Fredholm index $d\in\Z$ with compact zero set $|f^{-1}(0)|$ and $\Lambda:\Ww\to\Xx$ is a sc$^+$-multisection such that $\Lambda\circ f: X \to\Q^{\geq 0}$ is a compact, oriented, tame branched, ep$^+$-subgroupoid as in \cite[Thms.15.3.7, 15.4.1]{TheBook}.
Then the solution set $|S_{\Lambda\circ f}|: = |\supp(\La\circ f)|\subset |\Xx|$
carries a natural fundamental class $[\La\circ f]\in \check{H}_d(|S_{\Lambda\circ f}|;\Q)$. 
In particular, if $M:=|S_{\Lambda\circ f}|$ is an oriented manifold, and $\Lambda\circ f\equiv \theta \in \Q^{\geq 0}$ is constant, then the fundamental class is $[\La\circ f]=\theta [M]$. 

Moreover, consider two such sc$^+$-multisections  $\Lambda^0,\Lambda^1:\Ww\to\Xx$ whose solution sets $|S_{\Lambda^0\circ f}|, |S_{\Lambda^1\circ f}| \subset B$ are contained in a common open subset $B\subset|\Xx|$. Suppose that $\Theta:[0,1]\times \Xx \to\Q^{\geq 0}$ is a compact, oriented, tame branched ep$^+$-subgroupoid with support $|\supp\Theta| \subset [0,1]\times B$ and restrictions $\Theta|_{[0,\eps)\times\Xx}(t,x)=\La^0(f(x))$, $\Theta|_{(1-\eps,1]\times\Xx}(t,x)=\La^1(f(x))$ for some $\eps>0$. 
If the orientation of $\La^1\circ f=\Theta|_{\{1\}\times\Xx}$ given by boundary restriction of $\Theta$ as in \cite[Def.9.3.8]{TheBook} agrees with the orientation induced by $f$ via the solution set convention in \S\ref{ss:orient} (generalized to multisections in \cite[\S9]{TheBook}), and the orientation of $\La^0\circ f=\Theta|_{\{0\}\times\Xx}$ given by restriction of $\Theta$ is opposite to the orientation induced by $f$, then 
$$
(\phi^0)_* [\La^0\circ f] = (\phi^1)_* [\La^1\circ f] \in \check{H}_d(B;\Q) ,
$$
where $\phi^i: |S_{\Lambda^i\circ f}| \to B$ denotes the inclusion map for $i=0,1$. 
\end{lemma}
\begin{proof}
To explain what is meant by a \lq\lq natural fundamental class" we  first describe the structure of
the support $\supp(\theta)=\{x\in X \ | \ \theta(x)>0\}$ of the functor $\theta: = \La\circ f$.
By  \cite[Def.9.1.2]{TheBook},  each point $x\in \supp(\theta)$ has a local branching structure.
 This means that
there is an open neighbourhood $U(x)$ of $x\in X$ and a finite collection of open
smooth $n$-dimensional submanifolds $M_{x,i}\ni x, i\in I_x$ of $X$, (called local branches) together with weights $\si_{x,i}\in \Q^{>0}$ such that each inclusion $M_{x,i}\to U(x)$ is proper, and
\begin{align}\label{eq:Thetay}
\theta(y) = {\textstyle  \sum_{i\in I_x: y\in M_{x,i}} \si_{x,i} \quad \mbox{ for all }\ y\in U(x).}
\end{align}
In particular $U(x) = \bigcup_{i\in I_x} M_{x,i}$.
We define the {\bf branch locus}  $\Br$ of  $S: =|\supp(\theta)|\subset |\Xx|$ to consist of all $z\in S\subset |\Xx|$ that do not have a neighbourhood in $S$ that is entirely contained in one of the sets $|M_{x,i_0}|$. By \cite[Def.9.2.2,Lem.9.2.3]{TheBook}, the  complement  $S\less \Br$ is the image in $|\supp(\theta)|$ of the set of {\bf good points} in $\supp(\theta)$.   Note that by definition, given a local branching structure $(M_{x,i})_{i\in I_x}$
each good point $y$ has a neighbourhood $V(y) \subset \supp(\theta)$  such that $V(y)\cap M_{x,i} = 
V(y)\cap M_{x,j} $ for all $i,j\in I_x$ for which $y\in M_{x,i}\cap M_{x,j}$.  
Further $\Br$ is closed and nowhere dense in $S$ by \cite[Lem.9.2.7]{TheBook}.

With this understood, we require that the fundamental class $[\theta,\fo]$ has the following property:
for each local branching structure
  $(M_{x,i})_{i\in I_x}$ over $U(x)$ and good point $y\in U(x)$,
the restriction of $[\theta,\fo]$  to  $|V(y)|$ is given by
\begin{align}\label{eq:fundc}
{\textstyle \rho_{S, |V(y)|} ([\theta,\fo]) = \sum_{i\in I_x, y\in M_{x,i}} \si_{x,i}\,\mu_{|V(y)|}\in \check{H}^\infty_n\bigl(|V(y)|\bigr)}
\end{align}
where $\mu_{|V(y)|}$ denotes the pushforward by $\pi_\Xx: V(y)\to |V(y)|\subset |\Xx|$ of the fundamental class $\mu_{V(y)}$ of the oriented $n$-manifold $V(y)$ described in \eqref{eq:Afundc}. To see that
 the condition in \eqref{eq:fundc} makes sense, notice firstly that the map $\pi_\Xx: V(y)\to |V(y)|$ is proper
 because (as in \cite[Prop.9.1.12]{TheBook}) we can choose $V(y)$ to be a subset of a local uniformizer for $\Xx$, and secondly that
the pushforward class $\mu_{|V(y)|}$ is defined  for proper maps
by the discussion before  \eqref{eq:AlesAC}. Another important point about equation~\eqref{eq:Thetay} is that it uses the \c{C}ech homology theory $\check{H}^\infty_n(\cdot)$ that is analogous to  locally finite singular homology and is dual to compactly supported rational \c{C}ech cohomology.   Thus it equals the standard \c{C}ech homology $\check{H}_n(\cdot)$ on compact spaces, but has very different behavior on noncompact spaces.

We next claim that for each local branching structure
  $(M_{x,i})_{i\in I_x}$ over $U(x)$  the condition~\eqref{eq:fundc} determines the restriction of
$[\theta,\fo]$ to $|U(x)|\less\Br$.  Indeed, $|U(x)|\less\Br$ is a (possibly infinite) union of connected components $|C|$,
where $\pi_\Xx^{-1}(|C|)\cap U(x) =:C$ is the union of the submanifolds $V(y)$ with $|y|\in |C|$.
Thus $C$  can be identified with
an open subset of an oriented $n$-manifold.
The set $C$ may have a finite number of connected components $(C_i)_{i\in I_{x,|C|}}$ (where $I_{x,|C|}$ indexes the relevant local branches), but, because $|C|$ is disjoint from the branch locus,   each such component satisfies $|C_i| = |C|$.  This 
allows us to show that  the restriction of
$[\theta,\fo]$ to $|C|$ is simply a multiple of $\mu_{|C|}$. Indeed, because restricting to an open subset $V'\subset V$ takes the fundamental class of $V$ to that of $V'$,
this
 restriction  is   $\sum_{i\in I_{x,|C|}} \si_{x,i}\,\mu_{|C|}$.
%

Next observe that because
 $S$ is compact and the map $\pi_\Xx:X\to |\Xx|$ is open, $S$ has a finite open cover  by sets $|U_i|$, where each $U_i\subset X$ supports a local branching structure.  Assume for some $k\ge 0$ that  we have constructed  the fundamental class $[\theta]_{k}$ over $W_k : = \bigcup_{j=1}^{k} |U_{j}|$ 
so that \eqref{eq:fundc} holds for each $y$ with $|V(y)|\subset W_k$,
and for  $W: = |U_{k+1}|$ define
$[\theta]_{k,W} \in \check{H}_n(W_k\cap W)$ by restriction, $[\theta]_{k,W}: = 
\rho_{W_k, W_k\cap W}([\theta]_k)$. 
As noted above, the requirement \eqref{eq:fundc} 
  specifies the restriction  of $[\theta] $ to 
   $W\less \Br$.
   We will prove the following:\MS
   
\begin{itemlist}\item[{\rm (a)}]  there is a unique class $[\theta]_W\in \check{H}^\infty_{n}(W)$ such that
$\rho_{W, W\less Br}([\theta]) = [\theta]_{W\less \Br}$.
\item[{\rm (b)}]  $\check{H}^\infty_{n+1}(W_k\cap W) = 0$, and
\item[{\rm (c)}] $\rho_{W_k,W_k\cap W}([\theta]_k) = \rho_{W, W_k\cap W}([\theta]_{W}) \in \check{H}^\infty_n(W_k\cap W)$.
\end{itemlist}
\MS

Granted this, the  Mayer--Vietoris sequence  in \eqref{eq:AMV} implies that there is a unique class 
$[\theta]_{k+1}\in \check{H}^\infty_{n}(W_{k+1})$ that restricts to $[\theta]_k \oplus [\theta]_{W} \in 
\check{H}^\infty_{n}(W_{k})\oplus \check{H}^\infty_{n}(W)$.  
Then  $[\theta]_{k+1}$ satisfies \eqref{eq:fundc}  for each $y$ with $|y|\in |W_{k+1}|$, so that the construction of $[\theta]$ is completed by a finite number of these steps.

As a preliminary step towards (a), we prove that the homology of the branch locus  $\Br$ of $S$ vanishes for $q\ge n$; indeed
\begin{align} \label{eq:fundBr}
\check{H}^\infty_{q}(\Br\cap W_k) = 0, \quad  \forall q\ge n,  \ k.
\end{align}
To see this, we first claim that $\check{H}^\infty_{q}(A)=0, q\ge n$ for every closed, proper (i.e.\ $A\ne Y$) subset  $A$ of a connected $n$-manifold $Y$.  This holds because, by the
 universal coefficient theorem \cite[Cor.4.18]{Ma}, we may  identify 
$\check{H}^\infty_{q}(A)$ with the dual  $\Hom(\check{H}_c^{q}(A),\Q) $ of the corresponding cohomology group,
and $\check{H}_c^{q}(A)=0$ for these $A$ by   \cite[Ex.1,p78]{Ma}.
But each intersection
$\Br\cap W_k$
is a finite union of such subsets $A$, since $\Br$ is  closed and nowhere dense in $S$ by \cite[Lem.9.2.7]{TheBook}, and each $W_k$ is a finite union of open subsets of $\R^n$.
Therefore,  \eqref{eq:fundBr} holds  by an  inductive argument using the  Mayer-Vietoris sequence in (iv). 

Claim (a) now follows by applying the exact sequence \eqref{eq:AlesAC} with $i=n$, $A = \Br\cap W$ and $Y=W$.
Further (b) holds because of 
the exactness of $$
\check{H}^\infty_{n+1}\bigl(\Br\cap(W_k\cap W)\bigr)\to \check{H}^\infty_{n+1}(W_k\cap W)\to \check{H}^\infty_{n+1}\bigl((W_k\cap W) \less \Br\bigr),
$$  
and the fact that $(W_k\cap W) \less \Br$ is a disjoint union of connected  $n$-manifolds so that its homology vanishes in dimension $>n$.  
Similarly, to prove (c) it suffices to check that the two classes 
$\rho_{W_k,W_k\cap W}([\theta]_k), \  \rho_{W, W_k\cap W}([\theta]_W)$ in 
$\check{H}^\infty_{n}(W_k\cap W)$  have the same restriction to $(W_k\cap W) \less \Br$.  
 But, just as in the first paragraph of this proof, it follows from \eqref{eq:fundc} that for each connected component $|C|$ of this space 
the class $[\theta]_W$ has the form $\si_C |\mu_C|$, where $\si_C$ is the sum of the weights $\si_{x,i}$ 
of the local branches that map to $|C|$.   A similar claim holds for $[\theta]_k$.  
Thus the restriction to $|C|$ of each of these two classes is $\si_C |\mu_{C}|$. This completes the definition of $[\theta]$
and hence of $[\La\circ f]$. 
\MS

To prove the identity 
$(\phi^0)_* [\La^0\circ f] = (\phi^1)_* [\La^1\circ f] \in \check{H}_d(B;\Q)$ in the second paragraph,  note that  
by \eqref{eq:AlesAC} there is a commutative diagram with exact rows
\begin{align*}
\xymatrix
{
 \check{H}^\infty_{d+1}\bigl(|\supp(\Theta)|\cap \bigr((0,1)\times B\bigr)\bigr)\ar@{->}[d]\ar@{->}[r]^{\delta}&
\check{H}^\infty_{d}\bigl(|\supp(\theta_0)|\sqcup |\supp(\theta_1)| \bigr) \ar@{->}[d]\ar@{->}[r]
 & \check{H}^\infty_{d}\bigl((|\supp(\Theta)|\bigr) \ar@{->}[d] \\
 \check{H}^\infty_{d+1}\bigl((0,1)\times B\bigr) \ar@{->}[r] &
\check{H}^\infty_{d}\bigl(\{0,1\}\times B \bigr)\ar@{->}[r] &
 \check{H}^\infty_{d}\bigl([0,1]\times B\bigr)  
}
\end{align*}
where the vertical maps are induced by the inclusion $\phi:|\supp(\Theta)|\to [0,1]\times B$ and $\theta_i: = \Lambda^i\circ f$.  
The result will follow if we show that the boundary map  $$
\de:\check{H}^\infty_{d+1}\bigl(|\supp(\Theta)|\cap \bigr((0,1)\times B\bigr)\bigr)\to \check{H}^\infty_{d}\bigl(|\supp(\theta_0)|\sqcup |\supp(\theta_1)| \bigr)
$$ 
takes the fundamental class of $|\supp(\Theta)|\cap \bigr((0,1)\times B$ to the pair $(-[\theta_0], [\theta_1])$.  But by the naturality property
\eqref{eq:fundc} and condition (a) above,   it suffices to check this on each component of the set of good points; and it holds in this case  because, for each open subset $U\subset \R^d$, the map $\de:\check{H}_{d+1}^\infty\bigl((0,1)\times U\bigr) \to \check{H}^\infty_{d}\bigl(\{0, 1\}\times U\bigr)$ takes $\mu_{(0,1)\times U}$ to  $(-\mu_{\{0\}\times U}, \mu_{\{0\}\times U}).$  Since $\check{H}^\infty_{d}(Y) = \check{H}_{d}(Y)$ for compact spaces $Y$, this
 completes the proof.
\end{proof}

\begin{theorem}  \label{thm:polyVFC}
Every polyfold model without boundary $\bigl(f:\Xx\to\Ww, \psi \bigr)$ 
for a compact space $\oMm$
induces a well defined {\bf fundamental class} 
\begin{equation}\label{eq:polyVFC}
[\oMm]_f= \psi_* \bigl( \underset{\leftarrow}\lim \, 
(\phi_k)_*[ \Lambda_k\circ f  ] \bigr)
 \; \in \; \check H_d(\oMm;\Q)  . 
\end{equation} 
In particular, this limit exists and is independent of the choice of compactness controlling data $(N,\Uu)$, nested sequence of open subsets $B_k\subset B_{k-1} \subset |\Uu|\subset|\Xx|$ such that $\bigcap_{k\in\N} B_k=|f^{-1}(0)|$, and sc$^+$-multisections $\Lambda_k$ with  the following properties.

\begin{enumilist}
\item
Each $\Lambda_k:\Ww\to\Q^{\geq 0}$ is a sc$^+$-multisection as in \cite[Def.13.2.1]{TheBook} that is transverse, supported in $\Uu\cap\pi_\Xx^{-1}(B_k)$, and bounded $N(\Lambda_k)<1$ in the sense of \cite[Thm.15.3.7]{TheBook}. 
\item 
$\Lambda_k\circ f: X \to\Q^{\geq 0}$ is a compact, oriented, tame branched ep$^+$-subgroupoid of $\Xx$  in the sense of \cite[\S9]{TheBook}.  Its solution set $S_k: = |\supp(\La_k\circ f)|\subset |\Xx|$
carries a  fundamental class $[ \Lambda_k\circ f ] \in \check H_d(S_k;\Q)$,  that is pushed forward to $(\phi_k)_*[ \Lambda_k\circ f ] \in \check H_d(B_k;\Q)$ by the inclusion map $\phi_k: S_k \to B_k\subset|\Xx|$.
 \end{enumilist}
\end{theorem}

\begin{proof}
A nested sequence $B_k\subset B_{k-1} \subset |\Uu|$ with $\bigcap_{k\in\N} B_k=|f^{-1}(0)|$ exists since $|\Xx|$ is metrizable, so we can use balls around $|f^{-1}(0)|$ intersected with $|\Uu|$. 
For any choice of nested sequence, denote $\Uu_k:=\pi_\Xx^{-1}(B_k)\subset \Uu \subset X$. Then Remark~\ref{rmk:control-compact} guarantees that $(N,\Uu_k)$ controls compactness of $f$.  
Now \cite[Thm.15.3.7]{TheBook} guarantees the existence of $(N,\Uu_k)$-regular sc$^+$-multisections $\Lambda_k:\Ww\to\Q^{\geq 0}$. Their domain support in $\Uu_k$ \cite[Def.13.2.2.]{TheBook} guarantees that $\Lambda_k(f(x))=0$ for all $x\in X\less \Uu_k$, 
so that the realization $C_k$ of the ``perturbed solution set'' $|S_{\Lambda_k\circ f}|=| \{ x\in X \,|\, \Lambda_k(f(x))>0 \}|$ is compact: indeed, it is closed and contained in $\bigl| \{ x\in \Uu \,|\, f(x)\in W[1], N(f(x))< 1 \} \bigr|\bigr)$ by the bound $N(\Lambda_k)<1$, which is shorthand for $N(\Lambda_k(w))<1$ whenever $\Lambda_k(w)\ne 0$.  Further
 $|S_{\Lambda_k\circ f}|\subset B_k$. 

Next, $|S_{\Lambda_k\circ f}|$ is given a smooth structure by the fact that $\Lambda_k\circ f: X \to\Q^{\geq 0}$ is a compact, tame branched ep$^+$-subgroupoid; see \cite[Thm.15.3.7]{TheBook}. In addition, \cite[Thm.15.4.1]{TheBook} transfers the orientation of $f$ to a canonical orientation of this branched ep$^+$-subgroupoid (by applying the solution set convention in \S\ref{ss:orient} to each branch).  Thus $|S_{\Lambda_k\circ f}|$ has a natural fundamental class $[\La_k\circ f]\in \check{H}_d(|S_{\Lambda_k\circ f}|;\Q)$ by Lemma~\ref{lem:epfund}, which can be pushed forward by the continuous  inclusion map $\phi_k: S_k \to B_k\subset|\Xx|$.
\MS

It remains to prove that the limit $\underset{\leftarrow }\lim\,  (\phi_k)_* [\Lambda_k\circ f ]  \; \in \; \check H_d(|f^{-1}(0)|;\Q)$ exists and is independent of the choices involved.  
Here the main technical work is to construct a cobordism given any choice of compactness controlling data $(N',\Uu')$ and two $(N',\Uu')$-regular sc$^+$-multisections $\Lambda^i:\Ww\to\Q^{\geq 0}$, following the arguments of \cite[Cor.15.3.10]{TheBook}. 

As a first step, notice that 
$\Ti f(t,x):=(t,f(x))$ is a sc-Fredholm section of the strong bundle $[0,1]\times\Ww\to[0,1]\times\Xx$ with compact solution set $|\Ti f^{-1}(0)|=[0,1]\times|f^{-1}(0)|$, and $\Ti N(t,w):=N'(w)$ and $\Ti \Uu:=[0,1]\times \Uu'$ control compactness of $\Ti f$. 
Moreover, the orientation of $f$ and natural orientation of $[0,1]$ induce an orientation of $\Ti f$ by the product convention in \S\ref{ss:orient}. 
Next, instead of the general multisection extensions in \cite[\S14]{TheBook}, we can directly construct a sc$^+$-multisection $\Lambda':[0,1]\times \Ww\to\Q^{\geq 0}$ which extends $\Ti \Lambda|_{\{i\}\times\Ww}=\Lambda^i$ given on the boundary $\partial|[0,1]\times\Xx|=\{0,1\}\times|\Xx|$ and satisfies properties (1) and (2) of $(\Ti N,\Ti \Uu)$-regularity. 
For that purpose let $\beta:[0,1]\to[0,1]$ be a smooth function with $\beta|_{[0,\frac 14]}\equiv 1$, $1>\beta|_{(\frac 14,\frac 13)}>0$, $\beta|_{[\frac 13,1]}\equiv 0$. 
Then we obtain a sc$^+$-multisection $\Lambda':[0,1]\times\Ww\to\Q^{\geq 0}$ by extending
$\Lambda'(t,w):=\Lambda^{0}( \beta(t)^{-1} w)$ for $t<\frac 13$ and 
$\Lambda'(t,w):=\Lambda^{1}( \beta(1-t)^{-1} w)$ for $t>\frac 23$ 
with the trivial multisection for $\frac 13 \leq t \leq \frac 23$, that is $\Lambda'(t,0_x)=1$ and $\Lambda'(t,w)=0$ for all $w\ne 0_{P(w)}$. To check the sc$^+$-smoothness and $(\Ti N,\Ti \Uu)$-regularity, let $(\s_i,\sigma_i)_{i\in I}$ be local sections and weights representing $\Lambda^0$. Then $\Lambda'|_{[0,\frac 23]\times\Ww}$ is locally represented by $(\beta \s_i,\sigma_i)_{i\in I}$ and we can check that (1) $N'(\beta \s_i) < 1$ and (2) $\supp(\beta \s_i)\subset [0,1]\times \Uu'$. 
Analogously, $\Lambda'|_{[\frac 13,0]\times\Ww}$ inherits local section structures from $\Lambda^1$, and both of these section structures for $\Lambda'$ match with a trivial section structure $\s_i(x)=0_x$ with weights $\sum\sigma_i=1$ representing  $\Lambda'$ on $[\frac 13, \frac 23]\times\Ww$. 

Moreover, $\Lambda'$ already satisfies the transversality conditions (3),(4) in \cite[Thm.15.3.7]{TheBook} over the boundary of $[0,1]\times\Xx$. Indeed, for a local section structure $(\s_i,\sigma_i)_{i\in I}$ representing $\Lambda^0$, we have $\Lambda'$ for $t<\frac 14$ represented by the local sections $\s'_i(t,x)=\s_i(x)$ and weights $\sigma_i$. Then $\rT_{(\Ti f,\Lambda')}(t,x)$ is shorthand for the collection of linearized operators ${\rm D}(\Ti f - \s'_i)(t,x): \rT_{(t,x)}[0,1]\times X\to W_x$ for each $i\in I$ with $\Ti f(t,x)= \s'_i(t,x)$. For $0\leq t\leq\frac 14$ those are the operators $\R\times \rT_xX\to W_x, (\delta t, \delta x) \mapsto {\rm D}(f - \s_i)(x) \delta x$ for each $i\in I$ with $f(x)= \s_i(x)$. So surjectivity follows from transversality of $\Lambda^0\circ f$ and the kernels are $\ker {\rm D}(\Ti f - \s'_i)(t,x) = \R\times\ker {\rm D}(f - \s_i)(x)$.  
The latter satisfy the general position requirement \cite[Def.5.3.9]{TheBook} since the reduced tangent space\footnote{For readers tracing references in \cite{TheBook} it might help to note that $E_x\subsetneq E$ in \cite[Def.2.4.7]{TheBook} is defined on the prior page. In our setting, the partial cone is $C=[0,\infty)\oplus W$ for which $E_{(0,w)}=\{0\}\oplus W$.}  is $\rT_{(0,x)} ( [0,1]\times X)  = \{0\}\times \rT_x X$. Hence suitable complements of $\ker {\rm D}(\Ti f - \s'_i)$ are given by complements of $\ker {\rm D}(f - \s_i)(x)$. 
Analogously, transversality and general position for $t\geq\frac 34$ follows from the properties of $\Lambda^1$. 

Now the proof of \cite[Thm.15.3.9]{TheBook} provides an $(\Ti N,\Ti \Uu)$-regular sc$^+$-multisection $\Ti \Lambda:[0,1]\times \Ww\to\Q^{\geq 0}$ which extends $\Ti \Lambda|_{\{i\}\times\Ww}=\Lambda^i$ by adding to $\Lambda'$ finitely many contributions supported in the interior. 
In fact, the additional contributions can be chosen with support over $(\frac 14, \frac 34)\times\Xx$ so that we have
$\Theta|_{[0,\eps)\times\Xx}(t,x)=\La^0(f(x))$, $\Theta|_{(1-\eps,1]\times\Xx}(t,x)=\La^1(f(x))$ for $\eps=\frac 14$.
The resulting functor $\Theta:=\Ti\Lambda\circ\Ti f: [0,1]\times X \to\Q^{\geq 0}$ is a compact, tame branched ep$^+$-subgroupoid, and \cite[Thm.15.4.1]{TheBook} transfers the orientation of $\Ti f$ to a canonical orientation of this branched ep$^+$-subgroupoid by applying the solution set convention in \S\ref{ss:orient} to each branch. 
And as in Remark~\ref{rmk:interval-orient} the orientation of $\Ti f|_{\{1\}\times X}$ induced by boundary restriction of the product orientation of $\Ti f$ agrees with the orientation of $f$, and the orientation of $\Ti f|_{\{0\}\times X}$ is opposite to the orientation of $f$. Thus with the solution set orientation convention we obtain oriented branched ep$^+$-subgroupoid identifications $\La^1\circ f=\Theta|_{\{1\}\times\Xx}$ and $(\La^0\circ f)^-=\Theta|_{\{0\}\times\Xx}$, so that Lemma~\ref{lem:epfund} identifies the resulting fundamental classes 
\begin{equation}\label{eq:sameNU}
(\phi^0)_*[\Lambda^0 \circ f] = (\phi^1)_*[\Lambda^1 \circ f] \; \in \; \check{H}_d(|\Uu'|;\Q) . 
\end{equation}
Here $\phi^i: |S_{\Lambda^i\circ f}| \to B$ denotes the inclusion maps for $i=0,1$.

We will use this to prove existence of the limit 
$\underset{\leftarrow}\lim \, (\phi_k)_*[ \Lambda_k\circ f ]$ for the system of homeomorphisms 
$(\iota_{k\ell})_* : \check{H}_d(B_\ell;\Q)\to \check{H}_d(B_k;\Q)$ induced by the inclusions $\iota_{k\ell}: B_\ell \to B_k\subset |\Uu|$ for $k<\ell$ and any choice of $(N, \Uu_k)$-regular sc$^+$-multisections $\Lambda_k$. 
The inclusion $\Uu_\ell=\pi_\Xx^{-1}(B_\ell)\subset \Uu_k$ guarantees that $\Lambda_\ell$ is also $(N, \Uu_k)$-regular. Comparing it to $\Lambda_k$ in \eqref{eq:sameNU} for $(N',\Uu')=(N,\Uu_k)$ yields
$(\phi_{\ell,k})_*[\Lambda_\ell \circ f] = (\phi_k)_*[\Lambda_k \circ f] \; \in \; \check{H}_d(B_k;\Q)$, where the continuous map $\phi_{\ell,k}: |S_{\Lambda_\ell \circ f}|\to B_k$  is given by composing $\phi_\ell: 
|S_{\Lambda_\ell \circ f}| \to B_\ell$ with $\iota_{k\ell}$. 
So this identity becomes $(\iota_{k\ell})_*(\phi_\ell)_*[\Lambda_\ell \circ f] = (\phi_k)_*[\Lambda_k \circ f]$, as required for the limit. 
\MS

To prove independence of choices we need to identify the inverse limits arising from different choices, indexed by $i=0,1$, of compactness controlling data $(N^i,\Uu^i)$, nested sequence $B^i_k\subset B^i_{k-1}\subset |\Uu^i|$ with $\bigcap_{k\in\N} B^i_k=|f^{-1}(0)|$, and $(N^i,\Uu^i_k)$-regular multisections $\Lambda^i_k$ for $\Uu^i_k:=\pi_\Xx^{-1}(B^i_k)$. 
We can compare both limits to one arising from the compactness controlling data $(\max\{N^0,N^1\},\Uu^0_k\cap \Uu^1_k)$, the nested sequence $B^0_k\cap B^1_k \subset |\Uu^0_k\cap \Uu^1_k|$, and corresponding multisections. 
Thus it suffices to prove equality of the induced fundamental class for choices as above where either just the multisections differ, or we can (after shifting notation) assume $N^0\leq N^1$, $\Uu^1\subset \Uu^0$, and $B^1_k\subset B^0_k$. 
In both situations we can identify the limits by ensuring that the inclusion maps $\iota^{01}_k:B^1_k \hookrightarrow B^0_k$ yield
\begin{equation}\label{eq:wewantpoly}
(\phi^0_k)_*[\Lambda^0_k\circ f] = (\iota^{01}_k)_*(\phi^1_k)_*[\Lambda^1_k\circ f] \; \in \; \check{H}_d(B^0_k;\Q) 
\qquad \text{for sufficiently large}\; k  .
\end{equation} 
To do so first note that for fixed $k$, the multisection $\Lambda^1_k$ is also $(N^0,\Uu^0_k)$-regular, 
so that just the maps $\phi^1_k: |\Lambda^1_k\circ f| \to |\Uu^1_k|$ and $\phi^{1,0}_k: 
|\Lambda^1_k\circ f| \to |\Uu^0_k|$ differ by $\phi^{1,0}_k=\iota^{10}_k\circ\phi^0_k$. Thus we obtain
$$
(\phi^{1,0}_k)_*[\Lambda^1_k\circ f] = (\iota^{10}_k)_*(\phi^1_k)_*[\Lambda^1_k\circ f] \; \in \; \check{H}_d(|\Uu^0_k|;\Q) . 
$$
Now it remains to compare the classes arising from the two $(N^0,\Uu^0)$-regular multisections $\Lambda^1_k, \Lambda^0_k$. Here the prior cobordism construction applies, so that \eqref{eq:sameNU} yields
$$
(\phi^0_k)_*[\Lambda^0_k \circ f] = (\phi^{1,0}_k)_*[\Lambda^1 \circ f] \; \in \; \check{H}_d(|\Uu^0_k|;\Q) . 
$$
Combining both identities proves \eqref{eq:wewantpoly} and thus establishes independence of the fundamental class $[\oMm]_f$ from the choices in this construction. 
\end{proof}

\begin{rmk} \rm The above construction can be adapted to the more general case when the base polyfold $\Xx$ has formal boundary, so that the zero set $\Mm: = |f^{-1}(0)|$ of a polyfold-Fredholm section $f:\Xx\to \Ww$ has formal boundary $\p \oMm= |f^{-1}(0)|\cap |\p \Xx|$.  In this case, as stated in Corollary~\ref{cor:multisection}, there is a relative fundamental class $[\Mm]_f$ constructed as above that is an element of the \c{C}ech homology group  $\check{H}^\infty_d(\oMm\less \p\oMm)$.   Further details are left to the interested reader.
\end{rmk}

\subsection{The construction of \texorpdfstring{$\Xx_\Vv$}{X V} from \texorpdfstring{$\Xx$}{X}} \;\label{ss:Vdata}

Consider an ep-groupoid $\Xx=(X,\bX)$ as in Definition~\ref{def:poly}.
This subsection constructs the ep-groupoid $\Xx_\Vv$ of Theorem~\ref{thm:globstab} which is covered by local uniformizers $X_\Vv = \bigsqcup_J V_J$ and $\Mor_{\Xx_\Vv}(V_J,V_J)=G_J\times V_J$, along with a functor $\psi: \Xx_\Vv \to \Xx|_{V}$ which induces a homeomorphism
$|\psi|: |\Xx_\Vv| \stackrel{\simeq}\to |V|\subset |\Xx|$ to a neighbourhood of a given compact solution set $S\subset |\Xx|$. The construction proceeds in the following steps: 
\begin{itemlist}
\item
In Lemma~\ref{lem:Ui} we find a finite collection of local uniformizers $(U_i,G_i)_{i=1,\dots,N}$ for $\Xx$ (see Definition~\ref{def:preunif}) whose realizations  $\bigcup_i |U_i|\subset |\Xx|$ cover the compact solution set $S\subset|\Xx|$.  
\item
Lemma~\ref{lem:Udata} compiles these local uniformizers into data $(U_J,G_J, \rho_{IJ})_{J\subset A}$ 
which satisfies the conditions (a)-(c) of \S\ref{ss:mainres} for \'etale data, except for the (separation) property. Such data consists of translation groupoids $(U_J, G_J \times U_J)$ for each $J\subset A:=\{1,\ldots,N\}$ and \'etale maps $\rho_{IJ}:U_J\to U_I$ for $I\subset J$. It is related to the original ep-groupoid $\Xx$ by so-called footprint maps $\psi_J:U_J\to |\Xx|$ which lift homeomorphisms $\qu{U_J}{G_J} \simeq \bigcap_{i\in J} |U_i|=:F_J$. The $\rho_{IJ}$ can then be thought of as lifts of the inclusions $F_J\subset F_I$. 
\item
Lemma~\ref{lem:Vdata} constructs subsets $V_J\subset U_J$ from a cover reduction $F'_J\subset F_J$ of $S\subset\bigcup_{J\in A} F'_J$ such that the restrictions $(V_J,G_J, \rho_{IJ}|_{\TV_{IJ}})_{J\subset A}$ with $\TV_{IJ}=\rho_{IJ}^{-1}(V_I)\subset V_J$ form \'etale data as in \S\ref{ss:mainres}. This cover reduction is done to meet the separation properties: $\cl(V_{HI})\cap \cl(V_{HJ} ) = \emptyset$ unless $I\subset J$ or $J\subset I$, and $\cl(\TV_{IJ})\cap \cl(\TV_{HJ} ) = \emptyset$ unless $I\subset H$ or $H\subset I$.
 \item
Theorem~\ref{thm:reduce} then uses Proposition~\ref{prop:MMGa1} to compile the ep-groupoid $\Xx_\Vv$ from the \'etale data $(V_J,G_J, \rho_{IJ})$. Here each $(V_J,G_J)$ will turn out to be a local uniformizer for $\Xx_\Vv$, and the footprint maps $\psi_J: V_J\to |\Xx|$ are lifted to construct the functor $\psi:\Xx_\Vv\to \Xx$.
\end{itemlist}

\MS

Although the main application is to the polyfold case, all constructions in this subsection are valid when $\Xx$ is a general \'etale proper groupoid in the sense of Definition~\ref{def:etale} with a given  compact subset  $S\subset |\Xx|$ that is covered by a finite family of local uniformizers. The proof is a more detailed version of the arguments in \cite{Morb}.

\begin{lemma} \label{lem:Ui} 
Let $\Xx=(X,\bX)$ be an ep-groupoid as in Definition~\ref{def:poly}, and fix a compact subset $S\subset|\Xx|$. 
Then there exists a finite collection of local uniformizers $(U_i,G_i)_{i=1,\dots,N}$ as in Definition~\ref{def:preunif}, whose realizations  $(|U_i|)_{i=1,\dots,N}$ cover $S\subset \bigcup_{i=1}^N |U_i|\subset |\Xx|$.  
\end{lemma}
\begin{proof}
Given any object $x\in X$, the isotropy group $G_x=\{g\in\bX \,|\, s(g)=t(g)=x\}$ is finite by \cite[Proposition~7.1.12]{TheBook}, and \cite[Proposition~7.1.19]{TheBook} guarantees the existence of a local uniformizer $(U_x,G_x,\Ga_x)$ around $x$ as follows: 
\begin{itemlist}
\item
$U_x\subset X$ is an open neighbourhood of $x$, equipped with the natural $G_x$-action from \cite[Theorem~7.1.13]{TheBook}. 
\item
$\Ga_x:(U_x,G_x\times U_x)\to(X,\bX)$ is a functor consisting of the inclusion $U_x\to X$ and a bijection $G_x\times U_x \to \Mor_\Xx(U_x,U_x)$, which induces a homeomorphism $\qu{U_x}{G_x}\to |U_x|\subset |\Xx|$.
\end{itemlist}
This data also forms a local uniformizer in the sense of Definition~\ref{def:preunif} since these two notions are the same by Remark~\ref{rmk:locunif}~(ii). 

Since $S$ is compact and $S\subset \bigcup_{x\in \pi^{-1}(S)} |U_x|$,
we can choose a finite subcover $S\subset \bigcup_{i=1}^N |U_{x_i}|$ indexed by finitely many $x_1,\ldots,x_N\in X$. This proves the Lemma, where we rename $U_i:=U_{x_i}$ and $G_i:=G_{x_i}$. 
\end{proof}

The following constructions are formulated for general \'etale proper groupoids. The special case of an ep-groupoid (which represents a polyfold) is obtained by working in the \'etale category whose spaces are M-polyfolds, and whose maps are local sc-diffeomorphisms (that is, sc-smooth maps that have local sc-smooth inverses). 

\begin{lemma} \label{lem:Udata} Let $(U_i,G_i)_{i\in A}$  be a finite
collection of local uniformizers  of an \'etale proper groupoid $\Xx$ 
with footprints $(F_i: = |U_i|)_{i\in A}$, and for every $I\subset A$ define 
$F_I: = \cap_{i\in I} F_i$. 
Choose an identification $A= \{1,\dots,N\}$, which induces an order  $I=\{i_1<\dots<i_k\}$ on each $I\subset A$.
Then the following holds.

\begin{nenumilist}
\item
For every nonempty $I=\{i_1<\dots<i_k\}\subset A$ there is a translation groupoid $(U_I, G_I\times U_I)$ 
with footprint map $\psi_I:U_I\to  F_I\subset |\Xx|$ as follows:
\begin{enumlist}
\item
The domain
\begin{align*}
 U_I: = U_{i_1\dots i_k}&  := \bigl\{\ux= \bigl(x_{i_1}, m_{i_1i_2}, x_{i_2}, \ldots,m_{i_{k-1}i_k},x_{i_k}\bigr) \ \big| \\
&\quad \qquad  x_{i_\ell}\in U_{i_\ell} \, , \, m_{i_ji_\ell} \in \Mor_\Xx(x_{i_j},x_{i_\ell}), \ |x_{i_1}| \in F_I \bigr\}
\end{align*}
carries a natural 
\'etale structure so that 
$s_I: U_I \to U_{i_1}, (x_{i_1}, m_{i_1i_2}, \ldots ,x_{i_k}) \mapsto x_{i_1}$ and
$t_I: U_I \to U_{i_k}, (x_{i_1}, m_{i_1i_2}, \ldots ,x_{i_k}) \mapsto x_{i_k}$ are \'etale. 
\item
The finite group $G_I: = \prod_{i_j\in I} G_{i_j}$ acts by \'etale maps 
 on $U_I$,  with the action determined by $g_{i_\ell} \in G_{i_\ell}$ acting as\footnote{
In case $\ell=1$ resp.\ $\ell=k$ this is to be read as 
$\bigl( g_{i_1} * x_{i_1} \,,\, g_{i_1}^{-1} \circ m_{i_1i_2} \, \ldots  \bigr)$ resp.\
 $\bigl( \ldots  \, m_{i_{k-1}i_k} \circ g_{i_k} \,,\, g_{i_k} * x_{i_k} \bigr)$.
}
\begin{align}\label{eq:U1}
 g_{i_\ell} * (x_{i_1},m_{i_1i_2}, \ldots, x_{i_k} ) = 
   ( \ldots  \, m_{i_{\ell-1}i_ \ell} \circ g_{i_\ell} \,,\, g_{i_\ell} * x_{i_\ell} \,,\, g_{i_\ell}^{-1} \circ m_{i_\ell i_{\ell+1}} \, \ldots ).
\end{align}
Moreover, for any $i\in I$ the action of the subgroup\footnote{ For one element sets $I=\{i\}$ we interpret $G_{I\less\{i\}} = G_\emptyset=\{\id\}$ to be the trivial group.} $G_{I\less\{i\}}\subset G_I$ on $U_I$ is free. 

\item
The footprint map 
$$
\psi_I \;:\; U_I \;\to\; |\Xx|, \qquad (x_{i_1}, m_{i_1i_2}, x_{i_2}, \ldots,m_{i_{k-1}i_ k},x_{i_k}) \;\mapsto\; |x_{i_1}|=\ldots = |x_{i_k}|
$$ 
induces a homeomorphism $|\psi_I| : \qu{U_I}{G_I} \to F_I$.
\end{enumlist}

\item
For $I\subset J$,  define $U_{IJ}: = \psi_I^{-1}(F_J)$.  
Then there is an \'etale map 
$\rho_{IJ}: U_J \to U_I$ with image $U_{IJ}$ given by forgetting the $x_{j_\ell}$ for $j_\ell\in J\less I$ and suitably composing or forgetting morphisms.
That is, for $I = J \less\{j_\ell\}$ we define $\rho_{IJ}$ which forgets the object $x_{j_\ell}$ by 
$$
\rho_{(J\less\{j_\ell\})J}(\ldots, x_{j_{\ell-1}}, m , x_{j_\ell}, m', x_{j_{\ell+1}}, \ldots)  \,:=\,
(\ldots, x_{j_{\ell-1}}, m \circ m', x_{j_{\ell+1}}, \ldots) 
$$ 
for $1<\ell<|J|$, and for $\ell=1$ resp.\ $\ell=|J|$ by 
\begin{align*}
& \rho_{(J\less\{j_1\})J}(x_{j_1}, m , x_{j_2}, \ldots) := (x_{j_2}, \ldots)  ,
\\
& \rho_{(J\less\{j_{|J|}\})J}(\ldots, x_{j_{|J|-1}}, m , x_{j_{|J|}}) := (\ldots, x_{j_{|J|-1}}). 
\end{align*}
Then for $J\less I = (j_{\ell_1} < \dots < j_{\ell_k})$ we construct the map 
$\rho_{IJ}:=\rho_{J_k J_{k-1}} \circ \ldots \circ \rho_{J_1 J_0}$
by denoting $J_0:=J$ and $J_{n}:= J_{n-1} \less \{j_{\ell_n}\}$. 

These maps $\rho_{IJ}$ 
satisfy the following conditions for all $I\subset J \subset K$

\begin{enumlist}
\item $\rho_{II} =\id_{U_I}$;

\item
$\rho_{IJ}: U_J \to  U_I$ is an \'etale map  with closed graph that is the composite of
 the free quotient $U_J\to \qu{U_J}{G_{J\less I}}$ with a $G_I$-equivariant  \'etale homeomorphism 
 $\qu{U_J}{G_{J\less I}}\stackrel{\simeq} \to U_{IJ}\subset U_I$; in particular it is equivariant with respect to the projection $G_J\to G_I$
 and induces a proper map $\rho_{IJ}:U_J\to U_{IJ}\subset U_J$;

\item
$\psi_J = \psi_I\circ\rho_{IJ} $ for all $I\subset J$;

\item for  $H\subset I\subset J$ we have $\rho_{HJ} = \rho_{HI}\circ \rho_{IJ}$.
\end{enumlist}
\end{nenumilist}
\end{lemma}

\begin{rmk}\rm  \label{rmk:2actions}  In the following argument, we will frequently use the fact that the embeddings $\Ga_i(g,\cdot): U_i \to \Mor_\Xx(U_i,U_i)$ that are part of the local uniformizers $(U_i,G_i)$ yield actions of
$G_i$ and $G_j$ on the morphism space $\Mor_\Xx(U_i,U_j)$, 
\begin{align*}
G_i \times \Mor_\Xx(U_i,U_j)\; \to \; \Mor_\Xx(U_i,U_j), \qquad & (g,m) \; \mapsto\;  g\circ m := \Ga_i(g, g^{-1}*s(m) ) \circ m ,\\
G_j \times \Mor_\Xx(U_i,U_j) \; \to\;  \Mor_\Xx(U_i,U_j), \qquad & (g,m)\;  \mapsto\;  m\circ g := m \circ \Ga_j(g, t(m) ) . 
\end{align*}
The case $i=j$ results from the natural inner action of $G_i$ on the local uniformizer $(U_i,G_i,\Ga_i)$ in Remark~\ref{rmk:inngpact}~(ii). 
\hfill$\er$
\end{rmk}

\begin{proof}[Proof of Lemma~\ref{lem:Udata}]
 {\bf Definition of the $\mathbf U_I$:}\;
To understand the construction of the sets $U_I$ in (i), note that for one element sets $I=\{i\}$ the domain $U_{\{i\}}=U_i$ is naturally an \'etale space since it is an open subset of $X$. 
For a two element set $I = \{i<j\}$ the domain
$$
U_{\{i,j\}} =: U_{ij} = \bigl\{ \bigl(s(m), m, t(m)\bigr) \ \big| \  m\in\Mor_\Xx (U_i, U_j), \ |s(m)| \in F_{ij} \bigr\}
$$ 
is naturally identified with the open subset of morphisms $\Mor_\Xx(U_i , U_j )\subset \bX$, and thus also carries an \'etale structure. This identifies the maps $s_I,t_I$ with the source and target maps on $\Mor_\Xx(U_i , U_j )$, which are \'etale  by the  \'etale property of the groupoid $\Xx$. 

For general $I =\{i_1<\dots< i_k\}$ we can identify the space $U_{i_1\dots i_k}:= U_I$ with the set of $(k-1)$-tuples of composable morphisms $U_{i_1}\to U_{i_2}\to \dots \to U_{i_k}$ in $\Xx$.
This in particular implies that the footprint map $\psi_I(x_{i_1}, m_{i_1i_2}, \ldots , x_{i_k}) = |x_{i_1}|=\dots=|x_{i_k}| \in |\Xx|$ is well defined with image $F_I: = \bigcap_{i\in I}|U_i|$. 
\MS

\NI {\bf Proof of (i):}\;
For $k=1$ the properties of local uniformizers assert that $U_{i_1}$ is a $G_{i_1}$-invariant \'etale space with $s_{i_1}=t_{i_1}:U_{i_1}\to U_{i_1}$ the identity map, and $|\psi_{i_1}| : \qu{U_{i_1}}{G_{i_1}} \to |U_{i_1}|$ is a homeomorphism. 
Moreover, $G_{I\less\{i_1\}}=G_\emptyset$ is the trivial group, which does act freely.

To prove the regularity assertions in (i)(a) for a general $U_J=U_{j_1 \dots j_k}$ we may assume by induction that $U_{J'}:= U_{j_1\dots j_{k-1}}$ carries an \'etale structure such that both $s_{J'}$ and $$
t_{J'}: U_{J'} \to U_{j_{k-1}},\quad  (x_{j_1},m_{j_1j_2}, \ldots, x_{j_{k-1}}) \mapsto x_{j_{k-1}}
$$
 are \'etale. 
Then we identify
\begin{equation}\label{eq:UJind}
U_J = \bigl\{ \bigl( \ux , m, t(m) \bigr) \ \big| \   \ux \in U_{J'} \,,\, m\in\Mor_\Xx(U_{j_{k-1}},U_{j_k}) \,,\, t_{J'}(\ux)=s(m) \bigr\} 
\end{equation}
with the fiber product $U_{J'} \, \leftsub{t_{J'}}{\times_{s}} \Mor_\Xx(U_{j_{k-1}}, U_{j_k})$. 
This naturally carries an  \'etale structure since both $t_{J'}: U_{J'} \to X$ 
and $s:\bX\to X$ are \'etale. 
A proof in the polyfold case is given in \cite[Prop.5.2.20]{TheBook} and generalizes to other \'etale categories. 

In particular, we obtain charts from sufficiently small subsets of $ \Mor_\Xx(U_{j_{k-1}}, U_{j_k})\subset \bX$.  
On those charts, the map $t_J: U_J \to U_{j_k}$ is identified with a restriction of the target map $t:\bX\to X$, which is \'etale 
hence also induces an \'etale  structure. 
Note for later that the forgetful map $$
\rho_{J'J}:U_J\to U_{J'},\quad (\ux,m,t(m))\mapsto \ux
$$ 
is  \'etale, and that it is $G_{J'}$-equivariant and constant on $G_{j_k}$-orbits in $U_J$, so also equivariant w.r.t.\ the projection 
$\rho_{J'J}^G: G_J\to G_{J'}$ that forgets $g_{j_k}$.
In these charts, $s_J: U_J \to U_{j_1}$ is identified with $s_{J'}\circ \rho_{J'J}$ and thus also is  \'etale.

To prove the claims in (i)(b) about the action of $G_J$ on $U_J$ we may assume by induction that $G_{J'}:=\prod_{\ell=1}^{k-1} G_{j_\ell}$ acts as claimed on $U_{J'}$.
Then the action of $(\ug, g_{j_k}) \in G_{J'}\times G_{j_k}=G_J$ is
$$
(\ug', g_{j_k}) * (x_{j_1}, m_{j_1j_2},\ldots, x_{j_k} ) = \bigl( \ug' * (x_{j_1}, m_{j_1j_2}, \ldots, x_{j_{k-1}}) \,,\,  g_{j_{k-1}}^{-1}\circ m_{j_{k-1}j_k} \circ g_{j_k} \,,\, g_{j_k} * x_{j_k}  \bigr).
$$
Since $G_J$ is finite, it suffices to check 
that this action is  \'etale by fixing $\ug=(\ug', g_{j_k})$ and considering the map $\ux \mapsto \ug*\ux$ in a local chart given by the map $t_J(\ux)=x_k$. In these coordinates, the action $x_{j_k} \mapsto g_{j_k}*x_{j_k}$ is  \'etale by the properties of the local uniformizer. 
Moreover, this determines an action of $G_J$ on $U_J$ because
\begin{align*}
(\ug, g_{j_k}) * (\uh, h_{j_k}) * (\ux' , m , x_{j_k}  ) 
& = (\ug, g_{j_k}) * \bigl( \uh * \ux' \,,\,  h_{j_{k-1}}^{-1}\circ m \circ h_{j_k} \,,\, h_{j_k} * x_{j_k}  \bigr) \\
& =  \bigl( \ug * \uh * \ux' \,,\,  g_{j_{k-1}}^{-1}\circ \bigl( h_{j_{k-1}}^{-1}\circ m\circ h_{j_k} \bigr) \circ g_{j_k} \,,\, g_{j_k} * h_{j_k} * x_{j_k}  \bigr) \\
& =  \bigl(  (\ug \uh )* \ux' \,,\,  ( g_{j_{k-1}} h_{j_{k-1}} )^{-1} \circ m \circ ( g_{j_k}  h_{j_k} )  \,,\, (g_{j_k} h_{j_k} ) * x_{j_k} \bigr) \\
&= \bigl( (\ug, g_{j_k}) (\uh, h_{j_k})\bigr) * (\ux' , m , x_{j_k} ) .
\end{align*}
Next, freedom of the action of $G_{J\less\{j\}}$ on $U_J$ for $j\neq j_k$ follows directly from freedom of $G_{J'\less\{j\}}$ acting on $U_{J'}$.
For $j=j_k$ the action of $\ug=(g_{j_1},\ldots,g_{j_{k-1}})\in G_{J\less\{j_k\}}$ on $U_J$ is 
$$
(\ug, \id_{G_{j_k}}) * (\ux' , m , x_{j_k}  ) 
 =  \bigl( \ug * \ux' \,,\,  g_{j_{k-1}}^{-1}\circ m \,,\,  x_{j_k}  \bigr). 
$$
Here we use that fact that the action $m\mapsto g_{j_{k-1}}^{-1}\circ m$ of $G_{j_{k-1}}$ on $\Mor_\Xx(U_{j_{k-1}}, U_{j_k})$ is free. To see this, notice that otherwise there is $g \in G_{j_{k-1}}, g\ne \id,$ and $m\in \Mor_\Xx(U_{j_{k-1}}, U_{j_k})$ such that $g\circ m = m$.  In the notation of 
Definition~\ref{def:preunif}(iii) we therefore have $\Ga(g, s(x)) = \id_{s(x)}$.  But by assumption
 $\Ga:G\times U\to \Mor(U,U)$ is injective, so that this is possible only if $g=\id$.
So given any 
$$
(\ug, \id_{G_{j_k}}) * (\ux' , m , x_{j_k}  ) =  (\ux' , m , x_{j_k}  )
$$
 we deduce $g_{j_{k-1}}=\id$ and hence $$
 \ug*\ux'=(g_{j_1}, \ldots, g_{j_{k-2}}, \id) * \ux' = \ux'.
 $$ 
Now freedom of the action of $G_{J'\less\{j_{k-1}\}}$ on $U_{J'}$  implies $\ug=\id$, which confirms the desired freedom of $G_{J\less\{j_k\}}$ acting on $U_J$.

To prove  that $|\psi_J| : \qu{U_J}{G_J} \to F_J\subset |\Xx|$ is a homeomorphism we may assume by induction that $\psi_{J'}(\ux):= |x_{j_1}|=\ldots |x_{j_{k-1}}|$ induces a homeomorphism 
$\qu{U_{J'}}{G_{J'}} \to F_{J'}$. Then using \eqref{eq:UJind} we have 
$$
\psi_J(\ux,m,t(m)):=|x_{j_1}| = \ldots = |x_{j_{k-1}}|=|s(m)|=|t(m)|=|x_{j_k}|.
$$
 This implies $\psi_J(\ux,m,t(m))= \psi_{J'}(\ux)$ so we can factor $$
 \psi_J=\psi_{J'}\circ |\rho_{J'J}|,
 $$
  where $|\rho_{J'J}|: \qu{U_J}{G_J}\to \qu{U_{J'}}{G_{J'}}$ is induced by the forgetful map $$
  \rho_{J'J}: U_J \to U_J',  (\ux,m,t(m)) \mapsto \ux.
  $$ 
As we saw in (a), $\rho_{J'J}$ is \'etale, thus an open and continuous map. Since it is equivariant w.r.t.\ the projection $\rho_{J'J}^G: G_J\to G_{J'}$, it induces an open and continuous map $|\rho_{J'J}|$. The latter is also injective since for $\ux\in U_{J'}$ the preimage $$
\rho_{J'J}^{-1}(\ux)=\{ (\ux,m,t(m)) \in U_J \} \simeq \{ m\in \Mor_\Xx ( t'(\ux) , U_{j_k} )\}
$$
 is a $G_{j_k}$-orbit,
and we have the splitting $G_J=G_{J'}\times G_{j_k}$. 
Finally, the image of $|\psi_J|$ is as claimed:
\begin{align*}
 |\psi_J|\bigl(\qu{U_J}{G_J}\bigr) = \psi_J (U_J) & =  \{ \psi_{J'}(\ux) \,|\, \ux \in U_{J'}, 
\Mor_\Xx(t'(\ux), U_{j_k}) \neq \emptyset  \}  \\
& = \psi_{J'} (U_{J'}) \cap \psi_{j_k}(U_{j_k}) \\
& = F_{J'}\cap F_{j_k} \; = \; F_J .
\end{align*}

\NI {\bf Proof of (ii):}
Next, we check the properties of the projections $\rho_{IJ}$.  
The identity $\rho_{II} =\id_{U_I}$, footprint compatibility $\psi_I\circ\rho_{IJ} = \psi_J : \ux \mapsto |x_i|$ for any $i\in I\subset J$, and \rm cocycle condition $\rho_{HI}\circ\rho_{IJ} = \rho_{HJ}$ hold directly by construction. 
Equivariance of $\rho_{IJ}$ w.r.t.\ $G$ requires $\rho_{IJ}$ to be constant on the orbits of $G_{J\less I}$. This holds since the action of $g_{j_\ell}\in G_{j_\ell}$ for $j_\ell\in J\less I$ in \eqref{eq:U1} yields
$$
\rho_{IJ}( g_{j_\ell} * \ux) = 
 ( \ldots  \, m_{j_{\ell-1}j_\ell} \circ g_{j_\ell}^{-1} \circ g_{j_\ell} \circ m_{j_\ell j_{\ell+1}} \, \ldots ) = \rho_{IJ}( \ux).
$$ 
This also implies $G_I$-equivariance of the induced map $\qu{U_J}{G_{J\less I}} \to U_{IJ}$.

To prove that $\rho_{IJ}$ for $I\subset J$ is \'etale 
and induces an   \'etale  homeomorphism $\qu{U_J}{G_{J\less I}} \to U_{IJ}$, we work by induction over $|J|=k$. The statement holds for $k=1$ since $\rho_{II}=\id_{U_I}$ is the  
identity map on $U_I=U_{II}$ and $G_{I\less I}=\{\rm id\}$. 
Given $|J|\geq 2$ we have again the identity 
 $\rho_{JJ}=\id_{U_J}$ with $G_{J\less J}=\{\rm id\}$ for $I=J$, and for $I\subsetneq J$ can factor $\rho_{IJ}=\rho_{IJ'}\circ \rho_{J'J}$ with $J'=J\less\{j_i\}$ for some choice of $j_i\in J\less I$. 
By the induction hypothesis, $\rho_{IJ'}: U_{J'}\to U_I$ is \'etale  inducing an \'etale homeomorphism $\qu{U_{J'}}{G_{J'\less I}} \to U_{IJ'}= \psi_I^{-1}(|U_{J'}|)$, so it remains to consider the map 
$$
\rho_{J'J} \,:\; U_{J}\;\to\; U_{J'} , \qquad
(\ldots, x_{j_{i-1}}, m , x_{j_i}, m', x_{j_{i+1}}, \ldots) \;\mapsto\; 
(\ldots, x_{j_{i-1}}, m \circ m', x_{j_{i+1}}, \ldots) .
$$ 
For $j_i=j_k$ this is the $G_{J'}$-equivariant \'etale map 
 from (i)(a) that is constant on $G_{j_i}$-orbits.
For $j_i\neq j_k$ it is also \'etale 
 because it intertwines the maps $t_J:U_J\to U_{j_k}$ and $t_{J'}:U_{J'}\to U_{j_k}$ which are  
 \'etale by (i)(a); and it is also $G_{J'}$-equivariant and constant on $G_{j_i}$-orbits.
This already implies that $\rho_{IJ}=\rho_{IJ'}\circ \rho_{J'J}$ is \'etale.  
Moreover, $\rho_{J'J}$ induces a $G_{J'}$-equivariant 
\'etale map $\qu{U_J}{G_{j_i}}\to U_{J'}$, which we claim is injective and thus an \'etale homeomorphism 
 to its image.
Once this is established, we may quotient by the action of $G_{J'\less I}$ -- which is free by (i)(b) --  to see that the map 
$$
\qu{U_J}{G_{J\less I}}= \bigl(\qu{U_J}{G_{j_i}}\bigr) /\, G_{J'\less I} \to \qu{U_{J'}}{G_{J'\less I}}$$
 induced by $\rho_{J'J}$ is an  \'etale homeomorphism  
 to its image.
Composition with  
 the map $\qu{U_J'}{G_{J'\less I}} \to \psi_I^{-1}(F_{J'})$ induced by $\rho_{IJ'}$, which is an  \'etale homeomorphism, 
 shows that the map  $\qu{U_J}{G_{J\less I}} \to U_I$ induced by $\rho_{IJ}$ is also an  \'etale homeomorphism  to its image. Properness of the surjective map $\rho_{IJ}:U_J\to \rho_{IJ}(U_J)\subset U_I$ then follows since it is the composition of a free finite quotient map and a homeomorphism.

To check injectivity of the map $\qu{U_J}{G_{j_i}}\to U_{J'}$ induced by $\rho_{J'J}$, note that the preimage of $(\ldots, x_{j_{i-1}}, m'', x_{j_{i+1}}, \ldots)$ is given by tuples $(\ldots, x_{j_{i-1}}, m , x_{j_i}, m', x_{j_{i+1}}, \ldots)$ with $m\circ m' = m''$ and $x_{j_i}=t(m)=s(m')\in U_{j_i}$.  Due to the groupoid properties, $m''\in\Mor_\Xx(U_{j_{i-1}}, U_{j_{i+1}})$ determines $m,m'\in\Mor_\Xx$ up to $m\mapsto m\circ\gamma, m'\mapsto \gamma^{-1}\circ m'$, and the ambiguity is exactly $\gamma\in\Mor_\Xx(t(m), U_{j_i})\simeq G_{j_i}$
by the properties of local uniformizers. 
Finally, the image of $\rho_{J'J}$ is 
\begin{align*}
&\bigl\{ (\ldots, x_{j_{i-1}}, m'' , x_{j_{i+1}}, \ldots) \in U_{J'} \,\big|\, \exists\, m,m' \in \Mor_\Xx :  
m''=m\circ m', t(m)=s(m')\in U_{j_i}  \bigr\}  \\
& = \bigl\{ (\ldots, x_{j_{i-1}}, m'' , x_{j_{i+1}}, \ldots) \in U_{J'} \,\big|\, |x_{j_{i-1}}| = |x_{j_{i+1}}| \in F_{j_i} \bigr\} \; =\; \psi_{J'}^{-1}(F_{j_i}).
\end{align*}
Here the inclusion $\subset$ follows from $m\in\Mor_\Xx(x_{i-1},U_{j_i})$, $m' \in\Mor_\Xx(U_{j_i},x_{i+1})$, and towards the inclusion $\supset$ we know that 
$|x_{j_{i-1}}| \in |U_{j_i}|$ implies existence of a morphism $m\in\Mor_\Xx(x_{j_{i-1}}, U_{j_i})$;
so then $m':= m^{-1}\circ m''$ gives the required decomposition. 
%
%
\end{proof}

The next steps in the construction of $\Xx_\Vv$ are to choose a cover reduction of the covering $S\subset \bigcup_{i=1}^N F_i$ by the open subsets $F_i: = |U_i|\subset |\Xx|$ and construct the corresponding reduction of the data constructed in Lemma~\ref{lem:Udata}. 
In the context of Theorem~\ref{thm:globstab}, this covering reduction will be chosen compatible with the supports $C_i:=\supp\beta_i \subset F_i$ of a partition of unity $\beta_i:X\to[0,1]$ used in the constructions of \S\ref{ss:stabilize}. We state all cover reduction properties here, though only (i)-(iii) are needed for the construction of $\Xx_\Vv$. The following result applies in our setting since the realization $Y:=|\Xx|$ is metrizable by \cite[Thm.7.3.1]{TheBook}, and metric spaces are normal by \cite[Thm.\,32.2]{Mu}.

\begin{lemma}\label{lem:cov0}   Let $S\subset Y$  be a closed subset of a normal topological space $Y$ with a finite cover by open subsets $S\subset\bigcup_{i\in A} F_i $. 
Then there is a {\bf cover reduction} $\bigl(F'_I\bigr)_{I\subset A}$ by open subsets $F'_I\subset Y$ in the following sense:
\begin{enumilist}
\item
The closures are contained in the given intersections: $\ov{F'_I} \,\subset\, F_I := \bigcap_{i\in I} F_i$.  
\item
The closures intersect, i.e.\  $\ov{F'_I}\cap \ov{F'_J}\ne \emptyset$, only if $I\subset J$ or $J\subset I$. 
\item
The open sets cover the given set $S\,\subset \, \bigcup_{I\subset A} F'_I$.
\end{enumilist}
If, in addition,  $S\subset\bigcup_{i\in A} C_i$ is another cover by closed subsets $C_i\subset F_i$, then the cover reduction $\bigl(F'_I\bigr)_{I\subset A}$ can be chosen such that  
\begin{enumilist}
\item[{\rm (iv)}] 
$F'_I\cap C_j = \emptyset \quad \forall j\in A\less I$.
\end{enumilist}
\end{lemma}

\begin{proof}
Since $Y$ is normal, we may choose open subsets $F_i^0\sqsubset F_i$ that still cover $S$, where the notation $A\sqsubset B$ means that the closure of $A$ is contained in $B$. 
If closed subsets $C_i\subset F_i$ are given, we more over choose the $F_i^0$ so that $C_i \subset F_i^0$. 
Now we can follow the proof of \cite[Lemma~5.3.1]{MW1} by choosing nested open sets
\begin{align}\label{eq:cov01}
F_i^0\,\sqsubset\, G_i^1 \,\sqsubset\, F_i^1
\,\sqsubset\, G_i^2 \,\sqsubset\,\ldots \,\sqsubset\,
F_i^{N} = F_i . 
\end{align}
Then \cite[Lemma~5.3.1]{MW1} shows that the required cover reduction is achieved by
\begin{align}\label{eq:cov02}
F'_I \,: =\; \Bigl( {\textstyle\bigcap_{i\in I}} G_i^{|I|} \Bigr) \;\less\; \Bigl(\textstyle {\bigcup_{j\notin I}} \ov{F^{|I|}_j}\Bigr) .
\end{align}
If closed subsets $C_i\subset F_i$ were given, then for $j\notin I$ this construction satisfies $C_j\subset F_j^0 \subset F^{|I|}_j$ and thus $F'_I\cap C_j = \emptyset$.
\end{proof}

\begin{lemma}\label{lem:Vdata}  
Let  $(U_i,G_i)_{i\in A}$  be a finite set of local uniformizers as in Lemma~\ref{lem:Udata} whose footprints $(F_i:=|U_i|)_{i\in A}$ cover a compact subset $S\subset |\Xx|$, 
and let $\bigl(F'_I \subset\bigcap_{i\in I} F_i\bigr)_{I\subset A}$ be a cover reduction as in Lemma~\ref{lem:cov0}.
Choose an order of the set $A$ and for each $I\subset A$ define $V_I: = U_I\cap \psi_I^{-1}(F_I')$, where $U_I$ is as in Lemma~\ref{lem:Udata}.    Then $(V_I, G_I, \rho_{IJ})_{I\subset A}$ 
forms  \'etale data of type $V$ in the sense of \S\ref{ss:mainres} that is compatible with the footprint maps $\psi_I$ as follows:
\begin{itemlist}\item
For each $I\subset A$ the map
$\psi_I|_{V_I}: V_I \to |\Xx|$ induces a homeomorphism $|\psi_I| : \qu{V_I}{G_I}\stackrel{\simeq} \to F_I'$.  
\item
$\psi_J\big| _{\TV_{IJ}} = \psi_I\circ\rho_{IJ} $.
\end{itemlist}
\end{lemma}
\begin{proof}  We must check that the properties 
(a), (b) and (c) listed in \S\ref{ss:mainres} are satisfied.  
Note that (a) and (b) hold by construction. 
To check (c), 
for $I\subset J$ we define  $$
\TV_{IJ}: = \psi_J^{-1}(F'_I \cap F'_J)\subset V_J, \quad  
V_{IJ}: = \psi_I^{-1}(F'_I\cap F'_J)\subset V_I.
$$
  Then,  $\rho_{IJ}$ restricts to give a map $\TV_{IJ}\to V_{IJ}$ because $\psi_J = \psi_I\circ \rho_{IJ}: U_J\to |\Xx|$.  Further, the other properties listed in Lemma~\ref{lem:Udata}~(ii) readily imply that, for each $I\subset J$,
$\TV_{IJ}\subset V_J$ is a $G_J$-invariant open subset on which $G_{J\less I}$ acts freely, and $\rho_{IJ}: \TV_{IJ}\to V_I$ is a composite $\TV_{IJ}\to \qu{\TV_{IJ}}{G_{J\less I}} \to V_I$ of the quotient map with an injective \'etale map with open image $V_{IJ}:=\rho_{IJ}(\TV_{IJ})\subset V_I$.  Thus the identity and equivariance properties in \S\ref{ss:mainres}~(c) hold.  
To check the closed graph property, consider a sequence $(x_k,y_k) = (x_k, \rho_{IJ}(x_k))\in \TV_{IJ}\times V_{IJ} \subset V_J\times V_I$ with limit $(x_\infty, y_\infty)\in V_J\times V_I$.  Then $|x_k| = |y_k| \in F_J'\cap F_I'$ so that $|x_\infty| = |y_\infty|$.  But $|x_\infty|\in 
|V_J|= F_J'$ and
$|y_\infty|\in |V_I|=F_I'$. Therefore $|x_\infty|, | y_\infty|$ both lie in $ F_J'\cap F_I'$ which implies that $x_\infty\in \TV_{IJ}, y_\infty\in V_{IJ}$ and $(x_\infty, y_\infty)\in {\rm graph} \rho_{IJ}$, as required. 
 \MS

The next step is to check that $\TV_{HJ}\cap \TV_{IJ} = \rho_{IJ}^{-1}(\TV_{HI}\cap V_{IJ})$ when $H\subset I\subset J$.
To prove this, we argue by induction as in the proof of Lemma~\ref{lem:Udata}~(i), assuming this holds when  $J$ is replaced by $J' = J-\{j_i\}$, where $j_i\in J\less I$.  
We first compute the image $V_{IJ}$ in $V_I$ of the  \'etale homeomorphism  induced by $\rho_{IJ}=\rho_{IJ'}\circ \rho_{J'J}$ on $\qu{V_J}{G_{J\less I}}$  by 
 combining the identities $\rho_{J'J} (V_J)=\psi_{J'}^{-1}(F'_{j_i})$ with $\im\rho_{IJ'}= \psi_I^{-1}(F'_{J'})$ and $\psi_{J'}^{-1}=\rho_{I{J'}}^{-1} \circ \psi_I^{-1}$ to obtain
\begin{align*}
V_{IJ}:\; =\; \rho_{IJ'}\bigl( \rho_{J'J} (V_J) \bigr) 
& =\;  \rho_{IJ'}\bigl( \psi_{J'}^{-1}(F'_{j_i}) \bigr)
\; = \; \im \rho_{IJ'} \cap \psi_I^{-1}(F'_{j_i}) \bigr) \\
& 
=\; \psi_{I}^{-1}\bigl(F'_{J'}  \cap F'_{j_i}\bigr) 
\; = \;  \psi_{I}^{-1}(F'_J) \; = \; \psi_{I}^{-1}(F'_J\cap F'_I).
\end{align*}
Therefore 
\begin{align*}
 \rho_{IJ}^{-1}(\TV_{HI}\cap V_{IJ}) & =  \rho_{IJ}^{-1}(\psi_I^{-1}(F'_H) \cap \psi_I^{-1}(F'_J)) \\
 & = \psi_J^{-1}(F'_H\cap F'_J\cap F'_I)  =
\TV_{HJ}\cap \TV_{IJ}.
\end{align*}
It follows that  the domain of the composite 
$\rho_{HI}\circ \rho_{IJ}$, when restricted to the sets $V_{\bullet}\subset U_{\bullet}$, is precisely $\TV_{HJ}\cap \TV_{IJ}$. Hence the identity
$$
\rho_{HJ}|_{\TV_{HJ}\cap \TV_{IJ}} = \rho_{HI}\circ \rho_{IJ}
$$
holds since   $\rho_{HI}\circ \rho_{IJ}=\rho_{HJ}: U_J\to U_H$.  Thus the composition property holds.

Further, the separation conditions follow from property (ii) in Lemma~\ref{lem:cov0} because 
\begin{itemlist}\item[-]  if $H\subset I,J$ then $\cl(V_{HI}), \cl(V_{HJ}) \subset V_H$ and  $\psi_H(\cl(V_{HI})\cap \cl(V_{HJ})) = \ov{F'_H}\cap \ov{F'_I}\cap  \ov{F'_J}$,
 \item[-]  if $H,I\subset J$, then $\cl(\TV_{HI}), \cl(\TV_{HJ}) \subset V_J$ and 
 $\psi_J(\cl(\TV_{HI})\cap \cl(\TV_{HJ})) = \ov{F'_H}\cap \ov{F'_I}\cap \ov{F_J'}$.
 \end{itemlist}
This completes the proof that the tuples $(V_I, G_I, \rho_{IJ})_{I\subset A}$ form  \'etale data of type $V$.

Finally, the compatibility of this data with the footprint maps follows from the  properties (i)(c) and (ii)(c) in Lemma~\ref{lem:Udata}.
\end{proof}

We are now ready to prove part (ii) of Theorem~\ref{thm:globstab}.

\begin{theorem}\label{thm:reduce}  
Let  $(U_i,G_i)_{i\in A}$  be a finite set of local uniformizers as in Lemma~\ref{lem:Udata} whose footprints cover a given compact subset $S\subset\bigcup_{i\in A} |U_i|$, let $\bigl(F'_I \subset\bigcap_{i\in I} |U_i|\bigr)_{I\subset A}$ be a cover reduction as in Lemma~\ref{lem:cov0}, and choose an order $A=\{1,\ldots,N\}$. 
Then the \'etale groupoid $\Xx_\Vv$ constructed by Proposition~\ref{prop:MMGa1} from the \'etale data in Lemma~\ref{lem:Vdata} has the following properties with $V: = \pi_\Xx^{-1}\bigl( \bigcup_{I\subset A} \psi_I(V_I)  \bigr) = \pi_\Xx^{-1}\bigl( \bigcup_{I\subset A} F'_I  \bigr)\subset X$.

\begin{enumilist}\item
 The \'etale category
$\Xx_\Vv = (X_\Vv, \bX_\Vv)$ is covered by the local uniformizers $(V_I,G_I)_{I\subset A}$, and  is proper.
In detail, it is given by  
\begin{align*}
X_\Vv & =  \textstyle\bigsqcup_{I\subset A}  V_I =   \textstyle\bigsqcup_{I\subset A}
 \bigl\{ ( I, \ux ) \,\big|\, \ux=(x_{i_1},\ldots)   \in V_I \bigr\} , \\ 
 \Mor_{\Xx_\Vv}(V_I,V_J) 
&= \begin{cases} 
 \bigl\{ \bigl(I,J, \ux , \ug  \bigr)\phantom{^{-1}} \,\big|\,  \ux\in \TV_{IJ} , \ug\in G_I  \bigr\}  &; I\subset J \\
  \bigl\{ \bigl(J,I, \ux , \ug  \bigr)^{-1} \,\big|\,  \ux\in \TV_{JI} , \ug\in G_J  \bigr\}  &; J\subset I \\ 
  \ \emptyset &; \text{otherwise} 
 \end{cases}
\end{align*}
with the structure maps 
\begin{align*}
& s(I, J, \ux , \ug )  =  \bigl( I, \ug^{-1} * \rho_{IJ}(\ux) \bigr)  , 
\quad t(I, J, \ux , \ug )  = \bigl( J ,  \ux \bigr)  ,\qquad \mbox{ for } I\subseteq J\\ \notag
&\id_{I,\ux} =  (I,I,\ux, \underline{\id} ) , \qquad
(I, J, \ux , \ug)^{-1} = (J,I,\,\ug^{-1}*\ux,\, \ug^{-1} ),  \qquad \mbox{ for } I\supseteq J
\end{align*}
and composition given as in Proposition~\ref{prop:MMGa1}.  

\item
There is an \'etale functor $\psi: \Xx_\Vv\to \Xx$  
given by
\begin{align*}
&X_\Vv\to X, \qquad \bigl( I, (x_{i_1}, m_{i_1i_2}, \ldots , x_{i_k})\bigr) \mapsto x_{i_1}, \\
&\bX_\Vv\to \bX, \qquad\qquad\qquad\qquad\qquad\qquad\qquad\qquad\qquad\qquad\qquad\qquad\qquad\qquad
\end{align*}
$$
 \bigl(I=\{i_1<\cdot\cdot \},J=\{j_1<\cdot\cdot \},\ux, \ug \bigr)   \;\mapsto \;
\begin{cases}
m_{i_1i_2} \circ \ldots \circ m_{i_{\ell-1}i_\ell} &; \mbox{if } i_1 < j_1 = i_\ell, \\ 
g_{i_1} \circ
 \id_{x_{i_1}} &; \mbox{if } i_1 = j_1, \\
g_{j_\ell} \circ \bigl( m_{j_1j_2} \circ \ldots \circ m_{j_{\ell-1}j_\ell}\bigr)^{-1}  &;  \mbox{if }j_1 < i_1= j_\ell. 
\end{cases} 
$$
This is globally finite-to-one on both objects and morphisms, and lifts the footprint maps $\psi_I|_{V_I}$ of Lemma~\ref{lem:Vdata} in the sense that
 \begin{equation}\label{eq:psilift}
\pi_\Xx \circ \psi|_{V_I}   = \psi_I|_{V_I}  : V_I \to F'_I \subset |\Xx| \qquad\text{is surjective for all}\quad I\subset A . 
\end{equation}
Restricting the target to be the full subcategory $\Xx|_V$ of $\Xx$ with objects $V$, the functor $\psi: \Xx_\Vv\to \Xx|_{V}$ is an equivalence in the sense of \cite[Def.10.1.1]{TheBook}, that is $\psi$ restricts to bijections $\Mor_{\Xx_\Vv}(\ux,\ux')\to \Mor_\Xx( \psi(\ux), \psi(\ux'))$ for all $\ux,\ux'\in X_\Vv$, 
and induces a homeomorphism 
$$
|\psi| :|\Xx_\Vv|\to |V|\subset |\Xx|.
$$

\item 
$\Xx_\Vv$ is equipped with an inner action of $G := \prod_{i\in A} G_i$ given by $\ug*(I,\ux) = (I, \ug|_I*\ux)$ and 
$\al: \Obj_{\Xx_\Vv}\times G\to \Mor_{\Xx_\Vv}$, $((I,\ux), \ug)\mapsto (I,I,\ux, \ug|_I)$ with $(g_1,\ldots,g_N)|_I:=(g_{i_1},\ldots,g_{i_{|I|}})$ as in Proposition~\ref{prop:MMGa1}~(i).  
Moreover, $\psi$ is locally $G$-equivariant in the following sense: For $I=\{i_1<\cdot\cdot\}$ we have $\psi(V_I)\subset U_{i_1}\subset \Obj_{\Xx}$ and $\psi: V_I\to U_{i_1}$ is equivariant with respect to the projection $G\to G_{i_1}$ to the isotropy group acting on the local uniformizer $U_{i_1}$. 

\item 
$\Xx_\Vv^{\less G}$ is the \'etale (but usually not proper) groupoid given by Proposition~\ref{prop:Hcomplet}~(iii). 
In detail, this groupoid is given by the structure maps of $\Xx_\Vv$ as in Theorem~\ref{thm:reduce} restricted to
\begin{align*}
\Obj_{\Xx_\Vv^{\less G}}  &=  \ \textstyle
\Obj_{\Xx_\Vv}   \ = \     \bigsqcup_{I\subset A} V_I ,   \\
\Mor_{\Xx_\Vv^{\less G}}(V_I  ,V_J ) &=  \  
 \begin{cases} 
 \bigl\{ (I,J, \ux , \ug)\phantom{^{-1}} \,\big|\, \ux \in \TV_{IJ}, \ug \in G_{I\less H_{\ux}} \bigr\}  &; I\subset J \\
  \bigl\{ (J,I, \ux , \ug )^{-1} \,\big|\,  \ux \in \TV_{JI}, \ug \in G_{J\less H_{\ux}} \bigr\}  &; J\subset I \\ 
  \ \emptyset &; \text{otherwise} 
 \end{cases}
 \end{align*}
where we denote\footnote
{See Remark~\ref{rmk:complet}~(ii).}
 $H_\ux: = \min \{H \,|\, \ux\in \TV_{HJ}\}$. 
Further,  $\Xx_\Vv ^{\less G}$ is equipped with an action of $G$ so that the functor $\io:  \Xx_\Vv ^{\less G}\to  \Xx_\Vv $ given by inclusion is $G$-equivariant in the sense of Definition~\ref{def:gpact} and induces a homeomorphism $|\io|: |\Xx_\Vv ^{\less G}|\,/G 
\stackrel{\simeq}\to 
|\Xx_\Vv|$. 
\end{enumilist}
\end{theorem}

\begin{proof} We first prove (ii).
The definition of $\Xx_\Vv$  implies that each object and morphism in $\Xx$ has a finite number of preimages.  Hence $\psi$ is finite-to-one.
 In order to establish the other properties of  $\psi$, it is useful to consider the subcategory $\Bb_\Vv$ of $\Xx_\Vv$ that has the same objects as $\Xx_\Vv$, but has morphisms from $V_I$ to $V_J$ only if $I\subset J$.   
Thus
\begin{align*}
\Mor_{\Bb_\Vv}(V_I,V_J) & = \Mor_{\Xx_\Vv}(V_I,V_J), \quad \mbox{ if } I\subset J, \\
& = \emptyset \quad\mbox{ otherwise}
\end{align*}
 To see that $\psi: \Xx_\Vv \to \Xx$ is a well defined functor,   it will suffice to show that the natural lifts of the footprint maps 
$\psi_J: V_J\to F'_J\subset |\Xx|$
fit together to give maps $\psi: X_\Vv\to  X$ and $\psi:\bX_\Vv\to \bX$
 whose restriction to the subcategory $\Bb_\Vv$ is a functor  that is 
compatible with inverses.  For then we may apply  Lemma~\ref{lem:compos} to conclude
 that $\psi$  must coincide with the unique functorial extension of $\psi|_{\Bb_\Vv}$.

Consider the restriction of $ \psi$ to $\Bb_\Vv$.   Let us denote
$$
 I = \{i_1<i_2<\cdots \}, \;\; J= \{j_1< j_2<\cdot \},\;\;
  \ug = (g_{i_1}, g_{i_2},\cdots) \in G_I.
$$
To prove that $ \psi: \Bb_\Vv\to \Xx$  is a functor we must first check that
the given formulas are compatible with source and target map.  
Suppose first that   $j_1: =  i_\ell$ for some $\ell>1$ so that
$J\subset I$.  Then  $ (I,J,\ux,\ug) \in \Mor(V_I,V_J)$ has $\ug = (g_{j_1},\cdots) \in G_J$
so that $g_{i_1}=\id$.  Hence
$$
 \psi\bigl(s(I,J,\ux,\ug)\bigr) =  \psi (I,\ug^{-1}*\ux) = x_{i_1},
\quad  \psi\bigl(t(I,J,\ux,\ug)\bigr) =  \psi (J,\ux)  = x_{i_\ell} = x_{j_1}.
$$
Thus,  $ \psi( (I,J,\ux,\ug) ) = m_{i_1i_2}\circ \cdots\circ m_{i_{\ell-1}i_\ell} \in \Mor_{\Xx}(x_{i_1}, x_{i_\ell})$ has the same source and target.
If $i_1=j_1$ this is the source and target of $ \psi\bigl((I,J,\ux,\ug)\bigr) = g_{i_1}\circ \id_{x_1}$.  
Therefore the formulas are compatible with source and target maps if $J\subset I$ or if $i_1=j_1$. Finally, note that if $i_1=j_\ell$ for $\ell>1$ then $I\subset J$ so that $\ux\in V_J$, and we have
\begin{align*}
(s\times t)\, \bigl( \psi(I,J,\ux, \ug)\bigr) & = (s\times t)\,\bigl(g_{j_\ell} \circ 
 (m_{j_{\ell-1}j_\ell})^{-1}\cdots \circ  (m_{j_1j_2}) ^{-1}\bigr)\\ 
& = \bigl(g_{j_\ell}^{-1}* x_{j_\ell}, \   x_{j_1}\bigr) =  \bigl(g_{i_1}^{-1}* x_{i_1}, \   x_{j_1}\bigr)
\end{align*}
as required.
\MS

Now consider the composition rule for $I\subset J\subset K$, where we write $K: = \{k_1<k_2<\cdots\}$ so that
 $k_1\le j_1: = k_r \le i_1 = j_\ell = k_{s}$.  We must show that  
$$
 \psi(I,J, \ux,\ug)\circ \psi(J,K, \uy,\uh)  \ = \  \psi(I,K,\uy, \uh|_I \ug),
$$
when $ \ux\in V_J,\ \uy\in V_K, \ug\in G_I,\ \uh\in G_J$.
Since this is true if either  $k_1=j_1$ or $j_1= i_1$, we will assume that
$k_1<j_1: = k_r < i_1 = j_\ell = k_{s}$.
Since $J\subset K$, for any $j_a\in J$ there is $p<q$ so that $j_a = k_p < j_{a+1} = k_q$ and we may write\footnote
{
Here the superscripts $J, K$  indicate that the entries of $m,n$ are indexed by  pairs of  elements from the sets $J,K$ respectively. Thus $m^J_{ab}$ is the composite $m_{j_p j_{p+1}}\circ m_{j_{p+1} j_{p+2}}\circ  \dots \circ m_{j_{q-1} j_{q}}$ where $a = j_p < b = j_q$.}
\begin{align}\label{eq:comp00}
&\ux = (x_{j_1},m_{j_1j_2},\dots) \in V_J, \quad \uy: = (y_{k_1}, n_{k_1k_2},\dots) \in V_K,
\\ \notag
& m^J_{j_1j_\ell}: = m_{j_1j_2}\circ \dots\circ m_{j_{\ell-1} j_\ell},\quad  n^K_{j_aj_{a+1}}: = n^K_{k_p k_q} = n_{k_p k_{p+1}}\circ\dots\circ n_{k_{q-1}k_q}.  
\end{align}
Then the identity $\ux = \uh ^{-1}*\rho_{JK}(\uy)$ gives
$$
(x_{j_1},m_{j_1j_2}, x_{j_2} \dots) = (h_{j_1}^{-1}* y_{j_1},  h_{j_1}\circ n^K_{j_1j_2} \circ h_{j_2}^{-1}, h_{j_2}^{-1}* y_{j_2}, \  \cdots ),
$$
and more generally for each $a$ and $p<t$ we have  
\begin{align}\label{eq:comp0}
m_{j_a j_{a+1}} =  h_{j_a}\circ n^K_{j_aj_{a+1}} \circ h_{j_{a+1}}^{-1}, \quad 
m^J_{k_p k_t} =  h_{k_p}\circ n^K_{k_pk_t} \circ h_{k_t}^{-1}
\end{align}
Further, since $k_1<j_1: = k_r < i_1 = j_\ell = k_{s}$, 
\begin{align*}
& \psi\bigl((I,J, \ux,\ug)\bigr) = g_{j_\ell}\circ (m^J_{j_1 j_\ell})^{-1} =  g_{k_s}\circ (m^J_{k_r k_s})^{-1},\\
&\psi\bigl((J,K, \uy,\uh)\bigr)  = h_{k_r}\circ (n^K_{k_1 k_r})^{-1}, \quad
\psi\bigl((I,K, \uy,\uh|_I \ug)\bigr)  = h_{k_s}g_{k_s}\circ (n^K_{k_1 k_s})^{-1}
\end{align*}
Therefore, we need to show 
\begin{align}\label{eq:comp1}
g_{k_s} \circ 
 (m^J_{k_rk_s})^{-1}\circ
h_{k_r} \circ 
 (n^K_{k_{1}k_r})^{-1} = 
  h_{k_s}g_{k_s} \circ 
 (n^K_{k_{1}k_s})^{-1}.
  \end{align}
   But $ h_{k_s} g_{k_s}=   g_{k_s}\circ h_{k_s}$
 because of the categorical composition convention explained in
\eqref{eq:groupact}.
 Therefore \eqref{eq:comp1} follows from  \eqref{eq:comp0}.
 \MS
 
 Hence the restriction of $ \psi$ to  $\Bb_\Vv$ preserves its structure maps and so defines a functor 
 $\Bb_\Vv\to \Xx$.  Now let us check the compatibility of $ \psi$ with the inverse operation
$$
(I,J,\ux,\ug)\mapsto (I,J,\ux,\ug)^{-1}: = (J,I,\ug^{-1}*\ux, \ug^{-1}).
$$
 If $I\subset J$ with $i_1=j_\ell$, then $\ux\in V_J, \ug\in G_I$ and
 $g_{j_a}= \id$ for $a<\ell$ so that we have
\begin{align*}
  \psi((I,J,\ux,\ug))& = g_{j_\ell} \circ 
 (m_{j_{\ell-1}j_\ell})^{-1}\cdots \circ  (m_{j_1j_2}) ^{-1} = g_{j_\ell} \circ 
 (m^J_{j_1j_\ell})^{-1},\\
    \psi(J,I,\ug^{-1}*\ux, \ug^{-1}) &= 
m_{j_1j_2}\cdots \circ m_{j_{\ell-1}j_\ell}\circ g_{j_\ell}^{-1} = m^J_{j_1j_\ell}\circ g_{j_\ell}^{-1}.
\end{align*}
Hence $ \psi$ preserves the inverse operation.
Therefore $ \psi$ is a functor $\Xx_\Vv\to \Xx$ by 
Lemma~\ref{lem:compos}.   \MS

Next, note that the functor $ \psi:\Xx_\Vv\to\Xx$ is \'etale  on the object space since on each component $V_I\subset \Obj_{\Xx_\Vv}$ it is given by the restriction of the \'etale source map  $s_I$ from Lemma~\ref{lem:Udata}~(i)(a). 
On morphisms, the functor is \'etale because it consists of  \'etale actions of the finite isotropy groups by Remark~\ref{rmk:2actions},
 and the identity map as well as composition of morphisms in $\Xx$ are  \'etale since $\Xx$ is.  
Moreover, given any $(I,\ux),(I',\uy)\in\Obj_{\Xx_\Vv}$, we need to check that $ \psi$ restricts to a bijection $\Mor_{\Xx_\Vv}\bigl((I,\ux),(I',\uy)\bigr)\to \Mor_\Xx(x_1,y_1)$. 
In case $I\subset I'$ this is exactly the content of Lemma~\ref{lem:iteration} below, and composition with the (bijective and $ \psi$-compatible) inverse maps in $\Xx_\Vv$ and $\Xx$ establishes the bijection for the case $I'\subset I$. 

We now show that the induced map
\begin{align}\label{eq:proofinject}
| \psi| :|\Xx_\Vv|\overset{\simeq}{\to} F' :={\textstyle  \bigcup_{I\subset A} F'_I \subset |\Xx|}
\end{align}
is a homeomorphism between the realizations.
This map is surjective because $| \psi(V_I)|$ is the image $\psi_I (V_I) = F'_I$ of  the footprint map in Lemma~\ref{lem:Vdata}. To check that it is injective recall that $\psi_I, \psi_{I'}$ have disjoint images unless $I\subset I'$ or $I'\subset I$. 
Now for w.l.o.g.\ $I\subset I'$ with $|x_{i_1}|=| \psi(I,\ux)|=| \psi(I',\ux')|=|x'_{i'_1}|$ the equivalence $x_{i_1}\sim_\Xx x'_{i'_1}$ implies $\Mor_\Xx(x_{i_1},x'_{i'_1})\neq\emptyset$, which transfers to 
$\Mor_{\Bb_\Vv}((I,\ux), (I',\ux')) \neq\emptyset$
by Lemma~\ref{lem:iteration}. 
This proves $ (I,\ux) \sim_{\Xx_\Vv} (I',\ux')$, as required for injectivity of $| \psi|$. 
This map and its inverse are continuous since its lift $ \psi$ on objects is continuous with continuous local inverses.
This completes the proof of (ii).

To see that $\psi$ has the  local equivariance stated in (iii), note that since
$\psi: \Mor_{\Xx_\Vv}(V_I,V_I) \to \Mor_\Xx(V_{i_1},V_{i_1})$ is given by the formula $(I,I,\ux,\ug) \mapsto g_{i_1}\circ \id_{x_1}$, 
 the $G$ action on $V_I$ given by $\ux\mapsto \ug*\ux$ maps to the $G_{i_1}$ action on $V_{i_1}\subset U_{i_1}$.

It follows that each pair $(V_I,G_I)$ is a local uniformizer for $\Xx_\Vv$. Thus  $|\Xx_\Vv|$ is covered by finitely many uniformizers and hence, by Corollary~\ref{cor:prop1}, is proper because its realization $|\Xx_\Vv|\simeq F'$ is Hausdorff.
Thus (i) also holds.

Finally,  
the claims in (iv) about the relation between the  groupoids $\Xx_\Vv^{\less G}$ and $\Xx_\Vv$ are proved in 
Proposition~\ref{prop:MMGa10}~(iii).
\end{proof}

\begin{lemma} \label{lem:iteration}
In the setting of Lemma~\ref{lem:Vdata}, given 
\begin{align*}
& \qquad  I=(k_{\ell_1}= i_1<\dots) \; \subset\;  K = (k_1 < \ldots ), \quad  \mbox{\; and }\\
& \ux=(x_{i_1}, m_{i_1i_2}, \ldots)\in U_I, 
\ \; \uy = (y_{k_1},n_{k_1k_2},\ldots) 
\in U_K \mbox{\; with } \psi_I(\ux)=\psi_K(\uy)
\end{align*}
 there is  a bijection  $\Mor_{\Bb_\Vv}((I,\ux), (K,\uy)) \to \Mor_\Xx(x_{i_{1}},y_{k_1})$ given by
$$
U_K\times G_I \;\ni \;  \bigl( \uy ,(g_{i_1},\ldots,g_{i_{|I|}}) \bigr)  \mapsto g_{i_1} \circ n_{k_{\ell_1-1}k_{\ell_1}}^{-1} \circ\ldots\circ n_{k_1k_2}^{-1}, 
$$
where $i_s = k_{\ell_s},1\le s \le |I|$.
\end{lemma}
\begin{proof}
The morphism space with fixed source and target in $\Bb_\Vv$ is 
$$
\Mor_{\Bb_\Vv}((I,\ux), (K,\uy)) = \bigl\{ (\uy, \ug) \,\big|\, \ug\in G_I , \ug * \ux = \rho_{IK}(\uy) \bigr\},
$$
 where $\rho_{IK}(\uy)=(y_{k_{\ell_1}},n^K_{i_1i_2}, \ldots, n^K_{i_{|I|-1}i_{|I|}}, y_{k_{\ell_{|I|}}})\in U_I$,
 in the notation of \eqref{eq:comp00}.
This does not capture the morphisms $n_{k_1k_2},\ldots, n_{k_{\ell_1-1}k_{\ell_1}}$ in $\uy\in U_K$, which compose to 
$$
n^K_{k_1k_{\ell_1}}:=n_{k_1k_2}\circ\ldots\circ n_{k_{\ell_1-1}k_{\ell_1}}\in \Mor_\Xx(y_{k_1},y_{k_{\ell_1}}).
$$ 
Here by assumption we have\footnote
{Recall from Lemma~\ref{lem:Vdata} that $\psi_K:V_K\to |\Xx|$ is the footprint map, not to be confused with the functor $\psi: \Xx_\Vv\to \Xx$.} 
$$
\psi_{k_{\ell_1}}(y_{k_{\ell_1}})=\psi_K(\uy)= \psi_I(\ux)=\psi_{i_1}(x_{i_1}) \;\in \; |\Xx|,
$$ 
and since $|\psi_{k_{\ell_1}}|: V_{k_{\ell_1}}/G_{k_{\ell_1}} \to F'_{k_{\ell_1}}\subset |\Xx|$  is a homeomorphism, this implies $y_{k_{\ell_1}}=g_{i_1}* x_{i_1}$ for some $g_{i_1}\in G_{i_1}$ that is unique up to $G_{x_{i_1}}=\Mor_\Xx(x_{i_1},x_{i_1})\subset G_{i_1}$. Any such choice yields $n^K_{k_1k_{\ell_1}}\circ g_{i_1}^{-1}\in \Mor_\Xx(y_{k_1},x_{i_1})$, and taking the inverse in the groupoid $\Xx$ yields $g_{i_1}\circ (n^K_{k_1k_{\ell_1}})^{-1}\in\Mor_\Xx(x_{i_1},y_{k_1})$, which shows that the given map is well defined. 

To show that the map is a bijection, we will show that any $m\in \Mor_\Xx(x_{i_1},y_{k_1})$ has exactly one preimage. More precisely, given $m, \ux, \uy$ we will iteratively determine the entries in $\ug=(g_{i_1}, g_{i_2}, \ldots)\in G_I$ from the conditions $\ug * \ux = \rho_{IK}(\uy)$ and $g_{i_1}\circ (n^K_{k_1k_{\ell_1}})^{-1} = m$. 
The iteration starts with $g_{i_1}\in  G_{i_1}$ being determined uniquely by $$
\id_{x_{i_1}}\circ g_{i_1}= m\circ n^K_{k_1k_{\ell_1}} \in \Mor_\Xx(x_{i_1},y_{k_{\ell_1}}),\quad\mbox{where } k_{\ell_1} = i_1,
$$
 using notation from Remark~\ref{rmk:2actions}.
The properties of the local uniformizer in Definition~\ref{def:preunif} then also guarantees the first object in $\ug^{-1}* \rho_{IK}(\uy)$ to be $g_{i_1}^{-1} * y_{k_{\ell_1}} = x_{i_1}$. 
For the iteration step, assume $g_{i_1},\ldots, g_{i_{s-1}}$ are determined so that the first $(s-1)$ objects and $(s-2)$ morphisms in $\ug^{-1}* \rho_{IK}(\uy)$ agree with those given by $\ux=(x_{i_1},m_{i_1i_2},\ldots, m_{i_{|I|-1}i_{|I|}},x_{i_{|I|}})$. Then we see that $g_s \in G_{i_s}$ is uniquely determined by matching $m_{i_{s-1}i_s}\in\Mor_\Xx(x_{i_{s-1}},x_{i_s})$ with the corresponding morphism in $\ug^{-1}* \rho_{IK}(\uy)$, 
\begin{align*}
m_{i_{s-1}i_s}  \;=\;  g_{i_{s-1}}\circ  n^K_{i_{s-1}i_s} \circ g_{i_s}^{-1} & \quad
\Longleftrightarrow \quad 
\\
&\hspace{-,5in}
\id_{x_{i_s}} \circ g_{i_s} \;=\;   (m_{{i_{s-1}i_s}})^{-1}\circ g_{i_{s-1}} \circ n^K_{i_{s-1}i_s} \in \Mor_\Xx(x_{i_s}, y_{k_{\ell_s}}) .
\end{align*}
This indeed determines $g_{i_s}\in G_{i_s}$ since $x_{i_s}, y_{k_{\ell_s}} \in U_{i_s}$ and the local uniformizer $(U_{i_s},G_{i_s})$ comes with a bijection $\Ga_{i_s}: G_{i_s} \times U_{i_s}\to  \Mor_\Xx(U_{i_s}, U_{i_s}), (g, x) \mapsto \id_x\circ g$. Finally, the local uniformizer property 
$$
(s\times t)\circ\Ga_{i_s} : G_{i_s} \times U_{i_s} \to U_{i_s}\times U_{i_s} , \quad (g,x) \mapsto (x, g* x)
$$
 implies matching for the $s$-th object $y_{k_{\ell_s}}= t ( \id_{x_{i_s}}\circ g_{i_s} ) =  g_{i_s} * x_{i_s}$.
\end{proof}

 \begin{rmk}\rm \label{rmk:wonderful}  
  Let  $C\subset |\Xx_\Vv| = |\Xx|$  be any set such that its topology as a subspace of $|\Xx|$ is second countable.  
 We claim that there is a constant $K$ such that each sufficiently small open subset in $C$ pulls back to a collection of at most $K$ disjoint open sets in $\pi_{\Xx_\Vv}^{-1}(C)$.  This holds because  
 $|X_\Vv|$ is covered by the footprints of the finite number of uniformizers $(V_J, G_J), J\subset A,$, and each $|V_J|$  has this property with $K: = |G_J|$.
Hence  any countable collection of open subsets of $C$ that form a basis for the subspace topology of $C$ pulls back to 
a countable collection of open subsets of $\pi_{\Xx_\Vv}^{-1}(C)$ that  form a
basis for the  topology of $\pi_{\Xx_\Vv}^{-1}(C)$.
 Thus the inverse image $\pi_{\Xx_\Vv}^{-1}(C)$ of any second countable subset  $C\subset |\Xx_\Vv| = |\Xx|$  
  is second countable in the subspace topology.  Note that this might not hold for the map $X\to |\Xx|$.
\hfill$\er$
\end{rmk}

 \begin{rmk}\rm \label{rmk:XU}  
There is an analogous  \'etale groupoid  $\Xx_\Uu$ with objects $\bigsqcup_I U_I$, where $U_I$ is as in Lemma~\ref{lem:Udata}, and whose morphisms are generated by the projections $(\rho_{IJ}: U_J\to U_I)_{I\subset J}$
 and group actions $G_J\times U_J\to U_J$.  However, the full space of morphisms in $\Xx_\Uu$ is rather complicated to describe explicitly because there are morphisms from $U_I$ to $U_J$ whenever $|U_{I\cup J}| = F_{I\cup J} \ne \emptyset.$   (See \cite{MWiso,Morb} for a discussion of  such categories in the finite dimensional case.)  Further, the functor  $\psi$ extends to an
 \'etale equivalence $\psi:\Xx_\Uu\to \Xx$.  
 The category $\Xx_\Uu$  per se plays no role in the current paper.  
 \hfill$\er$
 \end{rmk}

 \subsection{The strong bundle \texorpdfstring{$\Ww_\Vv\to \Xx_\Vv$}{W V to X V}}\label{ss:reduce}

We now explain the structure of the pullback bundle $\Ww_\Vv:=\psi^*\Ww$ in the top row of diagram~\eqref{diag:11}   defined in Lemma~\ref{lem:bundle}.
Since we need to establish the properties of strong bundle and sc-Fredholm section functor, this subsection works explicitly with sc-groupoids and polyfolds.
However, many results extend to general \'etale proper categories.

Throughout we consider a sc-Fredholm section functor $f:\Xx\to \Ww$ of a strong bundle $P:\Ww\to\Xx$ over an ep-groupoid $\Xx$ with compact zero set $|f^{-1}(0)|$ as in Definition~\ref{def:poly}.
The bundle is given by structure maps $(P:W\to X, \mu:\bX\leftsub{s}{\times}_P W\to W)$, where $P:W\to X$ is a strong bundle over the M-polyfold of objects $X$, and $\mu$ defines the action of the morphisms of $\Xx$ on the object space $W$ of $\Ww$.  The  following result describes the structure of $\Ww_\Vv$ and the perturbation theory of $f_\Vv$ in detail.

\begin{theorem} \label{thm:reduceFred}
Let $(U_i,G_i)_{i=1,\ldots, N}$ be a finite collection of local uniformizers whose footprints $F_i:=|U_i|$ cover $S:=|f^{-1}(0)|\subset \bigcup_{i=1}^N F_i$, and let $\bigl(F'_I 
\subset |\Xx| \bigr)_{I\subset \{1,\ldots,N\}}$ be a cover reduction as in Lemma~\ref{lem:cov0}, with corresponding \'etale data $(V_J,G_J,\rho_{IJ})$ as in Lemma~\ref{lem:Vdata}. 
Then pullback with the \'etale functor $\psi:\Xx_\Vv\to\Xx$ from Theorem~\ref{thm:reduce} induces a sc-Fredholm section functor $f_\Vv:\Xx_\Vv\to \Ww_\Vv$ of a strong bundle $(P_\Vv:\Ww_\Vv\to \Xx_\Vv, \mu_\Vv)$ over the reduced ep-groupoid $\Xx_\Vv$, whose perturbation theory is equivalent to that of $f$. 
In detail:  

\begin{enumilist}\item 
The strong bundle $P_\Vv:\Ww_\Vv\to \Xx_\Vv$ is the pullback of 
$P:\Ww\to\Xx$ under $\psi$ as in Lemma~\ref{lem:bundle}. 
Using the specific structure from Theorem~\ref{thm:reduce}, this pullback is given by
\begin{align*}
P_\Vv :&\; \psi^*W:= X_\Vv\, \leftsub{\psi}{\times}_{P} W 
 \to X_\Vv,  \quad P_\Vv(I,\ux,w) = (I,\ux)  , \\
\mu_\Vv : & \;   \bX_\Vv\, \leftsub{s}{\times}_{P_\Vv} \psi^*W  \to \psi^*W,  \quad
 \mu_\Vv (I,J,\uy,\ug,w) =  \bigl( t(I,J, \uy,\ug ), \mu( \psi(I,J,\uy,\ug), w) \bigr), 
\end{align*}
where $(I,J,\uy, \ug, w)$ is an abbreviation for $\bigl( (I,J,\uy,\ug) , s(I,J,\uy,\ug),w) \bigr)\in  \bX_\Vv\, \leftsub{s}{\times}_{P_\Vv} \psi^*W$. 
The resulting groupoid $\Ww_\Vv= \psi^*\Ww$ is given by
\begin{align*}
W_\Vv =\psi^*W & =  \textstyle \bigsqcup_{I\subset A} W_I =  \textstyle\bigsqcup_{I\subset A} 
\bigl\{\bigl( I, \ux=(x_{i_1},\ldots) , w \bigr) \,\big|\, \ux \in V_I, w\in W_{x_{i_1}} \bigr\} , \\ 
\bW_\Vv ( W_I, W_J )  
&= \bigl\{ \bigl(I,J, \ux , \ug , w \bigr) \,\big|\, 
(I,J, \ux , \ug ) \in  \Mor_{\Xx_\Vv} ( V_I, V_J )  ,\;  w\in W_{\psi(s(I,J,\ux,\ug))}\bigr\}
\end{align*}
with the structure maps 
\begin{align}\label{eq:strWV}
s(I, J, \ux , \ug , w) & =  \bigl( s(I,J,\ux,\ug) , w \bigr)  , 
\qquad t(I, J, \ux , \ug , w)  = \bigl( t(I,J,\ux,\ug)  , \Hat\rho_{IJ}(\ux,\ug) w \bigr)  , \\ \notag
(I,J,\ux,\ug,w)&\circ  (J,K,\uy,\uh,w') = \bigl( (I,J,\ux,\ug)\circ  (J,K,\uy,\uh), w \bigr)\\ \notag
\id_{(I,\ux,w)} &= (\id_{(I,\ux)}, w), \qquad
(I, J, \ux , \ug , w)^{-1} = \bigl( (I, J, \ux, \ug)^{-1} , \Hat\rho_{IJ}(\ux,\ug) w).
\end{align}
Here $\Hat\rho_{IJ}$ lifts the morphisms in $\Mor_{\Xx_\Vv}(V_I,V_J)$ to a family of linear maps between 
 the fibers of the bundles $W_I\to V_I$, $W_J\to V_J$. It is given by 
\begin{align} \label{eq:hatrho0}
& \Hat\rho_{IJ}(\ux,\ug) w := \mu(  \psi(I,J,\ux,\ug), w)  \; \in \; W_{x_{j_1}}\; \quad \mbox{ for } \;\; I \subset J, \; w\in W_{(\ug^{-1} * \rho_{IJ}(\ux))_{j_1}} , \\ \notag 
&\Hat\rho_{IJ}(\ug^{-1}*\ux, \ug^{-1})  := (\Hat\rho_{JI}(\ux,\ug))^{-1} \qquad\qquad\qquad\qquad\quad  \mbox{ for } \;\; J\subset I, 
\end{align} 
and satisfies the following identity (with $\uv, \uk$ determined as in \eqref{eq:MMGa3})
\begin{equation} \label{eq:hatrho}
\Hat\rho_{JK}(\uy,\uh)  \, \Hat\rho_{IJ}(\ux,\ug) 
= \; \Hat\rho_{IK}\bigl(\rho_{I\vee K}(\uv) , \uk(\uh \ug )|_{I\wedge K}\bigr) .
\end{equation}

\item
The natural pullback functor $\Psi:\Ww_\Vv\to\Ww$ from Lemma~\ref{lem:bundle} restricts to a strong bundle equivalence $\Psi:\Ww_\Vv\to\Ww|_V$ in the sense of \cite[Def.10.4.1]{TheBook}, where 
$V = \pi_\Xx^{-1}\bigl( \bigcup_{I\subset A} F'_I  \bigr) = \pi_\Xx^{-1}( \im|\psi| )\subset X$ as in Theorem~\ref{thm:reduce}.

In the notation  of Theorem~\ref{thm:reduce}, this functor is given by 
$(I,\ux,w)\mapsto w$ on objects and $( I,J,\uy,\ug,w ) \mapsto (  \psi(I,J,\uy,\ug) , w)$ on morphisms. 

\item 
The sc-Fredholm section functor $f_\Vv:=\psi^*f: \Xx_\Vv\to \Ww_\Vv$ is the pullback of $f$ under $ \psi$ as in Lemma~\ref{lem:bundle}.  
Using the specific structure from Theorem~\ref{thm:reduce}, this functor is given by
\begin{align*}
f_\Vv(I,\ux) &= \bigl( I,\ux, f( \psi(I,\ux)) \bigr) = \bigl( I,\ux, f( x_{i_1}) \bigr) \quad\text{for}\; (I,\ux)\in X_\Vv, \\
f_\Vv(I,J,\uy,\ug) &= 
\bigl( I,J,\uy,\ug ,  f( \psi(I,\ug^{-1}*\rho_I(\uy))) \bigr) \qquad\text{for}\; (I,J,\uy,\ug)\in \bX_\Vv . 
\end{align*}
In particular, $\psi$ restricts to a functor between the full subcategories $\Xx_\Vv\supset f_\Vv^{-1}(0)\to f^{-1}(0)\subset \Xx$ and induces a homeomorphism  $|\psi|:|f_\Vv^{-1}(0)|\to |f^{-1}(0)|$ between the compact zero sets.

The inner action of $G: = \prod_{i\in A} G_i$ on $\Xx_\Vv$ naturally lifts to an inner action on $\Ww_\Vv$ as in Lemma~\ref{lem:actW}, and  $f_\Vv$ is $G$-equivariant.  

If, in addition, $f$ is orientable, then so is $f_\Vv$; any choice of orientation of $f$ induces an orientation of $f_\Vv$ via pullback with $\Psi$ as specified in \S\ref{ss:orient}. This definition is such that if $f,f_\Vv$ are transverse then $|\psi|:|f_\Vv^{-1}(0)|\to |f^{-1}(0)|$ is orientation preserving.

\item The perturbation theories of $f$ and $f_\Vv$ are equivalent as follows.
We can find compactness controlling data $(N,\Uu)$ for $f$ as in Definition~\ref{def:control-compact} with $|{\rm cl}_X(\Uu)|\subset |V|$. 
Then $(N_\Vv:=N\circ\Psi \,,\, \Uu_\Vv:= \psi^{-1}(\Uu))$ is compactness controlling data for $f_\Vv$ and we have a pushforward construction for regular perturbations:

For every $(N_\Vv,\Uu_\Vv)$-regular sc$^+$-multisection functor $\Lambda_\Vv:\Ww_\Vv\to\Q^{\ge 0}$ in the sense of Definition~\ref{def:NUregular} the pushforward $\Psi_*\Lambda_\Vv: \Ww|_V\to\Q^{\ge 0}$ extends to an $(N,\Uu)$-regular sc$^+$-multisection functor $\Ti\Lambda: \Ww\to\Q^{\ge 0}$ that is trivial\footnote{A multisection functor $\Lambda:\Ww\to\Q^{\ge 0}$ is called trivial at $x\in X$ if $\Lambda(0_x)=1$ and $\Lambda(w)=0$ for all $w\in W_x\less\{0_x\}$. } 
over $X\less \Uu\supset X\less \pi_\Xx^{-1}(|V|)$.  Moreover, 
$$
| \psi|\;:\;  \bigl|S_{\Lambda_\Vv\circ f_\Vv}=\{\ux \in X_\Vv | \Lambda_\Vv(f_\Vv(\ux))>0\} \bigr|
\;\to\; 
\bigl|S_{\Ti\Lambda\circ f}=\{q\in X\, |\, \Ti\Lambda(f(q))>0\} \bigr| 
$$ 
is a bijection between the perturbed zero sets, and the corresponding branched ep$^+$-subgroupoids 
$\Theta_\Vv:=\Lambda_\Vv\circ f_\Vv:\Xx_\Vv\to\Q^{\ge 0}$ and $\Theta:=\Ti\Lambda\circ f :\Xx\to\Q^{\ge 0}$
are equivalent in the sense that $\Theta_\Vv = \Theta\circ\psi$ is the \cite[Def.11.3.1]{TheBook} pullback under the equivalence $\psi:\Xx_\Vv\to\Xx|_V$. 
If $f$ is oriented, and $f_\Vv$ is equipped with the orientation from (iii), then the identification $| \psi|:  |S_{\Lambda_\Vv\circ f_\Vv}| \to |S_{\Ti\Lambda\circ f}|$ is orientation preserving in the sense that the \cite[Thm.11.3.5]{TheBook} pullback orientation of the branched ep$^+$-subgroupoid $(\Ti\Lambda\circ f)\circ\psi$
coincides with the orientation of $\Lambda_\Vv\circ f_\Vv$ induced from the orientation of $f_\Vv$ via the solution set convention in \S\ref{ss:orient}.

In the special case of a polyfold model without boundary in the sense of Definition~\ref{def:PolyfoldModel} for a moduli space $\oMm\simeq|f^{-1}(0)|$, that is when $\partial\Xx=\emptyset$ and $f:\Xx\to\Ww$ is oriented, the reduced sc-Fredholm description $\oMm\simeq|f_\Vv^{-1}(0)|$ from (iii) is a polyfold model that induces the same fundamental class $[\oMm]_{f_\Vv}= [\oMm]_f\in \check H_d(\oMm;\Q)$, defined in Theorem~\ref{thm:polyVFC}.

\end{enumilist}
\end{theorem}

\begin{proof}
Our constructions follow the theory of compatible perturbations in \cite{Wolfgang thesis}, which simplifies in our setting as $ \psi:\Xx_\Vv \to \Xx$ is a local sc-diffeomorphism and induces a topological embedding $| \psi|:|\Xx_\Vv| \to |\Xx|$, thus meets the ``topological pullback condition'' in \cite[Def.4.3]{Wolfgang thesis}. 
\MS

\NI {\bf Proof of (i):}\
Lemma~\ref{lem:bundle} proves the first part of (i):  $(P_\Vv, \mu_\Vv)$ is a strong bundle in the sense of \cite[Def.8.3.1]{TheBook}. The structure maps of $\Ww_\Vv$ result from writing out the construction, 
using the specific form of $\psi$ in Theorem~\ref{thm:reduce}~(ii).   To prove \eqref{eq:hatrho},
we use the fact that the composite $(m,w)\circ(m',\mu(m,w))= (m\circ m', w)$ has target 
\begin{equation}\label{eq:muprop}
\mu(m',\mu(m,w) ) = t(m',\mu(m,w)) = t(m\circ m', w) = \mu(m\circ m', w).
\end{equation}
Since by definition $\Hat\rho_{IJ}(\ux,\ug) w = \mu(\psi(I,J,\ux,\ug), w)$, we have
\begin{align*}
\Hat\rho_{JK}(\uy,\uh) \Hat\rho_{IJ}(\ux,\ug)  
&\;=\; \mu\bigl( \psi(J,K,\uy,\uh),\, \mu (  \psi(I,J,\ux,\ug) , w ) \bigr)  \\
&\;=\; \mu (  \psi(I,J,\ux,\ug) \circ   \psi(J,K,\uy,\uh), w)
\;=\;  
\Hat\rho_{IK}\bigl(\rho_{I\vee K}(\uv) , \uk(\uh \ug )|_{I\wedge K}\bigr),
\end{align*}
where the last expression is explained in \eqref{eq:MMGa3}.
\MS

\NI {\bf Proof of (ii):}
The existence of the strong bundle map $\Psi:\Ww_\Vv\to\Ww$  in (ii)  is established in Lemma~\ref{lem:bundle}, and its formula can be deduced from \eqref{eq:bundlemap}.   Moreover $\Psi:\Ww_\Vv\to \Ww|_V$ is a strong bundle equivalence by Lemma~\ref{lem:bundle}~(v).
\MS

\NI {\bf Proof of (iii):}
The proof that $f_\Vv$ is a sc-Fredholm section functor is given in  Lemma~\ref{lem:pullback}, while Lemma~\ref{lem:bundle}~(v) shows that $|\psi|: |f_\Vv^{-1}(0)| \to f^{-1}(0)|$ is a homeomorphism.  
Next, by Lemma~\ref{lem:actW}, the inner action of $G$ on $\Xx_\Vv$  lifts to $\Ww_\Vv$ and $f_\Vv$ is $G$-equivariant.  
We will denote the induced action on $W_\Vv$ by 
\begin{align}\label{eq:gpactW00}
(I,\ux,w )\mapsto \ug*(I,\ux,w ): & = \bigl(I, \ug|_I*\ux,\, \mu(\psi(I,I, \ug|_I*\ux,\ug|_I), w)\bigr)\\ \notag
&  = \bigl(I, \ug|_I*\ux, \, g_{i_1}*w\bigr)
\end{align}
where the second equality holds because $\psi(I,I, \ug|_I*\ux,\ug|_I) =  g_{i_1}$ by Theorem~\ref{thm:reduce}~(ii).
If, in addition, $f$ is orientable, then so is $f_\Vv$, and any choice of orientation of $f$ induces an orientation of $f_\Vv$ via pullback with $\Psi$, since this is a case of the equivalence convention in \S\ref{ss:orient}.   
It also is a special case of pushforward \cite[Thm.12.5.7]{TheBook} with the diagram $\Ww|_V \overset{\Psi}{\leftarrow}  \Ww_\Vv \overset{\rm Id}{\rightarrow} \Ww_\Vv$ which is a generalized strong bundle isomorphism in the sense of \cite[Def.10.5.2]{TheBook}. 
Moreover, Remark~\ref{rmk:equiv-orient} shows that if $f,f_\Vv$ are transverse then $|\psi|:|f_\Vv^{-1}(0)|\to |f^{-1}(0)|$ is orientation preserving. 
\MS

\NI {\bf Proof of (iv):}
The fact that  compactness controlling data $(N,\Uu)$ for $f$ with $|{\rm cl}_X(\Uu)|\subset |V|$ exists and pulls back to compactness controlling data $(N_\Vv, \Uu_\Vv)$ for $f_\Vv$ is proved in Lemma~\ref{lem:pullback}.  
Next note that when  $\Lambda_\Vv:\Ww_\Vv \to\Q^{\ge 0}$ is a sc$^+$-multisection functor 
 with support in $\Uu_\Vv$ then by Lemma~\ref{lem:pushLa}~(i) it may be pushed forward  to a multisection $\Lambda:=\Psi_*\Lambda_\Vv:\Ww|_{
 \TV}\to\Q^{\ge 0}$ that is defined over the saturation $\TV: = \pi_\Xx^{-1}(|V|)$ of $V$ in $\Xx$ and is sc$^+$-smooth. by construction; also see
 \cite[Lemma~11.5.2]{TheBook}. Since it has support in the saturated subset $\Uu\subset \TV$,
%
  it may be smoothly extended by the trivial multisection of $\Ww|_{X\less \Uu}$  (that equals $1$ along the zero section and is $0$ elsewhere)  to a sc$^+$-multisection of $\Ww\to \Xx$.

Next, suppose that $\Lambda_\Vv$ is $(N_\Vv,\Uu_\Vv)$-regular in the sense of Definition~\ref{def:NUregular}, that is
\begin{itemize}
\item[(1)] 
$N_\Vv(\Lambda_\Vv) = \sup\bigl\{ N_\Vv(\uw) \,\big|\, \uw\in W_\Vv[1], \Lambda_\Vv(\uw)>0 \bigr\}  \;<\; 1$; 

\item[(2)] 
$\text{\rm dom-supp}\,\Lambda_\Vv ={\rm cl}_{X_\Vv}\bigl\{ y \in X_\Vv \,\big|\, \exists\, \uw\in P_\Vv^{-1}(y)\less \{0\} : \Lambda_\Vv(\uw) >0  \bigr\}  \;\subset\; \Uu_\Vv =\psi^{-1}(\Uu)$;

\item[(3)]  
All local sections representing $\Lambda_\Vv$ are transverse to $f_\Vv$.
\end{itemize}

\NI
We now show that $\Lambda=\Psi_*\Lambda_\Vv$ is an $(N, \Uu)$-regular perturbation of $f|_\TV: \Xx|_\TV\to \Ww|_\TV$ (where as above $\TV$ denotes the saturation of $V$ in $\Xx$) because it  satisfies the corresponding versions of these conditions.\MS

\NI
{\bf Proof of (iv)-(1):}  We have
\begin{align*}
1> N_\Vv(\Lambda_\Vv) &=\sup\bigl\{ N_\Vv(\uw) \,\big|\, \uw\in W_\Vv[1],  \Lambda_\Vv(\uw)>0 \bigr\} \\
&= \sup\bigl\{ N(\Psi(\uw)) \,\big|\, \uw\in W_\Vv[1],  \Lambda(\Psi(\uw))>0 \bigr\} \\
& = \sup\bigl\{ N(w) \,\big|\, w \in\Psi (W_\Vv[1] ), \Lambda(w)>0 \bigr\} \\
& = \sup\bigl\{ N(w') \,\big|\, w' \in W[1] , \Lambda(w')>0 \bigr\} \;=\; N(\Lambda). 
\end{align*}
Here the last identity uses the fact that $N$ is invariant under morphisms and $|\Ww|_V|=|\Psi| (|\Ww_\Vv|)$, so every $w'\in W$  with $|P(w)|\in |V|$ is related by a morphism to some $w=\Psi(\uw)$. Moreover, for $w'\in W[1]$ the corresponding object $\uw$ in $N(w')=N(\Psi(\uw))$ lies in $W_\Vv[1]$ since both the morphisms of $\Ww$ and the strong bundle map $\Psi$ preserves the sc-filtrations. 
\MS

\NI
{\bf Proof of (iv)-(2):}   Note that the pullback identity $\Lambda_\Vv(\uo,w)=\Lambda(\Psi(\uo,w))=\Lambda(w)$ identifies
\begin{align*}
& \bigl\{ \uo \in X_\Vv \,\big|\, \exists (\uo,w)\in P_\Vv^{-1}(\uo)\less \{0\} : \Lambda_\Vv(\uo,w) >0  \bigr\}  \\
&\quad =  \bigl\{ \uo \in X_\Vv \,\big|\, \exists\, w\in P^{-1}(\psi(\uo))\less \{0\} : \Lambda(w) >0  \bigr\} \\
&\quad = \psi^{-1} \bigl( \bigl\{ o \in X \,\big|\, \exists\, w\in P^{-1}(o)\less \{0\} : \Lambda(w) >0  \bigr\} \bigr). 
\end{align*}
Now, $\pi_\Xx:\TV\to|V|\subset |\Xx|$ is continuous and $|\psi|:|\Xx_\Vv|\to |V|\subset|\Xx|$ is a homeomorphism, 
and the operations of closure and realization commute for saturated subsets of \'etale proper groupoids  by Lemma~\ref{lem:etale1}.  Further  the set 
$\bigl\{ o \in X \,\big|\,  \exists\, w\in P^{-1}(o)\less \{0\} : \Lambda(w) >0   \bigr\} $  is saturated because $\La$ is a functor.  Therefore
\begin{align*}
\bigl| \text{\rm dom-supp}\,\Lambda \bigr| 
&= 
\bigl|  {\rm cl}_{V} \bigl\{ o \in X \,\big|\,  \exists\, w\in P^{-1}(o)\less \{0\} : \Lambda(w) >0   \bigr\}  \bigr|\\
& = 
  {\rm cl}_{|V|}   \bigl|\bigl\{ o \in X \,\big|\,  \exists\, w\in P^{-1}(o)\less \{0\} : \Lambda(w) >0   \bigr\}  \bigr|\\
&=
{\rm cl}_{|V|} \bigl( |\psi| \bigl( \bigl|   \bigl\{ \uo \in X_\Vv \,\big|\, \exists (\uo,w)\in P_\Vv^{-1}(\uo)\less \{0\} : \Lambda_\Vv(\uo,w) >0  \bigr\}     \bigr|   \bigr) \bigr) \\
&=
 |\psi| \bigl(  {\rm cl}_{|\Xx_\Vv|} \bigl(\bigl|   \bigl\{ \uo \in X_\Vv \,\big|\, \exists (\uo,w)\in P_\Vv^{-1}(\uo)\less \{0\} : \Lambda_\Vv(\uo,w) >0  \bigr\}     \bigr|   \bigr) \bigr) \\
 &=
 |\psi| \bigl(  \bigl| {\rm cl}_{X_\Vv}   \bigl\{ \uo \in X_\Vv \,\big|\, \exists (\uo,w)\in P_\Vv^{-1}(\uo)\less \{0\} : \Lambda_\Vv(\uo,w) >0  \bigr\}     \bigr|   \bigr)   \\
&= |\psi| \bigl(  \bigl| \text{\rm dom-supp}\,\Lambda_\Vv \bigr| \bigr)  \subset |\psi|(|\Uu_\Vv|) = |\Uu|\subset |\Uu|. 
\end{align*}
Since $\Uu$ is a saturated subset of $\pi_\Xx^{-1}(|V|)$ and the operations of  closure and realization commute for saturated subsets of \'etale proper groupoids by Lemma~\ref{lem:etale1}, this implies $\text{\rm dom-supp}\,\Lambda\subset\Uu$ as required. \MS

\NI 
{\bf Proof of (iv)-(3):} 
This is formally stated as ``$\rT_{(f|_V,\Lambda)}$ is surjective on $\supp(\Lambda\circ f|_V)$''. To check this, consider any point in the perturbed solution groupoid, $o\in \TV\subset X$ with $(\Lambda\circ f)(o) >0$. 
Since $\im|\psi|=|V|$ we can find $\uo\in X_\Vv$ with $|\psi(\uo)|=|o|$ and, by construction of the pullback section $f_\Vv=\psi^*f$ and multisection $\Lambda_\Vv=\Lambda\circ\Psi$, we have
$$
(\Lambda_\Vv\circ f_\Vv)(\uo) = (\Lambda\circ\Psi)(\uo, f(\psi(\uo)) )  = \Lambda ( f(\psi(\uo))) > 0. 
$$
Now $|\psi(\uo)|=|o|$ implies the existence of a morphism $\psi(\uo)\to o$ in $\Xx|_V$, functoriality of $f$ lifts this to a morphism $f(\psi(\uo))\to f(o)$ in $\Ww|_V$, and functoriality of $\Lambda$ then implies $\Lambda ( f(\psi(\uo))) = \Lambda ( f(o))  > 0$. 
That means $\uo$ lies in the perturbed solution set $\supp(\Lambda_\Vv\circ f_\Vv)$, so if $\Lambda_\Vv$ is locally represented by sections $(\s_i: N_\uo \to W_\Vv|_{N_\uo})_{i\in\Ii}$, then by assumption ${\rm D}(f_\Vv-\s_i)(\uo): \rT_\uo X_\Vv \to P_\Vv^{-1}(\uo)$ is surjective for each $i\in\Ii$ with $\s_i(\uo)=f_\Vv(\uo)$. 
Now the corresponding pushforward sections $\s'_i(x)=\Psi\bigl(\s_i(\psi_{\uo}^{-1}(x))\bigr)$ near $\psi(\uo)$ have linearizations $$
{\rm D}(f-\s'_i)(\psi(\uo)) = \Psi\circ ({\rm D}(f_\Vv-\s_i)(\uo) ) \circ \rd\psi(\uo)^{-1}.
$$ Since  
$\Psi: P_\Vv^{-1}(\uo) \to P^{-1}(\psi(\uo))$ and $\rd\psi(\uo)$ are isomorphisms, these inherit surjectivity from ${\rm D}(f_\Vv-\s_i)$. 

To check that the surjectivity of the linearized sections at perturbed solutions also transfers to the original point $o$, we need to consider the sections $$
\Ti \s_i: N_o \to W|_{N_o}, x\mapsto\mu\bigl( t_m^{-1}(x) , \s'_i ( s_\Xx( t_m^{-1}(x)) ) \bigr)
$$
 obtained by transfer with morphisms near a choice of $m \in s^{-1}(\psi(\uo))\cap t^{-1}(o)$. To compute their linearization, note the identity $f(x) = \mu\bigl( t_m^{-1}(x) , f ( s_\Xx( t_m^{-1}(x)) ) \bigr)$, which is an expression of the functoriality of $f$ as in Definition~\ref{def:bundlemaps}. 
Since $\mu$ is linear in the second factor and we consider only linearizations of sections with $$
(f-\s'_i) ( s_\Xx( t_m^{-1}(o)) )=  (f-\s'_i) ( s_\Xx( m ) ) =  (f-\s'_i)(\psi(\uo)) = 0,
$$
 the resulting linearization of 
$$
(f-\Ti\s_i)(x) = \mu\bigl( t_m^{-1}(x) , (f-\s'_i) ( s_\Xx( t_m^{-1}(x)) ) \bigr)
$$ at $x=o$ is 
$$
{\rm D}(f-\Ti \s_i)(o) = \mu( m ,  \cdot ) \circ {\rm D}(f- \s'_i)(\psi(\uo)) \circ \rd s_\Xx (m) \circ \rd t_\Xx (m)^{-1}.
$$ 
Here $\rd s_\Xx$ and $\rd t_\Xx$ are isomorphisms as source and target maps are local sc-diffeomorphisms, and $\mu(m, \cdot ): W_{s(m)}\to W_{t(m)}$ is invertible with inverse given by $\mu(m^{-1}, \cdot )$, 
as can be verified from the algebraic properties of $\mu$ in Definition~\ref{def:bundle}~(iii).  
\MS

Thus we have shown that $\Lambda=\Psi_*\Lambda_\Vv$ is an $(N, \Uu)$-regular perturbation of $f|_\TV: \Xx|_\TV\to \Ww|_\TV$ (where as above $\TV: = \pi_\Xx^{-1}(|V|)$ is the saturation of $V$).   Since $\La$ has support in $\Uu\subset \TV$ we may extend it smoothly to a multisection $\Ti{\La}$ of $\Ww\to \Xx$  by defining $\Ti\La(0_x) = 1, x\in X\less \Uu$, and $\Ti\La(w) = 0$ for $w\in W|_{X\less \Uu}, w\ne 0$.  Since $\supp(\Ti\La) = \supp(\La)$ is unchanged, $\Ti\La$ is an $(N, \Uu)$-regular perturbation of $f: \Xx\to \Ww$.

If $\Xx$ has boundary (and corners), then good or general position of the kernels of linearizations at solutions on the boundary (see \cite[Def.3.6.6, 5.3.9]{TheBook}) transfers from $\Lambda_\Vv$ to $\Lambda$ because the kernels are identified via $\rd \psi (\uo)$, which also identifies the reduced tangent spaces of $\Xx_\Vv$ and $\Xx$. 

To check that $|\psi|$ identifies the perturbed zero sets, note that the construction of $f_\Vv$ by pullback and $\Ti\La$ by pushforward yields
$\Lambda_{\Vv}\circ f_\Vv \;=\; \Psi^*\Ti\Lambda\circ \psi^*f \;=\; \Ti\Lambda\circ f \circ \psi$. 
Here $\psi:\Xx_\Vv\to\Xx|_V$ is the equivalence of ep-groupoids from Theorem~\ref{thm:reduce}, so in particular induces a homeomorphism $|\psi| :|\Xx_\Vv|\to |V|\subset |\Xx|$. This homeomorphism restricts to a bijection 
$|\{\ux \in X_\Vv \,|\, \Lambda_\Vv(f_\Vv(\ux))>0\} | \simeq  |\{q\in X\, |\, \Ti\Lambda(f(q))>0\} |$
since $\Lambda_\Vv(f_\Vv(\ux))>0 \Leftrightarrow \Ti\Lambda(f(\psi(\ux)))>0$ and $ |\{q\in X\, |\, \Ti\Lambda(f(q))>0\} |\subset |V|=\im|\psi|$, where the latter inclusion is guaranteed by the facts that the unperturbed zero set $|f^{-1}(0)|$ is contained in $|V|$ and the perturbation $\Ti\Lambda: \Ww\to\Q^{\ge 0}$ is trivial over $X\less \TV$.

Next, if $f$ is oriented then this induces an orientation of the branched ep$^+$-subgroupoid $\Lambda_{\Vv}\circ f_\Vv: \Xx_\Vv\to\Q^{\ge 0}$ by \cite[Thm.15.4.3]{TheBook} -- using the solution set convention in \S\ref{ss:orient}. 
If $f_\Vv$ is equipped with any orientation, then this analogously orients the branched ep$^+$-subgroupoid $\Ti\Lambda\circ f: \Xx\to\Q^{\ge 0}$. 
Now the claim is that pulling back the orientation of $f$ to an orientation of $f_\Vv$ as in (iii) and then orienting the perturbed solution set $\Theta_\Vv:=\Lambda_{\Vv}\circ f_\Vv$ by the solution set convention for each local branch has the same result as orienting the perturbed solution set $\Theta:=\Ti\Lambda\circ f$ with the solution set convention and then pulling back this orientation to $\Theta_\Vv=\Theta\circ\psi$ as in \cite[Thm.11.3.5]{TheBook}. 
The latter pullback construction uses the fact that $\psi$ restricts to diffeomorphisms on all sufficiently small (finite dimensional) local branches. So the identification of these two orientation constructions can be checked at the level of local branches, where this is done in Remark~\ref{rmk:equiv-orient}. 

Finally, we consider the special case $\partial\Xx=\emptyset$ with $f$ oriented and a homeomorphism $\psi_\oMm: |f^{-1}(0)|\to \oMm$  to a compact Hausdorff space. 
Then the composition $\psi_\oMm\circ|\psi|: |f_\Vv^{-1}(0)|\to \oMm$ is also a homeomorphism by (iii), and we can prove equality of the induced fundamental classes $[\oMm]_{f_\Vv}= [\oMm]_f\in \check H_d(\oMm;\Q)$ defined in Theorem~\ref{thm:polyVFC} as follows: 
Given compactness controlling data $(N,\Uu)$, a nested sequence of open subsets $B_k\subset B_{k-1} \subset |U|\subset|\Xx|$ such that $\bigcap_{k\in\N} B_k=|f^{-1}(0)|$, we defined the first fundamental class
$[\oMm]_f := {\psi_\oMm}* \bigl( \underset{\leftarrow}\lim \, 
(\psi_k)_*[ \Lambda_k\circ f  ] \bigr)$, where $\Lambda_k$ is any choice of $(N,\Uu\cap\pi_\Xx^{-1}(B_k))$-regular perturbations of $f$ and $\psi_k: | \Lambda_k\circ f | \to B_k\subset|\Xx|$ are the corresponding continuous maps from the realization of the branched ep$^+$-subgroupoids $\Lambda_k\circ f$, defined as 
Theorem~\ref{thm:polyVFC}.

Since Theorem~\ref{thm:polyVFC} shows that the result is independent of the choices, we can choose $\Uu\subset X$ with  ${\rm cl}_{X}(\Uu)\subset \TV: = \pi_\Xx^{-1}\bigl(\bigcup_I F_I'\bigr)$ so that, as shown above, the pulled back data $(N_\Vv,\Uu_\Vv)$ controls compactness of $f_\Vv$. Since $|\Uu|\subset |\TV|=|V|$ we can also pull back the nested sequence of open subsets $B_k$  with the homeomorphism $|\psi|:|\Xx_\Vv|\to |\Xx|$ to obtain
 $B_{\Vv,k}:=|\psi|^{-1}(B_k)|\subset B_{\Vv,k-1} \subset |\Uu_\Vv|\subset|\Xx_\Vv|$
with intersection  $\bigcap_{k\in\N} B_{\Vv,k}=|f_\Vv^{-1}(0)|$. 
This data induces the  fundamental class
$$
[\oMm]_{f_\Vv} := ({\psi_\oMm} \circ |\psi|)_* \bigl( \underset{\leftarrow}\lim \, 
(\phi_{\Vv,k})_*[ \Lambda_{\Vv,k}\circ f_\Vv  ] \bigr),
$$
 where $\Lambda_{\Vv,k}$ are $(N_\Vv,\Uu_\Vv\cap\pi_{\Xx_\Vv}^{-1}(B_{\Vv,k}))$-regular perturbations of $f_\Vv$ and $\phi_{\Vv,k}: |\Lambda_{\Vv,k}\circ f_\Vv  |  \to B_{\Vv,k}\subset|\Xx_\Vv|$ are the corresponding continuous maps of the  realizations. 
 
Now we can compare the result with the  fundamental class $[\oMm]_{f}$ by choosing $\Lambda_k$ as the pushforward $\Psi_* \Lambda_{\Vv,k}$, trivially extended to $X\less \TV$ as above.
Then we saw above that $|\psi|$ restricts to homeomorphisms $\bigl(|S_{\Lambda_\Vv\circ f_\Vv}|, |\Lambda_\Vv\circ f_\Vv|\bigr) \simeq \bigl(|S_{\Ti\Lambda\circ f}|, |\Ti\Lambda\circ f|\bigr)$ and $B_{\Vv,k}\simeq B_k$, and hence identifies the corresponding
fundamental classes
 giving\footnote{
 Here, as in Remark~\ref{rmk:cech} and Theorem~\ref{thm:polyVFC}, the map $\phi_{\Vv,k}$ has image in the neighbourhood $B_{\Vv,k}\subset |\Xx_\Vv|$ while $\phi_k$ maps to $B_k\subset |\Xx|$.}
  $$
 |\psi|_* (\phi_{\Vv,k})_* [ \Lambda_{\Vv,k}\circ f_\Vv  ]  = (\phi_k)_*[ \Lambda_k\circ f  ] \;\;\in \;\; \check{H}_d(B_{k}),
 $$
With that we obtain as claimed
\begin{align*}
[\oMm]_{f_\Vv} &= (\psi_\oMm)_*  |\psi|_* \bigl( \underset{\leftarrow}\lim \, (\phi_{\Vv,k})_*[ \Lambda_{\Vv,k}\circ f_\Vv  ] \bigr) 
= (\psi_\oMm)_*  \bigl( \underset{\leftarrow}\lim  (\phi_k)_*[ \Lambda_k\circ f  ]\bigr) \;=\;[\oMm]_f . 
\end{align*}
This completes the proof.
\end{proof}

\begin{rmk}\label{rmk:XUW}\rm
Given a collection of $U$-data as in Lemma~\ref{lem:Udata}, one can, as mentioned in Remark~\ref{rmk:XU}, form a category $\Xx_\Uu$ analogous to $\Xx_\Vv$ together with an \'etale equivalence $\psi: \Xx_\Uu\to \Xx|_U$ where $U = \pi^{-1}\bigl(\bigcup_i  U_i\bigr)$. 
Similarly, there is a pullback bundle   $\Pp_\Uu: \Ww_\Uu\to \Xx_\Uu$ with section $f_\Uu$
that satisfies (appropriately modified versions of) the properties 
 in Theorem~\ref{thm:reduceFred}.  However, there are too many morphisms in $\Xx_\Uu$ for 
 the perturbation theory of multisections of $\Ww_\Uu\to \Xx_\Uu$ to be any easier to understand than that of the original bundle $\Ww\to \Xx$. 
 On the other hand the perturbation theory of $\Ww_\Vv\to \Xx_\Vv$ is more transparent. 
 Indeed, as explained in 
 as explained in Theorem~\ref{thm:reduce}~(iii),(iv), (and see also \S\ref{sec:genconstr}), $ \Xx_\Vv$ can be reconstructed from the nonsingular groupoid 
 $\Xx_\Vv^{\less G}$ by adding morphisms given by the $G$ action.  Moreover, $\Xx_\Vv^{\less G}$ is the completion of the poset $\Qq$, and we will see in Theorem~\ref{thm:stabilize}  that, as in Proposition~\ref{prop:globsym},  we can push forward  constant sections of the trivial bundle $\Qq \times E\to \Qq$  and then average over the $G$-action  to obtain  transverse multisections of $\Ww_\Vv\to \Xx_\Vv$.  
\hfill$\er$
\end{rmk}

\subsection{Fredholm stabilization over reduced ep-groupoids} \label{ss:stabilize}

Let $f:\Xx\to \Ww$ be a sc-Fredholm section functor of a strong bundle $(P: W\to X, \mu:\bX \leftsub{s}{\times}_P W\to W)$ over an ep-groupoid $\Xx=(X,\bX)$ with compact solution set $|f^{-1}(0)|\subset|\Xx|$.
Then Theorems~\ref{thm:reduce} and \ref{thm:reduceFred} construct the right hand square in the diagram \eqref{diag:11} and establish parts (i)--(iv) of the global stabilization  Theorem~\ref{thm:globstab}.  
To complete its proof, it remains to construct the stablization functor $\tau: \Xx_\Vv^{\less G}\times E \to \Ww_\Vv$ and establish the properties in part (v).  
Proposition~\ref{prop:MMGaE} showed that to construct such $\tau$ we  only need to construct a suitable functor ${\tau}$ from the trivial bundle $\Qq\times E\to \Qq$ to $\Ww_\Vv\to \Xx_\Vv$.  Before embarking on this construction, we first address the fact that
in order for the corresponding pushed forward multisections $\La_{\Vv,e}$ in Proposition~\ref{prop:globsym} to be both sc$^+$ and compactness controlled, we need to put extra conditions on the initial set of  local stabilizations used in Lemma~\ref{lem:Udata} as the basis for the construction of $\Xx_\Vv$

\begin{lemma} \label{lem:localtau}
Given compactness controlling data $(N,\Uu)$ for $f$ as in Definition~\ref{def:control-compact} there is a finite collection $(U_i,G_i)_{i\in A}$ of local uniformizers of $\Xx$ and $(N,\Uu)$-regular Fredholm stabilizations $\tau_i : X \times E_i \to W[1]$ as follows:
\begin{itemlist}
\item 
The domains of the local uniformizers are open subsets $U_i\subset X$ which cover the zero set $|f^{-1}(0)|\subset \bigcup_{i\in A} |U_i|$. The local morphisms over each $U_i$ are identified with $G_i$-actions on $U_i$ via bijections $\Ga_i: G_i\times U_i\to \Mor_\Xx(U_i,U_i)$ as in Definition~\ref{def:preunif}.  
\item
$E_i$ is a finite rank vector space with a linear action $G_i\times E_i \to E_i, (g,e)\mapsto g*e$. 
\item 
$\tau_i: X \times E_i \to W[1]$ is sc-smooth (thus sc$^+$ to $W$), linear in $E_i$, and a bundle map; in particular   
 $P(\tau_i(x,e))= x$. 
Moreover, $\tau_i|_{U_i\times E_i}$ is $G_i$-equivariant in the sense that $\tau_i(g*x,g *e) = \mu(\Ga_i(g,x) , \tau_i(x,e) )$ for all $g\in G_i, x\in U_i$. 

\item
$\tau_i$ is an $(N,\Uu)$-regular stabilization of the Fredholm section $f|_{U_i} : \Xx|_{U_i}\to \Ww|_{U_i}$ as follows: 
\begin{itemize}
\item[(1)] Each $E_i$ is equipped with a $G_i$-invariant norm $\|\cdot\|_{E_i}$ such that $N(\tau_i(x,e))\leq \| e \|_{E_i}$;   
\item[(2)]
$\supp\tau_i(\cdot,e)\subset\Uu$ for all $e\in E_i$; 
\item[(3)] for any unperturbed solution $y\in f^{-1}(0)\cap U_i$ the linearization of $f - \tau_i : X \times E_i \to W$,  $(y,e) \mapsto f(y) - \tau_i(y,e)$ at $(y,0)$ is surjective. 
If the solution lies in the boundary $y\in\partial X$, we can moreover achieve general (and hence good) position of $f-\tau_i$ at $y$ as in \cite[Def.3.6.6, 5.3.9]{TheBook}.
\end{itemize}
\end{itemlist}
\end{lemma}

\begin{proof}
As in the proof of Lemma~\ref{lem:Ui} we can find a local uniformizer $(U_x,G_x,\Ga_x)$ around each $x\in f^{-1}(0)$, and since $\Uu\subset X$ is a saturated neighbourhood of $f^{-1}(0)$, we can shrink their domains to $U_x\subset\Uu$. 
In addition, the nonlinear sc-Fredholm property \cite[Def.3.1.16]{TheBook} of $f$ includes a regularizing property due to which every $x\in f^{-1}(0)$ lies in the ``smooth part'' $X_\infty\subset X$ described in \S\ref{ss:orient}. Moreover, \cite[Thms.3.1.22, 3.1.26]{TheBook} guarantee that each linearization $\rD f(x):\rT_x X \to W_x$ at a solution is a sc-Fredholm operator in the sense of \cite[Def.1.1.9]{TheBook}.  
In fact the restriction to the reduced tangent space $\rD^R f(x):=\rD f(x)|_{\rT_x^R X}:\rT_x^R X \to W_x$ is sc-Fredholm since $\rT_x^R X\subset \rT_x X$ is a sc-subspace -- identical when $x$ lies in the interior of $X$, or of finite codimension given by the degeneracy index $d(x)\geq 1$ when $x\in\partial X$; see \cite[Def.2.4.7]{TheBook}.
The sc-Fredholm operator $\rD^R f(x)$ in particular has a finite dimensional sc-complement $C_x\subset W_x$ of the image, $W_x=C_x\oplus\im\rD^R f(x)$. 
By \cite[Prop.1.1.7]{TheBook}, the notion of a sc-complement guarantees that it lies in the ``smooth part'' of the fiber, that is in a local trivialization $\Phi: W|_{\text{nbhd}(x)} \to \bigcup_{u\in O} \im\Ga(u) \subset O\times F$ as in  \S\ref{ss:scale} we have $\Phi(C_x)\subset F_\infty=\bigcap_{m\in\N_0} F_m$ in the dense subset of smooth points in the sc-Banach space $F=(F_m)_{m\in\N_0}$.  
Now let $w_1,\ldots,w_{d_x}\in C_x$ be a basis, scaled such that $N(w_l)\leq \frac 12$ for each $l=1,\ldots,d_x$. 
Then \cite[Lemma~5.3.3]{TheBook} provides sc$^+$-sections $\s_{x,1},\ldots,\s_{x,d_x} : X \to W[1]$ with $\s_{x,l}(x)=w_l$, $N(\s_{x,l}(y))\leq 1$ for all $y\in X$ and $\supp \s_{x,l} \subset U_x$. 

Now $E_x:= (\R^{d_x})^{G_x} \cong \bigoplus_{h\in G_x} \{h\}\times\R^{d_x}$ is a vector space of dimension $\# G_x \cdot d_x$ with linear $G_x$-action defined on generators by $g*(h, \vec c):=(gh, \vec c)$. 
Then we construct $\tau_x: X \times E_x \to W[1]$ on generators $(h,\vec c)\in E_x$ by 
$$
\tau_x\bigl(y,  \bigl( h , \vec c=(c_1,\ldots,c_{d_x})\bigr) \bigr) :=
\textstyle \sum_{l=1}^{d_x}  c_l  \cdot \mu\bigl(\Ga_x(h , h^{-1}* y) , \s_{x,l}(h^{-1}*y) \bigr) 
$$
and extend linearly to $E_x=\bigl\{ e=\sum_{h\in G_x} (h , \vec c_h) \,|\, \vec c_h \in \R^{d_x} \bigr\}$. 
This map $\tau_x: X \times E_x \to W[1]$ is linear in $E_x$ and satisfies $P(\tau_x(y,e))=t( \Ga_x(h , h^{-1}* y) ) = h*h^{-1}*y= y$ by construction. 
It is sc$^\infty$ with values in $W[1]$ (thus sc$^+$ with values in $W$) since the sections $\s_{x,l}: X\to W$ are sc$^+$ (in particular sc$^\infty$ with values in $W[1]$) and $\mu:\bX \leftsub{s}{\times}_P W\to W$ is a strong bundle map (in particular restricts to a sc$^\infty$-map $\bX \leftsub{s}{\times}_P W[1]\to W[1]$). To check equivariance we use the algebraic properties of $\Ga$ in Definition~\ref{def:preunif} and  $\mu$ in Definition~\ref{def:bundle} on generators:  
\begin{align*}
\tau_x(g*y,  (g h , \vec c) ) 
&= \textstyle \sum_{l=1}^{d_x}  c_l  \cdot \mu\bigl(\Ga_x(  g h , ( g h )^{-1}* g*y) , \s_{x,l}(( g h)^{-1}*g*y) \bigr)  \\
&= \textstyle \sum_{l=1}^{d_x}  c_l  \cdot \mu\bigl(\Ga_x(  g h , h ^{-1}* g^{-1} * g *y) , \s_{x,l}(h^{-1}*g^{-1}*g*y) \bigr)  \\
&= \textstyle \sum_{l=1}^{d_x}  c_l  \cdot \mu\bigl(\Ga_x(  g h , h ^{-1} *y) , \s_{x,l}(h^{-1}*y) \bigr)  \\
&=  \textstyle \sum_{l=1}^{d_x}  c_l  \cdot \mu\bigl(  \Ga_x(h , h^{-1}* y) \circ \Ga_x(g, y) , \s_{x,l}(h^{-1}*y) \bigr)    \\
&=  \textstyle \sum_{l=1}^{d_x}  c_l  \cdot \mu\bigl(\Ga_x(g, y) , \mu\bigl(\Ga_x(h , h^{-1}* y) , \s_{x,l}(h^{-1}*y) \bigr)  \bigr)   \\
&= \mu\bigl(\Ga_x(g, y) , \textstyle \sum_{l=1}^{d_x}  c_l  \cdot \mu\bigl(\Ga_x(h , h^{-1}* y) , \s_{x,l}(h^{-1}*y) \bigr)  \bigr)   \\
&= \mu\bigl(\Ga_x(g, y) , \tau_x(y, (h ,\vec c)) \bigr) . 
\end{align*}
Towards $(N,\Uu)$-regularity note the following properties: 
\begin{itemize}
\item[(1)]
$N(\tau_x(y,e =\sum_{h\in G_x} (h,\vec c_h) )) \leq \sum_{h\in G_x} N(\tau_x(y, (h,\vec c_h) )) \leq \| e \|_{E_x}$ where we equip $E_x$ with the norm $\| \sum_{h\in G_x} (h,\vec c_h)\|_{E_x} := \sum_{h\in G_x} \| \vec c_h \|_{\ell^1}$ with $\| (c_1,\ldots,c_d)\|_{\ell^1} := |c_1|+\ldots+|c_d|$. Note that this norm is $G_x$-invariant by construction. Then $N(\tau_x(y,e))\leq  \| e \|_{E_x}$ follows by summing over the estimates on generators
\begin{align*}
N(\tau_x(y, (h,\vec c_h) ))& \leq 
\textstyle \sum_{l=1}^{d_x}  | c_l |  N \bigl( \mu\bigl(\Ga_x(h , h^{-1}* y) , \s_{x,l}(h^{-1}*y) \bigr) \bigr) \\
&= \textstyle \sum_{l=1}^{d_x}  | c_l |   N \bigl( \s_{x,l}(h^{-1}*y) \bigr) 
\leq \textstyle \sum_{l=1}^{d_x}  | c_l |  
\end{align*}
where we used compatibility of $N$ with the bundle morphisms encoded by $\mu$ as in \cite[Def.12.2.1~(3)]{TheBook}. 

\item[(2)]
$\supp\tau_x(\cdot,e)\subset U_x \subset \Uu$ for all $e\in E_x$ since $\supp \s_{x,l} \subset U_x$, $U_x$ is invariant under $G_x$, and $U_x \subset \Uu$. 

\item[(3)]  
$\rD^R (f-\tau_x) (x,0) : \rT_x^R X \times E_x \to W_x=C_x\oplus\im\rD^R f(x)$ is surjective since the 
derivatives of 
$\tau_x (x,  (\id , \vec c ) ) =  \sum_{l=1}^{d_x}  c_l  \cdot \mu\bigl(\Ga_x(\id  ,  x) , \s_{x,l}( x) \bigr) 
=  \sum_{l=1}^{d_x}  c_l  \cdot w_l$
in the direction of $\vec c$ span $C_x={\rm span}\{w_1,\ldots,w_{d_x}\}$. 
\end{itemize}

Note that (3) implies in particular that $\rD (f-\tau_x) (x,0) : \rT_x X \times E_x \to W_x=C_x + \im\rD f(x)$ is surjective. 
For $x\in\partial X$, we claim that (3) also implies general/good position \cite[Def.3.6.6, 5.3.9]{TheBook}. Towards this, choose a sc-complement of the reduced kernel $\rT_x^R X = P \oplus \ker\rD^R (f-\tau_x) (x,0)$, then the surjectivity in (3) implies that $\rD(f-\tau_x)(x,0): P \to W_x$ is a sc-isomorphism. 
As a result, $P$ is also a complement of the full kernel, $\rT_x X = P \oplus \ker\rD (f-\tau_x) (x,0)$, which confirms general position as the complement $P$ lies in the reduced tangent space. And this is a special case of the complement required by the good position property \cite[Def.3.1.24]{TheBook} for $N=\ker\rD (f-\tau_x) (x,0)$. Indeed, in a local chart around $x\simeq (0,0)\in [0,\infty)^{d(x)}\times E=:C_x$,  we have $P\subset \rT^R_xX\simeq \{0\}\times E$ and hence $n\in C_x \Leftrightarrow n+p\in C_x$ for $(n,p)\in N\times P$.

Finally, note that the local uniformizers provide an open cover of the zero set $|f^{-1}(0)|\subset\bigcup_{x\in f^{-1}(0)} |U'_x|$. Since the latter is compact, there is a finite cover $|f^{-1}(0)|\subset\bigcup_{i \in A} |U_{x_i}|$, and we rename all the structures by replacing the subscripts $x_i$ with $i$ in the finite set $A$. 
\end{proof}

Lemma~\ref{lem:localtau} provides a cover of the zero set $|f^{-1}(0)|\subset \bigcup_{i\in A} |U_i| =: F$, which lifts to a covering $\pi_\Xx^{-1}(F) = \bigcup_{i\in A} \pi_\Xx^{-1}(|U_i|)$ by saturated open sets $\pi_\Xx^{-1}(|U_i|)\subset X$. 
Now \cite[Thm.5.5.10, Cor.5.5.17, Thm.7.5.7]{TheBook}, and our assumptions in Definition~\ref{def:poly} guarantee the existence of a sc-smooth partition of unity $\sum_{i\in A} |\beta_i| \big|_F \equiv 1$ subordinate to this saturated open cover in the sense of  \cite[Def.7.5.5]{TheBook}. Since the index set $A=\{1,\ldots,N\}$ is finite, this notion simplifies to
\begin{align}\label{eq:punity}
\mbox{(a)} &\qquad   \supp\beta_ i \subset \pi_\Xx^{-1}(|U_i|) \mbox{ for } \ i=1,\ldots,N,\\ \notag
\mbox{(b)}  &\qquad  \mbox{each } \beta_i:\Xx\to [0,1] \ \mbox{  is a sc-smooth functor.}
\end{align}
Here the category $[0,1]$ has objects $[0,1]$ and only identity morphisms.

When constructing the left hand square in the diagram \eqref{diag:11}, we need to choose the cover reduction that defines $f_\Vv:\Xx_\Vv\to\Ww_\Vv$ to be compatible with such a partition of unity.

\begin{theorem} \label{thm:stabilize}
Given 
\begin{itemlist}
\item[-] a finite collection 
$(U_i,G_i,\tau_i: X \times E_i \to W)_{i\in A}$
of local uniformizers and Fredholm stabilizations as in Lemma~\ref{lem:localtau}, 
\item[-] a choice of order $A=\{1,\ldots, N\}$, and 
\item [-] a partition of unity $\sum_{i\in A} |\beta_i| \big|_{\bigcup_{i\in A} |U_i|} \equiv 1$ as in (a)-(b) above,
\end{itemlist}
 choose 
\begin{itemlist}\item [-]  a cover reduction $\bigl(F'_I \subset |\Xx| \bigr)_{I\subset A}$ 
of $F_i:=|U_i|$ that is compatible with $C_i:=\supp|\beta_i|$ as in Lemma~\ref{lem:cov0}~(iv), 
\item [-] compactness controlling data $(N,\Uu)$ for $f$ with  
$|\cl_X(\Uu)|\subset |V|= \bigcup_{I\subset A} F'_I $, and
\item [-] a sc-smooth bump function $|\be_\Uu|:|V|\to [0,1]$ that equals $1$ on a neighbourhood of $|f^{-1}(0)|$ and has support in $|\Uu|$,
\end{itemlist}
and let 
\begin{itemlist}\item [-] 
$f_\Vv:\Xx_\Vv\to \Ww_\Vv$ be the sc-Fredholm section functor  constructed in Theorem~\ref{thm:reduceFred}, and 
\item [-] $\Xx_\Vv ^{\less G}$ be the (nonproper) nonsingular subgroupoid  of $\Xx_\Vv$ equipped with a $G$ action so that $|\Xx_\Vv ^{\less G}|\,/G \simeq |\Xx_\Vv|$ as in Theorem~\ref{thm:reduce}~(iv).
\end{itemlist}
Then this data induces a $G$-equivariant Fredholm stabilization functor $\tau :  \Xx_\Vv^{\less G}\times E\  \to \ \Ww_\Vv$ for $f_\Vv$ as follows:
\begin{enumilist}
\item  The finite dimensional vector space $E:= E_1 \times \ldots \times E_N$ is equipped with the diagonal action 
$\ug*\ue:= (g_1*e_1,\ldots, g_N *e_N)$ of the finite group $G:=G_1\times \ldots \times G_N$. 
The  \'etale (but usually not proper) groupoid $\Xx_\Vv^{\less G}\times E$ is the product of $\Xx_\Vv^{\less G}$ from Theorem~\ref{thm:reduce}~(iv) with the category with object space $E$ and only identity morphisms.

\item 
There is a commutative diagram
\begin{align}\label{diag:14}
\xymatrix
{
\Xx_\Vv^{\less G} \times E\ar@{->}[d]_{\pr}\ar@{->}[r]^{\;\;\;\; \tau}  & \Ww_\Vv\ar@{->}[d]^{\Pp} \\
\Xx_\Vv^{\less G}   \ar@{->}[r]^{\io}& \Xx_\Vv 
}
\end{align}
Here $\pr: \Xx_\Vv^{\less G} \times E \to \Xx_\Vv^{\less G}$ is the strong bundle over the \'etale but generally nonproper groupoid $\Xx_\Vv^{\less G}$ given as in Definition~\ref{def:bundle} by the projection $\pr: (I,\ux,\ue) \mapsto (I,\ux)$ on objects and the lift of the target map 
$$
\mu_\pr: \bX_\Vv^{\less G} \times E \to  \Xx_\Vv^{\less G} \times E, \quad  (I,J,\ux,\ug,\ue)\mapsto  \bigl( t(I,J,\ux,\ug) , \ue \bigr).
$$

\NI The inclusion $\io$ is the identity map on objects and the inclusion on morphisms.

\NI The Fredholm stabilization functor $\tau :  \Xx_\Vv^{\less G} \times E  \to \Ww_\Vv$ is the strong bundle map given by\footnote
{
Here we use the abbreviated notation of Proposition~\ref{prop:MMGa1}. Therefore in a tuple $(I,J,\ug,\ux)$ we assume $\ux \in \TV_{(I\wedge J)(I\vee J)}$ and denote by $\rho_I(\ux), \rho_J(\ux)$ its images in $V_I, V_J$ respectively.}  
\begin{align*}
X_\Vv^{\less G} \times E & \to W_\Vv=
 \bigl \{(J,\ux,w)\in 
 X_\Vv\, \leftsub {\psi} {\times}_P W \ \big| \ \psi(J,\ux) = P(w) \bigr\}, 
\\ 
&\qquad \qquad 
( J, \ux , \ue) \mapsto \bigl( J,\ux , \tau_J (\ux, \ue) \bigr), \\
\bX_\Vv^{\less G}\times E & \to \bW_\Vv
= \bigl \{(I,J,\ux,w)\in  \bX_\Vv\, \leftsub {s} {\times}_P W 
 \ \big| \  \psi(I,\rho_I(\ux) ) = P(w) \bigr\},
\\ & \qquad \qquad 
(I, J,  \ux, \ug,  \ue)  \mapsto \bigl(  (I, J, \ux , \ug) , \tau_{IJ} (\ux, \ug, \ue) \bigr), 
\end{align*}
in terms of the maps $\tau_J: V_J  \times E \to W$ and
$\tau_{IJ}:  \Mor_{\Xx_\Vv^{\less G}}(V_I,V_J) \times E 
 \to W$ 
which represent the functor $\Psi\circ\tau$ and are
given as follows when $I=\{i_1<\ldots<i_{|I|}\}$ 
\begin{align}\label{eq:tauI} 
&\tau_I \bigl( (x_{i_1}, m_{i_1i_2}, \ldots, m_{i_{|I|-1}i_{|I|}}, x_{i_{|I|}}) \,,\, (e_1,\ldots,e_N) \bigr)  := \be_\Uu(x_{i_1})\beta_{i_1}(x_{i_1}) \tau_{i_1}(x_{i_1}, e_{i_1}) \\  \notag
&\hspace{1.2in} \textstyle
+ \sum_{n=2}^{|I|}  \be_\Uu(x_{i_1})\mu \bigl(  m_{i_{n-1}i_n}^{-1}\circ\ldots\circ m_{i_1i_2}^{-1}  , \beta_{i_n}(x_{i_n}) \tau_{i_n}(x_{i_n}, e_{i_n}) \bigr) \ \in\  W_{x_{i_1}},
  \\  \notag
&  \tau_{IJ} (\ux,\ug,\ue) :=  \tau_I \bigl( \ug^{-1}*\rho_{I}(\ux), \ue\bigr) \in W_{g_{i_1}^{-1}*x_{i_1}}  . 
\end{align}
This induces a sc-smooth functor $\tau :  \Xx_\Vv^{\less G} \times E  \to \Ww_\Vv[1]$.

\item  The functor $\tau$ is $G$-equivariant w.r.t.\ the product action of $G$ on $\Xx_\Vv^{\less G} \times E$ and the inner action on $\Ww_\Vv$ described in Theorem~\ref{thm:reduceFred}~(iii), 
so that on objects
$$
\tau(\ug*\uo,\ug *\ue) = \ug* \tau(\uo,\ue) = \mu_\Vv\bigl(\alpha(\ug, \ug*\uo) , \tau(\uo,\ue) \bigr)
\qquad
\forall \ug\in G, \uo\in X_\Vv . 
$$ 
\item[{\rm (iv)}] 
The functor $\tau$ is a compactness-controlled Fredholm stabilization in the following sense: 
The pullback $(N_\Vv:=N\circ\Psi,\, \Uu'_\Vv:=\psi^{-1}(\Uu'))$ of the restriction $(N,\Uu')$ of the data $(N,\Uu)$ chosen for Lemma~\ref{lem:localtau} controls compactness of $f_\Vv$  in the sense of \cite[Def.12.2.1,15.3.1]{TheBook}
by Theorem~\ref{thm:reduceFred}~(iv). 
Moreover, the family of maps $\bigl(\tau_\ue: X_\Vv  \to W_\Vv,  \uo\mapsto \tau(\uo,\ue) \bigr)_{\ue\in E}$ is $(N_\Vv,\Uu'_\Vv)$-regular in the sense
\begin{itemlist}
\item[(1)]
$N_\Vv(\tau_\ue(\uo))\leq \|\ue\| $ for $\ue\in E$, $\uo\in  X_\Vv = X_\Vv^{\less G}$, and a $G$-invariant norm on $E$;
\item[(2)]
$\supp\tau_\ue = {\rm cl}\{ \uo\in  X_\Vv \,|\, \tau_\ue(\uo)\ne 0_\uo \} \subset \Uu'_\Vv$ for all $\ue\in E$; 
\item[(3)]
there is a $G$-invariant neighbourhood $O_\tau\subset E$ of $0$ such that
the map $X_\Vv\times O_\tau \to W_\Vv$, $(\uo,\ue) \mapsto f_\Vv(\uo) - \tau(\uo,\ue)$ is transverse to the zero section of $W_\Vv\to X_\Vv$ in the sense that its linearization at every zero is onto. 
Moreover, it is in general (hence good)  position at zeros in the boundary $\partial X_\Vv\times O_\tau$ in the sense that the reduced linearizations are onto. 
\end{itemlist}
\end{enumilist}
\end{theorem}

\begin{proof} {\bf Proof of (i):}   The groupoid $\Xx_\Vv^{\less G} \times E $ is well defined and \'etale 
by Lemma~\ref{lem:MMGaE}~(ii),  but, by Remark~\ref{rmk:complet}~(iii), not in general proper.
In particular, its object and morphism spaces inherit an M-polyfold structure from 
the object and morphism spaces of $\Xx_\Vv^{\less G}$ in Theorem~\ref{thm:reduce}~(iv) 
by Cartesian products with the finite dimensional vector space $E$, which has a unique sc-Banach space structure $(E_k=E)_{k\in\N_0}$. 
In this sc-smooth structure, the structure maps of $\Xx_\Vv^{\less G} \times E $ are local sc-diffeomorphisms because they are products of the structure maps of $\Xx_\Vv^{\less G}$ with  the identity. 
This proves (i). 
\MS

\NI {\bf Proof of (ii)--(iii):}
First recall that the finite dimensional vector space $E$ carries a unique sc-structure $(E_k=E)_{k\in\N_0}$, which in particular coincides with its shift that appears in the notion of strong bundle. 
Thus the strong bundle structure of $\pr: \Xx_\Vv^{\less G} \times E \to \Xx_\Vv^{\less G}$ only requires $\mu_\pr: \Mor_{\Xx_\Vv^{\less G}} \times E \to  \Xx_\Vv^{\less G} \times E$ to be linear in $E$ and sc-smooth, which $\mu_\pr(I,J,\ux,\ug,\ue)=\bigl( t(I,J,\ux,\ug) , \ue \bigr)$ is since the target map $t$ is sc-smooth. 

The \'etale functor $\io:\Xx_\Vv^{\less G}\to \Xx_\Vv$ is the identity on objects and an  inclusion on morphisms.

The functor $\tau: \Xx_\Vv^{\less G} \times E  \to \Ww_\Vv$ in \eqref{diag:14} is defined on objects 
by the formulas in \eqref{eq:tauI}.  
The summands for $n\geq 2$ in this expression can be rewritten as
\begin{align*}
& \mu \bigl(  m_{i_{n-1}i_n}^{-1}\circ\ldots\circ m_{i_1i_2}^{-1}  , \beta_{i_n}(x_{i_n}) \tau_{i_n}(x_{i_n}, e_{i_n}) \bigr) \\
&\qquad\quad 
=
\mu( m_{i_{n-1}i_n}^{-1})\bigl( \ldots\bigl( \mu(m_{i_1i_2}^{-1})  (\beta_{i_n}(x_{i_n}) \tau_{i_n}(x_{i_n}, e_{i_n}))\bigr)\ldots\bigr) , 
\end{align*}
where we denote the fiberwise restrictions of the strong bundle map $\mu:\bX\leftsub{s}{\times}_P W\to W$  in Definitions~\ref{def:poly} and \ref{def:bundle} by
\begin{equation}\label{eq:mu}
\mu (m) \,:\; W_{s(m)} \to W_{t(m)}, \quad w \mapsto \mu(m,w) . 
\end{equation}
These summands in \eqref{eq:tauI} are then well defined in $W_{x_{i_1}}[1]$ because we have $$
\beta_{i_n}(x_{i_n})\in [0,1], \quad \tau_{i_n}(x_{i_n}, e_{i_n})\in W_{x_{i_n}}[1], \quad m_{i_ki_{k+1}}\in \Mor_\Xx(x_{i_k}, x_{i_{k+1}})
$$ so that $\mu(m_{i_ki_{k+1}}^{-1}): W_{x_{i_{k+1}}}[1]\to W_{x_{i_k}}[1]$ is a linear map between the appropriate fibers of $W\to X$, preserving the shifted sc-structure. 
Moreover, each of these summands is linear in $e_{i_n}$ since $\tau_{i_n}$ is linear on fibers and each $\mu(m^{-1})$ in \eqref{eq:mu} is linear. This shows that $\tau$ on objects is well defined and linear in $E=E_1\times\ldots\times E_N$. Note also that each $\tau_I$ is trivial (i.e.\ the zero map)  in the factors $E_j$ for $j\notin I$. (For a more precise claim see \eqref{eq:tauHrestrict} below.)
\MS

To check that $\tau$ is a strong bundle map in the sense of Remark~\ref{def:poly} and Definition~\ref{def:bundlemaps} it remains to show that it  is sc-smooth with respect to both sc-structures on $W$ (thus in particular $\tau :  \Xx_\Vv^{\less G} \times E  \to \Ww_\Vv[1]$  should be sc-smooth), and also is compatible with morphisms. 
  
Towards the former, note that the first summand in \eqref{eq:tauI}, $\beta_{i_1}(x_{i_1}) \tau_{i_1}(x_{i_1}, e_{i_1})\in E$ varies sc-smoothly with $$
\bigl(
\ux = (x_{i_1}, m_{i_1i_2}, \ldots, m_{i_{|I|-1}i_{|I|}}, x_{i_{|I|}}) \,,\, 
\ue = (e_1,\ldots,e_N) \bigr)\in V_I\times E
$$
 since by construction of the sc-smooth structure on $V_I\subset U_I$  in Lemma~\ref{lem:Udata} the projection $U_I\ni\ux=(x_{i_1}, m_{i_1i_2}, \ldots, )\mapsto x_{i_1}\subset X$ is a local sc-diffeomorphism, and $\tau_{i_1}$ is sc-smooth as map to $W[1]$ by construction in Lemma~\ref{lem:localtau}, and hence also sc-smooth after composition with the inclusion $W[1]\to W$. 
The further summands 
$$
\bigl( \ux \,,\, \ue \ \bigr) \mapsto 
\mu( m_{i_{n-1}i_n}^{-1}) \ldots \mu(m_{i_1i_2}^{-1})  \beta_{i_n}(x_{i_n}) \tau_{i_n}(x_{i_n}, e_{i_n})
$$
are sc-smooth to $W[1]$ and hence $W$, since  $\tau_{i_n}$ has this property by construction, the maps in \eqref{eq:mu} are strong bundle maps, and each of the maps
\begin{align*}
V_I &\to X, \qquad 
\ux  = (x_{i_1}, m_{i_1i_2}, \ldots, m_{i_{|I|-1}i_{|I|}}, x_{i_{|I|}}) \mapsto x_{i_n} \\
V_I &\to \bX , \qquad 
\ux = (x_{i_1}, m_{i_1i_2}, \ldots, m_{i_{|I|-1}i_{|I|}}, x_{i_{|I|}}) \mapsto m_{i_{\ell-1}i_\ell}^{-1} 
\end{align*}
is sc-smooth as follows: In a chart for $U_I\supset V_I$ these maps are represented by $x_{i_1}\mapsto x_{i_n}$ resp. $x_{i_1}\mapsto m_{i_{\ell-1}i_{\ell}}$, which can be expressed as composition of target maps, local inverses of source maps, and inversion --- all of which are sc-smooth since they are restrictions of the \'etale structure maps of $\Xx$. 
\MS

To see that $\tau$ is compatible with morphisms, it suffices to check that it satisfies the conditions in
Proposition~\ref{prop:MMGaE}, namely
\begin{enumilist}\item[{\rm (a)}] it is compatible with the morphisms in the poset $\Qq\times E$;
\item[{\rm (b)}] the induced map on object spaces is $G$ equivariant;
\item[{\rm (c)}]  its restriction to each fiber $(I,y)\times E$ is linear with kernel  that includes the subspace $E_{A\less H_y}$, where $H_y = \min \{H\ \big| \ y\in \TQ_{HI}\}$;  
\item[{\rm (d)}]
 $\Pp\circ {\tau} = \io\circ \pr: \Qq\times E\to \Xx_\Vv$, where $\io:  \Xx_\Vv^{\less G} \to \Xx_\Vv$ is the inclusion.
\end{enumilist}
For (a), we must  check that $\tau$ is compatible with the morphisms  from $V_I\times E\to V_J\times E$ with $I\subset J$. 
  Further, by 
 Lemma~\ref{lem:bundlemaps}~(ii), it suffices  to check the condition in Definition~\ref{def:bundlemaps}~(ii). Namely, we are given the map $\tau: Q\times E\to W_\Vv$ on objects and must check that
$$
 {\tau}(\mu_{\Qq}(m,e)) = \mu_\Vv(\io(m), w), \quad w: = {\tau}(s(m),e)\in W_{\psi(s(m)},
$$
  where $\mu_{\Qq}, \mu_\Vv$ are the structure maps of 
 the two bundles as in  Definition~\ref{def:bundle}. 
 Since $\Ww_\Vv = \Psi^*(\Ww)$ is the pullback of
 the original bundle $\Ww\to \Xx$ over the equivalence $\psi:\Xx_\Vv\to \Xx$, it suffices to check that
 $$
\psi\circ {\tau}(\mu_{\Qq}(m,\ue)) = \mu(\psi(\io(m)), w), 
 $$
 where $\mu$ is the structure map of $\Ww\to \Xx$ and $w$ is as before.
We next simplify the left hand side using the fact that the structure map $\mu_{\Qq}$
of the trivial bundle $\Qq\times E\to \Qq$  is simply\footnote
  {
  Recall from Lemma~\ref{lem:complet} that the morphisms in the poset $\Qq$ do not involve any group elements, though those in its groupoid completion $\Qq$ do.}
 $$
 \mu_{\Qq}: {\bQ}\, \leftsub{s}{\times}_P  (\Qq\times E) \to   \Qq\times E,  \
 (I,J,\ux, \ue)\mapsto (J,\ux,\ue) = \bigl( t(I,J,\ux), \ue\bigr).
 $$
  In other words, it is the product of the target map in $\Qq$ with $\id_{E}$. 
 Thus, we must check that
  for every morphism $m = (I,J,\ux, \id) \in \Qq$ with $I\subset J$ we have
\begin{align}\label{eq:checkfunct}
\psi\bigl(\tau (J,\ux, \ue)\bigr) & = \mu\bigl(\psi(I,J,\ux, \id), \tau_I(\ux,\ue)\bigr)
\end{align}
where $\tau_I$ is as in \eqref{eq:tauI}.

If we write $I = (i_1<i_2<\dots) \subset J = (j_1<j_2\dots)$ where $i_1 = j_{\ell_1}$, and  $(J,\ux) = (x_{j_1},m_{j_1j_2}\dots )$, then \eqref{eq:tauI} defines 
$$
\psi\bigl(\tau (J,\ux, \ue)\bigr)  = \tau_J(\ux,\ue)
$$ 
as an element in $W_{x_{j_1}}$.
Further, because the reduction $(F'_I)_{I\subset A}$ is  compatible with the supports $(C_i)_{i\in A}$ of the  partition of unity $(|\be_{j}|)_{j\in A}$
 as in
Lemma~\ref{lem:cov0}~(iv),
$\be_{j_k}(x_{j_k}) = 0$ if $\ux\in \TV_{IJ}\subset V_J$ and $j_k\in J\less I$.    Indeed, because $|x_{j_1}| =
|x_{j_2}| = \dots = |x_{j_{|J|}}|$ for all $\ux\in V_J$, $\be_{j_k}(x_{j_k}) = 0$ for all $j_k\notin H_{|\ux|}  =
\min  \{H\ | \ \ux\in \TV_{HJ}\}$.  Moreover, this set $H_{|\ux|} = H_{x_{j_1}}$ as defined in Proposition~\ref{prop:MMGaE0}.  Thus, if we denote the projection $E \to \prod_{i\in H_{|\ux|}}E_i$ by $\ue\mapsto \ue|_{H}$ we have
\begin{align}\label{eq:tauHrestrict}
\tau_J(\ux,\ue) = \tau_J(\ux,\ue|_{H_{|\ux|}}), \quad H_{|\ux|}: = H_{x_{j_1}}.
\end{align}
Using this, we check that 
\begin{align}\label{eq:checkfunct1}
 \tau_J(\ux,\ue) = \mu\bigl((m^J_{j_1j_{\ell_1}})^{-1}, \tau_I(\ux, \ue)\bigr)
\end{align}
where $m^J_{j_1j_{\ell_1}} = m_{j_1j_{2}}\circ\dots\circ m_{j_{\ell_1-1}j_{\ell_1}}$ as in \eqref{eq:comp00}.
Further, when $\ug = \id$ we have
$\tau_{IJ}(\ux,\id,\ue) = \tau_I(\rho_I(\ux), \ue)$, where $\rho_I(\ux)=
\rho_{IJ}(\ux)=
(x_{j_{\ell_1}}, m^J_{j_{\ell_1}j_{\ell_2}},\dots)$  as in \eqref{eq:comp00}.  
On the other hand,   when $I\subset J$, 
  $$
   \mu\bigl(\psi(I,J,\ux,\id), w\bigr) =  \mu\bigl((m^J_{j_1 j_{\ell_1}})^{-1}, w\bigr), \quad w\in W_{x_{j_{\ell_1}}}
   $$
by the definition of $\psi:\Xx_\Vv\to \Xx$ in Theorem~\ref{thm:reduce}~(ii).
Therefore, when $w = \tau_I(\ux,\ue)$ we obtain
\begin{align*}
  \mu\bigl(\psi(I,J, \ux,\id), \tau_I(\ux,\ue)\bigr) & = \mu\bigl((m^J_{j_1 j_{\ell_1}})^{-1}, \tau_I(\ux,\ue)\bigr)\\
  & = \tau_J(\ux,\ue) \quad\mbox{ by } \eqref{eq:checkfunct1}\\
  & = \psi\bigl(\tau (J,\ux, \ue)\bigr),
\end{align*}
which proves \eqref{eq:checkfunct}.  Thus (a) holds.  

The above proof also establishes condition (c), since this follows from \eqref{eq:tauHrestrict}.  Further (d)
holds since $\tau:\Xx_\Vv^{\less G} \times E \to \Ww_\Vv$ is given by E$(J,\ux,\ue)\mapsto (J,\ux,\tau_J(\ux,\ue))$ and hence lies over the inclusion  $\Xx_\Vv^{\less G}\to \Xx_\Vv$.
Thus it remains to check (b), namely that the induced map on objects $\tau_J: V_J\times E\to W$  is $G$-equivariant for each $J$. 
 
The following calculation shows that  this follows from the equivariance of the local  stabilizations $\tau_i$ and the functoriality of the  construction.
 Indeed, in the notation of \eqref{eq:tauI} we have for $\ug\in G_{J}$ and $\ux\in V_J$  
\begin{align}\label{eq:Gequivtau}
&\tau_J(\ug*\ux,\ug*\ue)  \\ \notag
& \quad = \tau_J \bigl(g_{j_1}* x_{j_1}, (g_{j_1})^{-1}\circ m_{j_1j_2}\circ g_{j_2}, \ldots , g_{j_{|J|}}*x_{j_{|J|}}) \,,\, (g_{j_1}*e_{j_1},\ldots, g_{j_{|J|}}*e_{j_{|J|}}) \bigr) \\ \notag
& \quad 
= \be_\Uu(x_{j_1}) \textstyle{
 \sum_{n=1}^{|H_\ux|} \mu \Bigl( \bigl((g_{j_1})^{-1}\circ m^J_{j_1 j_{h_{n}}}\circ g_{j_{h_n}}\bigr)^{-1}, \beta_{h_n}(x_{h_n}) \,\tau_{h_n}(g_{h_n}*x_{h_n}, g_{h_n}*e_{j_{h_n}}) \Bigr) 
   }\\\notag
 & \quad 
= \be_\Uu(x_{j_1}) \textstyle {  \sum_{n=1}^{|H_\ux|} \mu \Bigl(\bigl( (g_{j_1})^{-1}\circ m^J_{1 h_{n}}\circ g_{h_n}\bigr)^{-1}, \beta_{h_n}(x_{h_n}) \,g_{h_n}* \tau_{h_n}(x_{h_n}, e_{j_{h_n}}) \Bigr) 
}
\\\notag
 & \quad 
= \be_\Uu(x_{j_1}) \textstyle {  \sum_{n=1}^{|H_\ux|} \mu \bigl(  (m^J_{1 h_{n}})^{-1}\circ g_{j_1}, \beta_{h_n}(x_{h_n})\, \tau_{h_n}(x_{h_n}, e_{j_{h_n}}) \bigr)
\ = \  g_{j_1}* \tau_J(\ux, \ue),   
}
 \end{align}
 where the last equality holds by \eqref{eq:muprop}.
On the other hand, if $\ug \in G_{A\less J}$ then $\ug*\ue -\ue \in E_{ G_{A\less H_{\ux}}}$ and $\ug*\ux=\ux$  so that
 $$
 \tau_J(\ug*\ux, \ug*\ue)= \tau_J(\ux,\ug*\ue - \ue) + \tau_J(\ux,\ue) = \tau_J(\ux,\ue).$$
 This completes the proof of condition (b).  
 
 Therefore all conditions in Proposition~\ref{prop:MMGaE} hold, and $\tau$ has a unique $G$-equivariant extension to a functor $\tau:\Xx_\Vv^{\less G}\times E\to \Ww_\Vv$.  This proves (ii) and (iii).
\MS 

\NI {\bf Proof of (iv):} It remains to check the $(N_\Vv,\Uu_\Vv)$-regularity properties (1)--(3) in (iv). 
\MS

\NI 
To check (1) we estimate for $\uo=(I,\ux=(x_1, m_{12}, \ldots, m_{(|I|-1)|I|}, x_{|I|}))\in X_\Vv$ and $\ue=(e_1,\ldots,e_N)\in E=E_1\times\ldots\times E_N$
\begin{align*} 
 N_\Vv(\tau_\ue(\uo))
&=  N( \Psi ( \tau_\ue(\uo)))  =  N( \tau_I(\ux,\ue)))  \\
&\leq N\bigl(  \be_\Uu(x_{1}) \beta_{i_1}(x_1) \tau_{i_1}(x_1, e_{i_1})  \bigr) + \textstyle \sum_{n=2}^{|I|}  N \bigl( \mu \bigl(  m_{(n-1)n}^{-1}\circ\ldots\circ m_{12}^{-1}  , \be_\Uu(x_{1}) \beta_{i_n}(x_n) \tau_{i_n}(x_n, e_{i_n}) \bigr) \,
\bigr) \\
&\leq  \textstyle \sum_{n=1}^{|I|}  \be_\Uu(x_{1})  \beta_{i_n}(x_n) N \bigl( \tau_{i_n}(x_n, e_{i_n}) \bigr)
\;\leq\;  \textstyle \sum_{n=1}^{|I|}   \| e_{i_n}\|_{E_i} , 
\end{align*}
where we used the fact that $N$ is compatible with the bundle morphisms encoded in $\mu$, and that $0\leq\beta_i\leq 1$. This proves (1) with the $G$-invariant norm $\|\ue\|:=\sum_{i=1}^n  \| e_i\|_{E_i}$ on $E$. Moreover, $\ue\in O'_\tau:= \{ (e_1,\ldots,e_N) \,\|\, \sum_{i=1}^N \| e_{i}\|_{E_i} <1 \}$ guarantees the usual compactness control estimate $N_\Vv(\tau_\ue(\uo))<1$. 
\MS

\NI 
Note that (2) holds, namely $\supp\tau_\ue  \subset \Uu_\Vv=\psi^{-1}(\Uu)$, because  the bump function $\be_\Uu$ that occurs in the formula for $\tau_I$ was chosen to have support in $|\cl_X(\Uu)|\subset |V|$.
%
\MS

\NI 
Checking (3)  asks us to establish the transversality (and general position along $\partial\Xx_\Vv$) of the map $F_\tau: X_\Vv\times E \to W_\Vv$, $(\uo,\ue) \mapsto f_\Vv(\uo) - \tau(\uo,\ue)$ on a product neighbourhood of  $X_\Vv\times\{0\}$. For that purpose we will first establish that the differential, restricted to the reduced tangent space, $\rD^R F_\tau (\uo,\underline 0):=\rD F_\tau (\uo,\underline 0)|_{\rT^R_\uo X_\Vv \times E }$ is surjective for every unperturbed solution $f_\Vv(\uo) = \tau(\uo,\underline 0)=0_\uo$. (See \cite[p.46 and Def.2.4.7]{TheBook} for the notion of reduced tangent space at boundary points. In the interior, we denote the tangent space by $\rT^R:=\rT$.) 
To check this at $\uo=(I,\ux)$, note that $$
\rD^R F_\tau (\uo,\underline 0) : \rT^R_\uo X_\Vv \times E = \rT^R_\ux V_I \times E \to  W_{\psi(\uo)} \simeq (\Ww_\Vv)_\uo$$
 is the linear operator  
\begin{align*}
(\underline{X}, \ue) &\;\mapsto\;  \rD^R f (\psi(\uo))(\rd\psi (\uo) \underline{X} ) - \rD^R \tau_I (\cdot, \underline 0) (\ux) \underline{X} - \rD \tau_I (\ux, \cdot) (\underline 0) \ue  \notag \\
&\;=\; \rD^R f (\psi(\uo))(\rd\psi(\uo) \underline{X} ) - \tau_I (\ux, \ue)   
\end{align*}
since $\tau_I$ is linear in the second factor, in particular $\tau_I (\ux, \underline 0)=0$. 
Since $\psi_I$ is a local sc-diffeomorphism to a neighbourhood of $o:=\psi(\uo)$ (and at boundary points restricts to local sc-diffeomorphisms between the boundary and corner strata), 
surjectivity of $\rD^R F_\tau (\uo,\underline 0)$ is equivalent to surjectivity of 
\begin{equation}
\rT_o X \times E  \;\to\;   W_o , \quad
(Y, \ue) \;\mapsto\;  \rD^R f (o) Y - \tau_I (\ux, \ue)   . 
\label{eq:DF}
\end{equation}
To check surjectivity of this operator we will use of the compatibility between cover reduction and partition of unity achieved in Lemma~\ref{lem:cov0}: Since $|o=\psi(\uo)|\in F'_I$ and $F'_I\cap \supp |\beta_j|=\emptyset$ for $j\notin I$ we have $\sum_{i\in I} |\beta_i|(|o|)=1$  and can thus choose an index $i_n\in I$ with $|\beta_{i_n}|(|o|)>0$. 
Now recall that we are working at a solution $\uo=(I,\ux)$ given by some tuple $\ux=(x_1, m_{12}, \ldots, m_{(|I|-1)|I|}, x_{|I|})$ of objects $x_n\in U_{i_n}$ and morphisms $m_{(n-1)n}\in\Mor_{\Xx}(x_{n-1},x_n)$ between them, so that $|\psi(\uo)=o=x_1|=\ldots=|x_{|I|}|$. In this notation we found an index $1\leq n \leq |I|$ with $\beta_{i_n}(x_n)=|\beta_{i_n}|(|x_n|)>0$.
For this index we also know from Lemma~\ref{lem:localtau}~(3) that the reduced linearization of 
$f - \tau_{i_n} \,: \; X \times E_{i_n} \to W$ at $(x_n,0)$ is surjective. 
Here we used functoriality of $f$ to deduce $f(x_n)=0$ from $f(x_1=\psi(\uo))=0$. 
More precisely, there is a sc-diffeomorphism $m:=t \circ s_{x_n}^{-1} : \text{Nbhd}(x_n)\to\text{Nbhd}(x_1)$ between neighbourhoods of $x_n,x_1\in X$, which is given by a local inverse $s_{x_n}^{-1}$ of the source map near $s\bigl(  m_{(n-1)n}^{-1}\circ\ldots\circ m_{12}^{-1}\bigr)=x_n$. It lifts to sc-isomorphisms $M_x:= \mu( s_{x_n}^{-1}(x), \,\cdot\, ) : W_{x}\to W_{m(x)}$  between the fibers of the bundle $W\to X$. 
With this structure, local functoriality of $f$ can be expressed by $M_x(f(x)) = f ( m(x))$. This specifies to $M_{x_n}(f(x_n)) = f(x_1)=0$, which implies $f(x_n)=0$ since $M_{x_n}$ is invertible, and the relationship between linearized operators $M_{x_n}\circ\rD f(x_n) = \rD f(x_1) \circ \rd m (x_n)$. 
With that we can compute the image of the operator \eqref{eq:DF} applied to $\rT_{o=x_1}^R X=\rd m (x_n)(\rT_{x_n}^R X)$ and the subspace $\ldots \{0\} \times E_{i_n} \times \{0\} \ldots \subset E$ 
\begin{align*}
& \rD f (x_1)(\rT_o^R X) - \tau_I (\ux, \ldots \{0\} \times E_{i_n} \times \{0\} \ldots ) \\
& =  M_{x_n}  \bigl( \rD f(x_n) (\rT_{x_n}^R X)  \bigr) - M_{x_n}\bigl( \beta_{i_n}(x_n) \tau_{i_n}(x_n, E_{i_n}) \bigr)  \\
& =  M_{x_n}  \bigl( \rD f(x_n) (\rT_{x_n}^R X)  -  \tau_{i_n}(x_n, \beta_{i_n}(x_n) E_{i_n}) \bigr)  \\
&= \bigl( M_{x_n} \circ \rD  ( f - \tau_{i_n} ) (x_n,0) \bigr) (\rT_{x_n}^R X \times E_{i_n})  \;=\; W_{x_n}. 
\end{align*}
Here we used the fact that $M$ is an isomorphism, $\beta_{i_n}(x_n)\ne 0$, and $ \rD^R  ( f - \tau_{i_n} ) (x_n,0)$ is surjective. This prove surjectivity for \eqref{eq:DF} and thus for the linearizations of $F_\tau=f_\Vv-\tau_I$ at any point in the unperturbed solution set $f_\Vv^{-1}(0)\times\{0\} \subset X_\Vv\times E$. 

Now the implicit function theorem \cite[3.1.22, 3.1.26]{TheBook} guarantees transversality of solutions of $F_\tau$ in a neighbourhood of $f_\Vv^{-1}(0)\times\{0\}$. We can moreover apply \cite[Thm~3.1.22]{TheBook} on all local boundary and corner strata intersecting $f_\Vv^{-1}(0)\times\{0\}$, where $F_\tau$ restricts to a transverse section since its reduced linearizations are surjective, and we can conclude that the reduced linearizations (the linearizations of restrictions to boundary and corner strata) are surjective in a neighbourhood of $f_\Vv^{-1}(0)\times\{0\}$. 

In fact, our claim is that we can find a neighbourhood $O_\tau\subset E$ of $0$ so that the reduced linearizations of $F_\tau: X_\Vv\times O_\tau \to W_\Vv$ are surjective at all solutions -- and thus all solutions are transverse and in general (hence good) position, as explained in the proof of Lemma~\ref{lem:localtau}.
We will prove this by contradiction, assuming that $E\ni \ue_k\to 0$ as $k\to\infty$ has solutions $f_\Vv(\uo_k)=\tau(\uo_k,\ue_k)$ at which $\rD^R F_\tau(\uo_k,\ue_k)$ fails to be surjective. 
From the already established controls (1) of the size $N_\Vv(\tau)$ and (2) the support of $\tau$ we can deduce that $\uo_k\in\Uu_\Vv$ and $N_\Vv(f_\Vv(\uo_k))=N_\Vv(\tau(\uo_k,\ue_k))\leq 1$ for sufficiently large $k$, and thus $|\uo_k|$ lie in the set
$|\{ \uo \in \Uu_\Vv \,|\, N_\Vv(f_\Vv(\uo))\leq 1 \} |$ whose closure in $|\Xx_\Vv|$ is compact, since $(N_\Vv,\Uu_\Vv)$ control compactness as in Definition~\ref{def:control-compact}. So we can find a convergent subsequence $|\uo_k|\to |\uo_\infty|$. And since the groupoid $\Xx_\Vv$ is proper by Theorem~\ref{thm:reduce}~(i), we can in fact find a further subsequence with limit $\uo_k\to\uo_\infty\in X_\Vv$. This satisfies $f_\Vv(\uo_\infty)=\lim_{k\to\infty} f_\Vv(\uo_k)=\lim_{k\to\infty} \tau(\uo_k,\ue_k) = \tau(\uo_\infty, \underline{0})= 0$ by continuity of $f_\Vv$ and $\tau$.  
Now the subsequence $(\uo_k,\ue_k)\to (\uo_\infty,0)$ for sufficiently large $k$ must lie in the neighbourhood of $f_\Vv^{-1}(0)\times\{0\}$ on which surjectivity of $\rD^R F_\tau$ was already established above -- in contradiction to the assumption. 

This argument provides a neighbourhood $O_\tau\subset E$ of $0$ with the claimed transversality properties. To make it $G$-invariant, we replace it by $\bigcap_{g\in G} g*O_\tau$.
\end{proof}

\begin{proof}[Proof of Corollary~\ref{cor:multisection} ]
To simplify notation, we will denote $X_\Vv:=\Obj_{\Xx_\Vv}=\Obj_{\Xx_\Vv^{\less G}}$ and $W_\Vv:=\Obj_{\Ww_\Vv}=\psi^* W$. 
Then here is the first statement to be proved: 

\MS\NI
(i) {\it Each $\ue\in E$ induces sc$^+$-multisection functors $\La_{\Vv,\ue} : \Ww_\Vv\to \Q^{\ge 0}$ given by 
$$
\La_{\Vv,\ue}(w):= \tfrac{1}{\# G} \ \# \{ \uh\in G \,|\, \tau(P_\Vv(w) ,\uh*\ue)= w  \}
$$
and $\La_{\ue}:= \Psi_*\La_{\Vv,\ue}: \Ww|_{V}\to \Q^{\ge 0}$ given by pushforward \cite[Lemma~11.5.2]{TheBook}.  
These multisections are structurable in the sense of \cite[\S13.3]{TheBook}, and $\La_{\Vv,\ue}$ is globally structured in the sense of 
Definition~\ref{def:multisdef0}.
}

\MS
\NI {\bf Proof of (i):}   It follows from Theorem~\ref{thm:stabilize} that we are in the situation of
Proposition~\ref{prop:globsym0}.   Hence there are multisection functors $\La_{\Vv,\ue}$ that have all the required properties, except possibly that they may not be sc$^+$. However, they are sc$^+$ because
 the maps $\tau_i$ that define the local stabilizations as in Lemma~\ref{lem:localtau} have been  chosen to be sc$^+$.  This proves (i).
\MS

\MS\NI
(ii) {\it
There is a comeagre (and thus dense) subset $O^\pitchfork_\tau\subset O_\tau$ such that 
for all $\ue\in O^\pitchfork_\tau$ the functors 
$\Theta_{\Vv,\ue}: =\La_{\Vv,\ue}\circ f_\Vv : \Xx_\Vv\to\Q^{\ge 0}$ and $\Theta_\ue: =\La_{\ue}\circ f : \Xx|_V\to\Q^{\ge 0}$ are transversal and in general position over the boundary in the sense of \cite[\S15.2]{TheBook}. 
This gives both the structure of compact, tame branched ep$^+$-subgroupoids 
in the sense of \cite[\S9.1]{TheBook}. 
}

The following is a version of the proof of \cite[Thm.15.3.7]{TheBook} which is simplified by our prior construction of a global equivariant Fredholm stabilization. 
Recall from Theorem~\ref{thm:stabilize}~(iv) that we found compactness controlling data $(N_\Vv,\Uu_\Vv)$ for $f_\Vv$ and a $G$-invariant neighbourhood $O_\tau\subset \{\ue \in E\,|\, \|\ue\|<1\}$ of $0$ such that, formulated with the global section structure $\bigl(\s_h(\ux)=\tau(\ux,h*\ue)\bigr)_{h\in G}$ in (i), 
\begin{itemlist}
\item[(1)] $N_\Vv(\tau(\ux,\ue))\leq \|\ue\|$ is bounded by a $G$-invariant norm on $E$. In particular, 
$\ue\in O_\tau$ guarantees $N_\Vv(\s_h(\ux))<1$ for all $h\in G$ and $\ux\in X_\Vv$; 
\item[(2)]  $\tau|_{X_\Vv\less\Uu_\Vv\times E} \equiv 0$ and in fact (taking closures)  $\supp \s_h \subset \Uu_\Vv$ for all $h\in G$; 
\item[(3)]
$X_\Vv\times O_\tau \to W_\Vv,  (\ux,\ue) \mapsto f_\Vv(\ux) - \tau(\ux,\ue)$ is transverse and in general position on the boundary, in the sense that the reduced linearizations are onto at all solutions.  
\end{itemlist}

\MS\NI
If $\partial X_\Vv=\emptyset$, then $O^\pitchfork_\tau\subset O_\tau$ will be the set of regular values of the projection
\begin{align}\label{eq:univmod}
\Pi \,: \; \Ti S:= \bigl\{ (\ux,\ue) \in X_\Vv\times O_\tau \,\big|\, f_\Vv(\ux) = \tau(\ux,\ue) \bigr\}  \;\to\; O_\tau, \qquad (\uo,\ue)\mapsto \ue . 
\end{align}
If $\partial X_\Vv\ne\emptyset$, then we will need to intersect with the regular values of $\Pi$ restricted to boundary and corner strata. 
In either case, we need to make sure that $\Ti S$ is second countable, so we can apply Sard's Theorem in countably many charts. This will follow from the compactness controls in (1),(2), and the fact that the realization map $X_\Vv\to |\Xx_\Vv|$ is globally finite-to-one.
While $\Ti S \subset X_\Vv\times O_\tau$ is not closed under morphisms in $\Mor_{\Xx_\Vv}\less \Mor_{\Xx_\Vv^{\less G}}$, we can still consider its realization $|\Ti S|=\pi_{\Xx_\Vv\times E}(\Ti S)\subset |\Xx_\Vv\times E|= |\Xx_\Vv| \times E$. 
Since $\tau$ vanishes outside of $\Uu_\Vv\times E$, we have $\Ti S \less (\Uu_\Vv\times O_\tau) \subset f_\Vv^{-1}(0)\times O_\tau \subset \Uu_\Vv\times O_\tau$ and thus
\begin{align*}
\bigl| \Ti S \bigr| &= \bigl| \bigl\{ (\ux,\ue) \in \Uu_\Vv\times O_\tau \,\big|\, f_\Vv(\ux) = \tau(\ux,\ue) \bigr\} \bigr| \\
&\subset \bigl| \bigl\{ \ux \in \Uu_\Vv \,\big|\, f_\Vv(\ux)\in W_\Vv[1],   N_\Vv(f_\Vv(\ux))\leq 1  \bigr\} \bigr| \times O_\tau  \; =: \,  B_{f_\Vv}\times O_\tau
\end{align*}
where we used $f_\Vv(\ux)=\tau(\ux,\ue)\in W_\Vv[1]$ and $N_\Vv( f_\Vv(x))  = N_\Vv(\tau(\ux,\ue)) \leq \|\ue\| <1$.  
Now the closure ${\rm cl}(B_{f_\Vv})\subset |\Xx_\Vv|$ is compact by the compactness control property in Definition~\ref{def:control-compact}. Since $|\Xx_\Vv|$ is metric, this makes ${\rm cl}(B_{f_\Vv})$ a compact metric space and hence second countable. 
Its preimage $\pi_{\Xx_\Vv}^{-1}({\rm cl}(B_{f_\Vv})\subset X_\Vv$ inherits second countability by 
Remark~\ref{rmk:wonderful}.
Finally, we can conclude that $\Ti S \subset \pi_{\Xx_\Vv}^{-1}({\rm cl}(B_{f_\Vv})\times E$ is a subset of a product of second countable spaces and thus itself second countable. 

Moreover, $\Ti S$ inherits the structure of a smooth manifold with boundary and corners by the implicit function theorem \cite[3.1.22, 3.1.26]{TheBook}. And its boundary and corner strata, identified by the degeneracy index $d_{\Ti S}: \Ti S \to \N_0$ are given by the boundary strata of the M-polyfold $X_\Vv\times O_\tau$, that is $d_{\Ti S}(\ux,\ue)=d_{X_\Vv}(\ux)$ (see the transversality part of \cite[Thm.5.3.10]{TheBook}, which does not require a compact solution set).

Next, $\Ti S$ has finite dimension given by the Fredholm index of $f$ plus $\dim E$. So $\Pi:\Ti S \to O_\tau$ is smooth since it is the restriction of a sc-smooth map to a sub-M-polyfold of finite dimension, where sc-smoothness coincides with smoothness. We can now apply Sard's Theorem to any restriction of $\Pi$ to the domain $F\subset\Ti S$ of a local chart for $\Ti S$ or its boundary and corner strata. 
Here the corner strata -- connected components of $\{s\in \Ti S \,|\, d_{\Ti S}(s) = l \}$ for $l\geq 2$ -- might not strictly be submanifolds, but can still be covered with countably many local charts to $\R^{\dim\Ti S - l}$ (given by intersections of local faces as in \cite[Prop.2.5.16]{TheBook}).  
Sard's Theorem now assures that the regular values of each restriction $\Pi|_F$ are a comeagre set\footnote{The complement of a comeagre set is of Lebesgue measure $0$, or, equivalently, a countable union of nowhere dense subsets.} 
and we define $O^\pitchfork_\tau$ to be the intersection of all these comeagre sets. 
This proves that $O^\pitchfork_\tau\subset O_\tau$ is comeagre since countable intersections preserve the comeagre property. In addition, $O_\tau$ is a Baire space by the Baire category theorem for locally compact Hausdorff spaces, and this ensures that every comeagre subset is dense. 

Now $\ue\in E$ is a regular value of $\Pi : \Ti S\to O_\tau$ if and only if for all $\ux\in f_\Vv^{-1}(\ue)$ the projection $\ker \rD(f_\Vv-\tau) (\uo,\ue) \supset \rT_{(\ux,\ue)} F  \to E$ is onto, i.e.\ if for every $Y\in E$ there exists $X\in\rT_\ux X_\Vv$ so that $\rD(f_\Vv-\tau(\cdot,\ue))(\ux)X=\tau(\ux,Y)$, or equivalently
$\im \tau(\ux,\cdot)\subset \im\rD(f_\Vv-\tau(\cdot,\ue))(\ux)$. 
 This implies $$
 \im\rD(f_\Vv-\tau)(\ux,\ue)=\im\rD(f_\Vv-\tau(\cdot,\ue))(\ux)+\im\tau(\ux,\cdot)=\im\rD(f_\Vv-\tau(\cdot,\ue))(\ux),
 $$ and thus surjectivity of $\rD(f_\Vv-\tau)(\ux,\ue)$ (which holds at all points $(\ux,\ue)\in\Ti S$) implies surjectivity of $\im\rD(f_\Vv-\tau(\cdot,\ue))(\ux)$ for regular values $\ue$ of $\Pi$. 

To establish general position at solutions $\ux\in f_\Vv^{-1}(\ue)\cap \partial X_\Vv$ it is important to recall that, because $\Ti S$ is the solution of an equation in general position, the boundary and corner strata $F\subset\partial\Ti S$ are given by intersection with the boundary and corner strata of $X_\Vv\times E$.  Therefore, since $E$ has no boundary, each boundary or corner stratum $F$ has the form 
 $\Ti S \cap ( F_{X_\Vv}\times E)$ for some local boundary or corner stratum $F_{X_\Vv}\subset \partial X_\Vv$.  
Now our notion of regularity $\ue\in O^\pitchfork_\tau$ also guarantees that $\ue$ is a regular value of $\Pi |_{F} : F\to O_\tau$ for each boundary or corner stratum (or chart thereof) $F\subset\Ti S$. That means for all $\ux\in f_\Vv^{-1}(\ue)\cap F$ the projection $\ker \rD(f_\Vv-\tau) (\ux,\ue) \supset \rT_{(\ux,\ue)} F  \to E$ is onto, where $\rT_{(\ux,\ue)} F = \ker \rD(f_\Vv-\tau) (\ux,\ue) \cap (\rT_\ux F_{X_\Vv}\times E)$. 
As above, this is equivalent to  
$\im \tau(\ux,\cdot)\subset \im\rD(f_\Vv-\tau(\cdot,\ue))(\ux) |_{\rT_\ux F_{X_\Vv}}$, so that surjectivity of  $\rD(f_\Vv-\tau)(\ux,\ue) |_{\rT_\ux F_{X_\Vv}\times E}$ (which holds at all points $(\ux,\ue)\in\partial\Ti S$) implies surjectivity of $\im\rD(f_\Vv-\tau(\cdot,\ue))(\ux) |_{\rT_\ux F_{X_\Vv}}$ for regular values $\ue$ of $\Pi|_F$. 

Summarizing, we have defined $O^\pitchfork_\tau\subset O_\tau$ and shown that it is comeagre (and hence dense) via Sard's Theorem. Then we identified $O^\pitchfork_\tau$ with the set of $\ue\in O_\tau$ for which 
$f_\Vv-\tau(\cdot,\ue)$ is transverse and in general position. 
The latter perspective implies $G$-invariance of $O^\pitchfork_\tau$, since the $G$-equivariance of $f_\Vv$ and $\tau$ near a solution $(\ux,\ue)$ and the fact that the action of $G$ is \'etale  imply that
$$
\bigl(\rD ( f_\Vv-\tau (\cdot,g*\ue) ) (g*x)\bigr) \circ L_g = M_g \circ  \rD ( f_\Vv-\tau (\cdot,e) ) (x),
$$
where $L_g:\rT_x X_\Vv\to \rT_{g*\ux} X_\Vv$  is the linearized action  and the sc-isomorphism $M_g$ is given by
 $$M_g:= \mu_\Vv \bigl(\al(g*\ux,g) ,  \cdot \bigr) : (W_\Vv)_x \to  (W_\Vv)_{g*\ux}.
 $$ 
Here $L_g$ intertwines the tangent spaces of boundary and corner strata since the action of $G$ is given by local sc-diffeomorphisms.  Thus, given $g\in G$, the regularity of $\ue$ is equivalent to regularity of $g*\ue$. 
Thus for $\ue\in O^\pitchfork_\tau$ all elements in the global section structure $\bigl( \s_h= \tau(\cdot ,h*\ue) \bigr)_{h\in G}$ for $\La_{\Vv,\ue}$ in (i), when subtracted from $f_\Vv$, are transverse and in general position. This makes the perturbation $(f,\La_{\Vv,\ue})$ transverse and in general position in the sense of \cite[\S15.2]{TheBook}, so that \cite[Thm.15.2.22, 15.2.25]{TheBook} imply that  $\Theta_{\Vv,\ue}: =\La_{\Vv,\ue}\circ f_\Vv : \Xx_\Vv\to\Q^{\ge 0}$ is a tame branched ep$^+$-subgroupoid in the sense of \cite[Def.9.1.2]{TheBook}. 
In fact, $$
\supp\Theta_{\Vv,\ue}=\bigcup_{h\in G} (f_\Vv-\s_h)^{-1}(0)$$
 has a global branching structure given by the submanifolds $M_h=(f_\Vv-\s_h)^{-1}(0)\subset (X_\Vv)_\infty$ with weights $\sigma_h=\frac 1{\# G}$. 
Here the inclusion in the smooth part $(X_\Vv)_\infty$ follows from the regularization property \cite[Def.3.1.16]{TheBook} of the Fredholm section $f_\Vv$ and the sc$^+$ property of $\s_h$ by iterating the argument
$$
\ux\in (X_\Vv)_k \cap (f_\Vv-\s_h)^{-1}(0) \quad\Rightarrow\quad f_\Vv(\ux)=\s_h(\ux)\in (W_\Vv)_{k+1}  \quad\Rightarrow\quad \ux\in (X_\Vv)_{k+1}. 
$$
Finally, $\Theta_{\Vv,\ue}$ is compact in the sense of \cite[Def.9.1.6]{TheBook} since 
$| \supp\Theta_{\Vv,\ue} | =\bigcup_{h\in G} | M_h | \subset |\Xx_\Vv|$ is a closed subset of the precompact subset
$\bigl| \bigl\{ \ux \in \Uu_\Vv \,\big|\, f_\Vv(\ux)\in W_\Vv[1],   N_\Vv(f_\Vv(\ux))\leq 1  \bigr\} \bigr|$. 
Here we again used compactness control by $(\Uu_\Vv,N_\Vv)$ and properties (1),(2) to deduce
$\ux\in M_h \Rightarrow f_\Vv(\ux) = \s_h(\ux) \in W_\Vv[1]$ with $N(f_\Vv(\ux))=N(\s_h(\ux))<1$, and $M_h\subset \Uu_\Vv$ since $\ux\notin\Uu_\Vv \Rightarrow f_\Vv(\ux)\ne 0 , \s_h(\ux)=0  \Rightarrow \ux\notin M_h$.

These results transfer to the pushforward $\La_{e}$  by Theorem~\ref{thm:reduceFred}~(iv) as follows: We proved above that for each $\ue\in O^\pitchfork_\tau$ the sc$^+$-multisection $\La_{\Vv,\ue}$ is an $(N_\Vv,\Uu_\Vv)$-regular perturbation of $f_\Vv$. 
This is what gives $\Theta_{\Vv,\ue}$ the structure of a compact, tame branched ep$^+$-subgroupoid. 
Now Theorem~\ref{thm:reduceFred}~(iv) guarantees that $\La_{\ue}:= \Psi_*\La_{\Vv,\ue}:\Ww|_V\to\Q^{\geq 0}$ (as well as its trivial extension to $\Ww$) is an $(N,\Uu)$-regular perturbation of $f$. 
This in turn gives $\Theta_\ue$ the structure of a compact, tame branched ep$^+$-subgroupoid. 
Alternatively, we could observe that 
$$
\Theta_{\Vv,\ue} \;=\; \Lambda_{\Vv,\ue}\circ f_\Vv \;=\; \Psi^*\Lambda_\ue\circ \psi^*f \;=\; \Lambda_\ue\circ f \circ \psi \;=\;  \Theta_\ue\circ\psi \;:\;  \Xx_\Vv \to \Q^{\geq 0}
$$ 
is given by composition with the equivalence $\psi$ in the sense of \cite[Def.10.1.1]{TheBook} from Theorem~\ref{thm:reduce}~(ii). Then one can check in \cite[\S9.1]{TheBook} that the conditions for a compact, tame branched ep$^+$-subgroupoid are preserved by $\psi$ since it is a local sc-diffeomorphism on objects and induces a homeomorphism $|\psi| :|\Xx_\Vv|\to |V|\subset |\Xx|$. 

This completes the proof of (ii).  
With that, (iii) and (iv) are special cases of Theorem~\ref{thm:reduceFred}~(iv). 
\end{proof}

\begin{appendix}
\numberwithin{theorem}{section}
\numberwithin{equation}{section}

\section{A note on rational \v{C}ech homology}\label{app:Cech}

The rational \v{C}ech cohomology $\check{H}^*(X,\Q)$ of a paracompact Hausdorff space $X$ can be defined as the direct limit of the cohomology groups of a sequence of simplicial complexes given by the nerves of appropriate open covers of $X$.\footnote{
See \cite[p.257]{Hat} for a sketch of this construction.} As is shown for example in Eilenberg--Steenrod~\cite[Ch~IX]{EilS},
the resulting cohomology theory satisfies all the standard axioms together with a continuity property sometimes called tautness: If $X\subset Y$ is contained in an ambient paracompact Hausdorff space $Y$ and $U_{i+1}\subset U_i\subset Y$ is a nested sequence of open sets such that
$X = \bigcap_i U_i$, then $\check{H}^*(X,\Q) = \underset{\to}{\lim }\, \check{H}^*(U_i,\Q)$ is the direct limit. 

Correspondingly,  \cite[Ch~IX]{EilS} defines the \v{C}ech homology groups to be the inverse limit of the homology groups of the nerves.  However, these groups fail the exactness axiom in general, because inverse limits even of vector spaces are not well behaved unless these spaces are finite dimensional.  More precisely, if
$(i_j: W_{j+1}\to W_j)_{j\ge 1}$  is a sequence of linear maps between $\Q$-vector spaces, consider the map
$$ \textstyle
d: \prod_{j\ge 1} W_j \to \prod_{j\ge 1} W_j,\quad (w_j) \mapsto (w_j-i_j(w_{j+1})),
$$
and define
\begin{align}\label{eq:A2}
\underset{\leftarrow}{\lim }\, W_j: = \ker d,\qquad \underset{\leftarrow}{\lim }^1\, W_j = \coker d.
\end{align}
Then the resulting homology groups 
 fail exactness because $ \underset{\leftarrow}{\lim }^1\, W_j $ does not vanish for an arbitrary system 
 $(i_j: W_{j+1}\to W_j)_{j\ge 1}$ of $\Q$-vector spaces and linear maps. (For example, consider the case when $W_j =\Q \oplus \Q\oplus\cdots$  is the sum of countably many copies of $\Q$ and $i_j$ is the shift operator $i_j(r_1,r_2,\dots)= (0,r_1,r_2,\dots)$ for each $j$.)  
 This is not a problem if we restrict to (locally) compact spaces, since then we can assume that the nerves are finite, so that their homology groups are finite dimensional and  $\underset{\leftarrow}{\lim }\,^1$ vanishes.  In this case, there are other various definitions of rational \v{C}ech (or Alexander--Spanier) homology (see for example Massey~\cite{Ma} or Spanier~\cite[Ch~6.6]{Span}) which give the same result by the uniqueness theorem in Spanier~\cite[Ch~6,~Ex~D]{Span}.

Since polyfolds are not locally compact,  we cannot use this version of \v{C}ech homology in the proof of
Theorem~\ref{thm:polyVFC} since here we take an inverse limit  in a polyfold setting.  
 Most refinements of this theory (such as Steenrod homology) are concerned with dealing with problems arising from the use of more general coefficient groups and are defined in the locally compact setting.  The solution of this problem in the case of  $\Q$ coefficients, is apparently  \lq well-known' but not easy to chase down in the literature.  The idea is to avoid the problems with
$\underset{\leftarrow}{\lim }^1\,$ 
by 
\begin{quote}
{\it  defining the rational \v{C}ech homology  of a paracompact Hausdorff space $X$ to be the dual of its rational cohomology.}
\end{quote}
With this definition, it is still true that the homology of $X$ is the inverse limit of the homology of an appropriate sequence of nerves.
Indeed, it follows from the definitions that given any direct system $ (i_j:V_j\to V_{j+1})_{j\ge 1}$ of linear maps between $\Q$-vector spaces with corresponding  inverse dual  sequence
\begin{align}\label{eq:ij*}
i_j^*: V_{j+1}^*\to V_j^*\qquad V^*_j: = \Hom(V_j,\Q),
\end{align}
there is an isomorphism
$$
\underset{\leftarrow}{\lim }\, V_j^*: = \underset{\leftarrow}{\lim }\, \Hom(V_j,\Q)   \stackrel{\simeq}\to \Hom(\underset{\to}{\lim }\, V_j, \Q).
$$
The key point here is that, if the $V_j$ are vector spaces, then $\underset{\to}{\lim }\, V_j = (\oplus_j V_j)/K$ where  $\oplus _j V_j$ has elements that are finite sums
$$
\sum \la_i (j_i,v_i),\quad\mbox{  with }\;\;  v_i\in V_{j_i}, \ \la_i\in \Q,
$$
 and 
$K$ is the vector subspace generated by the elements $
(j,v) - (j+1, i_{j}(v)).
$
Thus the elements  in $\underset{\leftarrow}{\lim }\,  V_j^*$ are sequences $(\al_j\in V_j^*)$ such that $$
i_j^*(\al_{j+1}) = \al_{j},\quad\forall\ j, \quad\mbox{ or equivalently }\;\; \al_{j+1} \circ i_j = \al_j.
$$

Further, the argument in \cite[Ch~IX,Thm7.6]{EilS} shows that this homology theory does satisfy the exactness axiom because of the following fact.

\begin{lemma} \label{lem:app1}  Let $(i_j^*: V_{j+1}^*\to V_j^*)_{j\le 1}$ be the dual of  a direct system 
$ (i_j:V_j\to V_{j+1})_{j\ge 1}$ of  $\Q$-vector spaces and linear maps as in \eqref{eq:ij*}.  Then
$\underset{\leftarrow}{\lim }^1\,  V_j^* = 0$.\end{lemma}
\begin{proof}
Denote by  $i_{k,j}$ the composite  $i_k \circ\ldots\circ i_{j}: V_j\to V_{k+1}$ and
let $W_j: = \{ v\in V_j \ | \ \exists\ k\geq j :  i_{k,j}(v) = 0\} \subset V_j$ and let $\ov V_i: = V_i/W_i$.
Then $i_j$ restricts to a map $i_j:W_j\to W_{j+1}$, and evidently $\underset{\to}{\lim }\, W_i = 0$.  Therefore $\underset{\to}{\lim }\, V_i  \simeq \underset{\to}{\lim }\, \ov V_i$.

Since the maps $\ov V_i\to \ov V_{i+1}$ are injective, their duals $\ov V_{i+1}^*\to \ov V_{i}^*$ are surjective.  Hence  $\underset{\leftarrow}{\lim }^1\,  \ov V_j^* = 0$. Indeed, given $(\al_j)\in  \prod \ov V_j^*$, choose $\be_1=0$ and then inductively choose $\be_{j+1}\in \ov V_{j+1}^*$ so that 
$\al_j = \be_j-i_j^*(\be_{j+1})$.  
Then $d((\be_j)) = (\al_j)$ so that $\underset{\leftarrow}{\lim }^1\,  \ov V_j^* = 0$ by \eqref{eq:A2}.

The exact sequence
$
0 \to   \ov V_j^*  \to  V_j^* \to W_j^*\to 0
$
induces
$$
0 \to  \underset{\leftarrow}{\lim }\, \ov V_j^*  \to  \underset{\leftarrow}{\lim }\,  V_j^* \to \underset{\leftarrow}{\lim }\, W_j^*\to
\underset{\leftarrow}{\lim }^1\, \ov V_j^* =0\to  \underset{\leftarrow}{\lim }^1\,  V_j^* \to \underset{\leftarrow}{\lim }^1\, W_j^* \to 0.
$$
But $\underset{\leftarrow}{\lim }\, W_j^*=0$.  To see this, note that if $(\al_j)\in \underset{\leftarrow}{\lim }\, W_j^*$,
and $\exists \ w_j, \al_j(w_j)\ne 0$ then, given $k$ such that $i_{k,j}(w_j) = 0$, we have
$$
\al_j(w_j) =i_{k,j}^*( \al_{k}) (w_j) =\al_{k} \circ i_{k,j}(w) = 0,
$$
a contradiction.    Hence $\underset{\leftarrow}{\lim }\, V_j^* \simeq \underset{\leftarrow}{\lim }\,  \ov V_j^*$.

It remains to check that $ \underset{\leftarrow}{\lim }^1\, W_j^*=0$, which holds iff 
$$
{\textstyle 
d:\prod_j W_j^*\to \prod_j W_j^*, \quad (\be_j) }\mapsto (\be_j- i_j^*(\be_{j+1}))\quad\mbox{ is surjective.} 
$$
We can write $W_j = \oplus_{j<k} W_{j}^k$ 
where $$
w\in W_{j}^k\less \{0\}\  \Longrightarrow\ i_{\ell,j}(w)\ne 0, j<\ell< k,\ \mbox{ and } \  i_{k,j}(w)=0.
$$
Moreover we may choose the summands $W_{j}^k$ for $j = 1,2,3$ in turn,  so that  $i_j = \oplus\, i_j^k$, where  $i_j^k: W_j^k\to W_{j+1}^k$ 
and $W_{j}^k=0$ for $j\ge k$.
Note that  $i_{j}^k: W_j^k\to W_{j+1}^k$ is injective for $j+1<k$, and is the zero map for $j+1=k$. 
In other words,  the sequence $(i_j: W_j\to W_{j+1})_{j\ge 1}$ decomposes into a direct sum of sequences,
$(i_{j}^k:W_j^k\to W_{j+1}^k)_{1\le j< k}$,
 each of finite length.
 
Hence the dual sequence $(i_j^*: W_{j+1}^*\to W_j^*)$ is a direct product of sequences of finite length
$$
0 = (W_{k}^k)^*\to (W_{k-1}^k)^* \to \cdots \to (W_{1}^k)^*,\qquad k\ge 1.
$$
But for each of these sequences the operator $d^k$ is surjective: indeed given $(\al_j^k)_{1\le j< k}$ with $\al_j^k \in (W_j^k)^*$
we may choose  $\be_j^k\in (W_j^k)^*$ for $j=k-1, k-2,\dots$ by setting $\be_{k-1}^k = \al_{k-1}^k$, and inductively,
$$
\be_j^k = \al_j^k + (i_j^k)^*(\be_{j+1}^k).
$$
Then $d^k(\prod_j\be_j^k) = \prod_j \al^k_j$, so that $d = \prod d^k$ is surjective.
\end{proof}

The above discussion justifies the claims made about \v{C}ech homology in Remark~\ref{rmk:cech}.  However the proof of Lemma~\ref{lem:epfund} requires detailed knowledge about the version of this homology theory that holds for locally compact spaces and is the dual of cohomology with compact supports.  
 This theory, denoted $\check{H}_*^\infty$,  is carefully developed in \cite[Ch.4]{Ma}.  
 Here are some of its important properties  that are established in \cite[Ch4]{Ma}; many are similar to those of  locally finite singular homology.  We assume throughout that $Y$ is 
 a locally compact Hausdorff space.

\begin{itemlist}\item[(i)]
 If $Y$ is a connected orientable $n$-manifold, then 
 \begin{align}\label{eq:Afundc}  \check{H}^\infty_i(Y) = 0, i>n,\;\;\mbox {  and } \;\;\check{H}_n^\infty(Y) = \Q
 \end{align}
   with  generator $\mu_Y$ given by the fundamental class;
 further  
 $\check{H}^\infty_i(\R^n) = 0$ unless $i=n$;
\item[(ii)] 
if $U\subset Y$ is open, there is an induced restriction 
\begin{align}\label{eq:ArhoYU}
\rho_{Y,U}: \check{H}^\infty_i(Y)\to \check{H}^\infty_i(U)
\end{align}
that is compatible with composition: for $V\subset U\subset Y, \ \rho_{Y,V} = \rho_{U,V}\circ \rho_{Y,U}$.
\item[(iii)]   If  $f: A\to Y$ is continuous and proper, then there is an induced pushforward  $f_*: \check{H}^\infty_i(A)\to \check{H}^\infty_i(Y)$;  moreover, given a   proper inclusion $\io: A\to Y$,  there is a functorial long exact sequence
 \begin{align}\label{eq:AlesAC}
 \cdots\to \check{H}^\infty_i(A)\stackrel{\io_*}\to \check{H}^\infty_i(Y) \stackrel{\rho_{Y,Y\less A}}\longrightarrow 
 \check{H}^\infty_i(Y\less A) \stackrel{\p }\to \check{H}^\infty_{i-1}(A) \to \cdots
 \end{align} 
 \item[(iv)]    if $Y=U\cup V$ where $U,V$ are open, there is an exact Mayer--Vietoris sequence of the form: 
\begin{align}\label{eq:AMV}
 \cdots \to  \check{H}^\infty_{i+1}(U\cap V) \to  \check{H}^\infty_i(Y) \to  \check{H}^\infty_i(U)\oplus  \check{H}^\infty_i(V) \to  \check{H}^\infty_i(U\cap V) \to\cdots
\end{align}
 \end{itemlist}

\end{appendix}

\end{document}